\documentclass[11pt]{article}

\usepackage{amsmath}
\usepackage{amssymb}
\usepackage{amsfonts}
\usepackage{mathrsfs}
\usepackage{graphicx}
\usepackage{color,xcolor}
\usepackage{float}
\usepackage[listofformat=subsimple]{subfig}
\usepackage[normalem]{ulem}
\usepackage[linktocpage,colorlinks,linkcolor=blue,anchorcolor=blue, citecolor=blue,urlcolor=blue]{hyperref}
\usepackage{xurl}
\usepackage[inline]{enumitem}
\usepackage{multirow}
\usepackage{makecell}
\usepackage{breakcites}
\usepackage{bbm}
\usepackage{wrapfig}
\allowdisplaybreaks[4]

\usepackage{mleftright}
\usepackage[inline]{asymptote}
\usepackage{lmodern}
\usepackage{microtype}
\usepackage{tikz}
\usetikzlibrary{backgrounds}
\usetikzlibrary{arrows}
\usetikzlibrary{shapes,shapes.geometric,shapes.misc}

\tikzstyle{tikzfig}=[baseline=-0.25em,scale=0.5]

\pgfkeys{/tikz/tikzit fill/.initial=0}
\pgfkeys{/tikz/tikzit draw/.initial=0}
\pgfkeys{/tikz/tikzit shape/.initial=0}
\pgfkeys{/tikz/tikzit category/.initial=0}

\pgfdeclarelayer{edgelayer}
\pgfdeclarelayer{nodelayer}
\pgfsetlayers{background,edgelayer,nodelayer,main}

\tikzstyle{none}=[inner sep=0mm]

\newcommand{\tikzfig}[1]{%
{\tikzstyle{every picture}=[tikzfig]
\IfFileExists{#1.tikz}
  {\input{#1.tikz}}
  {%
    \IfFileExists{./figures/#1.tikz}
      {\input{./figures/#1.tikz}}
      {\tikz[baseline=-0.5em]{\node[draw=red,font=\color{red},fill=red!10!white] {\textit{#1}};}}%
  }}%
}

\tikzstyle{every loop}=[]

\newtheorem{theorem}{Theorem}[section]

\newtheorem{lemma}[theorem]{Lemma}
\newtheorem{example}[theorem]{Example}
\newtheorem{corollary}[theorem]{Corollary}
\newtheorem{conjecture}[theorem]{Conjecture}
\newtheorem{proposition}[theorem]{Proposition}

\newtheorem{assumption}[theorem]{Assumption}

\def\pn{\par\smallskip\noindent}

\newenvironment{myproof}[1] {\pn {\em Proof of {#1}.}}{\hfill $\Box$ \vskip 0.2truein}
\DeclareMathAlphabet\mathbfcal{OMS}{cmsy}{b}{n}

\newcommand{\D}{\mathbb{D}}
\newcommand{\E}{\mathbb{E}}
\newcommand{\I}{\mathbb{I}}
\newcommand{\J}{\mathbb{J}}
\newcommand{\N}{\mathbb{N}}

\newcommand{\R}{\mathbb{R}}
\newcommand{\LL}{\mathbb{L}}

\newcommand{\SI}{\mathbb{S}}
\newcommand{\HI}{\mathbb{H}}
\newcommand{\TT}{\mathbb{T}}
\newcommand{\U}{\mathbb{U}}
\newcommand{\V}{\mathbb{V}}
\newcommand{\W}{\mathbb{W}}
\newcommand{\X}{\mathbb{X}}

\newcommand{\Z}{\mathbb{Z}}

\newcommand{\ba}{\mathbf{a}}
\newcommand{\bb}{\mathbf{b}}
\newcommand{\bc}{\mathbf{c}}
\newcommand{\be}{\mathbf{e}}
\newcommand{\bi}{\mathbf{i}}
\newcommand{\bj}{\mathbf{j}}
\newcommand{\bx}{\mathbf{x}}
\newcommand{\by}{\mathbf{y}}
\newcommand{\bz}{\mathbf{z}}
\newcommand{\bu}{\mathbf{u}}
\newcommand{\bv}{\mathbf{v}}
\newcommand{\bw}{\mathbf{w}}

\newcommand{\BD}{\mathbf{D}}
\newcommand{\BP}{\mathbf{P}}
\newcommand{\BQ}{\mathbf{Q}}
\newcommand{\BT}{\mathbf{T}}
\newcommand{\BU}{\mathbf{U}}
\newcommand{\BV}{\mathbf{V}}
\newcommand{\BS}{\mathbf{S}}
\newcommand{\BO}{\mathbf{O}}
\newcommand{\BX}{\mathbf{X}}

\newcommand{\CD}{\mathbf{D}}
\newcommand{\CE}{\mathbf{E}}
\newcommand{\CF}{\mathbf{F}}
\newcommand{\CX}{\mathbf{X}}
\newcommand{\CY}{\mathbf{Y}}
\newcommand{\CZ}{\mathbf{Z}}
\newcommand{\CU}{\mathbf{U}}

\newcommand{\CT}{\mathbf{T}}
\newcommand{\CS}{\mathbf{S}}
\newcommand{\CO}{\mathbf{O}}
\newcommand{\CL}{\mathbf{L}}
\newcommand{\CG}{\mathbf{G}}
\newcommand{\CH}{\mathbf{H}}

\newcommand{\HH}{\textnormal{H}}
\newcommand{\T}{\textnormal{T}}

\newcommand{\sign}{\textnormal{sign}}

\newcommand{\conv}{\operatorname{conv}}
\newcommand{\spn}{\operatorname{sp}}

\newcommand{\proj}{\operatorname{p}}
\newcommand{\Prob}{\operatorname{Pr}}
\newcommand{\Exp}{\operatorname{Ex}}

\usepackage{booktabs}
\usepackage{graphicx}
\usepackage{stmaryrd}
\usepackage{mathtools}
\def\multiset#1#2{\ensuremath{\mleft(\kern-.3em\mleft(\genfrac{}{}{0pt}{}{#1}{#2}\mright)\kern-.3em\mright)}}

\usepackage{multicol}

\usepackage{algorithm}
\usepackage{algorithmic}

\makeatletter
\newcommand{\bigcomp}{%
  \DOTSB
  \mathop{\vphantom{\sum}\mathpalette\bigcomp@\relax}%
  \slimits@
}
\newcommand{\bigcomp@}[2]{%
  \begingroup\m@th
  \sbox\z@{$#1\sum$}%
  \setlength{\unitlength}{0.9\dimexpr\ht\z@+\dp\z@}%
  \vcenter{\hbox{%
    \begin{picture}(1,1)
    \bigcomp@linethickness{#1}
    \put(0.5,0.5){\circle{1}}
    \end{picture}%
  }}%
  \endgroup
}
\newcommand{\bigcomp@linethickness}[1]{%
  \linethickness{%
      \ifx#1\displaystyle 2\fontdimen8\textfont\else
      \ifx#1\textstyle 1.65\fontdimen8\textfont\else
      \ifx#1\scriptstyle 1.65\fontdimen8\scriptfont\else
      1.65\fontdimen8\scriptscriptfont\fi\fi\fi 3
  }%
}

\makeatother

\let\originalmiddle=\middle
\def\middle#1{\mathrel{}\originalmiddle#1\mathrel{}}

\usepackage{tikz}
\usetikzlibrary{external}
\usepackage{tikz-3dplot}
\usetikzlibrary{calc}
\usepackage{pgfplots}
\usepackage{xxcolor}
\pgfplotsset{compat=1.16}
\usepgfplotslibrary{fillbetween}
\pgfdeclareradialshading[tikz@ball]{ball}{\pgfqpoint{0bp}{0bp}}{%
 color(0bp)=(tikz@ball!0!white);
 color(7bp)=(tikz@ball!0!white);
 color(15bp)=(tikz@ball!70!black);
 color(20bp)=(black!70);
 color(30bp)=(black!70)}
\makeatother

\tikzset{viewport/.style 2 args={
    x={({cos(-#1)*1cm},{sin(-#1)*sin(#2)*1cm})},
    y={({-sin(-#1)*1cm},{cos(-#1)*sin(#2)*1cm})},
    z={(0,{cos(#2)*1cm})}
}}

\pgfplotsset{only foreground/.style={
    restrict expr to domain={rawx*\CameraX + rawy*\CameraY + rawz*\CameraZ}{-0.05:100},
}}
\pgfplotsset{only background/.style={
    restrict expr to domain={rawx*\CameraX + rawy*\CameraY + rawz*\CameraZ}{-100:0.05}
}}

\def\addFGBGplot[#1]#2;{
    \addplot3[#1,only background, opacity=0.25] #2;
    \addplot3[#1,only foreground] #2;
}

\usetikzlibrary{3d}
\usetikzlibrary{calc}
\usetikzlibrary{babel} % for issues with some babel packages

\pgfmathsetmacro\xx{1/sqrt(2)}
\pgfmathsetmacro\xy{1/sqrt(6)}
\pgfmathsetmacro\zy{sqrt(2/3)}

\usetikzlibrary{3d}
\usetikzlibrary{calc, shadings} 
\usetikzlibrary{positioning,arrows.meta}
\usetikzlibrary{trees}
\usepgfplotslibrary{fillbetween}
\makeatletter
\def\tikz@lib@cuboid@get#1{\pgfkeysvalueof{/tikz/cuboid/#1}}

\def\tikz@lib@cuboid@setup{%
   \pgfmathsetlengthmacro{\vxx}%
      {\tikz@lib@cuboid@get{xscale}*cos(\tikz@lib@cuboid@get{xangle})*1cm}
   \pgfmathsetlengthmacro{\vxy}%
      {\tikz@lib@cuboid@get{xscale}*sin(\tikz@lib@cuboid@get{xangle})*1cm}
   \pgfmathsetlengthmacro{\vyx}%
      {\tikz@lib@cuboid@get{yscale}*cos(\tikz@lib@cuboid@get{yangle})*1cm}
   \pgfmathsetlengthmacro{\vyy}%
      {\tikz@lib@cuboid@get{yscale}*sin(\tikz@lib@cuboid@get{yangle})*1cm}
   \pgfmathsetlengthmacro{\vzx}%
      {\tikz@lib@cuboid@get{zscale}*cos(\tikz@lib@cuboid@get{zangle})*1cm}
   \pgfmathsetlengthmacro{\vzy}%
      {\tikz@lib@cuboid@get{zscale}*sin(\tikz@lib@cuboid@get{zangle})*1cm}
}

\def\tikz@lib@cuboid@draw#1--#2--#3\pgf@stop{%
    \begin{scope}[join=bevel,x={(\vxx,\vxy)},y={(\vyx,\vyy)},z={(\vzx,\vzy)}]
       \begin{scope}[canvas is yz plane at x=#1]
          \draw[cuboid/all faces,cuboid/edges,cuboid/right face] 
                (0,0) -- ++(#2,0) -- ++(0,-#3) -- ++(-#2,0) -- cycle;
          \draw[cuboid/all grids,cuboid/right grid] (0,0) grid (#2,-#3);
       \end{scope}
       \begin{scope}[canvas is xy plane at z=0]
          \draw[cuboid/all faces,cuboid/edges,cuboid/front face] 
                (0,0) -- ++(#1,0) --  ++(0,#2) -- ++(-#1,0) -- cycle;
          \draw[cuboid/all grids,cuboid/front grid] (0,0) grid (#1,#2);
       \end{scope}
       \begin{scope}[canvas is xz plane at y=#2]
          \draw[cuboid/all faces,cuboid/edges,cuboid/top face] 
                (0,0) -- ++(#1,0) --  ++(0,-#3) -- ++(-#1,0) -- cycle;
          \draw[cuboid/all grids,cuboid/top grid] (0,0) grid (#1,-#3);
       \end{scope}
       \draw[cuboid/hidden edges] (0,#2,-#3) -- (0,0,-#3) -- (0,0,0) 
                (0,0,-#3) -- ++(#1,0,0);
       \begin{scope}[canvas is yz plane at x=#1]
          \draw[cuboid/all faces,cuboid/right face,cuboid/edges,fill opacity=0] 
                (0,0) -- ++(#2,0) -- ++(0,-#3) -- ++(-#2,0) -- cycle;
       \end{scope}
       \begin{scope}[canvas is xy plane at z=0]
          \draw[cuboid/all faces,cuboid/front face,cuboid/edges,fill opacity=0] 
                (0,0) -- ++(#1,0) --  ++(0,#2) -- ++(-#1,0) -- cycle;
       \end{scope}
       \begin{scope}[canvas is xz plane at y=#2]
          \draw[cuboid/all faces,cuboid/top face,cuboid/edges,fill opacity=0] 
                (0,0) -- ++(#1,0) --  ++(0,-#3) -- ++(-#1,0) -- cycle;
       \end{scope}
       \path (0,#2,0) coordinate (-left top front)
                      coordinate (-left front top)
                      coordinate (-top left front)
                      coordinate (-top front left)
                      coordinate (-front top left)
                      coordinate (-front left top);
       \path (0,#2,-#3) coordinate (-left top rear)
                        coordinate (-left rear top)
                        coordinate (-top left rear)
                        coordinate (-top rear left)
                        coordinate (-rear top left)
                        coordinate (-rear left top);
       \path (0,0,-#3) coordinate (-left bottom rear)
                       coordinate (-left rear bottom)
                       coordinate (-bottom left rear)
                       coordinate (-bottom rear left)
                       coordinate (-rear bottom left)
                       coordinate (-rear left bottom);
       \path (0,0,0) coordinate (-left bottom front)
                     coordinate (-left front bottom)
                     coordinate (-bottom left front)
                     coordinate (-bottom front left)
                     coordinate (-front bottom left)
                     coordinate (-front left bottom);
       \path (#1,#2,0) coordinate (-right top front)
                       coordinate (-right front top)
                       coordinate (-top right front)
                       coordinate (-top front right)
                       coordinate (-front top right)
                       coordinate (-front right top);
       \path (#1,#2,-#3) coordinate (-right top rear)
                         coordinate (-right rear top)
                         coordinate (-top right rear)
                         coordinate (-top rear right)
                         coordinate (-rear top right)
                         coordinate (-rear right top);
       \path (#1,0,-#3) coordinate (-right bottom rear)
                        coordinate (-right rear bottom)
                        coordinate (-bottom right rear)
                        coordinate (-bottom rear right)
                        coordinate (-rear bottom right)
                        coordinate (-rear right bottom);
       \path (#1,0,0) coordinate (-right bottom front)
                      coordinate (-right front bottom)
                      coordinate (-bottom right front)
                      coordinate (-bottom front right)
                      coordinate (-front bottom right)
                      coordinate (-front right bottom);
       \coordinate (-left center) at (0,.5*#2,-.5*#3);
       \coordinate (-right center) at (#1,.5*#2,-.5*#3);
       \coordinate (-top center) at (.5*#1,#2,-.5*#3);
       \coordinate (-bottom center) at (.5*#1,0,-.5*#3);
       \coordinate (-front center) at (.5*#1,.5*#2,0);
       \coordinate (-rear center) at (.5*#1,.5*#2,-#3);
       \coordinate (-center) at (.5*#1,.5*#2,-.5*#3);
       \path (0,#2,-.5*#3) coordinate (-left top center) 
                           coordinate (-top left center);
       \path (.5*#1,#2,-#3) coordinate (-top rear center)
                            coordinate (-rear top center);
       \path (#1,#2,-.5*#3) coordinate (-right top center)
                            coordinate (-top right center);
       \path (.5*#1,#2,0) coordinate (-top front center)
                          coordinate (-front top center);
       \path (0,0,-.5*#3) coordinate (-left bottom center) 
                           coordinate (-bottom left center);
       \path (.5*#1,0,-#3) coordinate (-bottom rear center)
                            coordinate (-rear bottom center);
       \path (#1,0,-.5*#3) coordinate (-right bottom center)
                            coordinate (-bottom right center);
       \path (.5*#1,0,0) coordinate (-bottom front center)
                          coordinate (-front bottom center);
       \path (0,.5*#2,0) coordinate (-left front center) 
                           coordinate (-front left center);
       \path (0,.5*#2,-#3) coordinate (-left rear center)
                            coordinate (-rear left center);
       \path (#1,.5*#2,0) coordinate (-right front center)
                            coordinate (-front right center);
       \path (#1,.5*#2,-#3) coordinate (-right rear center)
                          coordinate (-rear right center);
    \end{scope}
}

\tikzset{
  pics/cuboid/.style = {
    setup code = \tikz@lib@cuboid@setup,
    background code = \tikz@lib@cuboid@draw#1\pgf@stop
  },
  pics/cuboid/.default={1--1--1},
  cuboid/.is family,
  cuboid,
  all faces/.style={fill=white},
  all grids/.style={draw=none},
  front face/.style={},
  front grid/.style={},
  right face/.style={},
  right grid/.style={},
  top face/.style={},
  top grid/.style={},
  edges/.style={},
  hidden edges/.style={draw=none},
  xangle/.initial=0,
  yangle/.initial=90,
  zangle/.initial=210,
  xscale/.initial=1,
  yscale/.initial=1,
  zscale/.initial=0.5
}

\newcommand{\tikzcuboidreset}{
\tikzset{cuboid,
  all faces/.style={fill=white},
  all grids/.style={draw=none},
  front face/.style={},
  front grid/.style={},
  right face/.style={},
  right grid/.style={},
  top face/.style={},
  top grid/.style={},
  edges/.style={},
  hidden edges/.style={draw=none},
  xangle=0,
  yangle=90,
  zangle=210,
  xscale=1,
  yscale=1,
  zscale=0.5
}
}

\newcommand{\tikzcuboidset}{\@ifstar\tikzcuboidset@star\tikzcuboidset@nostar} 
\newcommand{\tikzcuboidset@nostar}[1]{\tikzcuboidreset\tikzset{cuboid,#1}}
\newcommand{\tikzcuboidset@star}[1]{\tikzset{cuboid,#1}}
\makeatother

\usetikzlibrary{angles, quotes}

\makeatletter
\newif\ifnobrackets
\renewcommand\@cite[2]{\ifnobrackets\else[\fi{#1\if@tempswa , #2\fi}\ifnobrackets\else]\fi\nobracketsfalse}

\makeatother

\begin{document}

\title{On decomposability and subdifferential of the tensor nuclear norm}

\author{
Jiewen GUAN
\thanks{Department of Systems Engineering and Engineering Management, The Chinese University of Hong Kong, Shatin, Hong Kong. Email: seemjwguan@gmail.com}
    \and
Bo JIANG
\thanks{School of Information Management and Engineering, Shanghai University of Finance and Economics, Shanghai 200433, China. Email: isyebojiang@gmail.com}
    \and
Zhening LI
\thanks{School of Mathematics and Physics, University of Portsmouth, Portsmouth PO1 3HF, United Kingdom. Email: zheningli@gmail.com}
}

\date{\today}

\maketitle

\begin{abstract}

We study the decomposability and the subdifferential of the tensor nuclear norm. Both concepts are well understood and widely applied in matrices but remain unclear for higher-order tensors. We show that the tensor nuclear norm admits a full decomposability over specific subspaces and determine the largest possible subspaces that allow the full decomposability. We derive novel inclusions of the subdifferential of the tensor nuclear norm and study its subgradients in a variety of subspaces of interest. All the results hold for tensors of an arbitrary order. As an immediate application, we establish the statistical performance of the tensor robust principal component analysis, the first such result for tensors of an arbitrary order.

\vspace{0.25cm}
    

\vspace{0.25cm}
    
    
\end{abstract}

\section{Introduction}

The matrix nuclear norm, as the convex envelope of the matrix rank, has found a wide range of applications, particularly in mathematical optimization to search for low-rank matrices.
Its essential properties underlying the applications are the decomposability and the explicit characterization of the subdifferential. Specifically, the decomposability states that
\begin{equation}\label{eq:matrixdecomp}
\|\BT+\BS\|_*=\|\BT\|_*+\|\BS\|_*\text{ if $\BT^{\T}\BS=\BO$ and $\BT\BS^{\T}=\BO$}
\end{equation}
and the explicit characterization of the subdifferential states that
\begin{equation}\label{eq:matrixpartial}
\partial\|\BT\|_* = \bigl\{\BU\BV^{\T}+\BX:\BX^{\T}\BU=\BO,\, \BX\BV=\BO,\, \|\BX\|_\sigma\le 1\bigr\},
\end{equation}
where $\BT=\BU\BD\BV^{\T}$ is a compact singular value decomposition (SVD) and $\|\bullet\|_*$ and $\|\bullet\|_\sigma$ denote the nuclear norm and the spectral norm, respectively.
%and $\TT^{[2]}(\BT)$ is some subspace orthogonal to the subspace spanned by $\BT$. 
They play a critical role in many matrix optimization problems and statistical learning, such as 
compressed sensing~\cite[Corollary~5]{negahban2011estimation},
matrix recovery and completion~\cite{candes2009exact,candes2010power}, 
multivariate regression~\cite[Corollary~3]{negahban2011estimation}, 
% autoregressive processes~\cite[Corollary~4]{negahban2011estimation},
phase retrieval~\cite[Chapter~10.4]{wainwright2019high}, 
recommender systems~\cite[Example~10.2]{wainwright2019high}, 
and robust principal component analysis (PCA)~\cite{candes2011robust};
see a recent book~\cite{wainwright2019high} and the references therein for more examples.
% see, e.g.,~\cite[Lemma~2.3]{recht2010guaranteed} and~\cite[Equation~2.9]{recht2010guaranteed}

With the rapid developments of data sciences in various fields, the recent decade has witnessed a significant increase in research activities in tensors, the higher-order generalization of matrices. In the same spirit to matrices, the tensor nuclear norm is widely recognized as a convex surrogate for the intrinsic complexity of a tensor~\cite{derksen2016nuclear}
% matrix nuclear norm.
% is the convex envelope of the rank. %, despite there is still no generalization of this fact to tensors, to the best of our knowledge.  
% as a result, 
% the tensor nuclear norm 
% it has been applied to 
and has triggered many applications in statistical learning such as tensor completion~\cite{yuan2016tensor}, tensor regression~\cite{raskutti2019convex}, and tensor robust PCA~\cite{driggs2019tensor}.
% , and has yielded considerable statistical performance. 
However, the nice properties of the matrix nuclear norm,~\eqref{eq:matrixdecomp} and~\eqref{eq:matrixpartial}, do not carry over directly to the tensor nuclear norm. This has made many low-rank tensor optimization problems difficult and statistical implications on low-rank tensors unclear. %, including tensor completion and robust principal PCA as typical examples. 
In fact, the only known result on the decomposability of the tensor nuclear norm is the so-called weak decomposability~\cite[Lemma~1]{raskutti2019convex} that only works for third-order tensors. Although it has found applications in tensor regression~\cite{raskutti2019convex}, it does not completely address the entire issue. For the subdifferential of the tensor nuclear norm, only two limited inclusions,~\cite[Lemma~1]{yuan2016tensor} and~\cite[Theorem~1]{yuan2017incoherent}, have been proposed. While these inclusions have been applied to analyze the statistical performance of tensor completion~\cite{yuan2016tensor,yuan2017incoherent}, they are still unsatisfactory from the perspective of tensor analysis since many obvious subgradients have been excluded.

The main reason behind the full decomposability and the subdifferential characterization of the matrix nuclear norm is that every matrix admits an SVD. As a contrast, only a very special class of tensors, 
\textcolor{black}{such as orthogonally decomposable tensors~\cite{zhang2001rank,sturmfels2016tensors,robeva2017singular,hashemi2018spectral}}, 
admits 
% SVDs
\textcolor{black}{SVDs~\cite[Definition~4.1]{chen2009tensor}}. While this 
unfortunate fact is primarily responsible for
% unpleasant fact has mainly caused
the gaps of the properties, it indeed makes the problems subtler and more interesting. This paper aims to offer a more in-depth understanding of the decomposability and the subdifferential of the tensor nuclear norm. 

As the first main result, we find that the tensor nuclear norm can in fact be fully decomposable if we carefully choose the tensor subspaces in which $\CT$ and $\CS$ reside as those in~\eqref{eq:matrixdecomp}. It also directly points out why only a weak decomposability is possible for the subspace considered in~\cite[Lemma~1]{raskutti2019convex}. Our full decomposability applies to tensors of an arbitrary order, leading to a very natural generalization of the matrix case in~\eqref{eq:matrixdecomp}. Moreover, we optimize the subspaces that allow the full decomposability and determine the largest possible such subspace pairs.
% for $\CT$ and $\CS$ to reside 
% while keeping the full decomposability. 
The study has also resulted \textcolor{black}{in} a dual byproduct,
the decomposability of the tensor spectral norm, running in parallel to the tensor nuclear norm.

The full decomposability of the tensor nuclear norm offers a tool to study its subdifferential. As the next main result, we propose novel subdifferential inclusions of the tensor nuclear norm. They strictly enlarge the inclusion proposed in~\cite[Theorem~1]{yuan2017incoherent} that is the only known subdifferential inclusion for tensors of an arbitrary order. In particular, we show that a full spectral ball, $\bigl\{\BX:\|\BX\bigr\|_\sigma\le 1\}$, can be imposed in~\eqref{eq:matrixpartial} rather than $\|\BX\|_\sigma\le \frac{2}{d(d-1)}$ imposed in~\cite[Theorem~1]{yuan2017incoherent}, where $d$ is the order of the tensor. Our study indicates that there is no universal way to explicitly characterize the subdifferential of the tensor nuclear norm as that of the matrix nuclear norm in~\eqref{eq:matrixpartial}, supported by various approximations of the subdifferential and several interesting examples. Moreover, we investigate subgradients of the tensor nuclear norm and derive various bounds for the inclusion and exclusion of the subdifferential in all relevant subspaces. 

We believe that these developments can be important to applications, no matter in theory or in practice. As a precursor, we propose an immediate application to analyze the statistical performance of the nuclear-norm-based tensor robust PCA that aims to recover a planted low-rank ground-truth tensor superposed by a sparse corruption. In the matrix case, the robust PCA is already a remarkable instance of modern compressed sensing and has found many interesting applications such as video surveillance, face recognition, and community detection; see, e.g.,~\cite{candes2011robust,jiang2018low,wright2022high} and the references therein. A key component underlying its outstanding performance is the nuclear norm minimization that tends to reveal low-rank solutions.
% and data clustering; 
To the best of our knowledge, the statistical performance of the tensor robust PCA
% (based on the nuclear norm) 
has only been established for third-order tensors~\cite{driggs2019tensor}. Our result applies to tensors of an arbitrary order and exactly recovers the matrix case~\cite[Theorem~1.1]{candes2011robust} when $d=2$. Moreover, it enjoys some looser conditions for the exact recovery than those required in~\cite{driggs2019tensor} when $d=3$. One interesting insight from our analysis suggests that the conditions for the exact recovery of the tensor robust PCA of every order are highly likely to be identical.

The rest of this paper is organized as follows. We start with uniform notations, main concepts, and important properties of tensor operations and norms in Sections~\ref{sec:prepare}. As an essential part to the main results, we introduce tensor subspaces, projections, and tensor norms in subspaces in Section~\ref{sec:subspaces}, presented in a self-contained way for readers with minimal background knowledge. The decomposability of the tensor nuclear norm is discussed in Section~\ref{sec:decomp} along with the decomposability of the tensor spectral norm. The subdifferential of the tensor nuclear norm is discussed in Section~\ref{sec:subdiff}. As an application, the statistical performance of the tensor robust PCA is analyzed in Section~\ref{sec:TRPCA}. Finally, we conclude this paper with some remarks and future research directions in Section~\ref{sec:conclusion}.

\section{Preparations}\label{sec:prepare}

To support a better understanding of the theoretical results in later sections, we compile this section on uniform notations, necessary concepts and important properties of tensor operations, and tensor spectral and nuclear norms.

\subsection{Basic notations}

Throughout this paper, we uniformly adopt lowercase letters (e.g., $x$), boldface lowercase letters (e.g., $\bx=(x_i)$), boldface capital letters 
%(e.g., $\BX=\mleft(x_{i j}\mright)$), and calligraphic letters 
(e.g., $\CX=(x_{i_1 i_2 \dots i_d})$) 
to denote scalars, vectors, and $d$th order tensors (including matrices) with $d\ge2$, respectively. Denote $\R^{n_1\times n_2\times \dots\times n_d}$ to be the space of $d$th order real tensors of dimension $n_1\times n_2\times\dots \times n_d$. The same notation applies to a vector space and a matrix space when $d = 1$ and $d = 2$, respectively. Without loss of generality, we assume that $2\le n_1\le n_2\le \dots \le n_d$
and that the order of the tensor space, $d$, is a fixed parameter. All Greek letters are used to denote constants. In particular, $\theta>0$ and $\kappa>0$ are some sufficiently small and sufficiently large constants that are not made explicit in Section~\ref{sec:TRPCA}, respectively.

The Frobenius inner product of two tensors $\CT,\CS\in\R^{n_1 \times n_2 \times \dots \times n_d}$ is defined as 
$$
\langle\CT, \CS\rangle:=\sum_{i_1=1}^{n_1} \sum_{i_2=1}^{n_2} \dots \sum_{i_d=1}^{n_d} t_{i_1 i_2 \dots i_d} s_{i_1 i_2 \dots i_d}.
$$
Its induced Frobenius norm is naturally defined as $\|\CT\|_2:=\sqrt{\langle\CT, \CT\rangle}$. When $d=1$, the Frobenius norm reduces to the Euclidean norm of a vector. In a similar vein, we may define the $\ell_p$-norm of a tensor (also known as the H\"{o}lder $p$-norm) for $1\le p\le \infty$ by viewing the tensor as a vector, i.e.,
$$
\|\CT\|_p:=\mleft( \sum_{i_1=1}^{n_1}\sum_{i_2=1}^{n_2} \dots\sum_{i_d=1}^{n_d} |t_{i_1i_2\dots i_d}|^p\mright)^{\frac{1}{p}}.
$$
Three specific H\"{o}lder $p$-norms are used in this paper, namely $\|\CT\|_1=\sum_{i_1=1}^{n_1}\sum_{i_2=1}^{n_2} \dots \sum_{i_d=1}^{n_d}|t_{i_1 i_2 \dots i_d}|$, $\|\CT\|_2$ as the Frobenius norm, and $\|\CT\|_\infty$ as the largest entry of $\CT$ in absolute value.

All blackboard bold capital letters denote sets, such as $\R^n$, the Euclidean unit sphere $\SI^n:=\{\bx\in\R^n:\|\bx\|_2=1\}$ embedded in $\R^n$, and the set of positive integers $\N$. We denote $\{\be_1,\be_2,\dots,\be_n\}$ to be the standard basis of $\R^n$. % (in order), by , and by $\N:=\{1,2,\dots\}$ the set of positive integers. %and the Euclidean unit ball $\mathbb{B}^n:=\{\bx\in\R^n:\sum_{i=1}^n x_i^2\le 1\}$.
For any $n\in\N$, we let $[n]:=\{1,2,\dots,n\}$. We denote $\I^d:=\{\bi\in\N^d:i_k\in[n_k]\,\forall\,k\in[d]\}$ to be the set of entry indices for the tensor space $\R^{n_1\times n_2\times \dots\times n_d}$. As a result, the set $\{\be_{i_1}\otimes\be_{i_2}\otimes\dots\otimes\be_{i_d}:\bi\in\I^d\}$ becomes the standard basis of $\R^{n_1\times n_2\times \dots\times n_d}$, where $\otimes$ denotes the vector outer product.

Finally, two most frequently used operations of sets are $\proj$ for the orthogonal projection and $\spn$ for the span. Some other notations are self-explanatory, including $\dim$ for the dimension of a space, $\conv$ for the convex hull of a set, $\Prob$ for the probability measure, and $\Exp$ for the expectation.
% For a sub-Gaussian random variable $x$, we define its sub-Gaussian norm to be $\|x\|_{\psi_2}:=\inf \mleft\{t>0: \Exp (\exp (x^2 / t^2)) \le 2\mright\}$, and we refer the reader to~\cite[Section~2]{vershynin2018high} for more details on this norm. 

\subsection{Tensor operations}

A tensor $\CT\in\R^{n_1 \times n_2 \times \dots \times n_d}$ has $d$ modes. Fixing entry indices of $d-1$ modes except mode $k$ results \textcolor{black}{in} a vector in $\R^{n_k}$, called a mode-$k$ fiber. For matrices, mode-$1$ fibers are columns and mode-$2$ fibers are rows. The mode-$k$ matricization of $\CT$, denoted by $\BT_{(k)}\in\R^{n_k\times\prod_{i\ne k}n_i}$, is to arrange mode-$k$ fibers to be the columns of the resulting matrix. The mode-$k$ product between $\CT$ and a matrix $\BX\in\R^{n\times n_k}$, denoted by $\CT\times_k \BX\in\R^{n_1\times \dots \times n_{k-1}\times n\times n_{k+1}\times \dots\times n_d}$, changes every mode-$k$ fiber of $\CT$, say $\bv\in\R^{n_k}$, to $\BX\bv\in\R^n$, a mode-$k$ fiber of $\CT\times_k \BX$; 
% in another word
\textcolor{black}{in other words},
$$
\CS=\CT\times_k \BX \iff \BS_{(k)}=\BX \BT_{(k)}.
$$
The mode-$k$ contraction of $\CT$ by a vector $\bv\in\R^{n_k}$ is a tensor $\CT\times_k \bv\in\R^{n_1\times \dots \times n_{k-1}\times n_{k+1}\times \dots\times n_d}$ of order $d-1$ under the same mode-$k$ product rule by treating $\bv$ as a $1\times n_k$ matrix.

A rank-one tensor (also called a simple tensor) is a tensor that can be written as outer products of vectors, e.g., $\bx_1\otimes \bx_2\otimes\dots\otimes\bx_d$.
It is easy to verify that 
$$\mleft\|\bx_1\otimes \bx_2\otimes\dots\otimes\bx_d\mright\|_2=\prod_{k=1}^d\|
% \bx_d
\textcolor{black}{\bx_k}
\|_2.$$
Any tensor $\CT\in\R^{n_1 \times n_2 \times \dots \times n_d}$ uniquely defines a multilinear form
$$
    \CT(\bx_1, \bx_2, \dots, \bx_d):=\langle\CT, \bx_1 \otimes \bx_2 \otimes \dots \otimes \bx_d\rangle=\CT\times_1\bx_1\times_2\bx_2\dots\times_d\bx_d
    %=\sum_{i_1=1}^{n_1} \sum_{i_2=1}^{n_2} \dots \sum_{i_d=1}^{n_d} t_{i_1 i_2 \dots i_d}(x_1)_{i_1}(x_2)_{i_2} \dots(x_d)_{i_d}
$$
of vector entries $(\bx_1, \bx_2, \dots, \bx_d)$ where $\bx_k\in\R^{n_k}$ for $k\in[d]$. 
% see~\cite[Page~642]{bader2006algorithm} for some discussions on the subtlety behind the latter definition.
If any vector entry, say $\bx_1$, is missing and replaced by a $\bullet$, then 
\begin{equation}\label{eq:specdot}
\CT (\bullet,\bx_2,\bx_3,\dots,\bx_d) = \CT\times_2\bx_2\times_3\bx_3\dots\times_d\bx_d\in \R^{n_1}
\end{equation}
becomes a vector. Similarly, $\CT(\bullet, \bullet,\bx_3,\bx_4,\dots ,\bx_d) \in \R^{n_1\times n_2}$ is a matrix, and so on.
For a thorough introduction to tensor operations, we refer interested readers to Kolda and Bader~\cite{kolda2009tensor}, Ballard and Kolda~\cite{TensorTextbook}, and Nie~\cite[Chapter~11]{nie2023moment}.

\subsection{Tensor spectral norm and nuclear norm}\label{sec:def}

Given a tensor $\CT\in\R^{n_1\times n_2\times\dots\times n_d}$, its spectral norm is defined as
    \begin{equation}\label{eq:spec}
    \|\CT\|_\sigma:=\max \bigl\{\langle\CT, \bx_1 \otimes \bx_2 \otimes \dots \otimes \bx_d\rangle:\|\bx_k\|_2=1\,\forall\, k\in[d]\bigr\},
    \end{equation}
and its nuclear norm is defined as
    \begin{equation}
    \|\CT\|_*:=\min \mleft\{\sum_{i=1}^r|\lambda_i|:
    \CT=\sum_{i=1}^r \lambda_i\, \bx_1^i \otimes \bx_2^i \otimes \dots \otimes \bx_d^i,\,\|\bx_k^i\|_2=1\,
    \forall\, i\in[r] \mbox{ and } k\in[d],\,r \in \N\mright\}; \label{eq:nuclear}
    \end{equation}
see, e.g.,~\cite{lim2005singular,qi2005eigenvalues,friedland2018nuclear} for more details.
A decomposition $\sum_{i=1}^r \lambda_i\, \bx_1^i \otimes \bx_2^i \otimes \dots \otimes \bx_d^i$ of $\CT$ in~\eqref{eq:nuclear} {\color{black} is commonly known as a CANDECOMP/PARAFAC (CP) decomposition in the literature; see, e.g.,~\cite[Section~3]{kolda2009tensor}. A CP decomposition of $\CT$} that attains $\|\CT\|_*$ is called a nuclear decomposition.
% \textcolor{black}{(It is worth noting that every nuclear decomposition is necessarily , but not vice versa;
% % while the converse may not hold;
% in fact, a nuclear decomposition is nothing but a CP decomposition that minimizes $\sum_{i=1}^r|\lambda_i|$ 
% % among
% over
% all such decompositions with unit-norm factors.)}
These two definitions, {\color{black}\eqref{eq:spec} and~\eqref{eq:nuclear}}, reduce to the matrix spectral norm and nuclear norm when $d=2$. However, unlike their matrix counterparts, computing the tensor spectral norm~\cite{he2010approximation} and nuclear norm~\cite{friedland2018nuclear} are both NP-hard when $d\ge3$.

Written in terms of the multilinear form in~\eqref{eq:spec}, the tensor spectral norm enjoys a nice property, i.e.,
\begin{align}
\|\CT\|_\sigma&=\max_{\|\bx_k\|_2=1\,\forall\, k\in[d]} \CT(\bx_1,\bx_2, \dots,\bx_d) \nonumber\\ 
&=\max_{\|\bx_k\|_2=1\,\forall\, k\in[d-1]} \|\CT(\bx_1,\bx_2, \dots,\bx_{d-1},\bullet)\|_2 \nonumber\\ 
&=\max_{\|\bx_k\|_2=1\,\forall\, k\in[d-2]} \|\CT(\bx_1,\bx_2, \dots,\bx_{d-2},\bullet,\bullet)\|_\sigma \label{eq:specextend}\\
&=\max_{\|\bx_k\|_2=1\,\forall\, k\in[d-3]} \|\CT(\bx_1,\bx_2, \dots,\bx_{d-3},\bullet,\bullet,\bullet)\|_\sigma, \nonumber
\end{align}
and so on. {\color{black}In particular, the second equality follows immediately from~\eqref{eq:specdot} and the Cauchy-Schwarz inequality. The third equality follows from the fact that $\CT(\bx_1,\bx_2, \dots,\bx_{d-2},\bullet,\bullet) = \CT\times_1\bx_1\times_2\bx_2\dots\times_{d-2}\bx_{d-2}\in \R^{n_{d-1}\times n_d}$ and the definition of matrix spectral norm, i.e.,~\eqref{eq:spec} for $d=2$. Similar reasoning applies to the last equality.}
% \textcolor{black}{which directly follows by repeatedly applying the reasoning in (9.2)--(9.4) of (the full version of)~\cite{barak2012hypercontractivity}}. 

The tensor spectral norm is the dual norm of the tensor nuclear norm, and 
% vise versa
\textcolor{black}{vice versa}.
\begin{lemma}\label{lma:norm-duality}
If two tensors $\CT,\CS\in\R^{n_1\times n_2\times\dots\times n_d}$, then
    $$
 \langle\CT,\CS\rangle \le \|\CT\|_\sigma\|\CS\|_*.
    $$
Moreover,
    $$
    \|\CT\|_\sigma =\max _{\|\CZ\|_* \le 1}\langle\CT, \CZ\rangle \mbox{ and }
    \|\CT\|_* =\max _{\|\CZ\|_\sigma \le 1}\langle\CT, \CZ\rangle.
    $$
\end{lemma}
A proof can be found in, e.g.,~\cite[Lemma~21]{lim2013blind}. An optimal tensor $\CZ$ of either of the above two optimization problems is called a dual certificate of $\CT$. In addition, for any nuclear decomposition $\CT=\sum_{i=1}^r \lambda_i\, \bx_1^i \otimes \bx_2^i \otimes \dots \otimes \bx_d^i$ with $\lambda_i>0$ for $i\in[r]$ and any dual certificate $\CZ$ with $\|\CZ\|_\sigma=1$ and $\|\CT\|_*=\langle \CT,\CZ\rangle$, one always has $\mleft\langle \CZ, \bx_1^i \otimes \bx_2^i \otimes \dots \otimes \bx_d^i\mright\rangle=1$ for any $i\in[r]$; see~\cite[Lemma~4.1]{friedland2018nuclear}.

There are some trivial bounds relating the two norms to the H\"{o}lder $p$-norms mentioned earlier, i.e.,
\begin{equation}\label{eq:trivial}
\|\CT\|_\infty\le\|\CT\|_\sigma\le\|\CT\|_2\le\|\CT\|_*\le\|\CT\|_1
\end{equation}
for any tensor $\CT$; see, e.g.,~\cite[Proposition~4.2]{chen2020tensor}.

\section{\textcolor{black}{Tensor} subspaces}\label{sec:subspaces}

The essential reason resulting \textcolor{black}{in} the full decomposability of the tensor nuclear norm lies in properly chosen tensor subspaces. In this section, we introduce tensor subspaces and elaborate their notations as a key step to appreciate the main results. To provide a better picture of the tensor subspaces, we also discuss orthogonal transformations and projections for tensors. Finally, we prove key properties of the tensor spectral norm and nuclear norm over subspaces. These properties set a foundation for the main theoretical developments. The section is presented in a self-contained way for readers with minimal background knowledge.

% In this subsection we introduce some notations of subspaces that will be used throughout this; to build familiarity with and motivate the notations, a brief look ahead at how they behave in the decomposability of the tensor nuclear norm will also be offered in Section~\ref{sec:interlude}.

\subsection{Subspaces}
% \label{sec:subspaces}

We consider the general tensor space $\R^{n_1 \times n_2 \times \dots \times n_d}$ of order $d$. Although a tensor space is also a vector space, we exclusively use $\TT$ and $\U$ for tensor subspaces and $\V$ for vector subspaces. In particular, we use $\V_k$ to denote a subspace of $\R^{n_k}$ that corresponds to mode $k$ of the tensor space.

% Consider $$
% \TT_{\J}^{\I}:=\spn\mleft(\mleft\{\bigotimes_{k=1}^d \bv_k:
% \bv_k\in \V_k\,\forall\,k\notin\I\cup\J,\, 
% \bv_k\in \V_k^\perp\,\forall\,k\in\I,\, 
% \bv_k\in \R^{n_k}\,\forall\,k\in\J\mright\}\mright).
% $$
% consider remove $\U$ notation, simple using direct sum of $\TT$.

Given vector subspaces $\V_k\subseteq\R^{n_k}$ for $k\in[d]$, we denote
$$
\TT\bigl((\V_k)_{k=1}^d\bigr)=\TT(\V_1,\V_2,\dots,\V_d):=\spn\mleft(\bigotimes_{k=1}^d \V_k\mright)=\spn(\V_1\otimes\V_2\otimes \dots\otimes\V_d)
%\spn(\V_1\otimes\V_2\otimes \dots\otimes\V_d).
$$
to be the tensor subspace spanned by $(\V_1,\V_2,\dots,\V_d)$. Let $\V_k^\perp$ be the orthogonal complement of $\V_k$, i.e., $\V_k\oplus\V_k^\perp=\R^{n_k}$, where $\oplus$ stands for the direct sum. Given an index set $\I\subseteq[d]$, we denote
%$\V_k^\I=\V_k^\perp$ if $k\in\I$ or $\V_k$ itself if $k\notin \I$.
$$
    \V_k^\I:=
    \begin{dcases}
        \V_k^\perp & k\in\I \\
        \V_k & k\notin \I.
    \end{dcases}
$$
For example, we always have $\V_k^{\varnothing}=\V_k$ and $\V_k^{[d]}=\V_k^\perp$ for $k\in[d]$. For any $\I\subseteq[d]$, we call
$$
    \TT^\I\bigl((\V_k)_{k=1}^d\bigr):=\TT\bigl((\V_k^\I)_{k=1}^d\bigr)= \spn\mleft(\bigotimes_{k=1}^d \V_k^\I\mright)%=\range\mleft(\bigotimes_{k=1}^d\proj_{\V_k^\I}\mright).
$$
a basic subspace of $\R^{n_1 \times n_2 \times \dots \times n_d}$ defined by $(\V_1,\V_2,\dots,\V_d)$. In particular, we have 
$
\TT^\varnothing\bigl((\V_k)_{k=1}^d\bigr)=\TT\bigl((\V_k)_{k=1}^d\bigr)=\spn(\V_1\otimes\V_2\otimes \dots\otimes\V_d)
$
and
$
\TT^{[d]}\bigl((\V_k)_{k=1}^d\bigr)=\TT\bigl((\V_k^\perp)_{k=1}^d\bigr)=\spn(\V_1^\perp\otimes\V_2^\perp\otimes \dots\otimes\V_d^\perp).
$
%As $\V_k^\I=(\V_k^\perp)^{[d]\setminus\I}$ for all $k\in[d]$ and $\I\subseteq[d]$, i
It is also easy to see that
$$
    \TT^{[d]\setminus\I}\bigl((\V_k)_{k=1}^d\bigr)=\TT^\I\bigl((\V_k^\perp)_{k=1}^d\bigr).
$$
Given $\V_k\subseteq\R^{n_k}$ for $k\in[d]$, there are exactly $2^d$ basic subspaces of $\R^{n_1 \times n_2 \times \dots \times n_d}$ and
$$
\bigoplus_{\I\subseteq[d]}\TT^\I\bigl((\V_k)_{k=1}^d\bigr)=\R^{n_1 \times n_2 \times \dots \times n_d}.
$$

The following $\U$ subspaces are particularly important in our discussions, i.e.,
\begin{align*}
    \U^\I\bigl((\V_k)_{k=1}^d\bigr)&:=\bigoplus_{\I\subseteq\J\subseteq[d]}\TT^\J\bigl((\V_k)_{k=1}^d\bigr)
    =\spn\mleft(\mleft\{\bigotimes_{k=1}^d \bv_k:\bv_k\in \V_k^\perp\,\forall\,k\in\I,\, \bv_k\in \R^{n_k}\,\forall\,k\notin\I\mright\}\mright),\\
        \U_\I\bigl((\V_k)_{k=1}^d\bigr)&:=\U^\I\bigl((\V_k^\perp)_{k=1}^d\bigr)
    =\spn\mleft(\mleft\{\bigotimes_{k=1}^d \bv_k:\bv_k\in \V_k\,\forall\,k\in\I,\, \bv_k\in \R^{n_k}\,\forall\,k\notin\I\mright\}\mright). 
\end{align*}
% \begin{align*}
%     \U^\I\bigl((\V_k)_{k=1}^d\bigr)&:=\bigoplus_{\I\subseteq\J\subseteq[d]}\TT^\J\bigl((\V_k)_{k=1}^d\bigr)
%     =\spn\mleft(\mleft\{\bigotimes_{k=1}^d \bv_k:
%     \begin{gathered}
%     \bv_k\in \V_k^\perp\,\forall\,k\in\I\\
%     \bv_k\in \R^{n_k}\,\forall\,k\notin\I
%     \end{gathered}
%     \mright\}\mright)
% \end{align*}
% and
% \begin{align*}
%         \U_\I\bigl((\V_k)_{k=1}^d\bigr)&:=\U^\I\bigl((\V_k^\perp)_{k=1}^d\bigr)
%     =\spn\mleft(\mleft\{\bigotimes_{k=1}^d \bv_k:
%     \begin{gathered}
%     \bv_k\in \V_k\,\forall\,k\in\I \\
%     \bv_k\in \R^{n_k}\,\forall\,k\notin\I
%     \end{gathered}
%     \mright\}\mright);
% \end{align*}
For example, we have 
$\U^\varnothing\bigl((\V_k)_{k=1}^d\bigr)=\U_\varnothing\bigl((\V_k)_{k=1}^d\bigr)=\R^{n_1 \times n_2 \times \dots \times n_d}$, $
\U^{[d]}\bigl((\V_k)_{k=1}^d\bigr)=\TT\bigl((\V_k^\perp)_{k=1}^d\bigr)$, and $\U_{[d]}\bigl((\V_k)_{k=1}^d\bigr)=\TT\bigl((\V_k)_{k=1}^d\bigr)$.
% For example, we have $\U^\varnothing(\{\V_k\}_{k=1}^d)=\R^{n_1 \times n_2 \times \dots \times n_d}$ and $\U^{[d]}(\{\V_k\}_{k=1}^d)=\TT^{[d]}(\{\V_k\}_{k=1}^d)=\spn(\V_1^\perp\otimes\V_2^\perp\otimes \dots\otimes\V_d^\perp)$.

As an illustrating example, if $\V_1$ is a subspace of $\R^{n_1}$ and $\V_2$ is a subspace of $\R^{n_2}$, we have in the matrix space $\R^{n_1\times n_2}$ that
% \begin{align*}
% &\TT^\varnothing(\V_1,\V_2)=\spn(\V_1\otimes\V_2),\, &&\TT^{\{1\}}(\V_1,\V_2)=\spn(\V_1^\perp\otimes\V_2),\,\\
% &\TT^{\{2\}}(\V_1,\V_2)=\spn(\V_1\otimes\V_2^\perp),\,
% &&\TT^{\{1,2\}}(\V_1,\V_2)=\spn(\V_1^\perp\otimes\V_2^\perp), \\
% &\U^\varnothing(\V_1,\V_2)=\spn(\R^{n_1}\otimes\R^{n_2}),\, &&\U^{\{1\}}(\V_1,\V_2)=\spn(\V_1^\perp\otimes\R^{n_2}),\,\\
% &\U^{\{2\}}(\V_1,\V_2)=\spn(\R^{n_1}\otimes\V_2^\perp),\,
% &&\U^{\{1,2\}}(\V_1,\V_2)=\spn(\V_1^\perp\otimes\V_2^\perp)\\
% &\U_\varnothing(\V_1,\V_2)=\spn(\R^{n_1}\otimes\R^{n_2}),\, &&\U_{\{1\}}(\V_1,\V_2)=\spn(\V_1\otimes\R^{n_2}),\,\\
% &\U_{\{2\}}(\V_1,\V_2)=\spn(\R^{n_1}\otimes\V_2),\,
% &&\U_{\{1,2\}}(\V_1,\V_2)=\spn(\V_1\otimes\V_2).
% \end{align*}
\begin{align*}
&\TT^\varnothing(\V_1,\V_2)=\spn(\V_1\otimes\V_2), &&\U^\varnothing(\V_1,\V_2)=\spn(\R^{n_1}\otimes\R^{n_2}), &&\U_\varnothing(\V_1,\V_2)=\spn(\R^{n_1}\otimes\R^{n_2}), \\
&\TT^{\{1\}}(\V_1,\V_2)=\spn(\V_1^\perp\otimes\V_2), &&\U^{\{1\}}(\V_1,\V_2)=\spn(\V_1^\perp\otimes\R^{n_2}), &&\U_{\{1\}}(\V_1,\V_2)=\spn(\V_1\otimes\R^{n_2}), \\
&\TT^{\{2\}}(\V_1,\V_2)=\spn(\V_1\otimes\V_2^\perp), &&\U^{\{2\}}(\V_1,\V_2)=\spn(\R^{n_1}\otimes\V_2^\perp),&&\U_{\{2\}}(\V_1,\V_2)=\spn(\R^{n_1}\otimes\V_2),\\
&\TT^{\{1,2\}}(\V_1,\V_2)=\spn(\V_1^\perp\otimes\V_2^\perp),  &&\U^{\{1,2\}}(\V_1,\V_2)=\spn(\V_1^\perp\otimes\V_2^\perp), &&\U_{\{1,2\}}(\V_1,\V_2)=\spn(\V_1\otimes\V_2).
\end{align*}

For a given tensor $\CT\in\R^{n_1 \times n_2 \times \dots \times n_d}$, we denote $\spn_k(\CT)\subseteq\R^{n_k}$ to be the subspace spanned by the columns of $\BT_{(k)}$, the mode-$k$ matricization of $\CT$. In particular, $\spn_1(\BT)$ is the column space of $\BT$ and $\spn_2(\BT)$ is the row space of $\BT$ if $\BT$ is a matrix. We adopt the following shorthand notations for the tensor subspaces defined by $\CT$,
$$
\TT^\I(\CT):=\TT^\I\bigl((\spn_k(\CT))_{k=1}^d\bigr),~\U^\I(\CT):=\U^\I\bigl((\spn_k(\CT))_{k=1}^d\bigr),\text{ and }\U_\I(\CT):=\U_\I\bigl((\spn_k(\CT))_{k=1}^d\bigr).
$$
We also simply denote $\TT(\CT):=\TT^{\varnothing}(\CT)$.

% \subsubsection{Interlude: A preview of the decomposability of the tensor nuclear norm}\label{sec:interlude}

\subsection{Orthogonal transformations and projections}\label{sec:orthognoal}

Orthogonal transformations preserve many matrix properties, in particular to this paper, keeping the spectral norm and nuclear norm unchanged. Given any matrix $\BT\in\R^{n_1\times n_2}$, there is an orthogonal matrix $\BP\in\R^{n_1\times n_1}$ such that the nonzero entries of $\BP\BT$ only appear in the first $r_1$ rows where $r_1=\dim\bigl(\spn_1(\BT\bigr))$. There is also an orthogonal matrix $\BQ\in\R^{n_2\times n_2}$ such that the nonzero entries of $\BT\BQ$ only appear in the first $r_2$ columns where $r_2=\dim\bigl(\spn_2(\BT)\bigr)$. 
% Its
\textcolor{black}{It is}
obvious that $r_1=r_2$, the rank of $\BT$.

In fact, the same can be applied to tensors as well. Given a tensor $\CT\in\R^{n_1\times n_2\times\dots\times n_d}$ and a mode $k\in[d]$, there is an orthogonal matrix $\BP_k\in\R^{n_k\times n_k}$, such that the entries of $\CT\times_k\BP_k$ are nonzero only when its mode-$k$ index is at most $r_k$, where $r_k=\dim\bigl(\spn_k(\CT)\bigr)$. Essentially, any mode-$k$ fiber of $\CT$, say $\bv\in\R^{n_k}$, changes to $\BP_k\bv\in\R^{n_k}$ whose last $n_k-r_k$ entries become zeros. Therefore, by repeatedly applying mode-$k$ products with orthogonal matrices $\BP_1,\BP_2,\dots,\BP_d$, we obtain a tensor $\CS=\CT\times_1\BP_1\times_2\BP_2\dots\times_d\BP_d$ such that $s_{i_1i_2\dots i_d}=0$ if there exists a mode index $k$ with $i_k>\dim\bigl(\spn_k(\CT)\bigr)$. Besides, this transformation is completely reversible, in particular, $\CT=\CS\times_1\BP_1^\T\times_2\BP_2^\T\dots\times_d\BP_d^\T$.
%and it can be done simply by applying mode-$k$ products of $\CS$ with $\BP_1^\T,\BP_2^\T,\dots,\BP_d^\T$. 
This orthogonal transformation for tensors is more like an expanded Tucker decomposition~\cite{kolda2009tensor} whose core tensor is a shrunk version of $\CS$ by deleting peripheral zero entries. While we may not explicitly apply this fact in our study, bearing this in mind makes many properties of tensor subspaces easier to understand. For example, when we consider a subspace $\V_k$ or $\spn_k(\CT)$ of $\R^{n_k}$ with $\dim(\V_k)=r_k$, we can simply treat $\V_k$ as $\R^{r_k}\times\{0\}^{n_k-r_k}$ and $\V_k^\perp$ as $\{0\}^{r_k}\times\R^{n_k-r_k}$.

In fact, the matrix SVD makes the above even better. Since $r_1=r_2$, SVD results \textcolor{black}{in} the top-left $r_1\times r_1$ submatrix of $\CS$ to be diagonal with all of the singular values. With that structure, the decomposability of the matrix nuclear norm is straightforward and the subdifferential of the matrix nuclear norm admits an explicit and complete representation; see~\eqref{eq:matrixdecomp} and~\eqref{eq:matrixpartial}. However, the diagonal structure does not hold for tensors in general, and this makes the results of decomposability and subdifferential of the tensor nuclear norm unsatisfactory. % Well, it actually leads to many interesting findings as we will detail in the next two sections. 

Given a subspace $\V$ of $\R^n$, the orthogonal projection operator $\proj_{\V}$ of a vector or a set of vectors in $\R^{n}$ is frequently used in this paper. The outer products of projection operators, in particular, $\bigotimes_{k=1}^d\proj_{\V_k}$ with $\V_k$ being a subspace of $\R^{n_k}$ for $k\in[d]$, of a tensor $\CT\in\R^{n_1\times n_2\times \dots \times n_d}$ is defined as follows: Given any rank-one decomposition $\CT=\sum_{i=1}^r \bx^i_1 \otimes \bx^i_2 \otimes \dots \otimes \bx^i_d$, 
$$
\mleft(\bigotimes_{k=1}^d\proj_{\V_k}\mright)(\CT)
:= \proj_{\TT((\V_k)_{k=1}^d)}(\CT)
= \sum_{i=1}^r \proj_{\V_1}(\bx^i_1) \otimes \proj_{\V_2}(\bx^i_2) \otimes \dots \otimes \proj_{\V_d}(\bx^i_d).
$$
It is easy to check that the projection is invariant with respect to rank-one decompositions. In fact, if we let $\BP_k\in\R^{n_k\times n_k}$ be the projection matrix of $\proj_{\V_k}$ for $k\in[d]$, then it is not hard to show that 
$$
\mleft(\bigotimes_{k=1}^d\proj_{\V_k}\mright)(\CT)=\CT\times_1\BP_1\times_2\BP_2\dots\times_d\BP_d,
$$
\textcolor{black}{known as the multilinear orthogonal projection; see, e.g.,~\cite[Section~2.6]{de2008tensor} and~\cite[Section~3]{vannieuwenhoven2012new}, as well as the discussions therein.
% from this point of view 
% that may shed further light 
% are also available.
% }
% In another word
% \textcolor{black}{
In other words}, $\CT\times_k\BP_k$ projects all mode-$k$ fibers of $\CT$ onto the subspace $\V_k$. Therefore, we can take $\bigotimes_{k=1}^d\proj_{\V_k}$ as performing vector projections $d$ times, in any order of the modes $1,2,\dots,d$, as mode products allow swapping.

We state a property of the projection $\bigotimes_{k=1}^d\proj_{\V_k}$ below without proof.
\begin{lemma}
    If $\V_k$ is a subspace of $\R^{n_k}$ for $k\in[d]$, then the projection $\bigotimes_{k=1}^d \proj_{\V_k}$ is self-adjoint. Moreover, 
    \begin{equation*}
        \TT\bigl((\V_k)_{k=1}^d\bigr)=\bigl\{\CT\in\R^{n_1\times n_2\times \dots \times n_d}:\spn_k(\CT)\subseteq\V_k\,\forall\,k\in[d]\bigr\}.
    \end{equation*}
    % In particular, these further imply that $\bigotimes_{k=1}^d\proj_{\V_k}=\proj_{\TT((\V_k)_{k=1}^d)}$.
\end{lemma}

Finally, the norm of an operator $\proj$, in particular a projection or arithmetic operations of some projections in this paper, is defined as 
$$
\|\proj\|:=\max\bigl\{\|\proj(\CT)\|_2:\|\CT\|_2=1,\,\CT\in\R^{n_1\times n_2\times \dots \times n_d}\bigr\}.
$$
This includes the case of vector spaces when $d=1$.

\subsection{Tensor spectral and nuclear norms over subspaces}

The definitions of the tensor spectral norm and nuclear norm in Section~\ref{sec:def} involve working with all unit vectors. Instead, the following results state that, to evaluate the two norms of a tensor in $\TT\bigl((\V_k)_{k=1}^d\bigr)$, it suffices to work with vectors in the subspaces $\V_k$'s.
% can be seen as some generalized Banach's theorem~\cite{banach1938,chen2012maximum} with additional restrictions on the subspace $\TT(\{\V_k\}_{k=1}^d)$.
\begin{lemma}\label{thm:spec-subspace}
If $\V_k$ is a subspace of $\R^{n_k}$ for $k\in[d]$ and $\CT\in\TT\bigl((\V_k)_{k=1}^d\bigr)$, then
    $$
    \|\CT\|_\sigma=\max\mleft\{\mleft\langle\CT, \bigotimes_{k=1}^d\bx_k\mright\rangle:\bx_k\in\V_k\cap\SI^{n_k}\,\forall\,k\in[d]\mright\}.
    $$
Hence, $\bigl\|\proj_{\TT((\V_k)_{k=1}^d)}(\CT)\bigr\|_\sigma\le\|\CT\|_\sigma$ for any $\CT\in\R^{n_1\times n_2\times \dots \times n_d}$.
\end{lemma}
\begin{proof} 
Let $\|\CT\|_\sigma=\langle\CT, \bigotimes_{k=1}^d\by_k\rangle$, where $\by_k\in\SI^{n_k}$ for $k\in[d]$. Since $\CT\in\TT\bigl((\V_k)_{k=1}^d\bigr)$, we have $\bigl(\bigotimes_{k=1}^d\proj_{\V_k}\bigr)(\CT)=\CT$. As $\bigotimes_{k=1}^d\proj_{\V_k}$ is self-adjoint, we further have
$$\|\CT\|_\sigma%=\mleft\langle\CT, \bigotimes_{k=1}^d\by_k\mright\rangle
=\mleft\langle\mleft(\bigotimes_{k=1}^d\proj_{\V_k}\mright)(\CT), \bigotimes_{k=1}^d\by_k\mright\rangle=\mleft\langle\CT, \mleft(\bigotimes_{k=1}^d\proj_{\V_k}\mright)\mleft(\bigotimes_{k=1}^d\by_k\mright)\mright\rangle
  =\mleft\langle\CT, \bigotimes_{k=1}^d\proj_{\V_k}(\by_k)\mright\rangle.$$
  
If there is some $\proj_{\V_k}(\by_k)={\bf 0}$, then $\|\CT\|_\sigma=0$. This makes $\CT=\CO$ and the equality trivially holds. Otherwise, we have $\bigl\|\proj_{\V_k}(\by_k)\bigr\|_2\ne 0$ for every $k\in[d]$. By noticing that $\bigl\|\proj_{\V_k}(\by_k)\bigr\|_2\le \|\by_k\|_2=1 $, we have
    \begin{align*}
    \|\CT\|_\sigma&=\mleft\langle\CT, \bigotimes_{k=1}^d\proj_{\V_k}(\by_k)\mright\rangle
    \\& =\mleft\langle\CT, \bigotimes_{k=1}^d\frac{\proj_{\V_k}(\by_k)}{\bigl\|\proj_{\V_k}(\by_k)\bigr\|_2}\mright\rangle \prod_{k=1}^d\bigl\|\proj_{\V_k}(\by_k)\bigr\|_2\\
    &\le\max \mleft\{\mleft\langle\CT, \bigotimes_{k=1}^d\bx_k\mright\rangle:\bx_k\in\V_k\cap\SI^{n_k}\,\forall\,k\in[d]\mright\}\\
    &\le\max \mleft\{\mleft\langle\CT, \bigotimes_{k=1}^d\bx_k\mright\rangle:\bx_k\in\SI^{n_k}\,\forall\,k\in[d]\mright\}\\
    &=\|\CT\|_\sigma,
    \end{align*}
    implying the validity of the equality. % Otherwise, there exists some $i\in[d]$ such that $\bigl\|\proj_{\V_k}(\by_k)\bigr\|_2=0$, implying that $\|\CT\|_\sigma=\mleft\langle\CT, \bigotimes_{k=1}^d\proj_{\V_k}(\by_k)\mright\rangle=0$ from the second line of the above argument. In this case, $\CT$ becomes a zero tensor and the equality obviously holds.

    Finally, for any $\CT\in\R^{n_1\times n_2\times \dots \times n_d}$, as $\proj_{\TT((\V_k)_{k=1}^d)}(\CT)\in\TT\bigl((\V_k)_{k=1}^d\bigr)$, we have
    \begin{align*}
    \bigl\|\proj_{\TT((\V_k)_{k=1}^d)}(\CT)\bigr\|_\sigma&=\max \mleft\{\mleft\langle\mleft(\bigotimes_{k=1}^d\proj_{\V_k}\mright)(\CT), \bigotimes_{k=1}^d\bx_k\mright\rangle:\bx_k\in\V_k\cap\SI^{n_k}\,\forall\,k\in[d]\mright\}\\
    &=\max \mleft\{\mleft\langle\CT, \bigotimes_{k=1}^d\proj_{\V_k}(\bx_k)\mright\rangle:\bx_k\in\V_k\cap\SI^{n_k}\,\forall\,k\in[d]\mright\}\\
    &=\max \mleft\{\mleft\langle\CT, \bigotimes_{k=1}^d\bx_k\mright\rangle:\bx_k\in\V_k\cap\SI^{n_k}\,\forall\,k\in[d]\mright\}\\
    &\le \max \mleft\{\mleft\langle\CT, \bigotimes_{k=1}^d\bx_k\mright\rangle:\bx_k\in\SI^{n_k}\,\forall\,k\in[d]\mright\}\\
    &=\|\CT\|_\sigma,
    \end{align*}
completing the final piece.
\end{proof}

The orthogonal transformation of tensors discussed in Section~\ref{sec:orthognoal} makes the above statement more intuitive. If $\proj_{\V_k}$ simply keeps the first several entries unchanged and zeroing the remainder for every $k\in[d]$, to maximize $\langle\CT, \bigotimes_{k=1}^d\bx_k\rangle$ for a tensor $\CT\in\TT\bigl((\V_k)_{k=1}^d\bigr)$, an optimal $\bx_k$ must not have a nonzero entry outside $\V_k$ for every $k\in[d]$. {\color{black}We remark that the Frobenius-norm analogue of the second part of Lemma~\ref{thm:spec-subspace}, i.e., $\bigl\|\proj_{\TT((\V_k)_{k=1}^d)}(\CT)\bigr\|_2\le\|\CT\|_2$, has been established in~\cite[(3.2)]{vannieuwenhoven2012new}.}

The following is the dual version of Lemma~\ref{thm:spec-subspace}. 

%We remark the main idea of the proof is similar to that of~\cite[Theorem~5.1]{friedland2018nuclear} but we make the details preciser and clearer.
\begin{lemma}\label{prop:nuclear-U-equiv}
If $\V_k$ is a subspace of $\R^{n_k}$ for $k\in[d]$ and $\CT\in\TT\bigl((\V_k)_{k=1}^d\bigr)$, then
\begin{align*}
\|\CT\|_*
% &=\inf\mleft\{k\in\R_{++}:\frac{\CT}{k}\in \TT(\{\V_k\}_{k=1}^d)\cap\ksix\mright\}
% \\&
=\min\mleft\{\sum_{i=1}^r|\lambda_i|:\CT=\sum_{i=1}^r\lambda_i\bigotimes_{k=1}^d\bx_k^i,\, \bx_k^i\in\V_k\cap\SI^{n_k}\,
    \forall\, i\in[r]  \mbox{ and }k\in[d],\,r \in \N\mright\}.
\end{align*}
Furthermore, there exists a dual certificate $\CZ\in\TT\bigl((\V_k)_{k=1}^d\bigr)$ with $\|\CZ\|_\sigma=1$ such that $\langle \CT,\CZ\rangle=\|\CT\|_*$.
\textcolor{black}{Besides, $\bigl\|\proj_{\TT((\V_k)_{k=1}^d)}(\CT)\bigr\|_*\le\|\CT\|_*$ for any $\CT\in\R^{n_1\times n_2\times \dots \times n_d}$.}
% Let $\V_k$ be a subspace of $\R^{n_k}$ for all $k\in[d]$ and $\CT\in\TT\bigl((\V_k)_{k=1}^d\bigr)$. We have
% \begin{align}
% \|\CT\|_*
% % &=\inf\mleft\{k\in\R_{++}:\frac{\CT}{k}\in \TT(\{\V_k\}_{k=1}^d)\cap\ksix\mright\}
% % \\&
% =\min\mleft\{\sum_{i=1}^r|\lambda_i|:
% \begin{gathered}
% \CT=\sum_{i=1}^r\lambda_i\bigotimes_{k=1}^d\bx_k^i \\
% \bx_k^i\in\V_k\cap\SI^{n_k}\,\forall\, k\in[d]\text{ and}~i\in[r]\\
% r \in \N
% \end{gathered}
%     \mright\}.
% \end{align}
% Besides, there exists a dual certificate $\CZ\in\TT\bigl((\V_k)_{k=1}^d\bigr)$ with $\|\CZ\|_\sigma=1$ such that $\langle \CT,\CZ\rangle=\|\CT\|_*$.
\end{lemma}
\begin{proof} 
    We assume that $\CT\ne\CO$ as otherwise the results hold trivially. This means that $\dim(\V_k)\ge1$ for every $k\in[d]$ and further $\HI=\bigotimes_{k=1}^d(\V_k\cap\SI^{n_k})$ is nonempty. It is easy to see that $\conv(\HI)$ is convex, compact, and centrally symmetric
    \textcolor{black}{
    with respect to the origin, % (i.e., $\conv(\HI)=-\conv(\HI)$; in particular, this 
    implying that $\CO\in\conv(\HI)$}.
    % Besides, $\conv(\HI)$ has a nonempty relative interior~\cite[Theorem~6.2]{rock1997convex}. 
    Let $\HI_k\subseteq\V_k\cap\SI^{n_k}$ be any normalized basis of $\V_k$ for $k\in[d]$.
    % By~\cite[Exercise~1.2]{brondsted2012introduction}, 
    We have
    $$
   \TT\bigl((\V_k)_{k=1}^d\bigr)= \spn\mleft(\bigotimes_{k=1}^d\V_k\mright) =\spn\mleft(\bigotimes_{k=1}^d\HI_k\mright)\subseteq\spn(\HI)\subseteq\spn\bigl(\conv(\HI)\bigr)%=\aff(\conv(\HI))
   \subseteq\spn\mleft(\bigotimes_{k=1}^d\V_k\mright),$$
    implying that $\spn\bigl(\conv(\HI)\bigr)=\TT\bigl((\V_k)_{k=1}^d\bigr)$. 
    \textcolor{black}{In fact, $\CO$
    % \in\operatorname{int}()
    is an interior point of $\conv(\HI)$ relative to $\TT\bigl((\V_k)_{k=1}^d\bigr)$. 
    % Suppose
    If it is not, then the supporting hyperplane theorem (see, e.g.,~\cite[Section~2.5.2]{boyd2004convex}) implies that there exists a supporting hyperplane to $\conv(\HI)$ at $\CO$. This means that $\conv(\HI)$ is entirely contained in that hyperplane because
    $\conv(\HI)$ is centrally symmetric with respect to $\CO$,
    a contradiction to
    % which contradicts
    $\spn\bigl(\conv(\HI)\bigr)=\TT\bigl((\V_k)_{k=1}^d\bigr)$.}
    As a result, 
    \textcolor{black}{it follows from~\cite[Proposition~1.1.8]{thompson1996minkowski} that}
    $\conv(\HI)$ must be the unit ball of some norm 
    $\|\bullet\|_{\HH}:\TT\bigl((\V_k)_{k=1}^d\bigr)\rightarrow\R_+$,
    \textcolor{black}{namely the Minkowski functional~\cite[Definition~1.1.7]{thompson1996minkowski} of $\conv(\HI)$}. 
    
    We first claim that $\|\bigotimes_{k=1}^d\bx_k\|_{\HH}=1$ if $\bx_k\in\V_k\cap\SI^{n_k}$ for any $k\in[d]$. As $\bigotimes_{k=1}^d\bx_k\in\HI\subseteq\conv(\HI)$, it is obvious that $\|\bigotimes_{k=1}^d\bx_k\|_{\HH}\le 1$. It then suffices to show that $\bigotimes_{k=1}^d\bx_k$ is an extreme point of $\conv(\HI)$. If this is not the case, then we can rewrite $\bigotimes_{k=1}^d\bx_k=\lambda \CY+(1-\lambda)\CZ$ for some $\lambda\in(0,1)$ and $\CY,\CZ\in\conv(\HI)$ with $\CY\ne\CZ$. By Carathe{\'o}dory's theorem, we can further rewrite 
    $$
    \bigotimes_{k=1}^d\bx_k=\lambda \sum_{i=1}^{r_1}\alpha_i\CY_i+(1-\lambda)\sum_{j=1}^{r_2}\beta_j\CZ_j,
    $$ 
    where $\CY_i,\CZ_j\in\HI$ and $\alpha_i,\beta_j>0$ with $\sum_{i=1}^{r_1}\alpha_i=\sum_{j=1}^{r_2}\beta_j=1$ for every $i\in[r_1]$ and $j\in[r_2]$. % due to~\cite[Theorem~2.3]{barvinok2002course}. 
    There must be some $\CY_i$ or $\CZ_j$ that is not equal to $\bigotimes_{k=1}^d\bx_k$, as otherwise it leads to $\CY=\CZ=\bigotimes_{k=1}^d\bx_k$. Assume without loss of generality that $\CY_1\ne\bigotimes_{k=1}^d\bx_k$. We observe that $\bigotimes_{k=1}^d\bx_k-\lambda\alpha_1\CY_1$ is not a nonnegative multiple of $\lambda\alpha_1\CY_1$, as otherwise it leads to $\CY_1=\bigotimes_{k=1}^d\bx_k$. By the equality condition of the triangle inequality, we have
    $$
    1=\mleft\|\bigotimes_{k=1}^d\bx_k\mright\|_2<\mleft\|\lambda\alpha_1\CY_1\mright\|_2+\mleft\|\bigotimes_{k=1}^d\bx_k-\lambda\alpha_1\CY_1\mright\|_2 \le \lambda \sum_{i=1}^{r_1}\alpha_i\mleft\|\CY_i\mright\|_2+(1-\lambda)\sum_{j=1}^{r_2}\beta_j\mleft\|\CZ_j\mright\|_2=1,
    $$
    leading to a contradiction. 

    By vector scaling, any $\CT\in\TT\bigl((\V_k)_{k=1}^d\bigr)$ can be written as $\sum_{i=1}^r\lambda_i\bigotimes_{k=1}^d\bx_k^i$ where $\bx_k^i\in\V_k\cap\SI^{n_k}$ for any $i\in[r]$ and $k\in[d]$. By the triangle inequality, we have $\|\CT\|_{\HH}\le\sum_{i=1}^r|\lambda_i|\cdot\|\bigotimes_{k=1}^d\bx_k^i\|_{\HH}=\sum_{i=1}^r|\lambda_i|$, and so
    \begin{equation}\label{eq:nuclearsame}
        \|\CT\|_{\HH}\le \min\mleft\{\sum_{i=1}^r|\lambda_i|:\CT=\sum_{i=1}^r\lambda_i\bigotimes_{k=1}^d\bx_k^i,\,\bx_k^i\in\V_k\cap\SI^{n_k}\, \forall\, i\in[r] \mbox{ and } k\in[d] ,\,r \in \N\mright\}.
    % \|\CT\|_{\HH}\le \inf\mleft\{\sum_{i=1}^r|\lambda_i|:
    % \begin{gathered}
    % \CT=\sum_{i=1}^r\lambda_i\bigotimes_{k=1}^d\bx_k^i\\
    % \bx_k^i\in\V_k\cap\SI^{n_k}\,\forall\, k\in[d] \text{ and}~ i\in[r]\\
    % r \in \N
    % \end{gathered}
    % \mright\};
    \end{equation}
    We next claim that~\eqref{eq:nuclearsame} holds as an equality. Upon scaling, we only need to show that the equality holds for any $\CT\in\TT\bigl((\V_k)_{k=1}^d\bigr)$ with $\|\CT\|_{\HH}=1$. Because $\|\CT\|_{\HH}=1$ implies $\CT\in\conv(\HI)$, it follows that $\CT=\sum_{i=1}^s\mu_i\bigotimes_{k=1}^d\by^i_k$, where $\by^i_k\in\V_k\cap\SI^{n_k}$ and $\mu_i\ge0$ with $\sum_{i=1}^s\mu_i=1$ for $i\in[s]$ and $k\in[d]$. This decomposition is feasible in~\eqref{eq:nuclearsame}. As a result,
    \begin{align*}
    \|\CT\|_{\HH}&\le \min\mleft\{\sum_{i=1}^r|\lambda_i|:\CT=\sum_{i=1}^r\lambda_i\bigotimes_{k=1}^d\bx_k^i,\,\bx_k^i\in\V_k\cap\SI^{n_k}
\, \forall\, i\in[r] \mbox{ and } k\in[d],\,r \in \N\mright\}\\
  &\le \sum_{i=1}^s|\mu_i|=1=\|\CT\|_{\HH}.
    \end{align*}
    
%     \begin{align*}
%     \|\CT\|_{\HH}&\le \inf\mleft\{\sum_{i=1}^r|\lambda_i|:
%     \begin{gathered}
%     \CT=\sum_{i=1}^r\lambda_i\bigotimes_{k=1}^d\bx_k^i\\
%     \bx_k^i\in\V_k\cap\SI^{n_k}\,\forall\, k\in[d] \text{ and}~ i\in[r]\\
%     r \in \N
% \end{gathered}
% \mright\} \le \sum_{i=1}^s|\mu_i|=1=\|\CT\|_{\HH},
%     \end{align*}
%     as desired; in particular, this also implies that the infimum can indeed be attained.
    
    The final step is to claim the equivalence between $\|\bullet\|_{\HH}$ and $\|\bullet\|_*$ over $\TT\bigl((\V_k)_{k=1}^d\bigr)$. To start with, let us denote the dual norm of $\|\bullet\|_{\HH}$ to be $\|\bullet\|_{\HH'}:\TT\bigl((\V_k)_{k=1}^d\bigr)\rightarrow\R_+$, i.e., $\|\CT\|_{\HH'}=\max_{\|\CZ\|_{\HH}\le 1
    % ,\,\CZ\in\TT((\V_k)_{k=1}^d)
    }\langle\CT,\CZ\rangle$ for $\CT\in\TT\bigl((\V_k)_{k=1}^d\bigr)$. Given any $\CT\in\TT\bigl((\V_k)_{k=1}^d\bigr)$, let $\CX\in\TT\bigl((\V_k)_{k=1}^d\bigr)$ be such that $\|\CT\|_{\HH'}=\langle\CT,\CX\rangle$ with $\|\CX\|_{\HH}\le 1$, implying that $\CX\in\conv(\HI)$. Thus,  $\CX$ can be written as $\sum_{i=1}^{r_0}\gamma_i\CX_i$, where $\CX_i\in\HI$ and $\gamma_i\ge0$ with $\sum_{i=1}^{r_0}\gamma_i=1$ for any $i\in[r_0]$. This further implies that 
    $$
\|\CT\|_{\HH'}=\langle\CT,\CX\rangle =\mleft\langle\CT,\sum_{i=1}^{r_0}\gamma_i\CX_i\mright\rangle \le\mleft(\sum_{i=1}^{r_0}\gamma_i\mright)\max_{\CZ\in\HI}\langle\CT,\CZ\rangle=\max_{\CZ\in\HI}\langle\CT,\CZ\rangle\le \max_{\|\CZ\|_{\HH}\le 1}\langle\CT,\CZ\rangle=\|\CT\|_{\HH'}.
    $$
    As a result, for any $\CT\in\TT\bigl((\V_k)_{k=1}^d\bigr)$, we have
    \begin{equation*}\label{eq:vstar-spec}
    \|\CT\|_{\HH'}=\max_{\CZ\in\HI}\langle\CT,\CZ\rangle 
    = \max_{\bx_k\in\V_k\cap\SI^{n_k}\,\forall\,k\in[d]}\mleft\langle\CT,\bigotimes_{k=1}^d\bx_k\mright\rangle=\|\CT\|_\sigma,
    \end{equation*}
    where the last equality is due to Lemma~\ref{thm:spec-subspace}. Now, by comparing~\eqref{eq:nuclear} with that~\eqref{eq:nuclearsame} holds as an equality, we have
$$
\|\CT\|_*\le \|\CT\|_{\HH} =\max_{\|\CZ\|_{\HH'}\le 1} \langle\CT,\CZ\rangle =\max_{\|\CZ\|_\sigma\le 1,\,\CZ\in\TT((\V_k)_{k=1}^d)} \langle\CT,\CZ\rangle
\le  \max_{\|\CZ\|_\sigma\le 1} \langle\CT,\CZ\rangle=\|\CT\|_*.
$$
%completing the final piece.    
    % Now, let $\mathcal{V}_1$ be the restriction of the tensor nuclear norm to $\TT(\{\V_k\}_{k=1}^d)$. It is clear that $\mathcal{V}_1(\CT)\le\|\CT|_{\HH}$ for all $\CT\in\TT(\{\V_k\}_{k=1}^d)$ by definition. We next claim that they are in fact equal, and this will complete the proof. Assume this is not the case, then we must have $\mathbb{B}_{\mathcal{V}}\subsetneq\mathbb{B}_{\mathcal{V}_1}$, and thus $\mathbb{B}_{\mathcal{V}_1^\star}\subsetneq\mathbb{B}_{\mathcal{V}^\star}$ by polarity, where $\mathbb{B}_{\mathcal{V}}:=\{\CT\in\TT(\{\V_k\}_{k=1}^d):\|\CT|_{\HH}\le 1\}$ and similar applies to the others, and $\mathcal{V}_1^\star$ is the dual norm of $\mathcal{V}_1$ over $\TT(\{\V_k\}_{k=1}^d)$. Thus, we can find some $\CT\in\TT(\{\V_k\}_{k=1}^d)$ for which $\mathcal{V}^\star(\CT)\le 1<\mathcal{V}_1^\star(\CT)$. Then, by duality, we have that
    % $$
    % \mathcal{V}^\star(\CT)<\mathcal{V}_1^\star(\CT)=\max_{\mathcal{V}_1(\CZ)\le 1,\,\CZ\in\TT(\{\V_k\}_{k=1}^d)}\langle\CT,\CZ\rangle\le\max_{\|\CZ\|_*\le 1}\langle\CT,\CZ\rangle=\|\CT\|_\sigma,
    % $$
    % a contradiction to (\ref{eq:vstar-spec}).

Finally, given any $\CT\in\TT\bigl((\V_k)_{k=1}^d\bigr)$, let $\CS$ be a dual certificate of $\|\CT\|_*$, i.e., $\langle\CT,\CS\rangle=\|\CT\|_*$ and $\|\CS\|_\sigma=1$. %We assume that $\CT\ne\CO$ as otherwise the statement is trivial. 
It follows from Lemma~\ref{thm:spec-subspace} that $\bigl\|\proj_{\TT((\V_k)_{k=1}^d)}(\CS)\bigr\|_\sigma\le\|\CS\|_\sigma=1$. On the other hand, by Lemma~\ref{lma:norm-duality},
$$
\bigl\|\proj_{\TT((\V_k)_{k=1}^d)}(\CS)\bigr\|_\sigma
\ge \frac{\bigl\langle \CT, \proj_{\TT((\V_k)_{k=1}^d)}(\CS)\bigr\rangle}{\|\CT\|_*} 
=\frac{\bigl\langle\proj_{\TT((\V_k)_{k=1}^d)}(\CT),\CS \bigr\rangle}{\|\CT\|_*}=\frac{\langle\CT,\CS\rangle}{\|\CT\|_*}=1.
$$
Therefore, $\proj_{\TT((\V_k)_{k=1}^d)}(\CS)\in\TT\bigl((\V_k)_{k=1}^d\bigr)$ is also a dual certificate of $\|\CT\|_*$.

\textcolor{black}{To show the last claim, we observe from the duality and self-adjointness that for any $\CT\in\R^{n_1\times n_2\times \dots \times n_d}$,
\begin{align*}
    \bigl\|\proj_{\TT((\V_k)_{k=1}^d)}(\CT)\bigr\|_*
    % =\|\proj_{\TT((\V_k)_{k=1}^d)}(\CT)\|_* 
    &=\max _{\|\CZ\|_\sigma \le 1}\langle\proj_{\TT((\V_k)_{k=1}^d)}(\CT), \CZ\rangle\\
    &=\max _{\|\CZ\|_\sigma \le 1}\langle\CT, \proj_{\TT((\V_k)_{k=1}^d)}(\CZ)\rangle\\
   & \le \max _{\|\proj_{\TT((\V_k)_{k=1}^d)}(\CZ)\|_\sigma \le 1}\langle\CT, \proj_{\TT((\V_k)_{k=1}^d)}(\CZ)\rangle    \\
   & \le \max _{\|\CY\|_\sigma \le 1}\langle\CT, \CY\rangle\\
   &=\|\CT\|_*,
\end{align*}
where the first inequality is due to $\bigl\|\proj_{\TT((\V_k)_{k=1}^d)}(\CZ)\bigr\|_\sigma\le \|\CZ\|_\sigma \le 1$ by Lemma~\ref{thm:spec-subspace}. % we have $\bigl\{\proj_{\TT((\V_k)_{k=1}^d)}(\CZ):\|\CZ\|_\sigma \le 1\bigr\}\subseteq\{\CZ:\|\CZ\|_\sigma \le 1\}$, which immediately implies that
}
\end{proof}

\section{Decomposability of the tensor nuclear norm}\label{sec:decomp}

The decomposability of the matrix nuclear norm has been well understood. It has given rise to many important results in machine learning and statistical estimation~\cite[Section~10]{wainwright2019high}. Specifically, for two matrices $\BT,\BS\in\R^{n_1\times n_2}$, 
$$
\|\BT+\BS\|_*=\|\BT\|_*+\|\BS\|_*\text{ if $\BT^{\T}\BS=\BO$ and $\BT\BS^{\T}=\BO$}.
$$
%see, e.g.,~\cite[Section~10.2.1]{wainwright2019high}. 
In the notation of subspaces, it reads as
$$ 
\|\BT+\BS\|_*=\|\BT\|_*+\|\BS\|_* \text{ for any $\BT\in\TT(\V_1,\V_2)$ and $\BS\in\TT(\V_1^\perp,\V_2^\perp)$},
$$
where $\V_k$ is a subspace of $\R^{n_k}$ for $k\in[2]$.
% $\operatorname{col}(\BT)\perp\operatorname{col}(\BS)$ and $\operatorname{row}(\BT)\perp\operatorname{row}(\BS)$ that , where for $\V,\V'\subseteq\R^n$, $\V\perp\V'$ means $\bx^{\T}\by=0$ for all $\bx\in\V$ and $\by\in\V'$.
The main reason behind the above decomposability is that any matrix admits an SVD. After proper orthogonal transformations of the column space and the row space, a matrix becomes a diagonal one whose entries are its singular values. The nuclear norm, as the sum of singular values, automatically admits this type of decomposition.

The diagonalization fails to work for higher-order tensors in general. As a result, generalizing the decomposability to tensors remains unclear and unsatisfactory. To the best of our knowledge, the only known result is a weak decomposability~\cite[Lemma~2.1]{raskutti2019convex} that only works for third-order tensors. Specifically, the weak decomposability states that % for $\CT,\CS\in\R^{n_1\times n_2\times n_3}$,
\begin{equation}\label{eq:decomposability-yuan}
    \|\CT+\CS\|_*\ge\|\CT\|_*+\frac{1}{2}\|\CS\|_* \text{ for any $\CT\in\TT\bigl((\V_k)_{k=1}^3\bigr)$ and $\CS\in\bigoplus_{|\I|\ge2,\,\I\subseteq[3]}\TT^\I\bigl((\V_k)_{k=1}^3\bigr)$},
    % :=\range\mleft(\sum_{\I\subseteq[3],\,|\I|\ge 2}\bigotimes_{i=1}^3\mleft(\mathbbm{1}_{i\in\I}(\proj_{\V_k^\perp}-\proj_{\V_k})+\proj_{\V_k}\mright)\mright),
\end{equation}
% which is the subspace $\overline{\X}$ considered in~\cite{yuan2016tensor,yuan2017incoherent,raskutti2019convex} that plays the same role as our $\Y$ in the same figure
% Then,~\cite[Lemma~1]{raskutti2019convex} 
% if
% \begin{equation}\label{eq:mathbbX}
%     \CT\in\TT(\{\V_k\}_{k=1}^3),\,\CS\in\overline{\X}^\perp:=\range\mleft(\sum_{\I\subseteq[3],\,|\I|\ge 2}\bigotimes_{i=1}^3\mleft(\mathbbm{1}_{i\in\I}(\proj_{\V_k^\perp}-\proj_{\V_k})+\proj_{\V_k}\mright)\mright),
% \end{equation}
where $\V_k$ is a subspace of $\R^{n_k}$ for $k\in[3]$.

We remark that this weak decomposability has found successful applications in high-dimensional tensor regression~\cite{raskutti2019convex}. However, the full decomposability, $\|\CT+\CS\|_*=\|\CT\|_*+\|\CS\|_*$, turns out to be impossible over the subspace pair $\TT\bigl((\V_k)_{k=1}^3\bigr)$ and $\bigoplus_{|\I|\ge2,\,\I\subseteq[3]}\TT^\I\bigl((\V_k)_{k=1}^3\bigr)$ considered in~\eqref{eq:decomposability-yuan}. This is evidenced by the following example.
% Despite that the weak decomposability~\eqref{eq:decomposability-yuan} is valid over , , as demonstrated by.
\begin{example}\label{ex:limitation-yuan}
Let $\CT=\be_1\otimes\be_1\otimes\be_1\in\R^{2\times 2\times 2}$ %whose only nonzero entry is $t_{111}=1$ 
and $\CS=\be_1\otimes\be_2\otimes\be_2+\be_2\otimes\be_1\otimes\be_2+\be_2\otimes\be_2\otimes\be_1+\be_2\otimes\be_2\otimes\be_2\in\R^{2\times 2\times 2}$. %whose nonzero entries are $s_{122}=s_{212}=s_{221}=s_{222}=1$. 
It is obvious that $\CT\in\TT\bigl((\V_k)_{k=1}^3\bigr)$ and $\CS\in\bigoplus_{|\I|\ge2,\,\I\subseteq[3]}\TT^\I\bigl((\V_k)_{k=1}^3\bigr)$, where $\V_k=\spn(\be_1)\subseteq\R^2$ for $k\in[3]$. However,
    $$3.078\approx\mleft\|\CT+\CS\mright\|_*<\mleft\|\CT\mright\|_*+\mleft\|\CS\mright\|_*\approx 1+3.162.$$
\end{example}
The above nuclear norms are computed by a fully polynomial-time approximation scheme designed in~\cite{hu2022complexity}. In fact, we even have $\|\CT+\CS\|_*<\|\CS\|_*$ in Example~\ref{ex:limitation-yuan}, albeit $\CT$ and $\CS$ sit in two mutually orthogonal subspaces. This is a phenomenon that can never happen in the matrix case.

\subsection{A natural decomposability}

The weak but not full decomposability obviously raises a question on the subspace candidate that is orthogonal to $\TT\bigl((\V_k)_{k=1}^3\bigr)$, i.e., why $\bigoplus_{|\I|\ge2,\,\I\subseteq[3]}\TT^\I\bigl((\V_k)_{k=1}^3\bigr)$ has been chosen in~\eqref{eq:decomposability-yuan}. In the matrix space $\R^{n_1\times n_2}$, the diagonalization provides a very clear picture. Among the four basic subspaces induced by $(\V_1,\V_2)$, i.e., $\TT(\V_1,\V_2)$, $\TT^{\{1\}}(\V_1,\V_2)$, $\TT^{\{2\}}(\V_1,\V_2)$, and $\TT^{\{1,2\}}(\V_1,\V_2)$ in Figure~\ref{fig:matrixsubspaces}, the only candidate is $\TT^{\{1,2\}}(\V_1,\V_2)$.
\begin{figure}[!ht]
% [16]{r}{0.4\linewidth}
    \centering
    % \resizebox{\linewidth}{!}{
\begin{tikzpicture}[xscale=0.5, yscale=0.25]
	\begin{pgfonlayer}{nodelayer}
		\node [style=none] (0) at (-5, 5) {};
		\node [style=none] (1) at (-5, -5) {};
		\node [style=none] (2) at (5, -5) {};
		\node [style=none] (3) at (5, 5) {};
		\node [style=none] (4) at (0, 5) {};
		\node [style=none] (5) at (0, -5) {};
		\node [style=none] (6) at (5, 0) {};
		\node [style=none] (7) at (-5, 0) {};
		\node [style=none] (8) at (-2.5, 2.5) {$\mathbb{T}^{\{1\}}(\mathbb{V}_1,\mathbb{V}_2)$};
		\node [style=none] (9) at (2.5, 2.5) {$\mathbb{T}^{\{1,2\}}(\mathbb{V}_1,\mathbb{V}_2)$};
		\node [style=none] (10) at (2.5, -2.5) {$\mathbb{T}^{\{2\}}(\mathbb{V}_1,\mathbb{V}_2)$};
		\node [style=none] (11) at (-2.5, -2.5) {$\mathbb{T}(\mathbb{V}_1,\mathbb{V}_2)$};
		\node [style=none] (12) at (-5.75, -2.5) {$\mathbb{V}_1$};
		\node [style=none] (13) at (-2.5, -6.25) {$\mathbb{V}_2$};
		\node [style=none] (14) at (2.5, -6.25) {$\mathbb{V}_2^{\perp}$};
		\node [style=none] (15) at (-5.75, 2.5) {$\mathbb{V}_1^{\perp}$};
	\end{pgfonlayer}
	\begin{pgfonlayer}{edgelayer}
		\draw (0.center) to (3.center);
		\draw (3.center) to (2.center);
		\draw (2.center) to (1.center);
		\draw (1.center) to (0.center);
		\draw (5.center) to (4.center);
		\draw (7.center) to (6.center);
	\end{pgfonlayer}
\end{tikzpicture}
    % }
    \caption{The four basic subspaces of $\R^{n_1\times n_2}$.}
    \label{fig:matrixsubspaces}
\end{figure}

In the tensor space $\R^{n_1\times n_2\times n_3}$, however, %the subspace $\bigoplus_{|\I|\ge2,\,\I\subseteq[3]}\TT^\I\bigl((\V_k)_{k=1}^3\bigr)$
\begin{equation*}%\label{eq:direct-sum-decomp-yuan}
\bigoplus_{|\I|\ge2,\,\I\subseteq[3]}\TT^\I\bigl((\V_k)_{k=1}^3\bigr)=
\TT(\V_1^\perp,\V_2^\perp,\V_3)\oplus\TT(\V_1^\perp,\V_2,\V_3^\perp)\oplus\TT(\V_1,\V_2^\perp,\V_3^\perp)\oplus\TT(\V_1^\perp,\V_2^\perp,\V_3^\perp)
\end{equation*}
% $$\bigoplus_{|\I|\ge2,\,\I\subseteq\{1,2,3\}}\TT^\I(\{\V_k\}_{k=1}^3)=
% \TT^{\{1,2\}}(\{\V_k\}_{k=1}^3)\oplus\TT^{\{1,3\}}(\{\V_k\}_{k=1}^3)\oplus\TT^{\{2,3\}}(\{\V_k\}_{k=1}^3)\oplus\TT^{\{1,2,3\}}(\{\V_k\}_{k=1}^3)
% $$
% $$
% \bigoplus_{|\I|\ge2,\,\I\subseteq\{1,2,3\}}\TT^\I\bigl((\V_k)_{k=1}^3\bigr)=
% \TT(\V_1^\perp,\V_2^\perp,\V_3)\oplus\TT(\V_1^\perp,\V_2,\V_3^\perp)\oplus\TT(\V_1,\V_2^\perp,\V_3^\perp)\oplus\TT(\V_1^\perp,\V_2^\perp,\V_3^\perp),
% $$
 % $$\overline{\X}^\perp(\{\V_k\}_{k=1}^3):=\range\mleft(\proj_{\V_1}\otimes \proj_{\V_2^\perp}\otimes\proj_{\V_3^\perp} + \proj_{\V_1^\perp}\otimes \proj_{\V_2}\otimes\proj_{\V_3^\perp} +\proj_{\V_1^\perp}\otimes \proj_{\V_2^\perp}\otimes\proj_{\V_3} +\proj_{\V_1^\perp}\otimes \proj_{\V_2^\perp}\otimes\proj_{\V_3^\perp} \mright).$$
includes four out of the eight basic subspaces defined by $(\V_1,\V_2,\V_3)$, each of which \textcolor{black}{is} spanned by at least two $\V_k^\perp$'s. 
% is merely one of such candidates; in other words, there can be other choices. 
While this candidate has only resulted \textcolor{black}{in} a weak decomposability, it has its own theoretical merits that will be discussed later in Section~\ref{sec:subdiff}. By making a compromise on the subspace size, it turns out that the full decomposability is indeed possible for some other subspaces. We start with the most restrictive candidate, $\TT(\V_1^\perp,\V_2^\perp,\V_3^\perp)$, a natural generalization from the matrix case $\TT(\V_1^\perp,\V_2^\perp)$.
% , as well as a proper generalization of~\cite[(10.12)]{wainwright2019high}.
For comparison, an example of the subspaces of $\R^{n_1\times n_2\times n_3}$ mentioned earlier is offered in Figure~\ref{fig:subspaces}. % It is obvious that $\Y(\{\V_k\}_{k=1}^3)$ exhibits an essential distinction from $\overline{\X}(\{\V_k\}_{k=1}^3)$.

%%%%%%%%%%%%%%%%%%%%%%%%%%%%%%%%%%%%%%%%%%%%%%%%%%%%%%%%%%%%%%%%%%%%%%%%%%%%%%%

\begin{figure}[!h]
    \centering
     \subfloat[\normalsize$\TT\bigl((\V_k)_{k=1}^3\bigr)$]{
     \scalebox{1.2}{
        \begin{tikzpicture}
        \tikzcuboidset{all grids/.style={draw=black,thin,step=.25}}
\pic[thin,black] at (4,0,0) {cuboid};
        \tikzcuboidset{all grids/.style={draw=black,thin,step=.25}}
\pic[thin,black] at (4,1.133,0) {cuboid};
        \tikzcuboidset{all grids/.style={draw=black,thin,step=.25}}
\pic[thin,black] at (5.133,0,0) {cuboid};
        \tikzcuboidset{all grids/.style={draw=black,thin,step=.25}}
\pic[thin,black] at (5.133,1.133,0) {cuboid};
        \tikzcuboidset{all grids/.style={draw=black,thin,step=.25},all faces/.style={fill=black!50}}
\pic[thin,black] at (4,0+0.211,1.42) {cuboid};
        \tikzcuboidset{all grids/.style={draw=black,thin,step=.25}}
\pic[thin,black] at (4,1.133+0.211,1.42) {cuboid};
        \tikzcuboidset{all grids/.style={draw=black,thin,step=.25}}
\pic[thin,black] at (5.133,0+0.211,1.42) {cuboid};
        \tikzcuboidset{all grids/.style={draw=black,thin,step=.25}}
\pic[thin,black] at (5.133,1.133+0.211,1.42) {cuboid};

\node[inner sep=0pt, outer sep=0pt, baseline] at (3.7,1.9,1.42) {\textcolor{black!100}{\scalebox{0.83333333}{\normalsize$\V_1^\perp$}}};
\node[inner sep=0pt, outer sep=0pt, baseline] at (3.7,0.75,1.42) {\textcolor{black!100}{\scalebox{0.83333333}{\normalsize$\V_1$}}};
\node[inner sep=0pt, outer sep=0pt, baseline] at (4.55,0.45,1.42) {\textcolor{white!100}{\scalebox{0.83333333}{\normalsize$\V_2$}}};
\node[inner sep=0pt, outer sep=0pt, baseline] at (5.65,0.46,1.42) {\textcolor{black!100}{\scalebox{0.83333333}{\normalsize$\V_2^\perp$}}};
\node[inner sep=0pt, outer sep=0pt, baseline] at (3.95,2.65,1.42) {\textcolor{black!100}{\scalebox{0.83333333}{\normalsize$\V_3$}}};
\node[inner sep=0pt, outer sep=0pt, baseline] at (4.5,3,1.42) {\textcolor{black!100}{\scalebox{0.83333333}{\normalsize$\V_3^\perp$}}};
\end{tikzpicture}}}
 %    \quad
 %    \subfloat[$\Y\bigl((\V_k)_{k=1}^3\bigr)$]{
 %        \begin{tikzpicture}
	% 		\tikzcuboidset{all grids/.style={draw=black,thin,step=.25},all faces/.style={fill=black!50}}
	% \pic[thin,black] at (4,0,0) {cuboid};
	% 		\tikzcuboidset{all grids/.style={draw=black,thin,step=.25},all faces/.style={fill=black!50}}
	% \pic[thin,black] at (4,1.133,0) {cuboid};
	% 		\tikzcuboidset{all grids/.style={draw=black,thin,step=.25}}
	% \pic[thin,black] at (5.133,0,0) {cuboid};
	% 		\tikzcuboidset{all grids/.style={draw=black,thin,step=.25},all faces/.style={fill=black!50}}
	% \pic[thin,black] at (5.133,1.133,0) {cuboid};
	% 		\tikzcuboidset{all grids/.style={draw=black,thin,step=.25},all faces/.style={fill=black!50}}
	% \pic[thin,black] at (4,0+0.211,1.42) {cuboid};
	% 		\tikzcuboidset{all grids/.style={draw=black,thin,step=.25},all faces/.style={fill=black!50}}
	% \pic[thin,black] at (4,1.133+0.211,1.42) {cuboid};
	% 		\tikzcuboidset{all grids/.style={draw=black,thin,step=.25},all faces/.style={fill=black!50}}
	% \pic[thin,black] at (5.133,0+0.211,1.42) {cuboid};
	% 		\tikzcuboidset{all grids/.style={draw=black,thin,step=.25},all faces/.style={fill=black!50}}
	% \pic[thin,black] at (5.133,1.133+0.211,1.42) {cuboid};
	% % \node at (7,2.75,0) {\textcolor{orange!50}{$\TT(\I^1_1,\I^2_2,\I^3_2)\in\R^{4\times 4\times 4}$}};
	% \end{tikzpicture}}
    \qquad    
    \subfloat[\normalsize$\TT^{\{1,2,3\}}\bigl((\V_k)_{k=1}^3\bigr)$]{
    \scalebox{1.2}{
        \begin{tikzpicture}
			\tikzcuboidset{all grids/.style={draw=black,thin,step=.25}}
	\pic[thin,black] at (4,0,0) {cuboid};
			\tikzcuboidset{all grids/.style={draw=black,thin,step=.25}}
	\pic[thin,black] at (4,1.133,0) {cuboid};
			\tikzcuboidset{all grids/.style={draw=black,thin,step=.25}}
	\pic[thin,black] at (5.133,0,0) {cuboid};
			\tikzcuboidset{all grids/.style={draw=black,thin,step=.25},all faces/.style={fill=black!50}}
	\pic[thin,black] at (5.133,1.133,0) {cuboid};
			\tikzcuboidset{all grids/.style={draw=black,thin,step=.25}}
	\pic[thin,black] at (4,0+0.211,1.42) {cuboid};
			\tikzcuboidset{all grids/.style={draw=black,thin,step=.25}}
	\pic[thin,black] at (4,1.133+0.211,1.42) {cuboid};
			\tikzcuboidset{all grids/.style={draw=black,thin,step=.25}}
	\pic[thin,black] at (5.133,0+0.211,1.42) {cuboid};
			\tikzcuboidset{all grids/.style={draw=black,thin,step=.25}}
	\pic[thin,black] at (5.133,1.133+0.211,1.42) {cuboid};
	% \node at (7,2.75,0) {\textcolor{orange!50}{$\TT(\I^1_1,\I^2_2,\I^3_2)\in\R^{4\times 4\times 4}$}};
	\end{tikzpicture}}}
    \qquad
    \subfloat[\normalsize$\bigoplus_{|\I|\ge2}\TT^\I\bigl((\V_k)_{k=1}^3\bigr)$]{
    \scalebox{1.2}{
        \begin{tikzpicture}
			\tikzcuboidset{all grids/.style={draw=black,thin,step=.25}}
	\pic[thin,black] at (4,0,0) {cuboid};
			\tikzcuboidset{all grids/.style={draw=black,thin,step=.25},all faces/.style={fill=black!50}}
	\pic[thin,black] at (4,1.133,0) {cuboid};
			\tikzcuboidset{all grids/.style={draw=black,thin,step=.25},all faces/.style={fill=black!50}}
	\pic[thin,black] at (5.133,0,0) {cuboid};
			\tikzcuboidset{all grids/.style={draw=black,thin,step=.25},all faces/.style={fill=black!50}}
	\pic[thin,black] at (5.133,1.133,0) {cuboid};
			\tikzcuboidset{all grids/.style={draw=black,thin,step=.25}}
	\pic[thin,black] at (4,0+0.211,1.42) {cuboid};
			\tikzcuboidset{all grids/.style={draw=black,thin,step=.25}}
	\pic[thin,black] at (4,1.133+0.211,1.42) {cuboid};
			\tikzcuboidset{all grids/.style={draw=black,thin,step=.25}}
	\pic[thin,black] at (5.133,0+0.211,1.42) {cuboid};
			\tikzcuboidset{all grids/.style={draw=black,thin,step=.25},all faces/.style={fill=black!50}}
	\pic[thin,black] at (5.133,1.133+0.211,1.42) {cuboid};
	
	% \node at (7,2.75,0) {\textcolor{orange!50}{$\TT(\I^1_1,\I^2_2,\I^3_2)\in\R^{4\times 4\times 4}$}};
	\end{tikzpicture}}}
    \caption[The subspaces.]{Subspaces of $\R^{n_1\times n_2\times n_3}$ presented by shaded blocks
       % $\TT(\{\V_k\}_{k=1}^3)$, $\Y(\{\V_k\}_{k=1}^3)$ and $\Y^\perp(\{\V_k\}_{k=1}^3)$ for third-order tensors, where 
    % $\V_k$ is a subspace of $\R^{n_k}$ for $i=1,2,3$, 
    where each block represents a basic subspace of $\R^{n_1\times n_2\times n_3}$ and the union of all eight blocks represents $\R^{n_1\times n_2\times n_3}$ in each subfigure.}
    \label{fig:subspaces}
    %  considered in~\cite{yuan2016tensor,yuan2017incoherent,raskutti2019convex} for comparative usages.
    %  (as well as the misleading $\overline{\X}$ in~\cite{yuan2016tensor,yuan2017incoherent,raskutti2019convex})
    \end{figure}

%%%%%%%%%%%%%%%%%%%%%%%%%%%%%%%%%%%%%%%%%%%%%%%%%%%%%%%%%%%%%%%%%%%%%%%%%%%%%%%

% We remark that, 

% \begin{lemma}
%     Let $\V\subseteq\R^{n}$ be an arbitrary subspace. It holds that
%     $$
%     \TT_1(\V):=\mleft\{\boldsymbol{X}\in\R_{\operatorname{sym}}^{n\times n}:\operatorname{rowspan}(\boldsymbol{X})\subseteq\V,\,\operatorname{colspan}(\boldsymbol{X})\subseteq\V\mright\}=\spn\mleft(\mleft\{\bv\bv_{\T}:\bv\in\V\mright\}\mright)=:\TT_2(\V).
%     $$
% \end{lemma}
% 
% \begin{proof}
%     We prove the lemma by two items:
%     \begin{itemize}
%         \item `$\subseteq$': Take any $\boldsymbol{X}\in\TT_1(\V)$. Then since $\boldsymbol{X}$ is symmetric, it admits an eigenvalue decomposition $\boldsymbol{X}=\sum_{i=1}^r \lambda_i\bv_k \bv_k^{\T}$, which also implies that $\spn(\{\bv_k\}_{i=1}^r)\subseteq\V$. Therefore $\boldsymbol{X}\in\TT_2(\V)$.

%         \item `$\supseteq$': Take any $\boldsymbol{X}\in\TT_1(\V)$, then we have $\boldsymbol{X}=\sum_{i=1}^m\lambda_i\bv_k \bv_k^{\T}$ for some $m\in\N$ where $\bv_k\in\V$ for all $i$, which is clearly symmetric. Besides, $\operatorname{rowspan}(\boldsymbol{X})=\operatorname{colspan}(\boldsymbol{X})\subseteq\spn(\{\bv_k\}_{i=1}^m)\subseteq\V$. Therefore $\boldsymbol{X}\in\TT_1(\V)$.
%     \end{itemize}
% \end{proof}

% We shall thus generalize $\TT(\V):=\spn\mleft(\mleft\{\bv_{\otimes 3}:\bv\in\V\mright\}\mright)$, or more generally 

The main result in this subsection is as follows.
\begin{theorem}\label{thm:decomposability}
    If $d\ge 2$ and $\V_k$ is a subspace of $\R^{n_k}$ for $k\in[d]$, then the nuclear norm is decomposable over $\TT\bigl((\V_k)_{k=1}^d\bigr)$ and $\TT^{[d]}\bigl((\V_k)_{k=1}^d\bigr)$ of the tensor space $\R^{n_1\times n_2\times\dots\times n_d}$, i.e.,
    \begin{equation*}
        \|\CT+\CS\|_*=\|\CT\|_*+\|\CS\|_* \text{ for any $\CT\in\TT\bigl((\V_k)_{k=1}^d\bigr)$ and $\CS\in\TT^{[d]}\bigl((\V_k)_{k=1}^d\bigr)$}.
    \end{equation*}
\end{theorem}
\begin{proof}
Let $\CT=\sum_{i=1}^{r_1}\lambda_i\bigotimes_{k=1}^d\bx_k^i$ be a nuclear decomposition, where $\bx_k^i\in\V_k\cap\SI^{n_k}$ for $i\in[r_1]$ and $k\in[d]$. The existence of such decomposition is guaranteed by Lemma~\ref{prop:nuclear-U-equiv}. We may further assume that $\lambda_i>0$ for any $i\in[r_1]$ simply by flipping signs and removing zeros. Similarly, let $\CS=\sum_{i=1}^{r_2}\mu_i\bigotimes_{k=1}^d\by_k^i$ be a nuclear decomposition, where $\mu_i>0$ for all $i\in[r_2]$ and $\by_k^i\in\V_k^\perp\cap\SI^{n_k}$ for $i\in[r_2]$ and $k\in[d]$. It then follows by Lemma~\ref{prop:nuclear-U-equiv} again that there exist dual certificates $\CX\in\TT\bigl((\V_k)_{k=1}^d\bigr)$ and $\CY\in\TT\bigl((\V_k^\perp)_{k=1}^d\bigr)$ of $\|\CT\|_*$ and $\|\CS\|_*$, respectively. % with $\|\CX\|_\sigma=\|\CY\|_\sigma=1$, such that $\langle\CX,\bigotimes_{k=1}^d\bx_k^i\rangle=1$ for all $i\in[r_1]$ and $\langle\CY,\bigotimes_{k=1}^d\by_k^i\rangle=1$ for all $i\in[r_2]$. 

    With these preparations, we now show that \begin{equation}\label{eq:TS}
    \CT+\CS=\sum_{i=1}^{r_1}\lambda_i\bigotimes_{k=1}^d\bx_k^i+\sum_{i=1}^{r_2}\mu_i\bigotimes_{k=1}^d\by_k^i
    \end{equation}
    is actually a nuclear decomposition and $\CX+\CY$ is a dual certificate of $\|\CT+\CS\|_*$. It is clear that
      % By again~\cite[Lemma~4.1]{friedland2018nuclear}, it is equivalent to show that there exists some $\CZ$ with $\|\CZ\|_\sigma=1$, s.t.\ $\langle\CZ,\bigotimes_{k=1}^d\bx_k^i\rangle=\langle\CZ,\bigotimes_{k=1}^d\by_k^i\rangle=1$, $\forall\, i\in[r_1],j\in[r_2]$. We claim that letting $\CZ:=\CX+\CY$ fulfills these. 
  \begin{equation} \label{eq:vertical}
  \mleft\langle\CX+\CY,\bigotimes_{k=1}^d\bx_k^i\mright\rangle=1+\mleft\langle\mleft(\bigotimes_{k=1}^d\proj_{\V_k^\perp}\mright)(\CY),\bigotimes_{k=1}^d\bx_k^i\mright\rangle=1+\mleft\langle\CY,\bigotimes_{k=1}^d\proj_{\V_k^\perp}(\bx_k^i)\mright\rangle=1
    \end{equation}
     for all $i\in[r_1]$. For the same reason, we also have $\langle\CX+\CY,\bigotimes_{k=1}^d\by_k^i\rangle=1$ for any $i\in[r_2]$. Hence, 
     $$
\langle\CT+\CS,\CX+\CY\rangle=\mleft\langle\sum_{i=1}^{r_1}\lambda_i\bigotimes_{k=1}^d\bx_k^i+\sum_{i=1}^{r_2}\mu_i\bigotimes_{k=1}^d\by_k^i,\CX+\CY\mright\rangle=\sum_{i=1}^{r_1}\lambda_i+\sum_{i=1}^{r_2}\mu_i.
     $$
     Since $\CX\in\TT\bigl((\V_k)_{k=1}^d\bigr)$ and $\CY\in \TT\bigl((\V_k^\perp)_{k=1}^d\bigr)$, we actually have 
     $\|\CX+\CY\|_\sigma=\max\bigl\{\|\CX\|_\sigma,\|\CY\|_\sigma\bigr\}=1$, due to a result dual to this decomposability; see Theorem~\ref{thm:dualdecomp} to be presented soon. This means that
     $\|\CT+\CS\|_*\ge\langle\CT+\CS,\CX+\CY\rangle=\sum_{i=1}^{r_1}\lambda_i+\sum_{i=1}^{r_2}\mu_i$ by Lemma~\ref{lma:norm-duality}. On the other hand, by combining~\eqref{eq:TS} with~\eqref{eq:nuclear}, the definition of the nuclear norm, we have $\|\CT+\CS\|_*\le \sum_{i=1}^{r_1}\lambda_i+\sum_{i=1}^{r_2}\mu_i$. Therefore, $\|\CT+\CS\|_*=\sum_{i=1}^{r_1}\lambda_i+\sum_{i=1}^{r_2}\mu_i=\|\CT\|_*+\|\CS\|_*$.
\end{proof}

The dual version of Theorem~\ref{thm:decomposability}, i.e., the decomposability of the tensor spectral norm, is of independent interest. In particular, it has supported the above proof.
\begin{theorem}\label{thm:dualdecomp} 
    If $d\ge 2$ and $\V_k$ is a subspace of $\R^{n_k}$ for $k\in[d]$, then the spectral norm is decomposable over $\TT\bigl((\V_k)_{k=1}^d\bigr)$ and $\TT^{[d]}\bigl((\V_k)_{k=1}^d\bigr)$ of the tensor space $\R^{n_1\times n_2\times\dots\times n_d}$, i.e.,
    \begin{equation*}
        \|\CT+\CS\|_\sigma=\max\bigl\{\|\CT\|_\sigma,\|\CS\|_\sigma\bigr\} \text{ for any $\CT\in\TT\bigl((\V_k)_{k=1}^d\bigr)$ and $\CS\in\TT^{[d]}\bigl((\V_k)_{k=1}^d\bigr)$}.
    \end{equation*}
\end{theorem}
\begin{proof}
    Since $\bigl(\bigotimes_{k=1}^d\proj_{\V_k}\bigr)(\CT)=\CT$ and $\bigl(\bigotimes_{k=1}^d\proj_{\V_k^\perp}\bigr)(\CS)=\CS$, we have
    \begin{align*}
    \|\CT+\CS\|_\sigma&=\max_{\bv_k\in\SI^{n_k}\,\forall\,k\in[d]} \mleft\langle\CT+\CS, \bigotimes_{k=1}^d\bv_k\mright\rangle
    \\&=\max_{\bv_k\in\SI^{n_k}\,\forall\,k\in[d]} \mleft(\mleft\langle\mleft(\bigotimes_{k=1}^d\proj_{\V_k}\mright)(\CT), \bigotimes_{k=1}^d\bv_k\mright\rangle+\mleft\langle\mleft(\bigotimes_{k=1}^d\proj_{\V_k^\perp}\mright)(\CS), \bigotimes_{k=1}^d\bv_k\mright\rangle\mright)
    \\&=\max_{\bv_k\in\SI^{n_k}\,\forall\,k\in[d]} \mleft(\mleft\langle\CT, \bigotimes_{k=1}^d\proj_{\V_k}(\bv_k)\mright\rangle+\mleft\langle\CS, \bigotimes_{k=1}^d\proj_{\V_k^\perp}(\bv_k)\mright\rangle\mright)
    \\&\le\max_{\bv_k\in\SI^{n_k}\,\forall\,k\in[d]} \mleft(\|\CT\|_\sigma\prod_{k=1}^d\bigl\|\proj_{\V_k}(\bv_k)\bigr\|_2+\|\CS\|_\sigma\prod_{k=1}^d\bigl\|\proj_{\V_k^\perp}(\bv_k)\bigr\|_2\mright)
    \\&\le\max\bigl\{\|\CT\|_\sigma,\|\CS\|_\sigma\bigr\}\max_{\bv_k\in\SI^{n_k}\,\forall\,k\in[d]} \mleft(\prod_{k=1}^d\bigl\|\proj_{\V_k}(\bv_k)\bigr\|_2+ \prod_{k=1}^d\|\proj_{\V_k^\perp}\bv_k)\|_2\mright)
        \\&\le\max\bigl\{\|\CT\|_\sigma,\|\CS\|_\sigma\bigr\}\max_{\bv_k\in\SI^{n_k}\,\forall\,k\in[2]} \mleft(\prod_{k=1}^2\bigl\|\proj_{\V_k}(\bv_k)\bigr\|_2+\prod_{k=1}^2\bigl\|\proj_{\V_k^\perp}(\bv_k)\bigr\|_2\mright)
    % \\&=\max_{\bv_k\in\SI^{n_k}\,\forall\,k\in[d]} \mleft\{\begin{pmatrix}\|\proj_{\V_1}(\bv_1)\|_2 \\ \|\proj_{\V_1^\perp}(\bv_1)\|_2\end{pmatrix}^{\T}\mleft(\prod_{i=2}^{d-1}\begin{pmatrix}\bigl\|\proj_{\V_k}(\bv_k)\bigr\|_2 & 0 \\ 0 & \bigl\|\proj_{\V_k^\perp}(\bv_k)\bigr\|_2\end{pmatrix}\mright)\begin{pmatrix}\|\proj_{\V_d}(\bv_d)\|_2 \\ \|\proj_{\V_d^\perp}(\bv_d)\|_2\end{pmatrix}\mright\}
    \\&\le\max\bigl\{\|\CT\|_\sigma,\|\CS\|_\sigma\bigr\}\max_{\bv_k\in\SI^{n_k}\,\forall\,k\in[2]} \prod_{k=1}^2\mleft\|\begin{pmatrix}\bigl\|\proj_{\V_k}(\bv_k)\bigr\|_2 \\ \bigl\|\proj_{\V_k^\perp}(\bv_k)\bigr\|_2\end{pmatrix}\mright\|_2\\
    &=\max\bigl\{\|\CT\|_\sigma,\|\CS\|_\sigma\bigr\}.
    \end{align*}

    On the other hand, let $\|\CT\|_\sigma=\langle\CT, \bigotimes_{k=1}^d\bx_k\rangle$ with $\bx_k\in\V_k\cap\SI^{n_k}$ for $k\in[d]$ by Lemma~\ref{thm:spec-subspace}. It is obvious that 
    $$
    \mleft\langle\CS,\bigotimes_{k=1}^d\bx_k\mright\rangle=\mleft\langle\mleft(\bigotimes_{k=1}^d\proj_{\V_k^\perp}\mright)(\CS),\bigotimes_{k=1}^d\bx_k\mright\rangle=\mleft\langle\CS, \bigotimes_{k=1}^d\proj_{\V_k^\perp}(\bx_k)\mright\rangle=0,
    $$
    where the last equality is due to $\bx_k\in\V_k$ for every $k$. This means that
    $
     \|\CT+\CS\|_\sigma\ge \langle \CT+\CS, \bigotimes_{k=1}^d\bx_k\rangle = \|\CT\|_\sigma.    $
    For the same reason, we also have $\|\CT+\CS\|_\sigma\ge\|\CS\|_\sigma$. Therefore, $\|\CT+\CS\|_\sigma\ge\max\bigl\{\|\CT\|_\sigma,\|\CS\|_\sigma\bigr\}$, completing the whole proof.
    \end{proof}

We remark that the matrix case of Theorem~\ref{thm:dualdecomp} 
\textcolor{black}{has been established in~\cite[(2.2.10)]{tropp2015introduction} and as a special case of~\cite[(1.2)]{li2003norm}.}
% has not been seen before, to the best of our knowledge.
Both Theorem~\ref{thm:decomposability} and Theorem~\ref{thm:dualdecomp} include and generalize the matrix case on the decomposability of the nuclear norm and spectral norm.
It is also worth noting that the decomposability in both Theorem~\ref{thm:decomposability} and Theorem~\ref{thm:dualdecomp} is full, albeit for third-order tensors the subspace $\TT^{[3]}\bigl((\V_k)_{k=1}^3\bigr)$ is smaller than $\bigoplus_{|\I|\ge2,\,\I\subseteq[3]}\TT^\I\bigl((\V_k)_{k=1}^3\bigr)$ in~\eqref{eq:decomposability-yuan} where
only a weak decomposability is possible. That being said, the subspace $\TT^{[d]}\bigl((\V_k)_{k=1}^d\bigr)$ is still practically useful in, e.g., analyzing the tensor robust PCA in Section~\ref{sec:TRPCA}.

\subsection{An improved decomposability}

It is natural to ask whether the decomposability in Theorem~\ref{thm:decomposability} can be further improved, in the sense that the two subspaces $\TT\bigl((\V_k)_{k=1}^d\bigr)$ and $\TT^{[d]}\bigl((\V_k)_{k=1}^d\bigr)$ can be enlarged without destroying the full decomposability of the nuclear norm. The answer is indeed affirmative. Let us first recall the notation of $\U$-subspaces,
\begin{align*}
    \U^\I\bigl((\V_k)_{k=1}^d\bigr)&=\bigoplus_{\I\subseteq\J\subseteq[d]}\TT^\J\bigl((\V_k)_{k=1}^d\bigr)
    =\spn\mleft(\mleft\{\bigotimes_{k=1}^d \bv_k:
    %\begin{gathered}
    \bv_k\in \V_k^\perp\,\forall\,k\in\I,\,
    \bv_k\in \R^{n_k}\,\forall\,k\notin\I
   % \end{gathered}
    \mright\}\mright),\\
        \U_\I\bigl((\V_k)_{k=1}^d\bigr)&=\U^\I\bigl((\V_k^\perp)_{k=1}^d\bigr)
    =\spn\mleft(\mleft\{\bigotimes_{k=1}^d \bv_k:
   % \begin{gathered}
    \bv_k\in \V_k\,\forall\,k\in\I,\,
    \bv_k\in \R^{n_k}\,\forall\,k\notin\I
   % \end{gathered}
    \mright\}\mright). 
\end{align*}
%
% define the following subspaces based on an index set $\I\subseteq [d]$,
% \begin{align*}
%     \U^\I(\{\V_k\}_{k=1}^d)&:=\spn\mleft(\mleft\{\bigotimes_{k=1}^d \bv_k:\bv_k\in \V_k\,\forall\,i\in\I,\, \bv_k\in \R^{n_i}\,\forall\,i\notin\I\mright\}\mright), \\ %\mbox{ and }
%     \Y_\I(\{\V_k\}_{k=1}^d)&:=\U^\I(\{\V_k^\perp\}_{k=1}^d).
% \end{align*}
%
% $$
%     \V_k(\I):=\begin{cases}\V_k & i\in\I \\
%     \R^{n_i} & i\notin\I.
%     \end{cases}
% $$
% where for any subspace $\mathbb{Y}\subseteq\R^{n_1\times n_2\times \dots \times n_d}$, $i\in[d]$ and $\I\subseteq [d]$,
% $$
%     \mathcal{F}(\mathbb{Y},i,\I):=\begin{cases}\mathbb{Y} & i\in\I \\
%     \R^{n_i} & i\notin\I.
%     \end{cases}
% $$
%For instance, $\U^{\varnothing}(\{\V_k\}_{k=1}^d)=\R^{n_1\times n_2\times \dots \times n_d}$ and $\U^{[d]}(\{\V_k\}_{k=1}^d)=\TT(\{\V_k\}_{k=1}^d)$ that was already defined in~\eqref{eq:def-M}.
An example of relevant subspaces of $\R^{n_1\times n_2\times n_3}$ for $\I=\{1,2\}$ is presented in Figure~\ref{fig:improved-subspaces}.

\begin{figure}[!h]
    \centering
    \hspace{-0.5cm}
     \subfloat[\normalsize$\TT\bigl((\V_k)_{k=1}^3\bigr)$]{
     \scalebox{1.2}{
        \begin{tikzpicture}
        \tikzcuboidset{all grids/.style={draw=black,thin,step=.25}}
\pic[thin,black] at (4,0,0) {cuboid};
        \tikzcuboidset{all grids/.style={draw=black,thin,step=.25}}
\pic[thin,black] at (4,1.133,0) {cuboid};
        \tikzcuboidset{all grids/.style={draw=black,thin,step=.25}}
\pic[thin,black] at (5.133,0,0) {cuboid};
        \tikzcuboidset{all grids/.style={draw=black,thin,step=.25}}
\pic[thin,black] at (5.133,1.133,0) {cuboid};
        \tikzcuboidset{all grids/.style={draw=black,thin,step=.25},all faces/.style={fill=black!50}}
\pic[thin,black] at (4,0+0.211,1.42) {cuboid};
        \tikzcuboidset{all grids/.style={draw=black,thin,step=.25}}
\pic[thin,black] at (4,1.133+0.211,1.42) {cuboid};
        \tikzcuboidset{all grids/.style={draw=black,thin,step=.25}}
\pic[thin,black] at (5.133,0+0.211,1.42) {cuboid};
        \tikzcuboidset{all grids/.style={draw=black,thin,step=.25}}
\pic[thin,black] at (5.133,1.133+0.211,1.42) {cuboid};

\node[inner sep=0pt, outer sep=0pt, baseline] at (3.7,1.9,1.42) {\textcolor{black!100}{\scalebox{0.83333333}{\normalsize$\V_1^\perp$}}};
\node[inner sep=0pt, outer sep=0pt, baseline] at (3.7,0.75,1.42) {\textcolor{black!100}{\scalebox{0.83333333}{\normalsize$\V_1$}}};
\node[inner sep=0pt, outer sep=0pt, baseline] at (4.55,0.45,1.42) {\textcolor{white!100}{\scalebox{0.83333333}{\normalsize$\V_2$}}};
\node[inner sep=0pt, outer sep=0pt, baseline] at (5.65,0.46,1.42) {\textcolor{black!100}{\scalebox{0.83333333}{\normalsize$\V_2^\perp$}}};
\node[inner sep=0pt, outer sep=0pt, baseline] at (3.95,2.65,1.42) {\textcolor{black!100}{\scalebox{0.83333333}{\normalsize$\V_3$}}};
\node[inner sep=0pt, outer sep=0pt, baseline] at (4.5,3,1.42) {\textcolor{black!100}{\scalebox{0.83333333}{\normalsize$\V_3^\perp$}}};
\end{tikzpicture}}\label{fig:improved-subspaces-(a)}}
    % \quad    
    \hspace{-0.75em}
    \subfloat[\normalsize$\TT^{\{1,2,3\}}\bigl((\V_k)_{k=1}^3\bigr)$]{
    \scalebox{1.2}{
        \begin{tikzpicture}
			\tikzcuboidset{all grids/.style={draw=black,thin,step=.25}}
	\pic[thin,black] at (4,0,0) {cuboid};
			\tikzcuboidset{all grids/.style={draw=black,thin,step=.25}}
	\pic[thin,black] at (4,1.133,0) {cuboid};
			\tikzcuboidset{all grids/.style={draw=black,thin,step=.25}}
	\pic[thin,black] at (5.133,0,0) {cuboid};
			\tikzcuboidset{all grids/.style={draw=black,thin,step=.25},all faces/.style={fill=black!50}}
	\pic[thin,black] at (5.133,1.133,0) {cuboid};
			\tikzcuboidset{all grids/.style={draw=black,thin,step=.25}}
	\pic[thin,black] at (4,0+0.211,1.42) {cuboid};
			\tikzcuboidset{all grids/.style={draw=black,thin,step=.25}}
	\pic[thin,black] at (4,1.133+0.211,1.42) {cuboid};
			\tikzcuboidset{all grids/.style={draw=black,thin,step=.25}}
	\pic[thin,black] at (5.133,0+0.211,1.42) {cuboid};
			\tikzcuboidset{all grids/.style={draw=black,thin,step=.25}}
	\pic[thin,black] at (5.133,1.133+0.211,1.42) {cuboid};
	\end{tikzpicture}}\label{fig:improved-subspaces-(b)}}
    % \quad
    \hspace{-0.75em}
    \subfloat[\normalsize$\U_{\{1,2\}}\bigl((\V_k)_{k=1}^3\bigr)$]{
    \scalebox{1.2}{
        \begin{tikzpicture}
        \tikzcuboidset{all grids/.style={draw=black,thin,step=.25},all faces/.style={fill=black!50}}
\pic[thin,black] at (4,0+0.211,1.42) {cuboid=1--1--2};
        \tikzcuboidset{all grids/.style={draw=black,thin,step=.25}}
\pic[thin,black] at (4,1.133+0.211,1.42) {cuboid=1--1--2};
        \tikzcuboidset{all grids/.style={draw=black,thin,step=.25}}
\pic[thin,black] at (5.133,0+0.211,1.42) {cuboid=1--1--2};
        \tikzcuboidset{all grids/.style={draw=black,thin,step=.25}}
\pic[thin,black] at (5.133,1.133+0.211,1.42) {cuboid=1--1--2};

\node[inner sep=0pt, outer sep=0pt, baseline] at (3.7,1.9,1.42) {\textcolor{black!100}{\scalebox{0.83333333}{\normalsize$\V_1^\perp$}}};
\node[inner sep=0pt, outer sep=0pt, baseline] at (3.7,0.75,1.42) {\textcolor{black!100}{\scalebox{0.83333333}{\normalsize$\V_1$}}};
\node[inner sep=0pt, outer sep=0pt, baseline] at (4.55,0.45,1.42) {\textcolor{white!100}{\scalebox{0.83333333}{\normalsize$\V_2$}}};
\node[inner sep=0pt, outer sep=0pt, baseline] at (5.65,0.46,1.42) {\textcolor{black!100}{\scalebox{0.83333333}{\normalsize$\V_2^\perp$}}};
\node[inner sep=0pt, outer sep=0pt, baseline] at (4.28,2.80,1.42) {\textcolor{black!100}{\scalebox{0.83333333}{\normalsize$\R^{n_3}$}}};
\end{tikzpicture}}\label{fig:improved-subspaces-(c)}}
% \quad
    \hspace{-0.75em}
    \subfloat[\normalsize$\U^{\{1,2\}}\bigl((\V_k)_{k=1}^3\bigr)$]{
    \scalebox{1.2}{
        \begin{tikzpicture}
			\tikzcuboidset{all grids/.style={draw=black,thin,step=.25}}
	\pic[thin,black] at (4,0+0.211,1.42) {cuboid=1--1--2};
			\tikzcuboidset{all grids/.style={draw=black,thin,step=.25}}
	\pic[thin,black] at (4,1.133+0.211,1.42) {cuboid=1--1--2};
			\tikzcuboidset{all grids/.style={draw=black,thin,step=.25}}
	\pic[thin,black] at (5.133,0+0.211,1.42) {cuboid=1--1--2};
			\tikzcuboidset{all grids/.style={draw=black,thin,step=.25},all faces/.style={fill=black!50}}
	\pic[thin,black] at (5.133,1.133+0.211,1.42) {cuboid=1--1--2};
	
	\end{tikzpicture}}\label{fig:improved-subspaces-(d)}}
    \caption[The subspaces.]{
    Subspaces of $\R^{n_1\times n_2\times n_3}$ presented by shaded blocks where the union of all blocks represents $\R^{n_1\times n_2\times n_3}$ in each subfigure.
    }
    \label{fig:improved-subspaces}
    \end{figure}

For third-order tensors, the pair of subspaces shown in (\subref{fig:improved-subspaces-(a)}) and (\subref{fig:improved-subspaces-(b)}) of Figure~\ref{fig:improved-subspaces} and required in the natural decomposability in Theorem~\ref{thm:decomposability} can now be enlarged to a pair of subspaces shown in (\subref{fig:improved-subspaces-(c)}) and (\subref{fig:improved-subspaces-(d)}) of Figure~\ref{fig:improved-subspaces}, respectively. More generally, for higher-order tensors, the following result shows the full decomposability of the tensor nuclear norm over an improved subspace pair from Theorem~\ref{thm:decomposability}.
\begin{theorem}\label{thm:improved-decomposability}
If $\I\subseteq [d]$ with $|\I|\ge 2$ and $\V_k$ is a subspace of $\R^{n_k}$ for $k\in[d]$, then the nuclear norm is decomposable over $\U_\I\bigl((\V_k)_{k=1}^d\bigr)$ and $\U^\I\bigl((\V_k)_{k=1}^d\bigr)$ of the tensor space $\R^{n_1\times n_2\times\dots\times n_d}$, i.e.,
    \begin{equation*}
        \|\CT+\CS\|_*=\|\CT\|_*+\|\CS\|_* \text{ for any $\CT\in\U_\I\bigl((\V_k)_{k=1}^d\bigr)$ and $\CS\in\U^\I\bigl((\V_k)_{k=1}^d\bigr)$}.
    \end{equation*}
\end{theorem}
\textcolor{black}{The proof of Theorem~\ref{thm:improved-decomposability} 
follows a similar structure to that of Theorem~\ref{thm:decomposability} and is 
an immediate consequence of Lemma~\ref{prop:duality-decomp} and Theorem~\ref{thm:spec-decomp} to be introduced soon. We shall include it after the proof of Theorem~\ref{thm:spec-decomp}.}
% The proof of this result will be discussed slightly later.

It is important to remark and in fact easy to observe the monotonicity. The decomposability over $\U_{\I_1}\bigl((\V_k)_{k=1}^d\bigr)$ and $\U^{\I_1}\bigl((\V_k)_{k=1}^d\bigr)$ implies the decomposability over $\U_{\I_2}\bigl((\V_k)_{k=1}^d\bigr)$ and $\U^{\I_2}\bigl((\V_k)_{k=1}^d\bigr)$ if $\I_1\subseteq\I_2$, under which  $\U_{\I_2}\bigl((\V_k)_{k=1}^d\bigr)\subseteq\U_{\I_1}\bigl((\V_k)_{k=1}^d\bigr)$ and $\U^{\I_2}\bigl((\V_k)_{k=1}^d\bigr)\subseteq\U^{\I_1}\bigl((\V_k)_{k=1}^d\bigr)$. 
Therefore, Theorem~\ref{thm:improved-decomposability} is essentially for the strongest case $|\I|=2$ whereas the weakest case is for $\I=[d]$, the result in Theorem~\ref{thm:decomposability}.
% Of course, any $\I\subseteq [d]$ with $|\I|\ge 2$ would also validate the nuclear norm decomposability as they are weaker (e.g., Theorem~\ref{thm:decomposability}), but we would like to remark that these cases have already been included within the above theorem as the subspaces considered therein are the largest possible among all $|\I|\ge 2$.
% In a short word
\textcolor{black}{In short}, the tensor nuclear norm is decomposable over a pair of subspaces if they have at least two disjoint modes. We say that a mode is disjoint for a pair of subspaces if the projections of the subspaces onto the vector space of that mode are orthogonal to each other.

The above point is particularly clear for the matrix case that is included in Theorem~\ref{thm:improved-decomposability}. Among the four basic subspaces $\TT(\V_1,\V_2)$, $\TT^{\{1\}}(\V_1,\V_2)$, $\TT^{\{2\}}(\V_1,\V_2)$, and $\TT^{\{1,2\}}(\V_1,\V_2)$, as well as any direct sum of some basic subspaces, only the pair $\TT(\V_1,\V_2)$ and $\TT^{\{1,2\}}(\V_1,\V_2)$ and the pair $\TT^{\{1\}}(\V_1,\V_2)$ and $\TT^{\{2\}}(\V_1,\V_2)$ have at least two disjoint modes, resulting \textcolor{black}{in} their full decomposability. If a pair of subspaces share only one disjoint mode, such as $\TT(\V_1,\V_2)$ and $\U^{\{2\}}(\V_1,\V_2)$, nothing can be said about the decomposability; see a simple example below. Therefore, the full decomposability of the tensor nuclear norm over $\U_\I\bigl((\V_k)_{k=1}^d\bigr)$ and $\U^\I\bigl((\V_k)_{k=1}^d\bigr)$ is in fact in the maximal sense for any $\I\subseteq [d]$ with $|\I|= 2$.
\begin{example}% [Maximality of $\U_\I\bigl((\V_k)_{k=1}^d\bigr)$ and $\U^\I\bigl((\V_k)_{k=1}^d\bigr)$ when $|\I|= 2$]
\label{ex:local-optimality}
    % Consider the following two matrices in $\R^{2\times 2}$: $\bA_1:=\be_1\otimes \be_1$ and $\bA_2:=\epsilon\cdot\be_1\otimes \be_2 + \be_2\otimes \be_2$, . 
    % % It is clear that the two matrices are exactly inside in a pair of decomposable subspaces as introduced in Theorem~\ref{thm:improved-decomposability}. However, if we slightly perturb $\bA_2$ to be $\bA_2':=\be_1\otimes \be_2 + \be_2\otimes \be_2$, then 
    % It is clear that $\bA_1\in\U^\I(\{\V_k\}_{i=1}^2)$ and $\bA_2\in\Y_\I^\perp(\{\V_k\}_{i=1}^2)$, where $\V_k:=\R\times\{0\}$ for every $i=1,2$ and $\I=\{2\}$. However, it holds by~\cite[Exercise~2.6.3.2]{horn2012matrix} that}
    Let $\BT=\be_1\otimes\be_1\in\R^{2\times 2}$ and $\BS=\epsilon\,\be_1\otimes\be_2\in\R^{2\times 2}$, where $\epsilon>0$ is sufficiently small. It is obvious that $\BT\in\TT(\V_1,\V_2)$ and $\BS\in\TT(\V_1,\V_2^\perp)$, where $\V_1=\V_2=\spn(\be_1)$. 
    However, the inequality
    $$
        \sqrt{1+\epsilon^2} = \mleft\|\BT+\BS\mright\|_* \ge \mleft\|\BT\mright\|_*+\alpha\mleft\|\BS\mright\|_*=1+ \alpha\epsilon=\sqrt{1+2\alpha\epsilon+\alpha^2\epsilon^2}
  $$
  cannot hold for any $\alpha>0$ that is independent of $\epsilon$.
        % However, as
    % $$
    %     \|\BT+\BS\|_*=\sqrt{1+\epsilon^2}\text{ and }\|\BT\|_*+\tau\|\BS\|_*=1+ \tau\epsilon=\sqrt{1+2\tau\epsilon+\tau^2\epsilon^2}\text{ for any $\tau>0$},
    % $$
    % we know the inequality $\|\BT+\BS\|_*\ge\|\BT\|_*+\tau\|\BS\|_*$ fails for any $\tau>(\sqrt{1+\epsilon^2}-1)/\epsilon\in(0,1)$.
        % Consider two matrices $\BT=\begin{pmatrix} 1 & 0\\ 0 & 0 \end{pmatrix}$ and $\BS=\begin{pmatrix} 0 & \epsilon\\ 0 & 1 \end{pmatrix}$, where $\epsilon>0$ is a small number. 
    %  It is obvious that $\BT\in\TT(\V_1,\V_2)$ and $\BS\in\U^{\{2\}}(\V_1,\V_2)=\TT(\R^2,\V_2^\perp)$, where $\V_1=\V_2=\R\times\{0\}$. However, by~\cite[Exercise~2.6.3.2]{horn2012matrix}, one has
    % $$
    %     \mleft\|\BT+\BS\mright\|_*=\sqrt{\frac{\epsilon^2+2-\epsilon\cdot\sqrt{\epsilon^2+4}}{2}}+\sqrt{\frac{\epsilon^2+2+\epsilon\cdot\sqrt{\epsilon^2+4}}{2}}
    %     % \frac{\sqrt{2} \sqrt{3-\sqrt{5}}}{2}+\frac{\sqrt{2} \sqrt{3+\sqrt{5}}}{2}
    %     < 1+\sqrt{1+\epsilon^2}
    %     =\mleft\|\BT\mright\|_*+\mleft\|\BS\mright\|_*.
    % $$
\end{example}
%In this sense, the result of Theorem~\ref{thm:improved-decomposability} is actually tight.
% Although the negative example is designed for the matrix case, it can be easily extended to higher-order tensors by a simple canonical embedding to the tensor space. 
% Any local enlargement of $\U_\I(\{\V_k\}_{k=1}^d)$ and/or $\U^\I(\{\V_k\}_{k=1}^d)$ would destroy the decomposability. Therefore, the conditions in Theorem~\ref{thm:improved-decomposability} are locally optimal.

We would like to remark that Theorem~\ref{thm:improved-decomposability} is also important in tensor analysis. As an immediate application, it helps to single out a large class of tensors whose upper bounds of the nuclear
norm based on tensor partitions (see~\cite[Theorem~3.1]{li2016bounds} and~\cite[Theorem~3.1]{chen2020tensor}) are tight. Specifically, it is stated in~\cite[Theorem~3.1]{li2016bounds} that if a tensor $\CT$ is partitioned into any set of subtensors $\{\CT_1,\allowbreak\CT_2,\allowbreak\dots,\allowbreak\CT_m\}$, then
  \begin{align}
    \bigl\|(\|\CT_1\|_\sigma,\|\CT_2\|_\sigma,\dots,\|\CT_m\|_\sigma)^{\T}\bigr\|_\infty
    &\le \|\CT\|_\sigma \le
    \bigl\|(\|\CT_1\|_\sigma,\|\CT_2\|_\sigma,\dots,\|\CT_m\|_\sigma)^{\T}\bigr\|_2, \label{eq:snorm0}\\
    \bigl\|(\|\CT_1\|_*,\|\CT_2\|_*,\dots,\|\CT_m\|_*)^{\T}\bigr\|_2 &\le \|\CT\|_* \le
    \bigl\|(\|\CT_1\|_*,\|\CT_2\|_*,\dots,\|\CT_m\|_*)^{\T}\bigr\|_1.  \label{eq:nnorm0}
  \end{align}
Therefore, any tensor $\CT=\CT_1+\CT_2$, where $\CT_1\in\U_\I\bigl((\V_k)_{k=1}^d\bigr)$ and $\CT_2\in\U^\I\bigl((\V_k)_{k=1}^d\bigr)$ with $|\I|=2$, immediately becomes a tight example of the upper bound of~\eqref{eq:nnorm0} for $m=2$. The number of blocks, $m$, can be further increased if we further decompose $\CT_1$ and/or $\CT_2$ properly. Besides, this class of tensors already includes all the examples discussed in~\cite[Section~3.2]{li2016bounds} as a proper subset. For the same reason, the dual version of Theorem~\ref{thm:improved-decomposability}, i.e., Theorem~\ref{thm:spec-decomp} to be presented soon, does a similar job to the tightness of the lower bound of~\eqref{eq:snorm0}. 

% The proof of Theorem~\ref{thm:improved-decomposability} follows a similar structure to that of Theorem~\ref{thm:decomposability}. It is an immediate consequence of the following two results, Lemma~\ref{prop:duality-decomp} and Theorem~\ref{thm:spec-decomp}. The former
\textcolor{black}{We are now in a position to state Lemma~\ref{prop:duality-decomp}, which}
can be shown in the same way as that in the proof of Theorem~\ref{thm:decomposability}, and is thus left to interested readers.
% The only slightly tricky point is to verify $\langle\CY,\bigotimes_{j=1}^d\proj_{\mathcal{F}(\V_k^\perp,i,\I)}(\boldsymbol{a}^j_i)\rangle=0$ (in the notations used in that proof), where the structures of the two subspaces come into play. 
\begin{lemma}\label{prop:duality-decomp}
If $d\ge2$,  $\I\subseteq [d]$ with $|\I|\ge 1$, and $\V_k$ is a subspace of $\R^{n_k}$ for $k\in[d]$, then $\textnormal{(i)}\implies\textnormal{(ii)}$, where
\begin{enumerate}[label=\textnormal{(\roman*)}]
\item $\|\CT+\CS\|_\sigma=\max\bigl\{\|\CT\|_\sigma,\|\CS\|_\sigma\bigr\}$ for any $\CT\in \U_\I\bigl((\V_k)_{k=1}^d\bigr)$ and $\CS\in \U^\I\bigl((\V_k)_{k=1}^d\bigr)$;
\item $\|\CT+\CS\|_*=\|\CT\|_*+\|\CS\|_*$ for any $\CT\in\U_\I\bigl((\V_k)_{k=1}^d\bigr)$ and $\CS\in\U^\I\bigl((\V_k)_{k=1}^d\bigr)$.
\end{enumerate}
    % If $\I\subseteq [d]$ is nonempty, then (1) $\rightarrow$ (2), where
    % \begin{enumerate}[label=(\arabic*)]
    %     \item $\|\CT+\CS\|_\sigma=\max\{\|\CT\|_\sigma,\allowbreak\|\CS\|_\sigma\}$ for any $\CT\in\U_\I(\{\V_k\}_{k=1}^d)$ and $\CS\in\U^\I(\{\V_k\}_{k=1}^d))$.
    %     \item $\|\CT+\CS\|_*=\|\CT\|_*+\|\CS\|_*$ for any $\CT\in\U_\I(\{\V_k\}_{k=1}^d)$ and $\CS\in\U^\I(\{\V_k\}_{k=1}^d))$.
    % \end{enumerate}
\end{lemma}

We remark that the actual requirement of Lemma~\ref{prop:duality-decomp} is $|\I|\ge 1$ which is weaker than $|\I|\ge 2$ imposed in both Theorem~\ref{thm:improved-decomposability} and Theorem~\ref{thm:spec-decomp}. The condition $|\I|\ge 1$ is only required to show~\eqref{eq:vertical}. As long as we have one mode $k\in\I$, it suffices to have $\bigotimes_{k=1}^d\proj_{\V_k^\perp}(\bx_k^i)=\CO$ in~\eqref{eq:vertical} since $\bx_k^i\in\V_k$. % $\proj_{\V_k^\perp}(\bx_k^i)={\bf 0}$.
We believe that the reverse implication of Lemma~\ref{prop:duality-decomp} is also true, i.e., the two statements are in fact equivalent. Unfortunately, currently we are unable to verify this claim. %, because we are not aware of any result playing a dual role to~\cite[Lemma~4.1]{friedland2018nuclear} which is used in proving ``(1) $\rightarrow$ (2)''.

Theorem~\ref{thm:spec-decomp} is the dual version of Theorem~\ref{thm:improved-decomposability} and is a generalization of Theorem~\ref{thm:dualdecomp}.
\begin{theorem}\label{thm:spec-decomp}
If $\I\subseteq [d]$ with $|\I|\ge 2$ and $\V_k$ is a subspace of $\R^{n_k}$ for $k\in[d]$, then the spectral norm is decomposable over $\U_\I\bigl((\V_k)_{k=1}^d\bigr)$ and $\U^\I\bigl((\V_k)_{k=1}^d\bigr)$ of the tensor space $\R^{n_1\times n_2\times\dots\times n_d}$, i.e.,
$$\|\CT+\CS\|_\sigma=\max\bigl\{\|\CT\|_\sigma,\|\CS\|_\sigma\bigr\} \text{ for any $\CT\in \U_\I\bigl((\V_k)_{k=1}^d\bigr)$ and $\CS\in \U^\I\bigl((\V_k)_{k=1}^d\bigr)$}.
$$
\end{theorem}
\begin{proof} By monotonicity, it suffices to show the case of $|\I|=2$. Without loss of generality, we assume that $\I=\{1,2\}$. By Lemma~\ref{thm:spec-subspace}, we have
    \begin{align*}
    &\|\CT+\CS\|_\sigma
    \\=&\max_{\bv_k\in\SI^{n_k}\,\forall\,k\in[d]} \mleft(\mleft\langle\CT, \bigotimes_{k=1}^d\bv_k\mright\rangle+\mleft\langle\CS, \bigotimes_{k=1}^d\bv_k\mright\rangle\mright)
    \\=&\max_{\bv_k\in\SI^{n_k}\,\forall\,k\in[d]} \mleft(\mleft\langle\CT, \proj_{\V_1}(\bv_1)\otimes\proj_{\V_2}(\bv_2)\otimes\bigotimes_{k=3}^d\bv_k\mright\rangle+\mleft\langle\CS, \proj_{\V_1^\perp}(\bv_1)\otimes\proj_{\V_2^\perp}(\bv_2)\otimes\bigotimes_{k=3}^d\bv_k\mright\rangle\mright)
    \\=&\max_{\bv_k\in\SI^{n_k}\,\forall\,k\in[d]\setminus\{1\}}\sqrt{\mleft\|\CT\mleft(\bullet,\proj_{\V_2}(\bv_2),\bigotimes_{k=3}^d\bv_k\mright)\mright\|_2^2+\mleft\|\CS\mleft(\bullet,\proj_{\V_2^\perp}(\bv_2),\bigotimes_{k=3}^d\bv_k\mright)\mright\|_2^2}     \\
    =&\max_{\bv_k\in\SI^{n_k}\,\forall\,k\in[d]\setminus\{1\}}\sqrt{\mleft\|\CT\mleft(\bullet,\bullet,\bigotimes_{k=3}^d\bv_k\mright)\mright\|_\sigma^2\cdot\mleft\|\proj_{\V_2}(\bv_2)\mright\|_2^2+\mleft\|\CS\mleft(\bullet,\bullet,\bigotimes_{k=3}^d\bv_k\mright)\mright\|_\sigma^2\cdot\bigl\|\proj_{\V_2^\perp}(\bv_2)\bigr\|_2^2}     \\
    =&\max_{\bv_k\in\SI^{n_k}\,\forall\,k\in[d]\setminus[2]} \sqrt{\max\mleft\{\mleft\|\CT\mleft(\bullet,\bullet,\bigotimes_{k=3}^d\bv_k\mright)\mright\|_\sigma^2,\mleft\|\CS\mleft(\bullet,\bullet,\bigotimes_{k=3}^d\bv_k\mright)\mright\|_\sigma^2\mright\}}
    \\=&\max\bigl\{\|\CT\|_\sigma,\|\CS\|_\sigma\bigr\},
    \end{align*}
    where the third equality is due to Cauchy-Schwarz inequality and the fact that 
    $$
        \CT\mleft(\bullet,\proj_{\V_2}(\bv_2),\bigotimes_{k=3}^d\bv_k\mright)\in\V_1\text{ and }\CS\mleft(\bullet,\proj_{\V_2^\perp}(\bv_2),\bigotimes_{k=3}^d\bv_k\mright)\in\V_1^\perp,
    $$
    the fourth is due to the consistency of the matrix spectral norm, %and the fact that
    % $$
    %     \spn_2\mleft(\CT\mleft(\bullet,\bullet,\bigotimes_{k=3}^d\bv_k\mright)\mright)\in\V_2\text{ and }\spn_2\mleft(\CS\mleft(\bullet,\bullet,\bigotimes_{k=3}^d\bv_k\mright)\mright)\in\V_2^\perp,
    % $$
    the  second to last is due to $\bigl\|\proj_{\V_2}(\bv_2)\bigr\|_2^2+\bigl\|\proj_{\V_2^\perp}(\bv_2)\bigr\|_2^2=1$, 
    and the last is due to that
    $    \max_{\bv_k\in\SI^{n_k}\,\forall\,k\in[d]\setminus[2]} \bigl\|\CT(\bullet,\bullet,\bigotimes_{k=3}^d\bv_k)\bigr\|_\sigma=\|\CT\|_\sigma
    $
    holds for any tensor $\CT$; see~\eqref{eq:specextend}.
\end{proof}
{\color{black}
We are now ready to prove Theorem~\ref{thm:improved-decomposability}.
\begin{myproof}{Theorem~\ref{thm:improved-decomposability}}
    The condition of Theorem~\ref{thm:improved-decomposability} is the same to the condition of Theorem~\ref{thm:spec-decomp} and implies the condition of Lemma~\ref{prop:duality-decomp}. Therefore, by Theorem~\ref{thm:spec-decomp}, the spectral norm is decomposable over $\U_\I\bigl((\V_k)_{k=1}^d\bigr)$ and $\U^\I\bigl((\V_k)_{k=1}^d\bigr)$, which implies that the nuclear norm is decomposable over $\U_\I\bigl((\V_k)_{k=1}^d\bigr)$ and $\U^\I\bigl((\V_k)_{k=1}^d\bigr)$ by Lemma~\ref{prop:duality-decomp}.
\end{myproof}
}

Theorem~\ref{thm:spec-decomp} further implies a generalized decomposability of the nuclear norm.
\begin{theorem}\label{thm:nucearld2}
If $\I\subseteq [d]$ with $|\I|\ge 2$ and $\V_k$ is a subspace of $\R^{n_k}$ for $k\in[d]$, then
% Let $d\ge 2$ and $\V_k$ be a subspace of $\R^{n_k}$ for all $k\in[d]$; let $\I\subseteq [d]$ with $|\I|\ge 2$.
% % If $\V_k$ is a subspace of $\R^{n_k}$ for $k\in[d]$, 
% We have
% and $\I\subseteq [d]$ with $|\I|\ge 2$,
    $$
    \|\CT\|_*\ge \bigl\|\proj_{\U_\I((\V_k)_{k=1}^d)}(\CT)\bigr\|_* + \bigl\|\proj_{\U^\I((\V_k)_{k=1}^d)}(\CT)\bigr\|_*\text{ for any $\CT\in\R^{n_1\times n_2\times \dots \times n_d}$}.
    $$
\end{theorem}
\begin{proof}
Let $\CT=\CT_1+\CT_2+\CT_3$, where $\CT_1=\proj_{\U_\I((\V_k)_{k=1}^d)}(\CT)$ and $\CT_2=\proj_{\U^\I((\V_k)_{k=1}^d)}(\CT)$. As a result, $\CT_1$, $\CT_2$, and $\CT_3$ reside in mutually orthogonal subspaces. By Lemma~\ref{prop:nuclear-U-equiv}, there exist $\CZ_1\in\U_\I\bigl((\V_k)_{k=1}^d\bigr)$ and $\CZ_2\in\U^\I\bigl((\V_k)_{k=1}^d\bigr)$ with $\|\CZ_1\|_\sigma=\|\CZ_2\|_\sigma=1$ such that $\langle \CT_1,\CZ_1\rangle=\|\CT_1\|_*$ and $\langle \CT_2,\CZ_2\rangle=\|\CT_2\|_*$. By Theorem~\ref{thm:spec-decomp}, we also know $\|\CZ_1+\CZ_2\|_\sigma=\max\bigl\{\|\CZ_1\|_\sigma,\|\CZ_2\|_\sigma\bigr\}=1$. These, together with the duality in Lemma~\ref{lma:norm-duality}, further imply that
$$
\|\CT\|_*\ge\langle \CT_1+\CT_2+\CT_3, \CZ_1+\CZ_2 \rangle = \langle \CT_1, \CZ_1\rangle + \langle \CT_2, \CZ_2\rangle = \|\CT_1\|_*+\|\CT_2\|_*,
$$
where the first equality is due to that $\langle \CT_i, \CZ_j\rangle=0$ for any $i\ne j$.
\end{proof}

We remark that Theorem~\ref{thm:nucearld2} includes Theorem~\ref{thm:improved-decomposability} as a special case. Since $\CT_3=\CO$ under the circumstances of Theorem~\ref{thm:improved-decomposability}, we have $\|\CT\|_*\le \|\CT_1\|_*+\|\CT_2\|_*$ by the triangle inequality and so Theorem~\ref{thm:nucearld2} holds as an equality. It is also worth mentioning that Theorem~\ref{thm:nucearld2} applies to any tensor instead of a direct sum of two tensors in mutually orthogonal subspaces. This broadens its applicability. For example, restricting to the matrix case, Theorem~\ref{thm:nucearld2} reduces to
$$
\|\BT\|_*\ge \bigl\|\proj_{\TT(\V_1,\V_2)}(\BT)\bigr\|_* + \bigl\|\proj_{\TT(\V_1^\perp,\V_2^\perp)}(\BT)\bigr\|_*\text{ if $\V_k$ is a subspace of $\R^{n_k}$ for $k\in[2]$}.
$$
We are not aware of such a generalization of the decomposability of the matrix nuclear norm, to the best of our knowledge. For tensor spaces, Theorem~\ref{thm:nucearld2} also provides a new way to bound the tensor nuclear norm from below via the flexibility of $\U_\I\bigl((\V_k)_{k=1}^d\bigr)$ and $\U^\I\bigl((\V_k)_{k=1}^d\bigr)$.

We conclude this section with a discussion of pairs of subspaces used in the decomposability of the tensor nuclear norm. Although not explicitly presented, the subspace $\bigoplus_{|\I|\ge2,\,\I\subseteq[d]}\TT^\I\bigl((\V_k)_{k=1}^d\bigr)$ of $\R^{n_1\times n_2\times \dots \times n_d}$ proposed in~\cite{yuan2017incoherent} actually implies (by a property similar to Lemma~\ref{thm:connection} to be discussed soon) another weak decomposability as that of $\bigoplus_{|\I|\ge2,\,\I\subseteq[3]}\TT^\I\bigl((\V_k)_{k=1}^3\bigr)$ in $\R^{n_1\times n_2\times n_3}$, i.e.,
% \begin{equation*}% \label{eq:subspace-X}
%     \overline{\X}^\perp(\{\V_k\}_{k=1}^d):=\range\mleft(\sum_{j_1< j_2}\mleft(\bigotimes_{i=1}^{j_1-1}\proj_{\V_k}\mright)\otimes\proj_{\V_{j_1}^\perp}\otimes\mleft(\bigotimes_{i=j_1+1}^{j_2-1}\proj_{\V_k}\mright)\otimes\proj_{\V_{j_2}^\perp}\otimes\bigotimes_{i=j_2+1}^{d}(\proj_{\V_k}+\proj_{\V_k^\perp})\mright).
% \end{equation*}
%
% \begin{equation*}\label{eq:subspace-X}
%     \overline{\X}^\perp(\{\V_k\}_{k=1}^d):=\range\mleft( \sum_{|\I|\ge2,\,\I\subseteq[d]} \bigotimes_{i=1}^{d} \proj_{\V_k^\I} \mright) =\spn\mleft( \bigcup_{|\I|\ge2,\,\I\subseteq[d]} \bigotimes_{i=1}^{d} \V_k^\I \mright).
% \end{equation*}
$$
\|\CT+\CS\|_*\ge\|\CT\|_*+\frac{2}{d(d-1)}\|\CS\|_*\text{ for any $\CT\in\TT\bigl((\V_k)_{k=1}^d\bigr)$ and $\CS\in\bigoplus_{|\I|\ge2,\,\I\subseteq[d]}\TT^\I\bigl((\V_k)_{k=1}^d\bigr)$}.
$$
The subspace in which $\CS$ resides includes all the basic subspaces of $\R^{n_1\times n_2\times \dots \times n_d}$ spanned by at least two $\V_k^\perp$'s.
Therefore, this weak decomposability applies to the pair of subspaces, $\TT\bigl((\V_k)_{k=1}^d\bigr)$ and $\bigoplus_{|\I|\ge2,\,\I\subseteq[d]}\TT^{\I}\bigl((\V_k)_{k=1}^d\bigr)$, including $1$ and $2^d-d-1$ basic subspaces, respectively. 
The decomposability becomes (quadratically) weaker as $d$ increases. % as we will see in the discussion of subdifferential soon in Section~\ref{sec:subdiff}. 
By contrast, the natural and restrictive full decomposability in Theorem~\ref{thm:decomposability} applies to $\TT\bigl((\V_k)_{k=1}^d\bigr)$ and $\TT^{[d]}\bigl((\V_k)_{k=1}^d\bigr)$, both being one basic subspace, whereas the improved full decomposability in Theorem~\ref{thm:improved-decomposability} applies to $\U_\I\bigl((\V_k)_{k=1}^d\bigr)$ and $\U^\I\bigl((\V_k)_{k=1}^d\bigr)$, both including $2^{d-|\I|}$ basic subspaces and attaining the maximum $2^{d-2}$ when $|\I|=2$.

\section{Subdifferential of the tensor nuclear norm}\label{sec:subdiff}

The subdifferential of a convex function $f:\R^n\rightarrow\R$ at $\bx\in\R^n$ is defined as
$$
\partial f(\bx):=\bigl\{\bz\in\R^n:f(\by)\ge f(\bx)+\langle \bz, \by-\bx\rangle\,\forall\,\by\in\R^n\bigr\},
$$
and the elements of $\partial f(\bx)$ are called subgradients; see, e.g.,~\cite[Section~23]{rock1997convex} and~\cite[Section~1]{mordukhovich2006variational}.
% ,~\cite[Section~2]{clarke1990optimization},,~\cite[Section~8]{rockafellar2009variational},~\cite{li2020understanding},~\cite[Section~4.3]{cui2021modern} for a broader perspective. 
{\color{black}Applying to the tensor nuclear norm, we have for any $\CT\in\R^{n_1\times n_2\times \dots \times n_d}$ that
$$
    \partial \|\CT\|_*=\bigl\{\CZ\in\R^{n_1\times n_2\times \dots \times n_d}:\|\CY\|_*\ge \|\CT\|_*+\langle \CZ, \CY-\CT\rangle\,\forall\,\CY\in\R^{n_1\times n_2\times \dots \times n_d}\bigr\}.
$$
}As an immediate consequence of~\cite[Corollary~8.25]{rockafellar2009variational}, the subdifferential of a norm $\|\bullet\|_{\diamond}:\R^n\rightarrow\R$ has a representation 
$$\partial\|\bx\|_{\diamond}=\bigl\{\bz\in\R^n:\langle\bz,\bx\rangle=\|\bx\|_{\diamond},\,\|\bz\|_{\circ}\le 1\bigr\},$$ 
where $\|\bullet\|_{\circ}$
%$:=\max\{\langle\bv,\bz\rangle:\|\bv\|_{\diamond}\le1\}$ 
is the dual norm of $\|\bullet\|_{\diamond}$; \textcolor{black}{see (\ref{eq:exact}) for its specialization to the tensor nuclear norm}.
%where $\|\bullet\|_{\circ}$ defined by $\|\bz\|_{\circ}:=\max\{\langle\bv,\bz\rangle:\|\bv\|_{\diamond}\le1\}$ for all $\bz\in\R^n$ is the dual norm of $\|\bullet\|_{\diamond}$. 

It is well known that the subdifferential of the matrix nuclear norm has an explicit characterization. If $\CT=\BU\BD\BV^{\T}\in\R^{n_1\times n_2}$ is a compact SVD, then
\begin{equation} \label{eq:matrixinclu}
\partial\|\BT\|_* =\bigl\{\BU\BV^{\T}+\BX:\BX\in\TT^{\{1,2\}}(\BT), \|\BX\|_\sigma\le 1\bigr\};
\end{equation}
see, e.g.,~\cite[Example~2]{watson1992characterization} 
% and~\cite[Lemma~4.27]{wright2022high}
and~\cite{lewis2005nonsmooth}.
% ~\cite{lewis1995convex,lewis1996derivatives,lewis1999nonsmooth,lewis2003mathematics,lewis2005nonsmooth,lewis2005nonsmooth2}
We remark that $\BU\BV^{\T}$ is invariant over all compact SVDs of $\CT$ due to Autonne's uniqueness~\cite[Theorem~2.6.5]{horn2012matrix}. The importance of this representation has been well recognized in mathematical optimization and statistics.
% Besides, generalizing many interesting statistical learning results from matrices to tensors, such as the low-rank matrix recovery/completion~\cite{recht2010guaranteed,gross2011recovering} and robust PCA~\cite{candes2011robust}, requires an in-depth understanding of the subdifferential of the tensor nuclear norm. 

In the tensor space, however, only two limited subdifferential inclusions of the nuclear norm are known in the literature, to the best of our knowledge. It was first stated in \textcolor{black}{the discussion following}~\cite[Lemma~1]{yuan2016tensor} that for any third-order tensor $\CT\in\R^{n_1\times n_2\times n_3}$, % we have that (recall the definition of $\overline{\X}^\perp(\{\V_k\}_{k=1}^d)$ in (\ref{eq:subspace-X}))
\begin{equation*}\label{eq:subdiff-old}
{\color{black}\overline{\D}_1(\CT):=\mleft\{\CZ+\proj_{\bigoplus_{|\I|\ge2,\,\I\subseteq[3]}\TT^\I(\CT)}(\CX):\CZ\in\Z(\CT),\,\CX\in\R^{n_1\times n_2\times n_3},\,\|\CX\|_\sigma\le \frac{1}{2}\mright\} \subseteq \partial \|\CT\|_*,}
\end{equation*}
% In this paper, rather than 
% whereas $\overline{\D}_1(\CT)$
% i.e.,
% The original version of the subdifferential inclusion in~\cite[Lemma 1]{yuan2016tensor} states that
% \begin{equation*}\label{eq:subdiff-old}
% \overline{\D}_1(\CT):=\mleft\{\CZ+\proj_{\bigoplus_{|\I|\ge2,\,\I\subseteq[3]}\TT^\I(\CT)}(\CX):\CZ\in\Z(\CT),\,\CX\in\R^{n_1\times n_2\times n_3},\,\|\CX\|_\sigma\le \frac{1}{2}\mright\} \subseteq \partial \|\CT\|_*,
% \end{equation*}
% which is slightly different to the one that we structured
% \begin{equation*}\label{eq:subdiff-new}
%    \D_1(\CT)=\mleft\{\CZ+\CX:\CZ\in\Z(\CT),\,\CX\in\bigoplus_{|\I|\ge2,\,\I\subseteq[3]}\TT^\I(\CT)
%     % \range\mleft(\sum_{\I\subseteq[3],\,|\I|\ge 2}\bigotimes_{i=1}^3\mleft(\mathbbm{1}_{i\in\I}(\proj_{\V_k^\perp}-\proj_{\V_k})+\proj_{\V_k}\mright)\mright),
%     ,\,\|\CX\|_\sigma\le \frac{1}{2}\mright\} \subseteq \partial \|\CT\|_*.
%  %   =:\D_1(\CT),
% \end{equation*}
% \begin{equation*}
%    \D_1(\CT)= \mleft\{\CZ+\CX:\CZ\in\Z(\CT),\,\CX\in\bigoplus_{|\I|\ge2,\,\I\subseteq[3]}\TT^\I(\CT),
%     \,\|\CX\|_\sigma\le \frac{1}{2}\mright\} \subseteq \partial \|\CT\|_*.
% \end{equation*}
% In order , 
% $\D_1(\CT)$ , 
% }
where for any $\CT\in\R^{n_1\times n_2\times \dots \times n_d}$,
\begin{equation}\label{eq:z(t)}
    \Z(\CT):=
    \begin{dcases}
        \{\CO\} & \CT=\CO \\
        \bigl\{\CZ\in\TT(\CT):\langle\CZ,\CT\rangle=\|\CT\|_*,\,\|\CZ\|_\sigma=1\bigr\} & \CT\ne\CO.
    \end{dcases}
    %\quad\text{for all $\CT\in\R^{n_1\times n_2\times \dots \times n_d}$};
\end{equation}
% $\CZ\in\TT(\{\V_k\}_{k=1}^3)$ is the tensor for which $\|\CZ\|\in\{0,1\}$ and $\langle\CZ,\CT\rangle=\|\CT\|_*$; 
% {\color{black}
% which, 
% To restrict $\CX\in\bigoplus_{|\I|\ge2,\,\I\subseteq[3]}\TT^\I(\CT)$, $\overline{\D}_1(\CT)$
% asks $\|\CX+\CY\|_\sigma\le\frac{1}{2}$ for some $\CY\in\bigoplus_{|\I|\le1,\,\I\subseteq[3]}\TT^\I(\CT)$, but $\CY$ has nothing to do with the subgradient $\CZ+\CX$ itself.
Later, it was stated in \textcolor{black}{the discussion following}~\cite[Theorem~1]{yuan2017incoherent} that for any $\CT\in\R^{n_1\times n_2\times \dots \times n_d}$,
% (note that the symbol $\overline{\X}$ below is consistent with the one in Figure~\ref{fig:subspaces})
\begin{equation}\label{eq:subdiff-incoherent}
    \D_2(\CT):=\mleft\{\CZ+\CX:\CZ\in\Z(\CT),\,\CX\in\bigoplus_{|\I|\ge2,\,\I\subseteq[d]}\TT^\I(\CT),\,\|\CX\|_\sigma\le \frac{2}{d(d-1)}\mright\} \subseteq \partial \|\CT\|_*.
    %=:\D_2(\CT).
\end{equation}

{\color{black}In this paper, we work with a slightly different variant of $\overline{\D}_1(\CT)$, i.e.,
\begin{equation}\label{eq:subdiff-yuanming}
   \D_1(\CT):= \mleft\{\CZ+\CX:\CZ\in\Z(\CT),\,\CX\in\bigoplus_{|\I|\ge2,\,\I\subseteq[3]}\TT^\I(\CT),\,
    % \range\mleft(\sum_{\I\subseteq[3],\,|\I|\ge 2}\bigotimes_{i=1}^3\mleft(\mathbbm{1}_{i\in\I}(\proj_{\V_k^\perp}-\proj_{\V_k})+\proj_{\V_k}\mright)\mright),
     \|\CX\|_\sigma\le \frac{1}{2}\mright\} \subseteq \partial \|\CT\|_*,
 %   =:\D_1(\CT),
\end{equation}
which makes the component in $\bigoplus_{|\I|\ge2,\,\I\subseteq[3]}\TT^\I(\CT)$ explicit.
% making
% the $\Z(\CT)$-part and $\bigoplus_{|\I|\ge2,\,\I\subseteq[3]}\TT^\I(\CT)$-part explicit.
% which also restricts subgradients in the direct sum of $\Z(\CT)$ and $\bigoplus_{|\I|\ge2,\,\I\subseteq[3]}\TT^\I(\CT)$ but simply asks $\|\CX\|_\sigma\le\frac{1}{2}$.
This is mainly
% The main purpose of this choice is
to 
% align with
match the same convention in the work~\cite{yuan2017incoherent} for the general order $d$, i.e., an explicit part in $\bigoplus_{|\I|\ge2,\,\I\subseteq[d]}\TT^\I(\CT)$ adopted in $\D_2(\CT)$.
%$\CZ+\CX$-split with $\CX$ being constrained directly) adopted  (i.e., $\D_2(\CT)$ below), 
We believe that this formulation is more natural and transparent; see Section~\ref{sec:final-remark} for a further discussion.}
% motivation
% , additional
% and
% , and beyond.
% where
% \begin{equation*}
%     \overline{\X}^\perp:=\range\mleft(\sum_{j_1< j_2}\mleft(\bigotimes_{i=1}^{j_1-1}\proj_{\V_k}\mright)\otimes\proj_{\V_{j_1}^\perp}\otimes\mleft(\bigotimes_{i=j_1+1}^{j_2-1}\proj_{\V_k}\mright)\otimes\proj_{\V_{j_2}^\perp}\otimes\mleft(\bigotimes_{i=j_2+1}^{d}\mleft(\proj_{\V_k}+\proj_{\V_k^\perp}\mright)\mright)\mright),
% \end{equation*}
% where $\CZ$ and $\V_k$ are defined in the same way as above. 
We remark that although $\D_2(\CT)$ applies to tensors of an arbitrary order, it does not include $\D_1(\CT)$ as a special case. In fact, $\D_2(\CT)$ for $d=3$ is strictly smaller than $\D_1(\CT)$. The two inclusions have been applied in tensor completions~\cite{yuan2016tensor,yuan2017incoherent}. It is also worth mentioning that $\Z(\CT)$ consists of exactly one element if $\BT$ is a matrix, as a consequence of Von Neumann's trace inequality~\cite[Theorem~0.1]{rhea2011case}.

\subsection{Decomposability and subdifferential}\label{sec:connect}

To gain a better understanding of the subdifferential of the tensor nuclear norm, let us first establish its connections with the decomposability. This enables us to apply the results developed in Section~\ref{sec:decomp}.

% \begin{lemma} \label{thm:connection}
% % Let $\V_k$ be a subspace of $\R^{n_k}$ for all $k\in[d]$;
% Let $\tau\ge 0$ and $\I\subseteq[d]$ with $|\I|\ge 1$. We have $\textnormal{(i)}\allowbreak\rightarrow\allowbreak\textnormal{(ii)}\allowbreak\rightarrow\allowbreak\textnormal{(iii)}$, where:
% \begin{enumerate}[label=\textnormal{(\roman*)}]
% \item $\|\CZ+\CX\|_\sigma\le\max\{\|\CZ\|_\sigma,\|\CX\|_\sigma/\tau\}$ for all 
% $\CZ\in\TT(\CT)$, 
% $\CX\in\U^\I(\CT)$, and $\CT\in\R^{n_1\times n_2\times \dots \times n_d}$;
% % , and $\V_k\subseteq\mathbb{G}(\R^{n_k})$;
% \item $\{\CZ+\CX:\CZ\in\Z(\CT),\,\CX\in\U^\I(\CT),\,\|\CX\|_\sigma\le \tau\} \subseteq \partial \|\CT\|_*$ for all $\CT\in\R^{n_1\times n_2\times \dots \times n_d}$;
% % for all $\CT$ with $\spn_k(\CT)=\V_k$ and $k\in[d]$;
% \item $\|\CZ+\CX\|_*\ge\|\CZ\|_*+\tau\|\CX\|_*$ for all 
% $\CZ\in\TT(\CT)$, 
% $\CX\in\U^\I(\CT)$, and $\CT\in\R^{n_1\times n_2\times \dots \times n_d}$.
% % , and $\V_k\subseteq\mathbb{G}(\R^{n_k})$.
% \end{enumerate}
% % Here, the set $\mathbb{G}(\R^{n_k})$ consists of all subspaces of $\R^{n_k}$.
% \end{lemma}
\begin{lemma} \label{thm:connection}
If $d\ge2$, $\I\subseteq[d]$ with $|\I|\ge 1$, and $\tau> 0$, then
$\textnormal{(i)}\implies\textnormal{(ii)}\implies\textnormal{(iii)}$, where
\begin{enumerate}[label=\textnormal{(\roman*)}]
\item $\|\CT+\CS\|_\sigma\le\max\bigl\{\|\CT\|_\sigma,\|\CS\|_\sigma/\tau\bigr\}$ for any $\CT\in\R^{n_1\times n_2\times \dots \times n_d}$ and $\CS\in\U^\I(\CT)$;
\item $\bigl\{\CZ+\CX:\CZ\in\Z(\CT),\,\CX\in\U^\I(\CT),\,\|\CX\|_\sigma\le \tau\bigr\} \subseteq \partial \|\CT\|_*$ for any $\CT\in\R^{n_1\times n_2\times \dots \times n_d}$;
\item $\|\CT+\CS\|_*\ge\|\CT\|_*+\tau\|\CS\|_*$ for any $\CT\in\R^{n_1\times n_2\times \dots \times n_d}$ and $\CS\in\U^\I(\CT)$.
\end{enumerate}
\end{lemma}
\begin{proof}
%As $\U^\I(\{\V_k\}_{k=1}^d)$ is the direct sum of some basic subspaces of $\R^{n_1\times n_2\times \dots \times n_d}$ but not of $\TT(\{\V_k\}_{k=1}^d)$,.
Suppose that (i) holds. For any $\CZ+\CX$ in the left-hand-side set in (ii), $\CZ\in\Z(\CT)$ implies that $\U^\I(\CT)\subseteq\U^\I(\CZ)$ and so $\CX\in\U^\I(\CT)\subseteq\U^\I(\CZ)$. As a result, $\|\CZ+\CX\|_\sigma\le\max\bigl\{\|\CZ\|_\sigma,\|\CX\|_\sigma/\tau\bigr\}\le 1$. 
% ; this is because $\TT(\CT)=\TT\bigl((\V_k)_{k=1}^d\bigr)$ and $\U^\I(\CT)=\U^\I\bigl((\V_k)_{k=1}^d\bigr)$. 
Therefore, for any $\CY\in\R^{n_1\times n_2\times \dots \times n_d}$,
\begin{equation}\label{eq:in}
    \langle\CZ+\CX,\CY-\CT\rangle=\langle\CZ+\CX,\CY\rangle-\langle\CZ,\CT\rangle-\langle\CX,\CT\rangle\le\|\CZ+\CX\|_\sigma\|\CY\|_*-\|\CT\|_*\le\|\CY\|_*-\|\CT\|_*,
\end{equation}
where the first inequality is due to Lemma~\ref{lma:norm-duality},~\eqref{eq:z(t)}, and the orthogonality between $\CT\in\TT(\CT)$ and $\CX\in\U^\I(\CT)$. This shows that $\CZ+\CX\in\partial\|\CT\|_*$.

Suppose that (ii) holds.
% , then as $\TT(\CT)=\TT\bigl((\V_k)_{k=1}^d\bigr)$ and $\U^\I(\CT)=\U^\I\bigl((\V_k)_{k=1}^d\bigr)$ as mentioned earlier,
% we notice that $\spn_k(\CT)\subseteq\V_k$ for $k\in[d]$. As a result, $\TT(\CT)\subseteq\TT\bigl((\V_k)_{k=1}^d\bigr)$ and $\U^\I(\CT)\supseteq\U^\I\bigl((\V_k)_{k=1}^d\bigr)$.
For any $\CT\in\R^{n_1\times n_2\times \dots \times n_d}$ and $\CS\in\U^\I(\CT)$, there exists an $\CX\in\U^\I(\CT)$, such that $\langle \CX,\CS\rangle=\|\CS\|_*$ and $\|\CX\|_\sigma=1$ by Lemma~\ref{prop:nuclear-U-equiv}.
% $\CX\in\R^{n_1\times n_2\times \dots \times n_d}$ such that $\langle \CX,\CS\rangle=\|\CS\|_*$ and $\|\CX\|_\sigma\le1$. By Lemma~\ref{thm:spec-subspace}, it is easy to see that $\|\proj_{\U^\I(\{\V_k\}_{k=1}^d)}(\CX)\|_\sigma\le\|\CX\|_\sigma\le 1$. By letting $\CS=\proj_{\U^\I(\{\V_k\}_{k=1}^d)}(\CX)\in\U^\I(\{\V_k\}_{k=1}^d)\subseteq\U^\I(\CT)$, we have $\langle \CS,\CS\rangle=\langle \CX,\CS\rangle=\|\CS\|_*$ and $\|\CS\|_\sigma\le \|\CX\|_\sigma\le1$. 
%hence, we know $\CX\in\U^\I(\CT)$, which further implies that 
As a result, $\CZ+\tau\CX\in\partial\|\CT\|_*$ for any $\CZ\in\Z(\CT)$. Therefore, 
%It then follows by letting $\CY=\CT+\CS$ in~\eqref{eq:in} that
$$
\|\CT+\CS\|_*\ge \|\CT\|_* + \langle\CZ+\tau\CX,\CS\rangle= \|\CT\|_* + \langle\CZ,\CS\rangle +\tau\langle\CX,\CS\rangle= \|\CT\|_* + \tau\|\CS\|_*,
$$
where the last equality is due to the orthogonality between $\CZ\in\Z(\CT)\subseteq\TT(\CT)$ and $\CS\in\U^\I(\CT)$.
% between $\CZ\in\TT(\CT)\subseteq\TT\bigl((\V_k)_{k=1}^d\bigr)$ and $\CS\in\U^\I\bigl((\V_k)_{k=1}^d\bigr)$.
\end{proof}

In a nutshell, a decomposability of the spectral norm implies an inclusion of the subdifferential of the nuclear norm, which in turn implies a decomposability of the nuclear norm. We remark that Lemma~\ref{thm:connection} actually specializes to a universal version of Lemma~\ref{prop:duality-decomp}, i.e., $\textnormal{(i)}\implies\textnormal{(iii)}$ when $\tau=1$.
% We believe that the reverse implications of Lemma~\ref{thm:connection} also hold although we are unable to prove them. This echoes the same situation to that of Lemma~\ref{prop:duality-decomp}. Well, at least the decomposability of the spectral norm is stronger than that of the nuclear norm. 
It is worth noting that when $\tau=1$, the inequalities in (i) and (iii) of Lemma~\ref{thm:connection} actually become equalities, i.e., $\|\CT+\CS\|_\sigma=\max\bigl\{\|\CT\|_\sigma,\|\CS\|_\sigma\bigr\}$ due to Lemma~\ref{thm:spec-subspace} and $\|\CT+\CS\|_*=\|\CT\|_*+\|\CS\|_*$ due to the triangle inequality. As a result of Lemma~\ref{thm:connection} with $\tau=1$, the decomposability of the spectral norm discussed in Section~\ref{sec:decomp}, in particular Theorem~\ref{thm:spec-decomp}, immediately provides new inclusions of the subdifferential of the nuclear norm.
% , as an immediate consequence of Lemma~\ref{thm:connection} with $\tau=1$. 
\begin{corollary}\label{thm:newset}
If $\CT\in\R^{n_1\times n_2\times \dots \times n_d}$ and $\I\subseteq[d]$ with $|\I|\ge 2$, then
\begin{equation}\label{eq:newset}
\D^\I(\CT):=\bigl\{\CZ+\CX: \CZ\in\Z(\CT),\,\CX\in\U^\I(\CT),\,\|\CX\|_\sigma\le 1\bigr\} \subseteq \partial \|\CT\|_*.
\end{equation}
\end{corollary}

When $d=2$, Corollary~\ref{thm:newset} reduces to the maximal inclusion of the subdifferential of the matrix nuclear norm, i.e.,~\eqref{eq:matrixinclu}. It is important to remark that $\U^\I(\CT)$, a properly chosen subspace for $\CX$, has made the full stretch possible for $\CX$ whose spectral norm can be as large as $1$. Moreover, Corollary~\ref{thm:newset} offers the flexibility to choose any $\U^\I(\CT)$ as long as $|\I|\ge2$. As a result, by combining all possible $\I\subseteq [d]$ with $|\I|= 2$ for higher-order tensors, we can have a larger inclusion.
\begin{theorem}\label{thm:subdiff}
    If $d\ge2$ and $\CT\in\R^{n_1\times n_2\times \dots \times n_d}$, then
    % let $\spn_i(\CT):=\spn_i(\CT)$ and $\CZ\in\range\mleft(\bigotimes_{k=1}^d\proj_{\spn_i(\CT)}\mright)$ be the tensor for which $\|\CZ\|\in\{0,1\}$ and $\langle\CZ,\CT\rangle=\|\CT\|_*$. Then, 
    % \begin{equation}
    % \D(\CT) := \conv\mleft\{\CZ+\CX: \CZ\in\Z(\CT),\, \CX\in\bigcup_{|\I|\ge2,\,\I\subseteq[d]}\range\mleft(\bigotimes_{k=1}^d\proj_{\spn_i^\I(\CT)}\mright),\, \|\CX\|_\sigma\le 1\mright\} \subseteq \partial\|\CT\|_*, \label{eq:full-subdiff}
    % \end{equation}
    \begin{align}
    \D(\CT) &:=\conv\mleft(\bigcup_{|\I|=2,\,\I\subseteq[d]}\D^\I(\CT)\mright)\nonumber\\
    &=\mleft\{\CZ+\CX: \CZ\in\Z(\CT),\,\CX\in\conv\mleft(\bigl\{\CY:\|\CY\|_\sigma\le 1\bigr\}\bigcap\bigcup_{|\I|=2,\,\I\subseteq[d]}\U^\I(\CT)\mright)\mright\} \nonumber \\&\subseteq \partial\|\CT\|_*. \label{eq:full-subdiff}
    \end{align}
\end{theorem}
\begin{proof}
Because any subdifferential is convex, the inclusion $\D(\CT)\subseteq \partial\|\CT\|_*$ is immediate by Corollary~\ref{thm:newset}. The second equality can also be easily verified by noticing that $\Z(\CT)$ is convex (see the discussions before Example~\ref{ex:notsingle}) and orthogonal to $\bigcup_{|\I|=2,\,\I\subseteq[d]}\U^\I(\CT)$.
\end{proof}

We remark that $\bigcup_{|\I|=2,\,\I\subseteq[d]}\U^\I(\CT)$ does not include every direction of $\sum_{|\I|=2,\,\I\subseteq[d]}\U^\I(\CT)$. %=\bigoplus_{|\I|\ge2,\,\I\subseteq[d]}\TT^\I(\CT)$. 
However, even when the former is intersected by a spectral ball, its convex hull does contain all directions of the latter, simply because
$$\conv\mleft(\bigcup_{|\I|=2,\,\I\subseteq[d]}\U^\I(\CT)\mright)=\sum_{|\I|=2,\,\I\subseteq[d]}\U^\I(\CT)=\bigoplus_{|\I|\ge2,\,\I\subseteq[d]}\TT^\I(\CT).$$

Before we analyze in detail the subdifferential of the tensor nuclear norm in Section~\ref{sec:detail}, let us discuss the relations among $\D_1(\CT)$, $\D_2(\CT)$, and $\D(\CT)$. All these sets consist of subgradients in the form of $\CZ+\CX$, where $\CZ\in\TT(\CT)$ and $\CX\in\bigoplus_{|\I|\ge2,\,\I\subseteq[d]}\TT^\I(\CT)$. There is no difference in terms of $\CZ$ as it must be chosen exactly from $\Z(\CT)$ and there is also no difference in terms of the direction of $\CX$ available from $\bigoplus_{|\I|\ge2,\,\I\subseteq[d]}\TT^\I(\CT)$. However, the key difference among these three sets is the spectral size of $\CX$, i.e., how large $\|\CX\|_\sigma$ can be. The set $\D(\CT)$ clearly beats $\D_1(\CT)$ and $\D_2(\CT)$ for any $\CX\in\bigcup_{|\I|=2,\,\I\subseteq[d]}\U^\I(\CT)$ where a full stretch is attainable, i.e., $\|\CX\|_\sigma\le1$ instead of $\|\CX\|_\sigma\le\frac{1}{2}$ for $\D_1(\CT)$ and $\|\CX\|_\sigma\le\frac{2}{d(d-1)}$ for $\D_2(\CT)$. As a result, a lot of subgradients have been missed in $\D_1(\CT)$ and $\D_2(\CT)$. Moreover, $\frac{2}{d(d-1)}$ tends to zero as $d$ tends to infinity.
\begin{proposition} 
If $d\ge2$ and $\CT\in\R^{n_1\times n_2\times \dots \times n_d}$, then
\begin{align*}
    \D(\CT)\setminus\D_1(\CT)&\supseteq\mleft\{\CZ+\CX: \CZ\in\Z(\CT),\,  
    \CX\in\bigcup_{|\I|=2,\,\I\subseteq[3]}\U^\I(\CT),\,\frac{1}{2}<\|\CX\|_\sigma\le 1\mright\}\text{ if $d=3$}, \\
    \D(\CT)\setminus\D_2(\CT)&\supseteq\mleft\{\CZ+\CX: \CZ\in\Z(\CT),\,  \CX\in\bigcup_{|\I|=2,\,\I\subseteq[d]}\U^\I(\CT),\,\frac{2}{d(d-1)}<\|\CX\|_\sigma\le 1\mright\}.
\end{align*}
\end{proposition}

When $d=3$, there does exist an $\CX\in\bigoplus_{|\I|\ge2,\,\I\subseteq[3]}\TT^\I(\CT)\setminus\bigcup_{|\I|=2,\,\I\subseteq[3]}\U^\I(\CT)$ such that $\CZ+\CX\in\D(\CT)$ requires $\|\CX\|_\sigma\le\alpha$ for some $\alpha<\frac{1}{2}$, implying that $\D_1(\CT)\setminus \D(\CT)\ne\varnothing$ when $d=3$. To be specific, the tensor $\CX(\frac{1}{3})$ in Example~\ref{ex:yuan3}, adding any $\CZ\in\Z(\CT)$, resides on the boundary of % $\conv\mleft(\mleft\{\CY:\|\CY\|_\sigma\le 1\mright\}\bigcap\bigcup_{|\I|=2,\,\I\subseteq[d]}\U^\I(\CT)\mright)$ in 
$\D(\CT)$ but $\|\CX(\frac{1}{3})\|_\sigma=\frac{2}{3\sqrt{3}}<\frac{1}{2}$.
% as it is easy to check that
% $$
%     \frac{1}{3}\be_1\otimes\be_2\otimes\be_2 +\frac{1}{3} \be_2\otimes\be_1\otimes\be_2 +\frac{1}{3} \be_2\otimes\be_2\otimes\be_1
% $$
% is essentially the unique convex combination representation of $\CX(\frac{1}{3})$ in the convex hull mentioned earlier, yet $\|\CX(\frac{1}{3})\|_\sigma={2\sqrt{3}}/{9}<\frac{1}{2}$. 
However, $\D(\CT)$ always includes $\D_2(\CT)$ as a proper subset for any $d\ge3$. We believe that $\D(\CT)$ is more useful in applications.

\begin{proposition}
If $d\ge2$ and $\CT\in\R^{n_1\times n_2\times \dots\times n_d}$, then $\D_2(\CT)\subseteq\D(\CT)$.
\end{proposition}
\begin{proof}
Let $\CZ+\CX\in\D_2(\CT)$ where $\CZ\in\Z(\CT)$ and $\CX\in\bigoplus_{|\I|\ge2,\,\I\subseteq[d]}\TT^\I(\CT)$ with $\|\CX\|_\sigma\le\frac{2}{d(d-1)}$. It suffices to show that $\CX\in\conv\bigl(\mleft\{\CY:\|\CY\|_\sigma\le 1\mright\}\bigcap\bigcup_{|\I|=2,\,\I\subseteq[d]}\U^\I(\CT)\bigr)$.

For any $\I=\{i,j\}$ with $1\le i<j\le d$, denote $\W^\I(\CT)=\W^\I\bigl((\spn_k(\CT))_{k=1}^d\bigr)$, where
% \begin{align*}
% &\W^\I(\{\V_k\}_{k=1}^d)\\
% :=&\spn\mleft(\mleft\{\bigotimes_{k=1}^d \bv_k:\bv_k\in \V_k\,\forall\,k\in[j-1]\setminus\{i\},\, \bv_k\in \R^{n_k}\,\forall\,k\in[d]\setminus[j],\,\bv_k\in \V_k^\perp\,\forall\,k=i,j\mright\}\mright)
% \end{align*}
    \begin{equation*}
\W^\I\bigl((\V_k)_{k=1}^d\bigr):=\spn\mleft(\mleft\{\bigotimes_{k=1}^d \bv_k:
    \bv_k\in
    \begin{dcases}
        \V_k&k\in[j-1]\setminus\{i\}
        \\
        \R^{n_k}&k\in[d]\setminus[j]
        \\
        \V_k^\perp&k=i,j
    \end{dcases}\mright\}\mright)
    \end{equation*}
for any subspace $\V_k\subseteq\R^{n_k}$ for $k\in[d]$. It is easy to verify that $\bigoplus_{|\I|=2,\,\I\subseteq[d]}\W^\I(\CT)=\bigoplus_{|\I|\ge2,\,\I\subseteq[d]}\TT^\I(\CT)$ and $\W^\I(\CT)\subseteq\U^\I(\CT)$. % for all $\I$ with $|\I|=2$ and $\I\subseteq[d]$; 
Since there are exactly $s=\frac{d(d-1)}{2}$ different index sets $\I$ satisfying $|\I|=2$ and $\I\subseteq[d]$, we may denote them to be $\I_1,\I_2,\dots,\I_s$. As a result, we have $\CX\in\bigoplus_{i=1}^s\W^{\I_i}(\CT)$ and so $\CX$ can be uniquely decomposed as $\sum_{i=1}^{s}\CX_i$, where $\CX_i\in\W^{\I_i}(\CT)$ for $i\in[s]$.

It is important to observe that $\|\CX_i\|_\sigma\le\|\CX\|_\sigma\le \frac{1}{s}$ for any $i\in[s]$. To see why, let $\I_i=\{j_1,j_2\}$ with $j_1<j_2$. By Lemma~\ref{thm:spec-subspace}, there exist $\bv_k\in\SI^{n_k}$ with
    \begin{equation*}
    \bv_k\in%\SI^{n_k}\cap
    \begin{dcases}
        \spn_k(\CT)&k\in[j_2-1]\setminus\{j_1\}
        \\
        \R^{n_k}&k\in[d]\setminus[j_2]
        \\
        \spn_k^\perp(\CT)&k=j_1,j_2
    \end{dcases}
   % \quad\text{for all $k\in[d]$},
    \end{equation*}
for $k\in[d]$, 
% $$
% \bv_k\in \spn_k(\CT)\,\forall\,k\in[j_2-1]\setminus\{j_1\},\, \bv_k\in \R^{n_k}\,\forall\,k\in[d]\setminus[j_2] \mbox{ and } \bv_k\in \spn_k^\perp(\CT) \mbox{ for }k=j_1,j_2,
% $$
such that 
% $\|\CX_i\|_\sigma=\langle \CX_i, \bigotimes_{k=1}^d \bv_k\rangle$, which implies that
$$
\|\CX_i\|_\sigma 
= \mleft\langle \CX_i, \bigotimes_{k=1}^d \bv_k\mright\rangle
= \mleft\langle \proj_{\W^{\I_i}(\CT)}(\CX), \bigotimes_{k=1}^d \bv_k\mright\rangle
= \mleft\langle \CX, \proj_{\W^{\I_i}(\CT)}\mleft(\bigotimes_{k=1}^d \bv_k\mright)\mright\rangle
%= \mleft\langle \CX, \bigotimes_{k=1}^d\bv_k\mright\rangle
\le \|\CX\|_\sigma.
$$
% and the fact that $\proj_{\W^{\I_i}(\CT)}(\CX)=\CX_i$, it is not difficulty to see that $\|\CX_i\|_\sigma\le\|\CX\|_\sigma\le\frac{2}{d(d-1)}$, i.e., $\|\frac{d(d-1)}{2}\CX_i\|_\sigma\le 1$. 
As a result, we have
$$
\CX = \frac{1}{s}\sum_{i=1}^{s} s\CX_i \text{ with } \|s\CX_i\|_{\sigma}\le 1 \text{ and }s\CX_i\in \W^{\I_i}(\CT)\subseteq\U^{\I_i}(\CT) \text{ for } i\in[s],
$$
implying that $\CX\in\conv\bigl(\bigl\{\CY:\|\CY\|_\sigma\le 1\bigr\} \bigcap \bigcup_{|\I|=2,\,\I\subseteq[d]}\U^\I(\CT)\bigr)$.
% this is nothing but a convex combination of tensors in $\{\CY:\|\CY\|_\sigma\le 1\}\bigcap\U^\I(\CT)$ for some $|\I|=2$.
\end{proof}

\subsection{Subgradients in subspaces}\label{sec:detail}

A properly chosen subspace has not only provided better inclusions of the subdifferential of the tensor nuclear norm such as Corollary~\ref{thm:newset} and Theorem~\ref{thm:subdiff}, but also resulted \textcolor{black}{in} the full decomposability of the tensor nuclear norm such as Theorem~\ref{thm:improved-decomposability}. In this part we look into the details of subgradients of the tensor nuclear norm in various subspaces of interest. In particular, we look into the structure of the subdifferential and estimate the bounds of the spectral norm in relevant subspaces, in order to provide a class of tensors that must be subgradients and a class that cannot be. 

Let us first examine the matrix case. Recall that $\partial\|\BT\|_* =\bigl\{\BU\BV^{\T}+\BX:\BX\in\TT^{\{1,2\}}(\BT),\,\|\BX\|_\sigma\le 1\bigr\}$ if $\BT=\BU\BD\BV^{\T}$ is a compact SVD. The projections of any subgradient $\CG\in\partial\|\BT\|_*$ onto the four basic subspaces, $\TT(\BT)$, $\TT^{\{1\}}(\BT)$, $\TT^{\{2\}}(\BT)$, and $\TT^{\{1,2\}}(\BT)$, behave exactly as follows:
\begin{enumerate}[label=\textnormal{(\roman*)}]
    % \item It must be the unique matrix $\BU\BV^{\T}$ in $\TT^{\varnothing}(\BT)$;
    % \item It must be a zero matrix, both in the $\TT^{\{1\}}(\BT)$ and  $\TT^{\{2\}}(\BT)$;
    % \item It can be any matrix $\BX\in\TT^{\{1,2\}}(\CT)$ as long as $\|\BX\|_\sigma\le1$.
    \item $\proj_{\TT(\BT)}(\CG)=\BU\BV^{\T}$ which is unique although the SVD may not be unique;
    \item $\proj_{\TT^{\{1\}}(\BT)}(\CG)=\CO$ and $\proj_{\TT^{\{2\}}(\BT)}(\CG)=\CO$;
    \item $\proj_{\TT^{\{1,2\}}(\BT)}(\CG)$ can be any matrix 
    \textcolor{black}{in $\TT^{\{1,2\}}(\BT)$}
    %in the intersection of $\TT^{\{1,2\}}(\BT)$ and the spectral ball.
    as long as $\|\proj_{\TT^{\{1,2\}}(\BT)}(\CG)\|_\sigma\le1$.
\end{enumerate}
Perhaps it is the last property that motivated the construction of the subspace $\bigoplus_{|\I|\ge2,\,\I\subseteq[d]}\TT^\I(\CT)$ when the subdifferential of the tensor nuclear norm was first studied in~\cite{yuan2016tensor,yuan2017incoherent}. This is indeed a good choice as we will explain soon, but there are more interesting ones, such as $\U^\I(\CT)$ for any $|\I|\ge2$ in Corollary~\ref{thm:newset}. As the first fact we observed, none of the above three properties holds for higher-order tensors. Let us now examine the tensor case.

Given a nonzero $\CT\in\R^{n_1\times n_2\times \dots\times n_d}$, we have by definition that
\begin{align}
\partial\|\CT\|_*&=\bigl\{\CY\in\R^{n_1\times n_2\times \dots\times n_d}:\langle\CT,\CY\rangle=\|\CT\|_*,\,\|\CY\|_\sigma\le 1\bigr\}\nonumber \\ 
&= \mleft\{\CZ+\CX: \CZ\in\Z(\CT),\,\CX\in\bigoplus_{|\I|\ge1,\,\I\subseteq[d]}\TT^\I(\CT),\,\|\CZ+\CX\|_\sigma\le 1\mright\}, \label{eq:exact}
 %\\ &\subseteq\{\CZ+\CX: \CZ\in\Z(\CT), \CX\in \overline{\TT_{111}},\,\|\CX\|_\sigma\le 2\} \label{eq:upper}
\end{align}
where $\Z(\CT)=\bigl\{\CZ\in\TT(\CT):\langle\CZ,\CT\rangle=\|\CT\|_*,\,\|\CZ\|_\sigma=1\bigr\}$ from~\eqref{eq:z(t)}. This natural decomposition of subgradients into $\CZ+\CX$ allows to examine $\CZ$ and $\CX$ separately in two mutually orthogonal subspaces but one needs to make sure $\|\CZ+\CX\|_\sigma\le 1$.
%{(We note in passing that the major difference between (\ref{eq:exact}) and, e.g., (\ref{eq:newset}) is that the former controls the spectral size of $\CZ+\CX$ whereas the latter controls that of $\CX$ instead). 
Let us first consider the set $\Z(\CT)$.
\begin{lemma}\label{lma:convexity-ZT}
    If $\CT\in\R^{n_1\times n_2\times \dots\times n_d}$, then $\Z(\CT)$ is nonempty, compact, and convex.
\end{lemma}
\begin{proof}
It suffices to consider the case $\CT\ne\CO$. The nonemptiness and compactness follow easily from Lemma~\ref{prop:nuclear-U-equiv} and the definition in~\eqref{eq:z(t)}. For the convexity, given any $\CZ_1,\CZ_2\in\Z(\CT)$ and $\alpha_1,\alpha_2\ge0$ with $\alpha_1+\alpha_2=1$, we have
$$
\|\CT\|_*  = \alpha_1 \|\CT\|_* + \alpha_2 \|\CT\|_* = \alpha_1\langle  \CZ_1,\CT\rangle + \alpha_2 \langle\CZ_2,\CT\rangle =\langle \alpha_1 \CZ_1 +\alpha_2\CZ_2,\CT\rangle \le \|\alpha_1 \CZ_1 +\alpha_2\CZ_2\|_\sigma\|\CT\|_*,
$$
implying that $\|\alpha_1 \CZ_1 +\alpha_2\CZ_2\|_\sigma\ge1$. By noticing that $\|\alpha_1 \CZ_1 +\alpha_2\CZ_2\|_\sigma \le\alpha_1\|\CZ_1\|_\sigma + \alpha_2\|\CZ_2\|_\sigma=1$, we in fact have $\|\alpha_1 \CZ_1 +\alpha_2\CZ_2\|_\sigma=1$, implying that $\alpha_1 \CZ_1 +\alpha_2\CZ_2\in\Z(\CT)$.
\end{proof}

Unlike the matrix case, the set $\Z(\CT)$, however, may not be a singleton when $d\ge3$.
\begin{example}%[$|\Z(\CT)|>1$]
\label{ex:notsingle}
% Let $\CE_1=\sum_{i=1}^3\be_i\otimes\be_i\otimes\be_i\in\R^{3\times3\times 3}$ and $\CE_2(t)=\CE_1+t\,\be_1\otimes\be_2\otimes\be_3\in\R^{3\times3\times 3}$. It is obvious that $\CE_2(t)\in \TT(\CT)=\R^{3\times 3\times 3}$. For any $-1\le t\le 1$, it can be verified that $\|\CE_2(t)\|_\sigma=1$ (see Lemma~\ref{thm:E1}) and $\langle \CE_2(t),\CE_1\rangle=3=\|\CT\|_*$. Therefore, $\CE_2(t)\in\Z(\CE_1)$ for any $-1\le t\le 1$.
Let $\CT=\sum_{i=1}^3\be_i\otimes\be_i\otimes\be_i\in\R^{3\times3\times 3}$ and $\CZ(t)=\CT+t\,\be_1\otimes\be_2\otimes\be_3\in\R^{3\times3\times 3}$. It is obvious that $\CZ(t)\in \TT(\CT)=\R^{3\times 3\times 3}$. For any $-1\le t\le 1$, it can be verified that $\bigl\|\CZ(t)\bigr\|_\sigma=1$ (see Lemma~\ref{thm:E1}) and $\bigl\langle \CZ(t),\CT\bigr\rangle=3=\|\CT\|_*$. Therefore, we have $\CZ(t)\in\Z(\CT)$ for any $-1\le t\le 1$.
\end{example}
%We leave the calculation of $\bigl\|\CZ(t)\bigr\|_\sigma$ and $\|\CT\|_*$ in Example~\ref{ex:notsingle} to interested readers.

We next turn to the subspace $\bigoplus_{|\I|\ge1,\,\I\subseteq[d]}\TT^\I(\CT)$ in~\eqref{eq:exact} where $\CX$ resides. Unlike the matrix case where the subspaces $\TT^{\{1\}}(\BT)$ and $\TT^{\{2\}}(\BT)$ are both intangible for any matrix $\BT$, a nonzero $\CX$ is in fact possible in, e.g., $\TT^{\{3\}}(\CT)$, of a third-order tensor $\CT$. 
\begin{example}%[Tangibility of $\TT^{\{1\}}(\CT)$]
\label{ex:1perp}
Let $\CT=\sum_{i=1}^2\be_i\otimes\be_i\otimes\be_i\in\R^{2\times 2\times 3}$. We have $\spn_1(\CT)=\spn_2(\CT)=\R^2$ and $\spn_3(\CT)=\spn(\be_1,\be_2)$. Let $\CZ=\CT\in\Z(\CT)$ since $\langle\CZ,\CT\rangle=2=\|\CT\|_*$ and $\|\CZ\|_\sigma=1$. We have the following two observations in contrast.
\begin{enumerate}[label=\textnormal{(\roman*)}]
    \item Let $\CX(t)=t\,\be_1\otimes\be_2\otimes\be_3\in\TT^{\{3\}}(\CT)$. For any $-1\le t\le 1$, it can be verified that $\bigl\|\CZ+\CX(t)\bigr\|_\sigma=1$ (see Lemma~\ref{thm:E2}). Therefore, $\CZ+\CX(t)\in\partial\|\CT\|_*$ for any $-1\le t\le 1$.
    \item Let $\CY(t)=t\,\be_1\otimes\be_1\otimes\be_3\in\TT^{\{3\}}(\CT)$. For any $t\ne0$, it can be verified that $\bigl\|\CZ+\CY(t)\bigr\|_\sigma>1$ (see Lemma~\ref{thm:E2}). Therefore, $\CZ+\CY(t)\notin\partial\|\CT\|_*$ for any $t\ne0$.
\end{enumerate}
\end{example}
Although the inclusions discussed in Section~\ref{sec:connect} focus on basic subspaces that are spanned by at least two $\spn_k^\perp(\CT)$'s, no matter for $\D_1(\CT)$, $\D_2(\CT)$, or $\D(\CT)$, it is important not to disregard basic subspaces that are spanned by only one $\spn_k^\perp(\CT)$. Another interesting observation is that, for some subspace where $\CX$ resides, e.g., $\bigoplus_{|\I|\ge2,\,\I\subseteq[3]}\TT^\I(\CT)$ in $\D_1(\CT)$, the tensor $\CX$ can surprisingly go beyond the full stretch, i.e., $\|\CX\|_\sigma>1$.
\begin{example}%[$\|\CX\|_\sigma>1$ in (\ref{eq:exact})]
\label{ex:yuan3}
Let $\CT=\be_1\otimes\be_1\otimes\be_1\in\R^{2\times 2\times 2}$. We have $\spn_k(\CT)=\spn(\be_1)$ for $k\in[3]$. Let $\CZ=\CT\allowbreak\in\allowbreak\Z(\CT)$ since $\langle\CZ,\CT\rangle=1=\|\CT\|_*$ and $\|\CZ\|_\sigma=1$. 

Let $\CX(t)=t(\be_1\otimes\be_2\otimes\be_2 + \be_2\otimes\be_1\otimes\be_2 + \be_2\otimes\be_2\otimes\be_1) \in\bigoplus_{|\I|\ge2,\,\I\subseteq[3]}\TT^\I(\CT)$. It can be verified that $\bigl\|\CZ+\CX(t)\bigr\|_\sigma=1$ if and only if $-1\le t\le \frac{1}{2}$ (see Lemma~\ref{thm:E3}). Therefore, $\CZ+\CX(t)\in\partial\|\CT\|_*$ for any $-1\le t\le \frac{1}{2}$. However, it can also be verified that $\bigl\|\CX(t)\bigr\|_\sigma=\frac{2|t|}{\sqrt{3}}$ (see Lemma~\ref{thm:E3}), in particular, $\|\CX(-1)\|_\sigma=\frac{2}{\sqrt{3}}>1$. Moreover, the allowed stretches of $\CX(t)$ are different along two opposite directions, positive $t$ making at most $\|\CX(\frac{1}{2})\|_\sigma=\frac{1}{\sqrt{3}}$ while negative $t$ leading to even $\|\CX(-1)\|_\sigma=\frac{2}{\sqrt{3}}$.
\end{example}

The above examples make the subdifferential of the tensor nuclear norm much more complicated and interesting. From what we have observed, the basic subspace $\TT^{[d]}(\CT)$ resembles the most to $\TT^{\{1,2\}}(\CT)$ of the matrix space, evidenced by not only the full decomposability in Theorem~\ref{thm:decomposability} but also the inclusion of the subdifferential in Corollary~\ref{thm:newset} for $\I=[d]$. However, the above examples call attention that every $\CX\in\bigoplus_{|\I|\ge1,\,\I\subseteq[d]}\TT^\I(\CT)$ should be taken care of. The rest of this subsection is focused on the size of $\|\CX\|_\sigma$ in these subspaces.

Given a tensor $\CT\in\R^{n_1\times n_2\times \dots \times n_d}$ and a nonempty $\J\subseteq 2^{[d]}\setminus\varnothing$, it defines a subspace $\X^\J(\CT):=\bigoplus_{\I\in\J}\TT^\I(\CT)$, a direct sum of $|\J|$ basic subspaces of $\R^{n_1\times n_2\times \dots \times n_d}$ defined by $(\spn_k(\CT))_{k=1}^d$. It is also obvious that $\X^\J(\CT)$ is orthogonal to $\TT(\CT)=\X^{\{\varnothing\}}(\CT)$ since $\varnothing\notin\J$. Let us define
\begin{align}
    \underline{\tau}\bigl(\X^\J(\CT)\bigr)&:=\max\bigl\{t:\|\CZ+\CX\|_\sigma\le 1 \,\forall\, \CZ\in\Z(\CT),\, \CX\in\X^\J(\CT),\, \|\CX\|_\sigma\le t\bigr\},\nonumber\\
    \overline{\tau}\bigl(\X^\J(\CT)\bigr)&:=\max\bigl\{\|\CX\|_\sigma:\CZ\in\Z(\CT),\,\CX\in\X^\J(\CT),\,\|\CZ+\CX\|_\sigma\le 1\bigr\}, \label{eq:tau_upper}
\end{align}
and further
\begin{align*}
    \underline{\tau}(\X^\J)&:=\min\bigl\{\underline{\tau}\bigl(\X^\J(\CT)\bigr):\CT\in\R^{n_1\times n_2\times \dots \times n_d}\bigr\},\\
    \overline{\tau}(\X^\J)&:=\max\bigl\{\overline{\tau}\bigl(\X^\J(\CT)\bigr):\CT\in\R^{n_1\times n_2\times \dots \times n_d}\bigr\}.
\end{align*}
In other words, $\underline{\tau}(\X^\J)$ is the tight lower bound of $\|\CX\|_\sigma$ and $\overline{\tau}(\X^\J)$  is the tight upper bound,
% and $\overline{\tau}(\X^\J)$ are respectively the tight and upper bounds of $\|\CX\|_\sigma$, 
in the sense that for any $\CT\in\R^{n_1\times n_2\times \dots \times n_d}$,
\begin{align*}
%(\TT(\CT)\oplus\X^\J(\CT))\cap
\partial\|\CT\|_*&\supseteq\bigl\{\CZ+\CX: \CZ\in\Z(\CT),\,\CX\in\X^\J(\CT),\,\|\CX\|_\sigma\le \underline{\tau}(\X^\J)\bigr\}, %\mbox{ for any } \CT\in\R^{n_1\times n_2\times \dots \times n_d}, 
\\ (\TT(\CT)\oplus\X^\J(\CT))\cap\partial\|\CT\|_*&\subseteq \bigl\{\CZ+\CX: \CZ\in\Z(\CT),\,\CX\in\X^\J(\CT),\,\|\CX\|_\sigma\le \overline{\tau}(\X^\J)\bigr\}. % \mbox{ for any } \CT\in\R^{n_1\times n_2\times \dots \times n_d}.
\end{align*}

As an example, the inclusion $\D_1(\CT)\subseteq\partial\|\CT\|_*$ in~\eqref{eq:subdiff-yuanming} implies that
$\underline{\tau}\mleft(\bigoplus_{|\I|\ge2,\,\I\subseteq[3]}\TT^\I(\CT)\mright)\ge\frac{1}{2}$ for any $\CT\in\R^{n_1\times n_2\times n_3}$, or simply $\underline{\tau}\mleft(\bigoplus_{|\I|\ge2,\,\I\subseteq[3]}\TT^\I\mright)\ge\frac{1}{2}$. The inclusion $\D_2(\CT)\subseteq\partial\|\CT\|_*$ in~\eqref{eq:subdiff-incoherent} implies that
$\underline{\tau}\mleft(\bigoplus_{|\I|\ge2,\,\I\subseteq[d]}\TT^\I\mright)\ge\frac{2}{d(d-1)}$. Although $\D(\CT)$ includes $\D_2(\CT)$ as a proper subset, it does not help to increase the lower bound of $\underline{\tau}\mleft(\bigoplus_{|\I|\ge2,\,\I\subseteq[d]}\TT^\I\mright)$. However, Corollary~\ref{thm:newset} and Theorem~\ref{thm:spec-decomp} actually imply that $\underline{\tau}(\U^\I)=\overline{\tau}(\U^\I)=1$ for any $\I\subseteq [d]$ with $|\I|\ge 2$. 

%It is worth noting that by definition, we have $\underline{\tau}\bigl(\X^\J(\CT)\bigr)\le\overline{\tau}\bigl(\X^\J(\CT)\bigr)$, and so $\underline{\tau}(\X^\J)\le\overline{\tau}(\X^\J)$. As an aside,
Some obvious monotonicity holds for the two bounds. If $\J_1\subseteq\J_2$, then $\X^{\J_1}\subseteq\X^{\J_2}$ and so
$$
\underline{\tau}(\X^{\J_1})\ge \underline{\tau}(\X^{\J_2}) \text{ and } \overline{\tau}(\X^{\J_1}) \le \overline{\tau}(\X^{\J_2}).
$$
Moreover, %there is also a monotonicity in the order $d$: 
for fixed $\J$, increasing the order $d$ will decrease $\underline{\tau}(\X^{\J})$ but increase $\overline{\tau}(\X^{\J})$ in the weak sense, the same to increasing any dimension $n_k$ while fixing the others.

We now present the main results on the two bounds for various tensor subspaces of interest.
\begin{theorem}\label{thm:bounds}
In the tensor space $\R^{n_1\times n_2\times \dots \times n_d}$ of order $d\ge3$,
% with $n_1\ge3$,
% with $\max_{1\le k\le d}{n_k}\ge 3$,
\begin{enumerate}[label=\textnormal{(\roman*)}]
    \item $\underline{\tau}(\TT^\I)=0$ and $\overline{\tau}(\TT^\I)=1$ for any $\I\subseteq [d]$ with $|\I|=1$; 
    
  %  {\color{black}$\underline{\tau}(\TT^{\{k\}})=0$ and $\overline{\tau}(\TT^{\{k\}})=\min\{1,n_k-2\}$ for any $k\in[d]$}
    
    \item $\underline{\tau}(\TT^\I)=\overline{\tau}(\TT^\I)=1$ for any $\I\subseteq [d]$ with $|\I|\ge 2$;
    
    \item $\underline{\tau}(\U^\I)=0$ and $\overline{\tau}(\U^\I)=1$ for any $\I\subseteq [d]$ with $|\I|=1$;
    
  %  {\color{black}$\underline{\tau}(\U^{\{k\}})=0$ and $\overline{\tau}(\U^{\{k\}})=\min\{1,n_k-2\}$ for any $k\in[d]$}
    
    \item $\underline{\tau}(\U^\I)=\overline{\tau}(\U^\I)=1$ for any $\I\subseteq [d]$ with $|\I|\ge 2$;
    
    \item \mbox{$\frac{1}{2}\le \underline{\tau}\mleft(\bigoplus_{|\I|\ge2,\,\I\subseteq[3]}\TT^\I\mright)\le\frac{1}{\sqrt{3}}\approx 0.577$ and $1.155\approx\frac{2}{\sqrt{3}}\le \overline{\tau}\mleft(\bigoplus_{|\I|\ge2,\,\I\subseteq[3]}\TT^\I\mright)\le\frac{1+\sqrt{2}}{2}\approx 1.207$};
    
    \item $\underline{\tau}\mleft(\bigoplus_{|\I|\ge1,\,\I\subseteq[3]}\TT^\I\mright)=0$ and $1.155\approx\frac{2}{\sqrt{3}}\le\overline{\tau}\mleft(\bigoplus_{|\I|\ge1,\,\I\subseteq[3]}\TT^\I\mright)\le\frac{3}{2}$;
    
    \item \mbox{$\frac{1}{3}\le \underline{\tau}\mleft(\bigoplus_{|\I|\ge2,\,\I\subseteq[4]}\TT^\I\mright)\le\frac{1}{2}$ and $1.207\approx\frac{1+\sqrt{2}}{2}\le \overline{\tau}\mleft(\bigoplus_{|\I|\ge2,\,\I\subseteq[4]}\TT^\I\mright)\le\frac{1+\sqrt{3}}{2}\approx 1.366$};
    
    \item $\underline{\tau}\mleft(\bigoplus_{|\I|\ge1,\,\I\subseteq[4]}\TT^\I\mright)=0$ and $1.207\approx\frac{1+\sqrt{2}}{2}\le\overline{\tau}\mleft(\bigoplus_{|\I|\ge1,\,\I\subseteq[4]}\TT^\I\mright)\le \frac{8}{5}$;

    \item $\frac{2}{d(d-1)}\le\underline{\tau}\mleft(\bigoplus_{|\I|\ge2,\,\I\subseteq[d]}\TT^\I\mright)\le\frac{1}{2}$ and $1.207 \approx\frac{1+\sqrt{2}}{2}\le \overline{\tau}\mleft(\bigoplus_{|\I|\ge2,\,\I\subseteq[d]}\TT^\I\mright)\le 2$ for $d\ge 5$;
        
    \item $\underline{\tau}\mleft(\bigoplus_{|\I|\ge1,\,\I\subseteq[d]}\TT^\I\mright)=0$ and $1.207 \approx\frac{1+\sqrt{2}}{2}\le \overline{\tau}\mleft(\bigoplus_{|\I|\ge1,\,\I\subseteq[d]}\TT^\I\mright)\le 2$ for $d\ge 5$.
\end{enumerate}
\end{theorem}
\begin{proof}
The proof is given case by case. As a general notation, $\CT$ is an arbitrary tensor or a specific example where $\underline{\tau}\bigl(\X^\J(\CT)\bigr)$ or $\overline{\tau}\bigl(\X^\J(\CT)\bigr)$ is concerned with and $\CZ\in\Z(\CT)$.
\begin{enumerate}[label=\textnormal{(\roman*)}]
\item
%  By canonically embedding the tensors in the orbit of $\CZ+\CY(t)$ in Example~\ref{ex:1perp} under the action of the permutation group (on indices) in $\R^{n_1\times n_2\times \dots \times n_d}$, we have $\underline{\tau}(\TT^{\I})\le 0$ for all $\I\subseteq [d]$ with $|\I|=1$; however, as $\underline{\tau}(\TT^{\I})\ge 0$ by definition, it follows that $\underline{\tau}(\TT^{\I})= 0$ for all $\I\subseteq [d]$ with $|\I|=1$, as desired. 
% In a similar vein, by embedding the orbit of $\CZ+\CX(t)$ in Example~\ref{ex:1perp}, we have $\overline{\tau}(\TT^{\I})\ge1$ for all $\I\subseteq [d]$ with $|\I|=1$; however, by Lemma~\ref{thm:spec-subspace},
% % and the requirement of a subgradient~\eqref{eq:exact}, 
% we also have for any $\CT$ that $\|\CX\|_\sigma\le\|\CZ+\CX\|_\sigma\le1$ for all $\CZ\in\Z(\CT)$ and $\CX\in\TT^{\I}(\CT)$ with $\|\CZ+\CX\|_\sigma\le1$, which further implies that $\overline{\tau}(\TT^{\I})\le1$, and so $\overline{\tau}(\TT^{\I})=1$ for all $\I\subseteq [d]$ with $|\I|=1$, as desired.

When $d=3$, the tensor $\CZ+\CY(t)\in\R^{2\times 2\times 3}$ in Example~\ref{ex:1perp} shows that $\underline{\tau}(\TT^{\{3\}})\le0$, implying that $\underline{\tau}(\TT^{\{3\}})=0$ as it cannot be negative. %Similar examples can show that $\underline{\tau}(\TT^{\I})=0$ for any $|\I|=1$. 
By the monotonicity with respect to $d$ and $n_k$, we have $\underline{\tau}(\TT^{\I})=0$ for any $\I\subseteq [d]$ with $|\I|=1$. Similarly, the tensor $\CZ+\CX(t)\in\R^{2\times 2\times 3}$ in Example~\ref{ex:1perp} shows that $\overline{\tau}(\TT^{\{3\}})\ge1$. %Similar examples can show that $\overline{\tau}(\TT^{\I})\ge1$ for any $|\I|=1$. 
By the monotonicity with respect to $d$ and $n_k$, we have $\overline{\tau}(\TT^{\I})\ge1$ for any $\I\subseteq [d]$ with $|\I|=1$. By Lemma~\ref{thm:spec-subspace} and the requirement of a subgradient in~\eqref{eq:exact}, we have $\|\CX\|_\sigma\le\|\CZ+\CX\|_\sigma\le1$ for any $\CZ\in\Z(\CT)$ and $\CX\in\TT^{\I}(\CT)$ with $\|\CZ+\CX\|_\sigma\le1$. This means that $\overline{\tau}(\TT^{\I})\le1$, implying that $\overline{\tau}(\TT^{\I})=1$.

\item Since $\TT^\I(\CT)\subseteq\U^\I(\CT)$ for any $\CT$, the monotonicity with respect to $\J$ of $\X^\J$ implies that $\underline{\tau}(\U^\I)\le\underline{\tau}(\TT^\I)\le \overline{\tau}(\TT^\I)\le \overline{\tau}(\U^\I)$. The results then follow immediately from (iv).

\item Since $\TT^\I(\CT)\subseteq\U^\I(\CT)$ for any $\CT$, the monotonicity with respect to $\J$ of $\X^\J$ implies that
$\underline{\tau}(\U^\I)\le\underline{\tau}(\TT^\I)=0$ and $\overline{\tau}(\U^\I)\ge\overline{\tau}(\TT^\I)=1$ by (i). Therefore, $\underline{\tau}(\U^\I)= 0$ as it cannot be negative. 
% however, as $\underline{\tau}(\U^\I)\ge 0$ by definition, it follows that $\underline{\tau}(\U^\I)=0$ for all $\I\subseteq [d]$ with $|\I|=1$, as desired. 
On the other hand, by Lemma~\ref{thm:spec-subspace} and the requirement of a subgradient in~\eqref{eq:exact}, we have $\|\CX\|_\sigma\le\|\CZ+\CX\|_\sigma\le1$ for any $\CZ\in\Z(\CT)$ and $\CX\in\U^{\I}(\CT)$ with $\|\CZ+\CX\|_\sigma\le1$. This means that $\overline{\tau}(\U^{\I})\le1$, implying that $\overline{\tau}(\U^{\I})=1$.
% and the requirement of a subgradient~\eqref{eq:exact}, 
% we have for any $\CT$ that $\|\CX\|_\sigma\le\|\CZ+\CX\|_\sigma\le1$ for all $\CZ\in\Z(\CT)$ and $\CX\in\U^{\I}(\CT)$ with $\|\CZ+\CX\|_\sigma\le1$; this further implies that $\overline{\tau}(\U^{\I})\le1$, and so $\overline{\tau}(\U^{\I})=1$ for all $\I\subseteq [d]$ with $|\I|=1$, as desired.

\item This is an immediate consequence of Corollary~\ref{thm:newset} and Theorem~\ref{thm:spec-decomp}.

\item $\underline{\tau}\mleft(\bigoplus_{|\I|\ge2,\,\I\subseteq[3]}\TT^\I\mright)\ge\frac{1}{2}$ is due to the inclusion $\D_1(\CT)\subseteq\partial\|\CT\|_*$; see also~\cite[Lemma~1]{yuan2016tensor}. $\underline{\tau}\mleft(\bigoplus_{|\I|\ge2,\,\I\subseteq[3]}\TT^\I\mright)\le\frac{1}{\sqrt{3}}$ and $\overline{\tau}\mleft(\bigoplus_{|\I|\ge2,\,\I\subseteq[3]}\TT^\I\mright)\ge\frac{2}{\sqrt{3}}$ are due to $\CZ+\CX(\frac{1}{2})$ and $\CZ+\CX(-1)$ in Example~\ref{ex:yuan3}, respectively.
To upper bound $\overline{\tau}\mleft(\bigoplus_{|\I|\ge2,\,\I\subseteq[3]}\TT^\I\mright)$, we notice that  for any $\CT$,
$$\bigoplus_{|\I|\ge2,\,\I\subseteq[3]}\TT^\I(\CT)\subseteq\TT^{\{2,3\}}(\CT)\oplus\U^{\{1\}}(\CT).$$ 
By the monotonicity and Lemma~\ref{thm:bridge},
\begin{align}
&\,\overline{\tau}\mleft(\bigoplus_{|\I|\ge2,\,\I\subseteq[3]}\TT^\I\mright) \nonumber\\%
\le&\,\max\mleft\{\|\CX_1+\CX_2\|_\sigma:\|\CZ+\CX_1+\CX_2\|_\sigma\le 1,\,\CZ\in\TT(\CT),\,\CX_1\in\TT^{\{2,3\}}(\CT),\,\CX_2\in\U^{\{1\}}(\CT)\mright\} \nonumber\\%\label{eq:top}\\
\le&\,\max\mleft\{x_1y_2z_2+x_2: x_1y_2z_2+x_2\le 1+x_1y_1z_1,\, %\|\bx\|_2=\|\by\|_2=\|\bz\|_2=1,\, \bx,\by,\bz \ge{\bf 0}
\bx,\by,\bz\in\SI^2\cap\R_{+}^2\mright\}\nonumber\\%\label{eq:sop}\\
=&\,\frac{1+\sqrt{2}}{2}, \nonumber
\end{align}
where the last equality is computed in Lemma~\ref{thm:optimization1}.
% \begin{align*}
% &\overline{\tau}\mleft(\bigoplus_{|\I|\ge2,\,\I\subseteq[3]}\TT^\I\mright)\\
% \le{}&\max_{\CT\in\R^{n_1\times n_2\times n_3}}\max\mleft\{\|\CX_1+\CX_2\|_\sigma:
% \begin{gathered}
% (\CZ,\CX_1,\CX_2)\in\TT(\CT)\times\TT^{\{2,3\}}(\CT)\times\U^{\{1\}}(\CT)\\
% \text{with}~\|\CZ+\CX_1+\CX_2\|_\sigma\le 1
% \end{gathered}
% \mright\}\\
% \le{}&\max\mleft\{x_1y_2z_2+x_2:
% \begin{gathered}
% x_1y_2z_2+x_2\le 1+x_1y_1z_1\\
% \bx,\by,\bz\in\SI^2\cap\R_{+}^2
% % \|\bx\|_2=\|\by\|_2=\|\bz\|_2=1,\, \bx,\by,\bz \ge{\bf 0}
% \end{gathered}
% \mright\}\\
% ={}&\frac{1+\sqrt{2}}{2};
% \end{align*}

\item Since $\underline{\tau}(\TT^\I)=0$ for any $\I\subseteq[3]$ with $|\I|=1$ from (i), $\underline{\tau}\mleft(\bigoplus_{|\I|\ge1,\,\I\subseteq[3]}\TT^\I\mright)=0$ by the monotonicity with respect to $\J$ of $\X^\J$. $\overline{\tau}\mleft(\bigoplus_{|\I|\ge1,\,\I\subseteq[3]}\TT^\I\mright)\ge\frac{2}{\sqrt{3}}$ is due to the tensor $\CZ+\CX(-1)$ in Example~\ref{ex:yuan3}. To upper bound $\overline{\tau}\mleft(\bigoplus_{|\I|\ge1,\,\I\subseteq[3]}\TT^\I\mright)$, we notice that for any $\CT$,
$$
\bigoplus_{|\I|\ge1,\,\I\subseteq[3]}\TT^\I(\CT)=\TT^{\{3\}}(\CT)\oplus\bigl(\TT^{\{2\}}(\CT)\oplus\TT^{\{2,3\}}(\CT)\bigr)\oplus\U^{\{1\}}(\CT).
$$
By the monotonicity and a similar idea to Lemma~\ref{thm:bridge},
\begin{align*}
&\overline{\tau}\mleft(\bigoplus_{|\I|\ge1,\,\I\subseteq[3]}\TT^\I\mright)\\
\le{}&%\max_{\CT\in\R^{n_1\times n_2\times n_3}}
\max\mleft\{\|\CX_1+\CX_2+\CX_3\|_\sigma: 
% \!\!\!\begin{array}{l}
\begin{gathered}
\|\CZ+\CX_1+\CX_2+\CX_3\|_\sigma\le 1,\,\CZ\in\TT(\CT),\\
\CX_1\in\TT^{\{3\}}(\CT),\,\CX_2\in\TT^{\{2\}}(\CT)\oplus\TT^{\{2,3\}}(\CT),\,\CX_3\in\U^{\{1\}}(\CT)
\end{gathered}
% \end{array}\!\!\!
\mright\}\\
% \max\mleft\{\|\CX_1+\CX_2+\CX_3\|_\sigma: 
% \begin{gathered}
% (\CZ,\CX_1)\in\TT(\CT)\times\TT^{\{3\}}(\CT)\text{ and}\\
% (\CX_2,\CX_3)\in\mleft(\TT^{\{2\}}(\CT)\oplus\TT^{\{2,3\}}(\CT)\mright)\times\U^{\{1\}}(\CT)\\
% \text{with}~\|\CZ+\CX_1+\CX_2+\CX_3\|_\sigma\le 1
% \end{gathered}
% \mright\} \\
\le{}&
% \max\mleft\{x_1y_1z_2+x_1y_2+x_2:\begin{array}{l}x_1y_1z_2+x_1y_2+x_2\le1+x_1y_1z_1,\\
%  \|\bx\|_2=\|\by\|_2=\|\bz\|_2=1,\, \bx,\by,\bz\ge{\bf 0}\end{array}\mright\}
\max\mleft\{x_1y_1z_2+x_1y_2+x_2:
%\begin{gathered}
x_1y_1z_2+x_1y_2+x_2\le1+x_1y_1z_1,\,
\bx,\by,\bz\in\SI^2\cap\R_{+}^2
%\end{gathered}
\mright\}\\
={}&\frac{3}{2},
\end{align*}
where the last equality is computed in Lemma~\ref{thm:optimization2}. % for the calculation of the last optimization problem.

\item $\underline{\tau}\mleft(\bigoplus_{|\I|\ge2,\,\I\subseteq[4]}\TT^\I\mright)\ge\frac{1}{3}$ generalizes~\cite[Lemma~1]{yuan2016tensor} from $d=3$ to $d=4$ and its proof is given by Lemma~\ref{lemma:yuan-lemma1-d=4}.  $\underline{\tau}\mleft(\bigoplus_{|\I|\ge2,\,\I\subseteq[4]}\TT^\I\mright)\le\frac{1}{2}$ and $\overline{\tau}\mleft(\bigoplus_{|\I|\ge2,\,\I\subseteq[4]}\TT^\I\mright)\ge\frac{1+\sqrt{2}}{2}$ are due to the tensors $\CZ+\CX(\frac{1}{3})$ and $\CZ+\CX(-\frac{1+\sqrt{2}}{3})$ in Example~\ref{ex:yuan4}, respectively. To upper bound $\overline{\tau}\mleft(\bigoplus_{|\I|\ge2,\,\I\subseteq[4]}\TT^\I\mright)$, we notice that for any $\CT$,
$$
\bigoplus_{|\I|\ge2,\,\I\subseteq[4]}\TT^\I(\CT)\subseteq\TT^{\{3,4\}}(\CT)\oplus(\TT^{\{2\}}\oplus\TT^{\{2,3\}}\oplus\TT^{\{2,4\}}\oplus\TT^{\{2,3,4\}})(\CT)\oplus\U^{\{1\}}(\CT).
$$
By the monotonicity and a similar idea to Lemma~\ref{thm:bridge},
\begin{align*}
&\overline{\tau}\mleft(\bigoplus_{|\I|\ge2,\,\I\subseteq[4]}\TT^\I\mright)\\
\le{}&\max\mleft\{\|\CX_1+\CX_2+\CX_3\|_\sigma:
% \!\!\!\begin{array}{l}
\begin{gathered}
\|\CZ+\CX_1+\CX_2+\CX_3\|_\sigma\le 1,\,\CZ\in\TT(\CT),\,\CX_1\in\TT^{\{3,4\}}(\CT),\\
\CX_2\in(\TT^{\{2\}}\oplus\TT^{\{2,3\}}\oplus\TT^{\{2,4\}}\oplus\TT^{\{2,3,4\}})(\CT),\,\CX_3\in\U^{\{1\}}(\CT)
% (\CZ,\CX_1,\CX_3)\in\TT(\CT)\times\TT^{\{3,4\}}(\CT)\times\U^{\{1\}}(\CT)\text{ and}\\
% \CX_2\in\mleft(\TT^{\{2\}}\oplus\TT^{\{2,3\}}\oplus\TT^{\{2,4\}}\oplus\TT^{\{2,3,4\}}\mright)(\CT)\\
% \text{with}~\|\CZ+\CX_1+\CX_2+\CX_3\|_\sigma\le 1
\end{gathered}
% \end{array}\!\!\!
\mright\}\\
\le{}&\max\mleft\{x_1y_1z_2w_2+x_1y_2+x_2:
x_1y_1z_2w_2+x_1y_2+x_2\le 1+x_1y_1z_1w_1,\,
\bx,\by,\bz,\bw\in\SI^2\cap\R_{+}^2
% \|\bx\|_2=\|\by\|_2=\|\bz\|_2=1,\, \bx,\by,\bz \ge{\bf 0}
\mright\}\\
={}&\frac{1+\sqrt{3}}{2},
\end{align*}
where the last equality is obtained using a similar calculation to Lemma~\ref{thm:optimization2}.

\item Since $\underline{\tau}(\TT^\I)=0$ for any $\I\subseteq[4]$ with $|\I|=1$ from (i), $\underline{\tau}\mleft(\bigoplus_{|\I|\ge1,\,\I\subseteq[4]}\TT^\I\mright)=0$ by the monotonicity with respect to $\J$ of $\X^\J$.
% As $\TT^\I\subseteq\bigoplus_{|\I|\ge1,\,\I\subseteq[4]}\TT^\I$ for all $\I\subseteq[4]$ with $|\I|=1$, it follows by monotonicity in $\J$ and (i) that $\underline{\tau}\mleft(\bigoplus_{|\I|\ge1,\,\I\subseteq[4]}\TT^\I\mright)=0$. 
$\overline{\tau}\mleft(\bigoplus_{|\I|\ge1,\,\I\subseteq[4]}\TT^\I\mright)\ge\frac{1+\sqrt{2}}{2}$ is due to the tensor $\CZ+\CX(-\frac{1+\sqrt{2}}{3})$ in Example~\ref{ex:yuan4}. To upper bound $\overline{\tau}\mleft(\bigoplus_{|\I|\ge1,\,\I\subseteq[4]}\TT^\I\mright)$, we notice that for any $\CT$,
\begin{align*}
&\bigoplus_{|\I|\ge1,\,\I\subseteq[4]}\TT^\I(\CT)\\
={}&\TT^{\{4\}}(\CT)\oplus\bigl(\TT^{\{3\}}(\CT)\oplus\TT^{\{3,4\}}(\CT)\bigr) \oplus(\TT^{\{2\}}\oplus\TT^{\{2,3\}}\oplus\TT^{\{2,4\}}\oplus\TT^{\{2,3,4\}})(\CT)\oplus\U^{\{1\}}(\CT).
\end{align*}
By the monotonicity and a similar idea to Lemma~\ref{thm:bridge},
\begin{align*}
&\overline{\tau}\mleft(\bigoplus_{|\I|\ge1,\,\I\subseteq[4]}\TT^\I\mright)\\
\le{}&\max\mleft\{\mleft\|\sum_{i=1}^4\CX_i\mright\|_\sigma\!: 
% \!\!\!\begin{array}{l}
\begin{gathered}
% (\CZ,\CX_1,\CX_2)\in\TT(\CT)\times\TT^{\{4\}}(\CT)\times\mleft(\TT^{\{3\}}\oplus\TT^{\{3,4\}}\mright)(\CT)\text{ and}\\
% (\CX_3,\CX_4)\in\mleft(\TT^{\{2\}}\oplus\TT^{\{2,3\}}\oplus\TT^{\{2,4\}}\oplus\TT^{\{2,3,4\}}\mright)(\CT)\times\U^{\{1\}}(\CT)\\
% \text{with}~\|\CZ+\CX_1+\CX_2+\CX_3+\CX_4\|_\sigma\le 1
\|\CZ+\CX_1+\CX_2+\CX_3+\CX_4\|_\sigma\le 1,\,\CZ\in\TT(\CT),\,\CX_2\in\TT^{\{3\}}(\CT)\oplus\TT^{\{3,4\}}(\CT), \\
\CX_1\in\TT^{\{4\}}(\CT),\CX_3\in(\TT^{\{2\}}\oplus\TT^{\{2,3\}}\oplus\TT^{\{2,4\}}\oplus\TT^{\{2,3,4\}})(\CT),\CX_4\in\U^{\{1\}}(\CT)
\end{gathered}
% \end{array}\!\!\!
\mright\}\\
\le{}&
% \max\mleft\{x_1y_1z_2+x_1y_2+x_2:\begin{array}{l}x_1y_1z_2+x_1y_2+x_2\le1+x_1y_1z_1,\\
%  \|\bx\|_2=\|\by\|_2=\|\bz\|_2=1,\, \bx,\by,\bz\ge{\bf 0}\end{array}\mright\}
\max\mleft\{x_1y_1z_1w_2+x_1y_1z_2+x_1y_2+x_2:
% \!\!\!\begin{array}{l}
\begin{gathered}
x_1y_1z_1w_2+x_1y_1z_2+x_1y_2+x_2\le1+x_1y_1z_1w_1,\\
\bx,\by,\bz,\bw\in\SI^2\cap\R_{+}^2
\end{gathered}
% \end{array}\!\!\!
\mright\}\\
={}&\frac{8}{5},
\end{align*} 
where the last equality is obtained using a similar calculation to Lemma~\ref{thm:optimization2}.

\item $\underline{\tau}\mleft(\bigoplus_{|\I|\ge2,\,\I\subseteq[d]}\TT^\I\mright)\ge\frac{2}{d(d-1)}$ is due to the inclusion $\D_2(\CT)\subseteq\partial\|\CT\|_*$; see also~\cite[Theorem~1]{yuan2017incoherent}. Both $\underline{\tau}\mleft(\bigoplus_{|\I|\ge2,\,\I\subseteq[d]}\TT^\I\mright)\le\frac{1}{2}$ and $\overline{\tau}\mleft(\bigoplus_{|\I|\ge2,\,\I\subseteq[d]}\TT^\I\mright)\ge\frac{1+\sqrt{2}}{2}$ are due to (vii) and the monotonicity with respect to $\J$ of $\X^\J$ and $d$. Finally, $\overline{\tau}\mleft(\bigoplus_{|\I|\ge2,\,\I\subseteq[d]}\TT^\I\mright)\le 2$ is a trivial bound since $\|\CX\|_\sigma\le \|\CX+\CZ\|_\sigma +\mleft\|-\CZ\mright\|_\sigma\le 2$ for any $\CZ\in\Z(\CT)$ and $\CX\in\X^\J(\CT)$ with $\|\CZ+\CX\|_\sigma\le1$ in~\eqref{eq:tau_upper}.
% by the requirement of a subgradient~\eqref{eq:exact}.

\item % Because $\TT^\I\subseteq\bigoplus_{|\I|\ge1,\,\I\subseteq[d]}\TT^\I$ for all $\I\subseteq [d]$ with $|\I|=1$, it follows by monotonicity in $\J$ and (i) that $\underline{\tau}\mleft(\bigoplus_{|\I|\ge1,\,\I\subseteq[d]}\TT^\I\mright)=0$. 
Since $\underline{\tau}(\TT^\I)=0$ for any $\I\subseteq[d]$ with $|\I|=1$ from (i), $\underline{\tau}\mleft(\bigoplus_{|\I|\ge1,\,\I\subseteq[d]}\TT^\I\mright)=0$ by the monotonicity with respect to $\J$ of $\X^\J$. $\overline{\tau}\mleft(\bigoplus_{|\I|\ge1,\,\I\subseteq[d]}\TT^\I\mright)\ge\frac{1+\sqrt{2}}{2}$ is due to (viii) and the monotonicity with respect to $\J$ of $\X^\J$ and $d$. Finally, $\overline{\tau}\mleft(\bigoplus_{|\I|\ge1,\,\I\subseteq[d]}\TT^\I\mright)\le 2$ is a trivial bound as in the proof of (ix).
\end{enumerate}
The proof is complete.
\end{proof}

The key idea to bound $\overline{\tau}\mleft(\bigoplus_{|\I|\ge2,\,\I\subseteq[3]}\TT^\I\mright)$ from above in (v) is to establish a link between a tensor optimization problem and a simple low-dimensional spherical optimization problem, i.e., Lemma~\ref{thm:bridge}. This link offers a tool of independent interest. While it is presented in a rather special structure, it can be easily extended to other similar structures, such as bounding $\overline{\tau}\mleft(\bigoplus_{|\I|\ge1,\,\I\subseteq[3]}\TT^\I\mright)$ in (vi), $\overline{\tau}\mleft(\bigoplus_{|\I|\ge2,\,\I\subseteq[4]}\TT^\I\mright)$ in (vii), and 
$\overline{\tau}\mleft(\bigoplus_{|\I|\ge1,\,\I\subseteq[4]}\TT^\I\mright)$ in (viii).

We remark that instead of using $\bigoplus_{|\I|\ge2,\,\I\subseteq[3]}\TT^\I(\CT)\subseteq\TT^{\{2,3\}}(\CT)\oplus\U^{\{1\}}(\CT)$ in the proof of (v) to bound $\overline{\tau}\mleft(\bigoplus_{|\I|\ge2,\,\I\subseteq[3]}\TT^\I\mright)$ from above, we may also use $\bigoplus_{|\I|\ge2,\,\I\subseteq[3]}\TT^\I(\CT)=\TT^{\{1,3\}}(\CT)\oplus\TT^{\{2,3\}}(\CT)\oplus\U^{\{1,2\}}(\CT)$ to obtain another upper bound via a similar argument. However, the bound is worse than that in (v), the same to other ways of partitioning $\bigoplus_{|\I|\ge2,\,\I\subseteq[3]}\TT^\I(\CT)$. For the same reason, the ways of partitioning $\bigoplus_{|\I|\ge1,\,\I\subseteq[3]}\TT^\I(\CT)$ in the proof of (vi), $\bigoplus_{|\I|\ge2,\,\I\subseteq[4]}\TT^\I(\CT)$ in (vii), and 
$\bigoplus_{|\I|\ge1,\,\I\subseteq[4]}\TT^\I(\CT)$ in (viii) are all the best.

\begin{lemma}\label{thm:bridge}
If $\CT\in\R^{n_1\times n_2\times n_3}$, then
\begin{align*}
&\,\max\mleft\{\|\CX_1+\CX_2\|_\sigma: \|\CZ+\CX_1+\CX_2\|_\sigma\le 1,\,\CZ\in\TT(\CT),\,\CX_1\in\TT^{\{2,3\}}(\CT),\,\CX_2\in\U^{\{1\}}(\CT)\mright\}\\
\le&\,\max\mleft\{x_1y_2z_2+x_2: x_1y_2z_2+x_2\le1+x_1y_1z_1,\, \bx,\by,\bz\in\SI^2\cap\R_{+}^2\mright\}.
\end{align*}
% \begin{align*}
% &
%\max\mleft\{\|\CX_1+\CX_2\|_\sigma: \|\CZ+\CX_1+\CX_2\|_\sigma\le 1,\,\CZ\in\TT(\CT),\,\CX_1\in\TT^{\{2,3\}}(\CT),\,\CX_2\in\U^{\{1\}}(\CT)\mright\}
% \max\mleft\{\|\CX_1+\CX_2\|_\sigma:
% \begin{gathered}
% (\CZ,\CX_1,\CX_2)\in\TT(\CT)\times\TT^{\{2,3\}}(\CT)\times\U^{\{1\}}(\CT)\\
% \text{with}~\|\CZ+\CX_1+\CX_2\|_\sigma\le 1
% \end{gathered}
% \mright\}\\
% \le{}& \max\mleft\{x_1y_2z_2+x_2:
% \begin{gathered}
% x_1y_2z_2+x_2\le 1+x_1y_1z_1\\
% \bx,\by,\bz\in\SI^2\cap\R_{+}^2
% % \|\bx\|_2=\|\by\|_2=\|\bz\|_2=1,\, \bx,\by,\bz \ge{\bf 0}
% \end{gathered}
% \mright\}.
% % \\
% \le{}&\max\mleft\{x_1y_2z_2+x_2: x_1y_2z_2+x_2\le1+x_1y_1z_1,\, \|\bx\|_2=\|\by\|_2=\|\bz\|_2=1,\, \bx,\by,\bz\ge{\bf 0}\mright\}.
% \end{align*}
% If $\V_k$ is a subspace of $\R^{n_k}$ for $k\in[3]$, then
% \begin{align*}
% &~~~\max\mleft\{\|\CX_1+\CX_2\|_\sigma: \|\CZ+\CX_1+\CX_2\|_\sigma\le 1,\,\CZ\in\TT(\{\V_k\}_{k=1}^3),\,\CX_1\in\TT^{\{2,3\}}(\{\V_k\}_{k=1}^3),\,\CX_2\in\U^{\{1\}}(\{\V_k\}_{k=1}^3)\mright\}\\
% &\le\max\mleft\{x_1y_2z_2+x_2: x_1y_2z_2+x_2\le1+x_1y_1z_1,\, x_1^2+x_2^2=y_1^2+y_2^2=z_1^2+z_2^2=1,\, \bx,\by,\bz\ge{\bf 0}\mright\}.
% \end{align*}
\end{lemma}
\begin{proof}
Let us denote $(\CZ, \CX_1,\CX_2)$ to be an optimal solution to the first optimization problem and let $\|\CX_1+\CX_2\|_\sigma=\langle \CX_1+\CX_2,\bv_1\otimes\bv_2\otimes\bv_3\rangle$ where $\bv_k\in\SI^{n_k}$ for $k\in[3]$. Let $\V_k=\spn_k(\CT)$ for $k\in[3]$ and
$$
\begin{array}{lll}
a_1=\bigl\|\proj_{\V_1}(\bv_1)\bigr\|_2, &b_1=\bigl\|\proj_{\V_2}(\bv_2)\bigr\|_2, &c_1=\bigl\|\proj_{\V_3}(\bv_3)\bigr\|_2,\\
a_2=\bigl\|\proj_{\V_1^\perp}(\bv_1)\bigr\|_2, &b_2=\bigl\|\proj_{\V_2^\perp}(\bv_2)\bigr\|_2, &c_2=\bigl\|\proj_{\V_3^\perp}(\bv_3)\bigr\|_2.
\end{array}
$$
It is obvious that $\|\ba\|_2=\|\bb\|_2=\|\bc\|_2=1$ and $\ba,\bb,\bc\ge{\bf 0}$.%; hence, $\ba,\bb,\bc\in\SI^2\cap\R_{+}^2$.

By applying Theorem~\ref{thm:spec-decomp} with $\CZ+\CX_1\in\U_{\{1\}}(\CT)$ and $\CX_2\in\U^{\{1\}}(\CT)$, we have
$$
%\|\CX_1\|_\sigma\le
\max\bigl\{\|\CZ+\CX_1\|_\sigma,\|\CX_2\|_\sigma\bigr\} = \|\CZ+\CX_1+\CX_2\|_\sigma\le 1.
$$
Again by applying Theorem~\ref{thm:spec-decomp} with $\CZ\in\U_{\{2\}}(\CT)$ and $\CX_1\in\U^{\{2\}}(\CT)$, we also have 
$$
\max\bigl\{\|\CZ\|_\sigma,\|\CX_1\|_\sigma\bigr\}= \|\CZ+\CX_1\|_\sigma\le\max\bigl\{\|\CZ+\CX_1\|_\sigma,\|\CX_2\|_\sigma\bigr\} \le 1.
$$
As a result, $\|\CZ\|_\sigma,\|\CX_1\|_\sigma, \|\CX_2\|_\sigma\le 1$.

% By
% % Theorem~\ref{thm:spec-decomp} 
% noticing that $\CZ+\CX_1\in\U_{\{1\}}(\CT)$ and $\CX_2\in\U^{\{1\}}(\CT)$ and applying Lemma~\ref{thm:spec-subspace}, we have
% $$
% %\|\CX_1\|_\sigma\le
% \max\{\|\CZ+\CX_1\|_\sigma,\|\CX_2\|_\sigma\} \le \|\CZ+\CX_1+\CX_2\|_\sigma\le 1;
% $$
% % Again by applying Theorem~\ref{thm:spec-decomp} with 
% Similarly, by noticing that $\CZ\in\U_{\{2\}}(\CT)$ and $\CX_1\in\U^{\{2\}}(\CT)$, we also have 
% $$
% \max\{\|\CZ\|_\sigma,\|\CX_1\|_\sigma\} \| \le \|\CZ+\CX_1\|_\sigma\le\max\{\|\CZ+\CX_1\|_\sigma,\|\CX_2\|_\sigma\} \le 1.
% $$
% As a result, we have $\max\{\|\CZ\|_\sigma,\|\CX_1\|_\sigma, \|\CX_2\|_\sigma\}\le 1$.

It follows from Lemma~\ref{thm:spec-subspace} that
\begin{align*}
\|\CX_1+\CX_2\|_\sigma 
&= \langle\CX_1, \bv_1 \otimes  \bv_2 \otimes \bv_3\rangle + \langle \CX_2,\bv_1 \otimes  \bv_2 \otimes \bv_3\rangle \\
&= \bigl\langle\CX_1,\proj_{\V_1}(\bv_1) \otimes  \proj_{\V_2^\perp}(\bv_2) \otimes \proj_{\V_3^\perp}(\bv_3)\bigr\rangle +  \bigl\langle\CX_2,\proj_{\V_1^\perp}(\bv_1) \otimes  \bv_2 \otimes \bv_3\bigr\rangle \\
&\le \|\CX_1\|_\sigma \bigl\|\proj_{\V_1}(\bv_1)\bigr\|_2 \bigl\|\proj_{\V_2^\perp}(\bv_2)\bigr\|_2 \bigl\|\proj_{\V_3^\perp}(\bv_3)\bigr\|_2 + \|\CX_2\|_\sigma \bigl\|\proj_{\V_1^\perp}(\bv_1)\bigr\|_2 \|\bv_2\|_2 \|\bv_3\|_2\\
&\le a_1b_2c_2+a_2.
\end{align*}
Moreover, by that
$$\langle\CZ,\bv_1\otimes\bv_2\otimes\bv_3\rangle+\|\CX_1+\CX_2\|_\sigma = 
\langle\CZ+\CX_1+\CX_2,\bv_1\otimes\bv_2\otimes\bv_3\rangle\le \|\CZ+\CX_1+\CX_2\|_\sigma\le 1$$
and applying Lemma~\ref{thm:spec-subspace} again, we also have
$$
\|\CX_1+\CX_2\|_\sigma\le 1 - \bigl\langle\CZ,\bv_1\otimes\bv_2\otimes\bv_3\rangle =1 + \langle-\CZ,\proj_{\V_1}(\bv_1) \otimes  \proj_{\V_2}(\bv_2) \otimes \proj_{\V_3}(\bv_3)\bigr\rangle \le 1+a_1b_1c_1.
$$

Let us now compare the two upper bounds of $\|\CX_1+\CX_2\|_\sigma$, $a_1b_2c_2+a_2$ and $1+a_1b_1c_1$. 
\begin{enumerate}[label=\textnormal{(\roman*)}]
    \item If $a_1b_2c_2+a_2>1+a_1b_1c_1$, consider $f(x) = a_1b_2\sqrt{1-x^2}+a_2$ and $g(x)=1+a_1b_1x$ defined on $[0,1]$. Since $f(x)$ is decreasing, $g(x)$ is increasing, $f(c_1)>g(c_1)$, and $f(1)=a_2\le 1\le g(1)$, there must exist some $d_1\in[c_1,1]$ such that $f(c_1)\ge f(d_1)=g(d_1)\ge g(c_1)$.
    
    \item If $a_1b_2c_2+a_2\le 1+a_1b_1c_1$, then by letting $d_1=c_1$, we have $f(c_1)=f(d_1)\le g(d_1)=g(c_1)$.
\end{enumerate}
To summarize both cases, there always exists some $d_1\in[c_1,1]$, such that
\begin{equation} \label{eq:first}
\|\CX_1+\CX_2\|_\sigma\le \min\bigl\{f(c_1),g(c_1)\bigr\} \le f(d_1)= a_1b_2d_2+a_2 \le g(d_1)= 1+a_1b_1d_1,
\end{equation}
where $d_2=\sqrt{1-d_1^2}$. Therefore, $(a_1,a_2,b_1,b_2,d_1,d_2)$ is feasible to the second optimization problem, which implies that $a_1b_2d_2+a_2$ is no more than the optimal value of the second problem. This, together with the fact that the optimal value of the first optimization problem is no more than $a_1b_2d_2+a_2$ by~\eqref{eq:first}, directly shows the claimed result.
\end{proof}

Let us take some time to digest the results in Theorem~\ref{thm:bounds}, focusing on the subspaces where $\CX$ resides.
% that can make $\CZ+\CX$ a subgradient in~\eqref{eq:exact}. 
For any basic subspace $\TT^\I(\CT)$ with $|\I|\ge2$, $\CZ+\CX$ is a subgradient if and only if $\|\CX\|_\sigma\le 1$, i.e., (ii). However, for a basic subspace $\TT^\I(\CT)$ with $|\I|=1$, nothing can be guaranteed and one has to check $\|\CZ+\CX\|_\sigma$ case by case (see Example~\ref{ex:1perp}), but any $\|\CX\|_\sigma>1$ definitely rules $\CZ+\CX$ out, i.e., (i). A direct sum of several basic subspaces can be a subspace that keeps the bounds perfect as long as the sum is a subset of $\U^\I(\CT)$ for some $|\I|=2$, the largest subspace to be perfect, i.e., (iv). Although $\bigoplus_{|\I|\ge2,\,\I\subseteq[d]}\TT^\I$ cannot keep the bounds perfect, it is the largest structure in the sense that every direction in $\bigoplus_{|\I|\ge2,\,\I\subseteq[d]}\TT^\I(\CT)$ can make $\CZ+\CX$ a subgradient. This is perhaps the most interesting aspect of the subspace, in which the lower bound $\frac{2}{d(d-1)}$ provides an assurance of the spectral size of $\CX$ albeit it can be conservative, i.e., (ix). For the two special cases $d=3,4$, we are able to nail both $\underline{\tau}\mleft(\bigoplus_{|\I|\ge2,\,\I\subseteq[d]}\TT^\I\mright)$ and $\overline{\tau}\mleft(\bigoplus_{|\I|\ge2,\,\I\subseteq[d]}\TT^\I\mright)$ down to a smaller range.%, i.e., (v) and (vii).

We remark in particular that $\frac{2}{d(d-1)}$, as the conservative lower bound of $\underline{\tau}\mleft(\bigoplus_{|\I|\ge2,\,\I\subseteq[d]}\TT^\I\mright)$ from $\D_2(\CT)$, was improved to $\frac{1}{2}$ in $\D_1(\CT)$, i.e., (v) for $d=3$. It was also improved to $\frac{1}{3}$ in (vii) for $d=4$, one of the most important results in Theorem~\ref{thm:bounds}. Both make us believe strongly that the $\frac{2}{d(d-1)}$ bound can be improved to $\frac{1}{d-1}$, i.e., the validity of (iii) in Conjecture~\ref{cj:bounds} below.

% We note in closing that we are almost sure about; should this be the case, then the bound $\frac{2}{d(d-1)}$ in $\D_2(\CT)$ in (\ref{eq:subdiff-incoherent}) can be further improved to $1/(d-1)$.
% hence the reslut (if proved) include results of $D_1$ and $D_2$.

We have not explored the cases for $d\ge 5$ in theory. However, we do have applied similar techniques to those used in Theorem~\ref{thm:bounds} and then resorted to computer programs to find the optimal values of relevant problems. We list our findings as a conjecture below.
\begin{conjecture}\label{cj:bounds}
In the tensor space $\R^{n_1\times n_2\times \dots \times n_d}$ of order $d\ge3$, % with $\max_{1\le k\le d}{n_k}\ge 3$,
\begin{enumerate}[label=\textnormal{(\roman*)}]
    \item $\underline{\tau}\mleft(\bigoplus_{|\I|\ge2,\,\I\subseteq[3]}\TT^\I\mright)=\frac{1}{\sqrt{3}}$ and $\overline{\tau}\mleft(\bigoplus_{|\I|\ge2,\,\I\subseteq[3]}\TT^\I\mright)=\frac{2}{\sqrt{3}}$;
    \item $\underline{\tau}\mleft(\bigoplus_{|\I|\ge2,\,\I\subseteq[d]}\TT^\I\mright)\le\mleft(\frac{d-2}{d}\mright)^{\frac{d-2}{2}}$ and in particular $\lim_{d\rightarrow\infty}\underline{\tau}\mleft(\bigoplus_{|\I|\ge2,\,\I\subseteq[d]}\TT^\I\mright)\le \frac{1}{e}$;
    \item $\underline{\tau}\mleft(\bigoplus_{|\I|\ge2,\,\I\subseteq[d]}\TT^\I\mright)\ge\frac{1}{d-1}$;   
    \item $\overline{\tau}\mleft(\bigoplus_{|\I|\ge2,\,\I\subseteq[d]}\TT^\I\mright)\le\frac{2(d-2)}{d-1}$ for $d\ge 5$;
    % $$
    %     \begin{cases}
    %     \frac{1+\sqrt{3}}{2} & d=4\\
    %         \frac{2(d-2)}{d-1} & d\ge5;
    % \end{cases}
    % $$
    % \item $\overline{\tau}(\bigoplus_{|\I|\ge2,\,\I\subseteq[d]}\TT^\I)\le\frac{2(d-2)}{d-1}$ if $d\ge5$;
    % \item $\overline{\tau}(\bigoplus_{|\I|\ge2,\,\I\subseteq[4]}\TT^\I)\le\frac{1+\sqrt{3}}{2}$.
    \item $\overline{\tau}\mleft(\bigoplus_{|\I|\ge1,\,\I\subseteq[d]}\TT^\I\mright)\le\frac{2d}{d+1}$;
    \item $\lim_{d\rightarrow\infty}\overline{\tau}\mleft(\bigoplus_{|\I|\ge2,\,\I\subseteq[d]}\TT^\I\mright)\ge\overline{\tau}\mleft(\bigoplus_{|\I|\ge2,\,\I\subseteq[200]}\TT^\I\mright)\ge 1.319$.
\end{enumerate}
\end{conjecture}

\subsection{Final remark of the subdifferential}\label{sec:final-remark}

% The 
\textcolor{black}{Recall that the}
original version of the subdifferential inclusion in~\cite{yuan2016tensor} states that
\begin{equation*}
% \label{eq:subdiff-old}
\overline{\D}_1(\CT)=\mleft\{\CZ+\proj_{\bigoplus_{|\I|\ge2,\,\I\subseteq[3]}\TT^\I(\CT)}(\CX):\CZ\in\Z(\CT),\,\CX\in\R^{n_1\times n_2\times n_3},\,\|\CX\|_\sigma\le \frac{1}{2}\mright\} \subseteq \partial \|\CT\|_*,
\end{equation*}
which is slightly different to the one that we structured
\begin{equation*}\label{eq:subdiff-new}
   \D_1(\CT)=\mleft\{\CZ+\CX:\CZ\in\Z(\CT),\,\CX\in\bigoplus_{|\I|\ge2,\,\I\subseteq[3]}\TT^\I(\CT)
    % \range\mleft(\sum_{\I\subseteq[3],\,|\I|\ge 2}\bigotimes_{i=1}^3\mleft(\mathbbm{1}_{i\in\I}(\proj_{\V_k^\perp}-\proj_{\V_k})+\proj_{\V_k}\mright)\mright),
    ,\,\|\CX\|_\sigma\le \frac{1}{2}\mright\} \subseteq \partial \|\CT\|_*.
 %   =:\D_1(\CT),
\end{equation*}
% \begin{equation*}
%    \D_1(\CT)= \mleft\{\CZ+\CX:\CZ\in\Z(\CT),\,\CX\in\bigoplus_{|\I|\ge2,\,\I\subseteq[3]}\TT^\I(\CT),
%     \,\|\CX\|_\sigma\le \frac{1}{2}\mright\} \subseteq \partial \|\CT\|_*.
% \end{equation*}
They both restrict subgradients in the direct sum of $\Z(\CT)$ and $\bigoplus_{|\I|\ge2,\,\I\subseteq[3]}\TT^\I(\CT)$. In order to restrict $\CX\in\bigoplus_{|\I|\ge2,\,\I\subseteq[3]}\TT^\I(\CT)$, $\D_1(\CT)$ simply asks $\|\CX\|_\sigma\le\frac{1}{2}$, whereas $\overline{\D}_1(\CT)$ asks $\|\CX+\CY\|_\sigma\le\frac{1}{2}$ for some $\CY\in\bigoplus_{|\I|\le1,\,\I\subseteq[3]}\TT^\I(\CT)$, but $\CY$ has nothing to do with the subgradient $\CZ+\CX$ itself.

The study on the subdifferential in this section is presented in the structure of $\D_1(\CT)$, i.e., simply checking $\|\CX\|_\sigma$, % and making a decision
rather than checking 
$
\min\bigl\{\|\CX+\CY\|_\sigma: \CY\in\bigoplus_{|\I|\le1,\,\I\subseteq[3]}\TT^\I(\CT)\bigr\}.
$
For the perfect inclusion in Corollary~\ref{thm:newset}, i.e.,
$$
\D^\I(\CT)=\bigl\{\CZ+\CX:\CZ\in\Z(\CT),\,\CX\in\U^\I(\CT),\,\|\CX\|_\sigma\le 1\bigr\} \subseteq \partial \|\CT\|_*\text{ for any $\I\subseteq [d]$ with $|\I|\ge 2$},
$$
both structures lead to the same set. This is because $\CX=\proj_{\U^\I(\CT)}(\CX+\CY)$ and so $\|\CX\|_\sigma\le\|\CX+\CY\|_\sigma\le1$ by Lemma~\ref{thm:spec-subspace}.
On the other hand, $\overline{\D}_1(\CT)$ includes $\D_1(\CT)$ as a proper subset since $\|\CX\|_\sigma\le\frac{1}{2}$ trivially implies that $\|\CX+\CY\|_\sigma\le\frac{1}{2}$ for $\CY=\CO$ and the tensor $\CZ+\CX(\frac{1}{2})$ in Example~\ref{ex:yuan33} belongs to $\overline{\D}_1(\CT)\setminus\D_1(\CT)$. Even though, both $\overline{\D}_1(\CT)$ and $\D_1(\CT)$ provide the lower bound $\frac{1}{2}$ of $\underline{\tau}\mleft(\bigoplus_{|\I|\ge2,\,\I\subseteq[3]}\TT^\I\mright)$ in Theorem~\ref{thm:bounds}. In fact, the constant $\frac{1}{2}$ in $\overline{\D}_1(\CT)$ turns out to be tight and this answers in the negative a comment raised in~\cite[Page~1039]{yuan2016tensor} to sharpen the constant.
\begin{proposition} 
The constant $\frac{1}{2}$ in $\overline{\D}_1(\CT)$ is tight.
\end{proposition}
The proposition is an immediate consequence of the following example.
\begin{example}\label{ex:yuan33}
Let $\CT=\be_1\otimes\be_1\otimes\be_1\in\R^{2\times 2\times 2}$. We have $\spn_k(\CT)=\spn(\be_1)$ for $k\in[3]$. The only tensor $\CZ\in\TT(\CT)$ that satisfies $\langle\CZ,\CT\rangle=\|\CT\|_*=1$ and $\|\CZ\|_\sigma=1$ is $\CZ=\CT$.

Let $\CX(t)=t(\be_1\otimes\be_2\otimes\be_2 + \be_2\otimes\be_1\otimes\be_2 + \be_2\otimes\be_2\otimes\be_1) \in\bigoplus_{|\I|\ge2,\,\I\subseteq[3]}\TT^\I(\CT)$ and $\CY(t)=-t\CT\in\bigoplus_{|\I|\le1,\,\I\subseteq[3]}\TT^\I(\CT)$. It can be verified that $\bigl\|\CX(t)\bigr\|_\sigma=\frac{2|t|}{\sqrt{3}}$ and $\bigl\|\CX(t)+\CY(t)\bigr\|_\sigma=|t|$ for any $t\in\R$ (see Lemma~\ref{thm:E3}). However, it can also be verified that $\bigl\|\CZ+\CX(t)\bigr\|_\sigma\le1$ if and only if $-1\le t\le \frac{1}{2}$ (see Lemma~\ref{thm:E3}), implying that $\CZ+\CX(t)\notin\partial\|\CT\|_*$ for any $t>\frac{1}{2}$.
\end{example}
Perhaps the tightness of $\overline{\D}_1(\CT)$ provides another reason to consider the structure of $\D_1(\CT)$ as we strongly believe that the constant $\frac{1}{2}$ in $\D_1(\CT)$ can be improved to $\frac{1}{\sqrt{3}}$, i.e., (i) in Conjecture~\ref{cj:bounds}. We leave it to future works.

\section{Tensor robust principal component analysis}\label{sec:TRPCA}

In this section, we establish the statistical performance of the nuclear-norm-based tensor robust PCA as an immediate application of our theoretical developments. The main result that we apply is the new inclusion~\eqref{eq:full-subdiff} in Theorem~\ref{thm:subdiff}, more specifically, Corollary~\ref{thm:newset}. It makes the study in the sequel as straightforward as that in the matrix case~\cite{candes2011robust}.

\subsection{Model and main result}

The tensor robust PCA aims to recover a low-rank ground-truth tensor $\CL\in\R^{n_1\times n_2\times\dots\times n_d}\setminus\{\CO\}$ that is superposed by a sparse corruption $\CS\in\R^{n_1\times n_2\times\dots\times n_d}$. Specifically, it is to find the ground-truth $\CL$ by solving the following convex optimization model
\begin{equation}\label{eq:PCP}
    \min\bigl\{\|\CT_1\|_*+\lambda\|\CT_2\|_1: \CT_1+\CT_2=\CL+\CS,\, \CT_1,\CT_2\in\R^{n_1\times n_2\times\dots\times n_d}\bigr\},
\end{equation}
where $\lambda> 0$ is a balancing parameter. For a comprehensive introduction to the model, the readers are referred to~\cite{candes2011robust} and a recent book~\cite{wright2022high}. 
% Like most tensor problems~\cite{hillar2013most}, 
Unfortunately,~\eqref{eq:PCP} is computationally intractable due to the NP-hardness to compute the tensor nuclear norm~\cite{friedland2018nuclear} albeit it is convex. However, understanding the statistical performance of~\eqref{eq:PCP} can still be of great importance, just as it is for relevant problems such as tensor completion~\cite{yuan2016tensor} and tensor regression~\cite{raskutti2019convex}.

It should be noted that not every $\CL$ is identifiable. For example, it is definitely impossible to recover a simultaneously low-rank and sparse $\CL$ due to ambiguity. As a result, some standard assumptions on $\CL$ and $\CS$ are required.
\begin{assumption}\label{assump:assump}
    The entries of $\CS$ are independent random variables, each being zero with probability $1-\rho$, positive with probability $\frac{\rho}{2}$, and negative with probability $\frac{\rho}{2}$, for a sufficiently small constant $\rho>0$. 
    % The sparsity patterns of $\CS$ follow from i.i.d.\ Bernoulli random variables with parameter $\rho$ for a sufficiently small constant $\rho>0$. The signs of nonzero entries of $\CS$ follow from i.i.d.\ symmetric Bernoulli random variables (taking $\pm1$ with equal probability). 
    There exist a constant $u_0>0$ and a sufficiently small constant $\theta_0>0$ such that
    \begin{equation*}
        \max_{k\in[d]}u_k \le u_0,~
        r_0:=\max_{k\in[d]}r_k\le \theta_0\frac{ (1-\rho)n_1}{u_0\ln^2{n_d}}, \text{ and }
        \min_{\CZ\in\Z(\CL)}\|\CZ\|_\infty\le\sqrt{\frac{u_0r_0}{n_1 n_d \ln^{\max\{2d-5,0\}}{n_d}}},
    % \begin{dcases}
    %     \max_{k\in[d]}u_0(\V_k) \le u_0
    %     \\
    %     % \begin{dcases}
    %     % \|\CZ\|_\infty\le\sqrt{\frac{u_0r_0}{n_1 n_2}} & d=2, \\
    %     % \|\CZ\|_\infty\le\sqrt{\frac{u_0r_0}{n_1 n_d \ln^{2d-5}{n_d}}} & d\ge 3,
    %     % \end{dcases}
    %     % \\
    %     \min_{\CZ\in\Z(\CL)}\|\CZ\|_\infty\le\sqrt{\frac{u_0r_0}{n_1 n_d \ln^{\max\{2d-5,0\}}{n_d}}}
    %     \\
    %     \max_{k\in[d]}r_k\le \theta_0 \frac{(1-\rho)n_1}{u_0\ln^2{n_d}}.
    % \end{dcases}
    \end{equation*}
where
$$
    r_k:=\dim\bigl(\spn_k(\CL)\bigr)\text{ and }u_k:=\frac{n_k}{r_k} \max_{i\in[n_k]}\bigl\|\proj_{\spn_k(\CL)}(\be_i)\bigr\|_2^2\text{ for $k\in[d]$}.
$$
%=\frac{\max _{j\in[n_i]}\mleft\|\proj_{\spn_k(\CL)}(\be_j)\mright\|_2^2}{n_i^{-1} \sum_{j=1}^{n_i}\mleft\|\proj_{\spn_k(\CL)}\be_j)\mright\|_2^2}.$$    
\end{assumption}
Recall $\Z(\CL)=\bigl\{\CZ\in\TT(\CL):\langle\CZ,\CL\rangle=\|\CL\|_*,\,\|\CZ\|_\sigma=1\bigr\}$ defined in~\eqref{eq:z(t)} for $\CL\ne\CO$ and the assumption that $2\le n_1\le n_2\le\dots\le n_d$ without loss of generality. In Assumption~\ref{assump:assump}, the first and last requirements of $\CL$ are known as the incoherence conditions in the literature~\cite{candes2011robust,yuan2016tensor,driggs2019tensor,lu2020tensor}. The assumption on $\CS$ means that its sparsity patterns follow Bernoulli distributions with parameter $\rho$ but there is no assumption on the magnitudes of the entries of $\CS$. We now state the main result in this section,
\textcolor{black}{deferring its proof to Section~\ref{sec:putting-everything-together}}.
% In summary, we have established the following main theorem in this paper. For ease of reference, all conditions are explicitly listed again.
\begin{theorem}\label{thm:main-thm}
    Under Assumption~\ref{assump:assump}, the convex optimization model~\eqref{eq:PCP} with $\lambda=\frac{1}{\sqrt{n_d}}$ exactly recovers $\CL$ and $\CS$ with high probability for any fixed $d\ge2$.
\end{theorem}
An event with high probability is one whose probability depends on a certain number, which %will be $\prod_{k=1}^d n_k$ and 
is $n_d$ in our case, and tends to 1 as $n_d$ tends to infinity, i.e., the probability of the event occurring can be made as close to 1 as desired. As mentioned in Theorem~\ref{thm:main-thm}, the order of the tensor space, $d\ge2$, is deemed as a fixed parameter.
% We use the Greek letters $\theta,\kappa>0$ to denote some small and sufficiently large constants that are not made explicit, such as $\theta_0$ in Assumption~\ref{assump:assump}.

To the best of our knowledge, Theorem~\ref{thm:main-thm} is the first result concerning the statistical behavior of~\eqref{eq:PCP} for tensors of an arbitrary order. Before proceeding to the proof, let us first compare our result with several existing ones on tensor robust PCA based on nuclear norms. First of all, when $d=2$, the special case of Theorem~\ref{thm:main-thm} exactly recovers the result on the matrix robust PCA; see~\cite[Theorem~5.3]{wright2022high} and~\cite[Theorem~1.1]{candes2011robust}.

There are many t-SVD-based methods in the literature that use the t-SVD~\cite{kilmer2013third} to define different tensor nuclear norms; see, e.g.,~\cite{gao2020enhanced,lu2020tensor,lu2021transforms}. These methods have a good computability and exhibit remarkable performance in numerical experiments but they only work for third-order tensors due to the inherent design mechanism of the t-SVD. This limitation has restricted the generality and versatility of the methods. By contrast, the model~\eqref{eq:PCP} is built on top of the vanilla tensor nuclear norm that applies to tensors of an arbitrary order. For third-order tensors, some of the conditions in these t-SVD-based methods seem a bit different from ours. Taking~\cite[Theorem~4.1]{lu2020tensor} as an example, the balancing parameter $\lambda$ therein is set as $\frac{1}{\sqrt{n_2 n_3}}$ while that in Theorem~\ref{thm:main-thm} is $\frac{1}{\sqrt{n_3}}$, which is consistent with the matrix case $\frac{1}{\sqrt{n_2}}$; see~\cite[Theorem~1.1]{candes2011robust}. As another example, in~\cite[Theorem~4.1]{lu2020tensor}, 
% (or to be more specific,~\cite[Equations~35-37]{lu2020tensor})
the counterpart of $\min_{\CZ\in\Z(\CL)}\|\CZ\|_\infty$ in its context needs to be bounded by $\sqrt{\frac{u_{\rm{t}} r_{\rm{t}}}{n_1 n_2 n_3^2}}$ where $r_{\rm{t}}$ is the tubal rank of $\CL$~\cite[Definition~4.4]{kilmer2013third} while ours is $\sqrt{\frac{u_0r_0}{n_1 n_3 \ln{n_3}}}$ in Assumption~\ref{assump:assump}. The former seems more restrictive if $r_{\rm{t}}$ is regarded analogously to $r_0$.
% $$
% r_{\rm{t}}(\CL)\le \theta_{\rm{t}}\frac{n_1 n_3}{u_0\ln^2{(n_2 n_3)}},\quad\text{which is asymptotically worse than our $r_0\le \theta_0\frac{ (1-\rho)n_1}{u_0\ln^2{n_3}}$}
% $$
% due to the appearance of $n_3$ if 

% The requirements on the incoherence conditions in~\cite[Definition~4.1]{lu2020tensor} also have similar issue.
%       Moreover, we highlight that the tensor tubal rank is not as standard as the Tucker rank considered here, and to our best knowledge, their connection is still not clear. These facts show that the generalization proposed in~\cite{lu2020tensor} has some undesirable discordance. We believe this is mainly due to the fact that t-SVD treats a third-order tensor as a matrix, and thus losses part of the intrinsic structures of tensors.

We are also aware of an unpublished work of tensor robust PCA based on the vanilla tensor nuclear norm~\cite{driggs2019tensor}. We highlight that our result has some advantages. Similar to the t-SVD-based methods, the analysis in~\cite{driggs2019tensor} also applies to third-order tensors only. Even for third-order tensors, 
        % We suspect the reason for they only consider the third-order case is because an easy-to-use inclusion of the subdifferential of the nuclear norm of higher-order tensors was unknown to them. 
some assumptions in~\cite{driggs2019tensor} can be stricter than ours. For example, the requirements of $\min_{\CZ\in\Z(\CL)}\|\CZ\|_\infty$ in~\cite[Theorem~1]{driggs2019tensor} and in Theorem~\ref{thm:main-thm} are 
        \begin{equation*}
            \min_{\CZ\in\Z(\CL)}\|\CZ\|_\infty\le\sqrt{\frac{u_{\rm{d}} r_{\rm{d}}}{n_1 n_2 n_3}}\text{ with } r_{\rm{d}}=\sqrt{\frac{r_1 r_2 n_3+r_1 r_3 n_2+ r_2 r_3 n_1}{n_1+n_2+n_3}}  \mbox{ and }\min_{\CZ\in\Z(\CL)}\|\CZ\|_\infty\le\sqrt{\frac{u_0r_0}{n_1 n_3 \ln{n_3}}}, %\mbox{ for }r_0=\max_{k\in[3]}r_k
        \end{equation*}
respectively. 
%whereas Assumption~\ref{assump:assump} imposes $\min_{\CZ\in\Z(\CL)}\|\CZ\|_\infty\le\sqrt{u_0r_0}/\sqrt{n_1 n_3 \ln{n_3}}$ instead. 
While it is difficult to compare the two upper bounds directly, in the case when $r_k=r$ and $n_k=n$ for $k\in[3]$, they become $O(\sqrt{\frac{ r}{n^3}})$ and $O(\sqrt{\frac{ r}{n^2\ln{n}}})$, respectively. It shows that our bound is clearly better in this case. That being said, the rank requirements in~\cite[Theorem~1]{driggs2019tensor} and in Theorem~\ref{thm:main-thm} are  
    \begin{equation*}
        r_{\rm{d}}\le \theta_{\rm{d}}\sqrt{\frac{n_1 n_2 n_3}{(n_1+n_2+n_3)\ln(n_1 n_2 n_3)}} \text{ and } r_0\le \theta_0\frac{(1-\rho)n_1}{u_0\ln^2{n_3}},
    \end{equation*}
    respectively.
% whereas Assumption~\ref{assump:assump} imposes $r_0\le \theta_0 {(1-\rho)n_1}/{(u_0\ln^2{n_3})}$ instead. 
In the case when $n_k=n$ for $k\in[3]$, the two upper bounds become $O(\frac{n}{\sqrt{\ln{n}}})$ and $O(\frac{n}{\ln^2{n}})$, respectively. It shows that our bound is slightly worse if $r_{\rm{d}}$ is regarded analogously to $r_0$.

{\color{black}
It is worth mentioning that the differences between the recovery bounds required in~\cite[Theorem~1]{driggs2019tensor} and in Theorem~\ref{thm:main-thm}
are driven by 
% the
different subspaces 
% geometries structures
coupled with different subdifferential inclusions, i.e.,~\eqref{eq:subdiff-yuanming} and~\eqref{eq:subdiff-nuclear-L}.
% respectively
%i.e., $\bigoplus_{|\I|\ge2,\,\I\subseteq[3]}\TT^\I
% (\CT)
$ 
%and
%$\TT^{[3]}
% (\CL)
$
% In fact, 
These differences largely govern the quantitative differences in 
% several
many subsequent 
% probabilistic
tools used in~\cite{driggs2019tensor} and in this chapter.
% are largely governed by the geometry.
For example, \cite[Lemma~7]{driggs2019tensor} and Lemma~\ref{lma:incoherence} to be introduced soon, both concerning the projections of standard basis tensors,
%(differing only in the projection structures), 
% but in different structures
% (though framed under different structural settings), 
state that
$$
\max_{\bi\in\I^3} \mleft\|\proj_{\bigoplus_{|\I|\le 1,\,\I\subseteq[3]}\TT^\I(\CL)}\mleft(\bigotimes_{k=1}^3 \be_{i_k}\mright)\mright\|_2^2 \le  \frac{v_{\rm{d}}^2r_{\rm{d}}^2 \sum_{k=1}^3n_k}{n_1 n_2 n_3}
\mbox{ and }
    \max_{\bi\in\I^3} \mleft\|\proj_{\bigoplus_{|\I|\le 2,\,\I\subseteq[3]}\TT^\I(\CL)}\mleft(\bigotimes_{k=1}^3 \be_{i_k}\mright)\mright\|_2^2\le u_0\sum_{k=1}^3\frac{r_k}{n_k},
$$ 
respectively.
% (We note in passing that, here,
% the projection structures $\bigoplus_{|\I|\le 1,\,\I\subseteq[3]}\TT^\I$
%  and $\bigoplus_{|\I|\le 2,\,\I\subseteq[3]}\TT^\I$
% are the orthogonal complements of $\bigoplus_{|\I|\ge2,\,\I\subseteq[3]}\TT^\I$ 
% and
% $\TT^{[3]}$
% mentioned earlier.) 
If $r_{\rm{d}}$ is regarded analogously to $r_0$ and $n_k=n$ for $k\in[3]$, then the above two bounds become $O(\frac{r_0^2}{n^2})$ and $O(\frac{r_0}{n})$, respectively.
The two bounds scale differently and propagate through subsequent 
% Bernstein-type
bounds in the analysis, resulting in different recovery bounds.
% is what ultimately changes the recovery bounds.
% where $v_{\rm{d}}$ is a constant similar to $u_{\rm{d}}$, and .
% which is a direct geometric consequence of 4-block (i.e.,) vs.\ 7-block (i.e.,)
% quantitative differences in the probabilistic tools
In any case, we believe that the recovery bounds in~\cite[Theorem~1]{driggs2019tensor} and in Theorem~\ref{thm:main-thm} offer complementary insights, when the order of the tensor is three.}

We remark that there are also a lot of works based on the so-called sum-of-nuclear-norms~\cite{liu2012tensor,huang2015provable}, i.e., using $\sum_{k=1}^d\lambda_k\|\BT_{(k)}\|_*$ as a tractable surrogate of $\|\CT\|_*$. We do not compare Theorem~\ref{thm:main-thm} with these results since their assumptions involve quantities that are absent from our framework.

The rest of this section is devoted to the proof of Theorem~\ref{thm:main-thm}. The overall framework of the proof is similar to that of the matrix robust PCA in~\cite[Chapter~5]{wright2022high} but there are quite a lot of details to be dealt with for tensors. 

\subsection{Unique optimality}

% \begin{assumption}
%     % Let $\CZ\in\TT(\{\V_k\}_{k=1}^d)$ be the tensor for which $\|\CZ\|=1$ and $\langle\CZ,\CT\rangle=\|\CT\|_*$.
%     We assume 
    
% \end{assumption}

% \begin{assumption}
%     We assume $\OO\sim\operatorname{Bernoulli}(\rho)$, and the nonzero entries of $\CS$ are independently chosen from $\operatorname{Unif}(\{\pm 1\})$.
% \end{assumption}

% \subsection{Unique optimality}

% and
% $$
% \partial\|\CL\|_*\supseteq\conv\mleft(\bigcup_{\ell=2}^d\bigcup_{j_\ell>\dots>j_1}\mleft\{\CZ+\CY:\CY\in\range\mleft(\bigotimes_{k=1}^d\mleft(\mathbbm{1}_{i\in\{j_k\}_1^\ell}(\proj_{\V_k^\perp}-\proj_{\V_k})+\proj_{\V_k}\mright)\mright)\cap\mathbb{B}\mright\}\mright),
% $$
% where $\V_k:=\spn_i(\CL)$, $\CZ\in\TT(\{\V_k\}_{k=1}^d)$ is the tensor for which $\|\CZ\|=1$ and $\langle\CZ,\CL\rangle=\|\CL\|_*$, $\OO := \operatorname{supp}\mleft(\CS\mright) \subseteq\bigtimes_{k=1}^d[n_i]$, and $\CE := \operatorname{sign}\mleft(\CS\mright) \in\{-1,0,1\}^{n_1 \times n_2\times \dots\times n_d}$. 

To prove Theorem~\ref{thm:main-thm}, we start to characterize the conditions under which $(\CL,\CS)$ is the unique optimal solution of~\eqref{eq:PCP}. They are essentially the existence of a tensor in the relative interior of $\partial\|\CL\|_*\cap\,\lambda\,\partial\|\CS\|_1$. % that is simultaneously a subgradient of $\|\bullet\|*$ and $\|\bullet\|_1$ at $\CL$ and $\CS$, respectively.

To manipulate $\partial\|\CL\|_*$, we resort to the inclusion in Corollary~\ref{thm:newset} with $\I=[d]$, i.e.,
\begin{equation}\label{eq:subdiff-nuclear-L}
\bigl\{\CZ+\CX:\CZ\in\Z(\CL),\,\CX\in\TT^{[d]}(\CL),\,\|\CX\|_\sigma\le 1\bigr\} \subseteq \partial \|\CL\|_*.
\end{equation}
In order to alleviate the lengthy notation in derivations, let us denote $\LL:=\bigoplus_{|\I|\le d-1,\,\I\subseteq[d]}\TT^\I(\CL)$ as the direct sum of all basic subspaces defined by $\CL$ except $\TT^{[d]}(\CL)$, or equivalently $\LL^\perp=\TT^{[d]}(\CL)$ as the orthogonal complement.

To study $\partial\|\CS\|_1$, we need to define subspaces based on the entries of a tensor.
%Let us denote $\I^d:=\{\bi=(i_1,i_2,\dots,i_d)^{\T}:i_k\in[n_k]\,\forall\,k\in[d]\}$ to be the set of entry indices for the tensor space $\R^{n_1\times n_2\times \dots\times n_d}$. 
For any tensor $\CT\in\R^{n_1\times n_2\times \dots\times n_d}$, we denote $\I(\CT):=\{\bi\in\I^d:t_{\bi}\ne0\}$ to be the support set of $\CT$, recalling that $\I^d=\bigl\{\bi\in\N^d:i_k\in[n_k]\,\forall\,k\in[d]\bigr\}$. We use the notation $\I(\CS)\sim\operatorname{Bernoulli}(\rho)$ to denote the sampling process of $\CS$ in Assumption~\ref{assump:assump}, i.e., $\Prob\{\bi\in\I(\CS)\}=\rho$ for every $\bi\in\I^d$ independently.
%; occasionally, we will also abuse this notation by applying it to general index sets without explicitly mentioning $\CS$. 
For any index set $\I\subseteq\I^d$, we also abuse $\I$ to denote a subspace of $\R^{n_1\times n_2\times \dots\times n_d}$ under no ambiguity, i.e., 
$$
\I=\spn\bigl(\{\be_{i_1}\otimes \be_{i_2}\otimes\dots\otimes\be_{i_d}:\bi\in\I\}\bigr).
$$
Therefore, $\dim(\I)=|\I|$ where the former $\I$ is a subspace and the latter is an index set. This also makes $\proj_{\I}$ self-explanatory. For example, $\proj_{\varnothing}$ is the zero operator, $\proj_{\I^d}$ is the identity operator as $\I^d=\R^{n_1\times n_2\times \dots\times n_d}$, and $\proj_{\{\bi\}}=\proj_{\spn(\{\be_{i_1}\otimes\be_{i_2}\otimes \dots\otimes\be_{i_d}\})}$, i.e., zeroing out all but the $\bi$th entry. It is also obvious that $\I^\perp=\spn\bigl(\{\be_{i_1}\otimes \be_{i_2}\otimes\dots\otimes\be_{i_d}:\bi\in\I^d\setminus\I\}\bigr)$. %; for this reason, we will also identify $\I^\perp$ with $\I^d\setminus\I$ as long as the distinction is clear from context. 
Let $\CE:=\sign(\CS)$, which is a random tensor with entries being $0$, $1$, or $-1$ since $\CS$ is random. It is easy to check (see also~\cite[Example~3.41]{beck2017first}) that 
\begin{equation}\label{eq:subdiff-l1-L}
\partial\|\CS\|_1=\bigl\{\CE+\CF : \proj_{\I(\CS)}(\CF)=\CO,\,\|\CF\|_\infty \le 1\bigr\}.
\end{equation}

\begin{lemma}\label{lma:uniqueness1}
    If $\LL\cap\I(\CS)=\{\CO\}$, or equivalently $\|\proj_{\LL}\proj_{\I(\CS)}\|< 1$, then $(\CL,\CS)$ is the unique optimal solution of~\eqref{eq:PCP} if there exists a $\CD\in\R^{n_1\times n_2\times \dots\times n_d}$ such that 
    $$
    \proj_{\LL}(\CD)\in\Z(\CL),~\bigl\|\proj_{\LL^\perp}(\CD)\bigr\|_\sigma< 1,~\proj_{\I(\CS)}(\CD)=\lambda\CE,\text{ and }\bigl\|\proj_{\I^\perp(\CS)}(\CD)\bigr\|_\infty<\lambda.
    $$
\end{lemma}
\begin{proof}
Since any feasible solution of~\eqref{eq:PCP} can be written as $(\CL+\CH,\CS-\CH)$ for some perturbation $\CH$, it suffices to show that $\|\CL+\CH\|_*+\lambda\|\CS-\CH\|_1>\|\CL\|_*+\lambda\|\CS\|_1$ for any $\CH\ne\CO$. 

By Lemma~\ref{prop:nuclear-U-equiv}, there exists an $\CX\in\LL^\perp$ with $\|\CX\|_\sigma=1$ such that $\bigl\langle\proj_{\LL^\perp}(\CH),\CX\bigr\rangle=\bigl\|\proj_{\LL^\perp}(\CH)\bigr\|_*$. As a result, $\proj_{\LL}(\CD)+\CX\in\partial\|\CL\|_*$ by~\eqref{eq:subdiff-nuclear-L}. Let $\CF\in\I^\perp(\CS)$ with $\|\CF\|_\infty=1$ such that $\bigl\langle\proj_{\I^\perp(\CS)}(\CH),\CF\bigr\rangle=-\bigl\|\proj_{\I^\perp(\CS)}(\CH)\bigr\|_1$. It is obvious that $\CE+\CF\in\partial\|\CS\|_1$ by~\eqref{eq:subdiff-l1-L}. Thus, by the definition of subgradient,
        \begin{align*}
            &\|\CL+\CH\|_*+\lambda\|\CS-\CH\|_1
            \\\ge\,& \|\CL\|_*+\lambda\|\CS\|_1+\bigl\langle\proj_{\LL}(\CD)+\CX,\CH\bigr\rangle+\lambda\langle\CE+\CF,-\CH\rangle
            \\=\,&\|\CL\|_*+\lambda\|\CS\|_1+\langle\CX,\CH\rangle-\lambda\langle\CF,\CH\rangle + \bigl\langle\proj_{\LL}(\CD)-\proj_{\I(\CS)}(\CD),\CH\bigr\rangle%\underbrace{\langle\proj_{\LL}(\CD)-\lambda\CE,\CH\rangle}_{\mathclap{=\langle\proj_{\LL}(\CD)-\proj_{\I(\CS)}(\CD),\CH\rangle}}
           % \\=\,&\|\CL\|_*+\lambda\|\CS\|_1+\|\proj_{\LL^\perp}(\CH)\|_*+\lambda\|\proj_{\I^\perp(\CS)}(\CH)\|_1 + \langle\proj_{\LL}(\CD)-\proj_{\I(\CS)}(\CD),\CH\rangle
            \\=\,&\|\CL\|_*+\lambda\|\CS\|_1+\bigl\|\proj_{\LL^\perp}(\CH)\bigr\|_*+\lambda\bigl\|\proj_{\I^\perp(\CS)}(\CH)\bigr\|_1 + \bigl\langle\proj_{\I^\perp(\CS)}(\CD)-\proj_{\LL^\perp}(\CD),\CH\bigr\rangle
            \\\ge\,& \|\CL\|_*+\lambda\|\CS\|_1+\bigl(1-\bigl\|\proj_{\LL^\perp}(\CD)\bigr\|_\sigma\bigr)\bigl\|\proj_{\LL^\perp}(\CH)\bigr\|_*+\bigl(\lambda-\bigl\|\proj_{\I^\perp(\CS)}(\CD)\bigr\|_\infty\bigr)\bigl\|\proj_{\I^\perp(\CS)}(\CH)\bigr\|_1
            \\>\,& \|\CL\|_*+\lambda\|\CS\|_1
        \end{align*}
    as long as $\max\bigl\{\bigl\|\proj_{\LL^\perp}(\CH)\bigr\|_*,\bigl\|\proj_{\I^\perp(\CS)}(\CH)\bigr\|_1\bigr\}>0$ since $\bigl\|\proj_{\LL^\perp}(\CD)\bigr\|_\sigma< 1$ and $\bigl\|\proj_{\I^\perp(\CS)}(\CD)\bigr\|_\infty<\lambda$.

    Suppose on the contrary that $\bigl\|\proj_{\LL^\perp}(\CH)\bigr\|_*=\bigl\|\proj_{\I^\perp(\CS)}(\CH)\bigr\|_1=0$. This means that $\CH\in\LL$ and $\CH\in\I(\CS)$, and so $\CH\in\LL\cap\I(\CS)=\{\CO\}$, a contradiction to $\CH\ne\CO$.
\end{proof}

To guarantee the exact recovery by solving~\eqref{eq:PCP}, it suffices to construct a dual certificate $\CD$ that satisfies the conditions stated in Lemma~\ref{lma:uniqueness1}. However, restricting $\proj_{\LL}(\CD)\in\Z(\CL)$ and $\proj_{\I(\CS)}(\CD)=\lambda\CE$ simultaneously can be demanding and in fact may not be possible. Thus, we need to relax these two conditions, from zero distance to $\frac{\lambda}{8}$ in terms of the Frobenius norm. As a compensation, the other two conditions need in turn to be restricted by shrinking their radii by half. 

% \subsection{Relaxed requirements}
\begin{lemma}\label{lma:relax}
If $\|\proj_{\LL}\proj_{\I(\CS)}\|< \frac{1}{2}$ and $\lambda<1$, then $(\CL,\CS)$ is the unique optimal solution of~\eqref{eq:PCP} if there exists a $\CD\in\R^{n_1\times n_2\times \dots\times n_d}$ such that
% $$
\begin{equation}\label{eq:existence-dual-certificate}
\min_{\CZ\in\Z(\CL)}\bigl\|\proj_{\LL}(\CD)-\CZ\bigr\|_2\le \frac{\lambda}{8},~\bigl\|\proj_{\LL^\perp}(\CD)\bigr\|_\sigma< \frac{1}{2},~\bigl\|\proj_{\I(\CS)}(\CD)-\lambda\CE\bigr\|_2\le\frac{\lambda}{8},\text{ and }\bigl\|\proj_{\I^\perp(\CS)}(\CD)\bigr\|_\infty<\frac{\lambda}{2}.
\end{equation}
% $$
\end{lemma}
\begin{proof}
    We first derive a couple of technical bounds. For any $\CH\in\R^{n_1\times n_2\times \dots\times n_d}$, we have
    % \begin{equation*}
        \begin{align*}
            \bigl\|\proj_{\I(\CS)}(\CH)\bigr\|_2&\le\bigl\|\proj_{\I(\CS)}\proj_{\LL}(\CH)\bigr\|_2+\bigl\|\proj_{\I(\CS)}\proj_{\LL^\perp}(\CH)\bigr\|_2 \\
            &\le\frac{1}{2}\|\CH\|_2+\bigl\|\proj_{\LL^\perp}(\CH)\bigr\|_2
            \\&\le\frac{1}{2}\bigl(\bigl\|\proj_{\I(\CS)}(\CH)\bigr\|_2+\bigl\|\proj_{\I^\perp(\CS)}(\CH)\bigr\|_2\bigr)+\bigl\|\proj_{\LL^\perp}(\CH)\bigr\|_2,
        \end{align*}
    which implies that $\bigl\|\proj_{\I(\CS)}(\CH)\bigr\|_2\le\bigl\|\proj_{\I^\perp(\CS)}(\CH)\bigr\|_2+2\bigl\|\proj_{\LL^\perp}(\CH)\bigr\|_2$. For the same reason, we also have $\bigl\|\proj_{\LL}(\CH)\bigr\|_2\le\bigl\|\proj_{\LL^\perp}(\CH)\bigr\|_2+2\bigl\|\proj_{\I^\perp(\CS)}(\CH)\bigr\|_2$. 

    Similar to the proof of Lemma~\ref{lma:uniqueness1}, we have $\CZ+\CX\in\partial\|\CL\|_*$, where $\CZ\in\Z(\CL)$ is an optimal solution of $\min_{\CZ\in\Z(\CL)}\bigl\|\proj_{\LL}(\CD)-\CZ\bigr\|_2$ (the existence is guaranteed by Lemma~\ref{lma:convexity-ZT}) and $\CX\in\LL^\perp$ with $\|\CX\|_\sigma=1$ and $\bigl\langle\proj_{\LL^\perp}(\CH),\CX\bigr\rangle=\bigl\|\proj_{\LL^\perp}(\CH)\bigr\|_*$. We also have $\CE+\CF\in\partial\|\CS\|_1$, where $\CF\in\I^\perp(\CS)$ with $\|\CF\|_\infty=1$ and $\bigl\langle\proj_{\I^\perp(\CS)}(\CH),\CF\bigl\rangle=-\bigl\|\proj_{\I^\perp(\CS)}(\CH)\bigr\|_1$. Therefore, for any $\CH\ne\CO$,
    % \begin{equation*}
        \begin{align*}
                    &\|\CL+\CH\|_*+\lambda\|\CS-\CH\|_1            \\
            \ge\,& \|\CL\|_*+\lambda\|\CS\|_1+\langle\CZ+\CX,\CH\rangle+\lambda\langle\CE+\CF,-\CH\rangle            \\
            =\,&\|\CL\|_*+\lambda\|\CS\|_1+ \bigl\langle\CX-\proj_{\LL^\perp}(\CD)-\lambda\CF+\proj_{\I^\perp(\CS)}(\CD)+\bigl(\CZ-\proj_{\LL}(\CD)\bigr)+\bigl(\proj_{\I(\CS)}(\CD)-\lambda\CE\bigr),\CH\bigr\rangle\\            
            % =\,&\|\CL\|_*+\lambda\|\CS\|_1+\|\proj_{\LL^\perp}(\CH)\|_*+\lambda\|\proj_{\I^\perp(\CS)}(\CH)\|_1 + \underbrace{\langle\CZ-\lambda\CE,\CH\rangle}_{\mathclap{=\langle(\CZ-\proj_{\LL}(\CD))+(\proj_{\LL}(\CD)-\proj_{\I(\CS)}(\CD))+(\proj_{\I(\CS)}(\CD)-\lambda\CE),\CH\rangle}} \\
            \ge\,&\|\CL\|_*+\lambda\|\CS\|_1+\mleft(1-\frac{1}{2}\mright)\bigl\|\proj_{\LL^\perp}(\CH)\bigr\|_*+\mleft(\lambda-\frac{\lambda}{2}\mright)\bigl\|\proj_{\I^\perp(\CS)}(\CH)\bigr\|_1 - \frac{\lambda}{8}\bigl\|\proj_{\LL}(\CH)\bigr\|_2 - \frac{\lambda}{8}\bigl\|\proj_{\I(\CS)}(\CH)\bigr\|_2
            \\
            \ge\,&\|\CL\|_*+\lambda\|\CS\|_1+\frac{1}{2}\bigl\|\proj_{\LL^\perp}(\CH)\bigr\|_*+\frac{\lambda}{2}\bigl\|\proj_{\I^\perp(\CS)}(\CH)\bigr\|_1 - \frac{3\lambda}{8}\bigl\|\proj_{\LL^\perp}(\CH)\bigr\|_2 - \frac{3\lambda}{8}\bigl\|\proj_{\I^\perp(\CS)}(\CH)\bigr\|_2
            \\
            \ge\,&\|\CL\|_*+\lambda\|\CS\|_1+\mleft(\frac{1}{2}-\frac{3\lambda}{8}\mright)\bigl\|\proj_{\LL^\perp}(\CH)\bigr\|_*+\frac{\lambda}{8}\bigl\|\proj_{\I^\perp(\CS)}(\CH)\bigr\|_1,
        \end{align*}
    % \end{equation*}
    where the penultimate inequality is due to the technical bounds derived earlier and the last one is due to the trivial bounds among tensor norms~\eqref{eq:trivial}. Since $\lambda<1$, the desired result follows immediately %as in the proof of Lemma~\ref{lma:uniqueness1} 
    by the condition $\|\proj_{\LL}\proj_{\I(\CS)}\|< \frac{1}{2}<1$, i.e., $\LL\cap\I(\CS)=\{\CO\}$.
\end{proof}

The task of the remaining subsections is to study the concentration behavior of $\|\proj_{\LL}\proj_{\I(\CS)}\|$ and to construct a dual certificate $\CD$ that satisfies the conditions in Lemma~\ref{lma:relax} with high probability. Using the idea in~\cite{candes2011robust}, we construct a low-rank part $\CD_1$ and a sparse part $\CD_2$ separately and then form the dual certificate $\CD=\CD_1+\CD_2$.

\subsection{Concentration behavior of $\|\proj_{\LL}\proj_{\I(\CS)}\|$}\label{sec:POPMless}

In this subsection, we show that $\|\proj_{\LL}\proj_{\I(\CS)}\|\le\delta$ holds with high probability for any $\delta\in(0,1]$, a stronger result than that required in Lemma~\ref{lma:relax}. As a necessary preparation, we first show that any standard basis of $\R^{n_1\times n_2\times\dots\times n_d}$ is far away from the subspace $\LL$ as a direct consequence of the incoherence conditions in Assumption~\ref{assump:assump}. 

\begin{lemma}\label{lma:incoherence}
If the condition 
$
u_k=\frac{n_k}{r_k} \max_{i\in[n_k]}\bigl\|\proj_{\spn_k(\CL)}(\be_i)\bigr\|_2^2\le u_0
$
for $k\in[d]$
in Assumption~\ref{assump:assump} holds, then 
$$
\max_{\bi\in\I^d}\,\mleft\|\proj_{\LL}\mleft(\bigotimes_{k=1}^d \be_{i_k}\mright)\mright\|_2^2\le u_0\sum_{k=1}^d\frac{r_k}{n_k}.
$$
\end{lemma}
\begin{proof}
    For any $\bi\in\I^d$, we have
    \begin{align*}
        &\mleft\|\proj_{\LL}\mleft(\bigotimes_{k=1}^d \be_{i_k}\mright)\mright\|_2^2\\
        =\,&\mleft\|\bigotimes_{k=1}^d \be_{i_k}-\bigotimes_{k=1}^d \proj_{\spn_k^\perp(\CL)}(\be_{i_k})\mright\|_2^2
        \\=\,&\mleft\|\sum_{\ell=1}^d\mleft(\mleft(\bigotimes_{k=1}^{\ell-1}\proj_{\spn_k^\perp(\CL)}(\be_{i_k})\mright)\otimes\mleft(\bigotimes_{k=\ell}^d\be_{i_k}\mright)-\mleft(\bigotimes_{k=1}^{\ell}\proj_{\spn_k^\perp(\CL)}(\be_{i_k})\mright)\otimes\mleft(\bigotimes_{k=\ell+1}^d\be_{i_k}\mright)\mright)\mright\|_2^2
        \\=\,&\sum_{\ell=1}^d\,\mleft\|\mleft(\bigotimes_{k=1}^{\ell-1}\proj_{\spn_k^\perp(\CL)}(\be_{i_k})\mright)\otimes\bigl(\be_{i_\ell}-\proj_{\spn_\ell^\perp(\CL)}(\be_{i_\ell})\bigr)\otimes\mleft(\bigotimes_{k=\ell+1}^d\be_{i_k}\mright)\mright\|_2^2\\
        \le\,&\sum_{\ell=1}^d\,\bigl\|\proj_{\spn_\ell(\CL)}(\be_{i_\ell})\bigr\|_2^2\\
        \le\,&u_0\sum_{\ell=1}^d \frac{r_\ell}{n_\ell},
    \end{align*}
    where the last equality holds because all the rank-one tensors there are mutually orthogonal and the last inequality follows from the condition on $u_k$.
\end{proof}

The most important result in this part is to bound the tail probability of $\bigl\|\proj_{\LL}(\proj_{\I^d}-q^{-1}\proj_{\I})\proj_{\LL}\bigr\|$, i.e., Proposition~\ref{prop:hdp-1-1}. To achieve this, we treat a projection onto a subspace in $\R^{n_1\times n_2\times \dots\times n_d}$, which is itself a linear operator, as a $\prod_{k=1}^d n_k$ by $\prod_{k=1}^d n_k$ matrix, and then apply the matrix Bernstein inequality~\cite[Theorem~1.4]{tropp2012user}:
% Following the strategy used in the proof of~\cite[Lemma~5]{yuan2016tensor}, we manage to similarly borrow the power of the matrix Bernstein inequality~\cite[Theorem~1.4]{tropp2012user}. {We remark that although this inequality only applies to matrices, when dealing with linear operators from $\R^{n_1\times n_2\times\dots\times n_d}$ to itself, it suffices to treat them as linear mappings from $\R^{\prod_{k=1}^d n_k}$ to itself instead, i.e., a matrix in $\R^{(\prod_{k=1}^d n_k)\times (\prod_{k=1}^d n_k)}$, and then apply the desired inequality. %This works out because $\R^{n_1\times n_2\times\dots\times n_d}$ is isometrically isomorphic to $\R^{\prod_{k=1}^d n_k}$, and thus the operator norms of the two representations of the concerned linear operator are identical by the very definition.}
If $\BX_1,\BX_2,\dots,\BX_m\in\R^{n\times n}$ are independent and self-adjoint random matrices such that $\Exp \BX_i=\BO$ and $\text{emax}(\BX_i)\le s$ almost surely for any $i\in[m]$, then for any $t\ge0$,
\begin{equation}\label{eq:bernstein}
\Prob\mleft\{ \text{emax}\mleft(\sum_{i=1}^m\BX_i\mright)\ge t\mright\}\le n\exp\mleft(\frac{-3t^2}{6\|\sum_{i=1}^m\Exp \BX_i^2\|_\sigma+2st}\mright),
\end{equation}
where $\text{emax}(\BX)$ denotes the largest eigenvalue of $\BX$. In particular, $\text{emax}(\BX)=\|\BX\|_\sigma$ if $\BX$ is positive semidefinite.

\begin{proposition}\label{prop:hdp-1-1}
    Under Assumption~\ref{assump:assump}, if $\I\sim\operatorname{Bernoulli}(q)$ and $t>0$, then
    \begin{equation*}
        \Prob\bigl\{\bigl\|\proj_{\LL}(\proj_{\I^d}-q^{-1}\proj_{\I})\proj_{\LL}\bigr\|\ge t\bigr\}\le \exp\mleft(\frac{-3t^2q}{u_0(6+2t)\sum_{k=1}^d\frac{r_k}{n_k}}\mright)\prod_{k=1}^d n_k.
    \end{equation*}
\end{proposition}
% {We remark that although this inequality only applies to matrices, when dealing with tensors, we may simply apply the inequality to (the symmetric embeddings of) the matricizations of the tensors and use the fact that the spectral norm of the matricization of a tensor is no less than the original one~\cite[Proposition~4.1]{wang2017operator}, although this manipulation will lead to a weaker result. In other words, this means that the direct generalization of the matrix Bernstein inequality to tensors still works (although being weaker), and owing to this we will not distinguish them in what follows.}
% Before we present the proof, we shall first remark that as $(\R^{n_1\times n_2\times\dots\times n_d},\|\bullet\|_2)\cong(\R^{\prod_{k=1}^d n_k},\|\bullet\|)$ and~\cite[Proposition~4.1]{wang2017operator}, it suffices view any linear operator $\mathcal{A}:(\R^{n_1\times n_2\times\dots\times n_d},\|\bullet\|_2)\rightarrow(\R^{n_1\times n_2\times\dots\times n_d},\|\bullet\|_2)$ as a matrix $\boldsymbol{A}:(\R^{\prod_{k=1}^d n_k},\|\bullet\|)\rightarrow(\R^{\prod_{k=1}^d n_k},\|\bullet\|)$ and apply the matrix Bernstein inequality to bound their spectral norms. 
\begin{proof}
    %Denote $\I^d:=\{\bi=(i_1,i_2,\dots,i_d):i_k\in[n_k]\,\forall\,k\in[d]\}$ to be the set of entry indices for the tensor space $\R^{n_1\times n_2\times \dots\times n_d}$ and denote $\proj_{\bi}:=\proj_{\spn(\be_{i_1}\otimes\be_{i_2}\otimes \dots\otimes\be_{i_d})}$, i.e., zeroing out all but the $\bi$th entries. It is clear that
    We first notice that 
    $$
    \proj_{\LL}(\proj_{\I^d}-q^{-1}\proj_{\I})\proj_{\LL}=\sum_{\bi\in\I^d}\proj_{\LL}(\proj_{\bi}-q^{-1}\proj_{\bi}\proj_{\I})\proj_{\LL},%_{:=\mathcal{M}_{i_1,i_2,\dots,m_d}},
    $$
   % where $\proj_{\varnothing}$ is a zero operator by default. 
    which is a sum of independent, zero-mean, and positive semidefinite random linear operators as
    \begin{equation}\label{eq:operator-case-by-case}
        \proj_{\LL}(\proj_{\bi}-q^{-1}\proj_{\bi}\proj_{\I})\proj_{\LL}=
        \begin{dcases}
            (1-q^{-1})\proj_{\LL}\proj_{\bi}\proj_{\LL} & \bi\in\I \\
            \proj_{\LL}\proj_{\bi}\proj_{\LL} & \bi\notin\I.
        \end{dcases}
    \end{equation}
    In order to apply the matrix Bernstein inequality to bound its tail probability, we need to control both $\bigl\|\proj_{\LL}(\proj_{\bi}-q^{-1}\proj_{\bi}\proj_{\I})\proj_{\LL}\bigr\|$ and $\bigl\|\sum_{\bi\in\I^d}\Exp\bigl(\proj_{\LL}(\proj_{\bi}-q^{-1}\proj_{\bi}\proj_{\I})\proj_{\LL}\bigr)^2\bigr\|$. In fact, for any $\bi\in\I^d$,
        \begin{align*}
       \bigl\|\proj_{\LL}(\proj_{\bi}-q^{-1}\proj_{\bi}\proj_{\I})\proj_{\LL}\bigr\|& \le \max\bigl\{|1-q^{-1}|,1\bigr\}\|\proj_{\LL}\proj_{\bi}\proj_{\LL}\|\\
           & \le q^{-1}\max_{\|\CX\|_2\le 1}\bigl\langle\proj_{\LL}\proj_{\bi}\proj_{\LL}(\CX),\CX\bigr\rangle \\
           & = q^{-1}\max_{\|\CX\|_2\le 1}\bigl\langle\proj_{\bi}\proj_{\LL}(\CX),\proj_{\bi}\proj_{\LL}(\CX)\bigr\rangle \\
        & = q^{-1}\max_{\|\CX\|_2\le 1}\mleft\langle\bigotimes_{k=1}^d\be_{i_k},\proj_{\LL}(\CX)\mright\rangle^2\\
        &    = q^{-1}\mleft\|\proj_{\LL}\mleft(\bigotimes_{k=1}^d \be_{i_k}\mright)\mright\|_2^2\\
        &\le q^{-1}u_0\sum_{k=1}^d\frac{r_k}{n_k},
    \end{align*}
    % \end{equation*}
    where the last inequality follows from Lemma~\ref{lma:incoherence}. On the other hand, since $\Exp\bigl(\proj_{\LL}(q^{-1}\proj_{\bi}\proj_{\I})\proj_{\LL}\bigr)=\proj_{\LL}\proj_{\bi}\proj_{\LL}$ and the variance of 
$\proj_{\LL}(q^{-1}\proj_{\bi}\proj_{\I})\proj_{\LL}$ is no more than its second moment, %$\allowbreak{\Exp(\proj_{\LL}(q^{-1}\proj_{\bi}\proj_{\I})\proj_{\LL})^2}$.
we have
        \begin{align*}
            \mleft\|\sum_{\bi\in\I^d}\Exp\bigl(\proj_{\LL}(\proj_{\bi}-q^{-1}\proj_{\bi}\proj_{\I})\proj_{\LL}\bigr)^2\mright\|&\le \mleft\|\sum_{\bi\in\I^d}\Exp\bigl(\proj_{\LL}(q^{-1}\proj_{\bi}\proj_{\I})\proj_{\LL}\bigr)^2\mright\|
            \\&=q^{-2}\mleft\|\sum_{\bi\in\I^d}\proj_{\LL}\proj_{\bi}\proj_{\LL}\proj_{\bi}\proj_{\LL}\Prob\{\bi\in\I\}\mright\|
            \\&=q^{-1}\mleft\|\sum_{\bi\in\I^d}\proj_{\LL}\proj_{\bi}\proj_{\LL}\proj_{\bi}\proj_{\LL}\mright\|
            \\&=q^{-1}\max_{\|\CX\|_2\le 1}\sum_{\bi\in\I^d}\bigl\langle\proj_{\LL}\proj_{\bi}\proj_{\LL}\proj_{\bi}\proj_{\LL}(\CX),\CX\bigr\rangle
            \\&=q^{-1}\max_{\|\CX\|_2\le 1}\sum_{\bi\in\I^d}\bigl\langle\proj_{\LL}\proj_{\bi}\proj_{\LL}(\CX),\proj_{\LL}\proj_{\bi}\proj_{\LL}(\CX)\bigr\rangle
            \\&=q^{-1}\max_{\|\CX\|_2\le 1}\sum_{\bi\in\I^d}\mleft\langle\bigotimes_{k=1}^d \be_{i_k},\proj_{\LL}(\CX)\mright\rangle^2\mleft\|\proj_{\LL}\mleft(\bigotimes_{k=1}^d \be_{i_k}\mright)\mright\|_2^2 
            \\&\le q^{-1}\max_{\bi\in\I^d}\mleft\|\proj_{\LL}\mleft(\bigotimes_{k=1}^d \be_{i_k}\mright)\mright\|_2^2
            \max_{\|\CX\|_2\le 1}\sum_{\bi\in\I^d}\mleft\langle\bigotimes_{k=1}^d \be_{i_k},\proj_{\LL}(\CX)\mright\rangle^2%_{\mathclap{=\mleft\|\proj_{\LL}(\CX)\mright\|_2^2}}
            \\
            &\le q^{-1}u_0\sum_{k=1}^d\frac{r_k}{n_k},
        \end{align*}
    % \end{equation*}
    where the last inequality is due to Lemma~\ref{lma:incoherence} and that 
    $$
    \sum_{\bi\in\I^d}\mleft\langle\bigotimes_{k=1}^d \be_{i_k},\proj_{\LL}(\CX)\mright\rangle^2=\bigl\|\proj_{\LL}(\CX)\bigr\|_2^2.
    $$
    The desired inequality then follows immediately from the matrix Bernstein inequality~\eqref{eq:bernstein}.
\end{proof}

Proposition~\ref{prop:hdp-1-1} immediately implies the following two results and shows that $\|\proj_{\LL}\proj_{\I(\CS)}\|\le\delta$ holds with high probability for any $\delta\in(0,1]$.

\begin{corollary}\label{cor:hdp-1}
Under Assumption~\ref{assump:assump}, for any $\epsilon\in(0,1]$, there exists a $\kappa_0>0$ depending on $\epsilon$ only, such that $\|\proj_{\LL}\proj_{\I^\perp}\|\le\sqrt{1-q+q\epsilon}$ holds with high probability as long as $q\ge \kappa_0 \frac{{d^2u_0r_0\ln{n_d}}}{n_1}$ and $\I\sim\operatorname{Bernoulli}(q)$.
\end{corollary}
\begin{proof} By that 
$$
q\ge \kappa_0 \frac{{d^2u_0r_0\ln{n_d}}}{n_1}\ge \kappa_0 u_0\mleft(\sum_{k=1}^d\frac{r_k}{n_k}\mright)\sum_{k=1}^d \ln{n_k}
$$ 
and letting $t=\epsilon$ in Proposition~\ref{prop:hdp-1-1}, we have
    \begin{align*}
    \|\proj_{\LL}\proj_{\I^\perp}\proj_{\LL}\|=\|\proj_{\LL}-\proj_{\LL}\proj_{\I}\proj_{\LL}\|&\le \|\proj_{\LL}-q\proj_{\LL}\|+q\bigl\|\proj_{\LL}(\proj_{\I^d}-q^{-1}\proj_{\I})\proj_{\LL}\bigr\|\le 1-q+q\epsilon
\end{align*}
holds with probability at least
% $$
% 1-\exp\left(\frac{-3\epsilon^2q}{u_0(6+2\epsilon)\sum_{k=1}^d\frac{r_k}{n_k}}\right)\prod_{k=1}^d n_k
% \ge 1-\exp\left(\frac{-3\kappa_0\epsilon^2 d\ln n_d}{6+2\epsilon}\right)\prod_{k=1}^d n_k
% \ge 1-\left(\prod_{k=1}^d n_k\right)^{1-\frac{3\kappa_0\epsilon^2}{6+2\epsilon}}
% $$
\begin{align*}
1-\exp\mleft(\frac{-3\epsilon^2q}{u_0(6+2\epsilon)\sum_{k=1}^d\frac{r_k}{n_k}}\mright)\prod_{k=1}^d n_k
&\ge 1-\exp\mleft(\frac{-3\kappa_0\epsilon^2 \sum_{k=1}^d \ln{n_k}}{6+2\epsilon}\mright)\prod_{k=1}^d n_k\\
&= 1-\mleft(\prod_{k=1}^d n_k\mright)^{1-\frac{3\kappa_0\epsilon^2}{6+2\epsilon}},
\end{align*}
a high probability for a large enough $\kappa_0$. Notice that $q\ge \kappa_0 \frac{d^2u_0r_0\ln{n_d}}{n_1}$ is a probability and so we must ensure that $\kappa_0 \frac{d^2u_0r_0\ln{n_d}}{n_1}<1$. 
% This can always be guaranteed since $r_0\le \theta_0 \frac{(1-\rho)n_1}{u_0\ln^2{n_d}}$ in Assumption~\ref{assump:assump} for a sufficiently small $\theta_0$.
This can always be guaranteed in high dimensions since we have $r_0\le \theta_0\frac{ (1-\rho)n_1}{u_0\ln^2{n_d}}$ in Assumption~\ref{assump:assump} and so
$
    %r_0\le \theta_0\frac{ (1-\rho)n_1}{u_0\ln^2{n_d}}\rightarrow 
    \frac{\kappa_0 {d^2u_0r_0\ln{n_d}}}{n_1}\le\frac{\kappa_0 d^2 \theta_0 (1-\rho)}{\ln{n_d}}\rightarrow 0$ as $n_d\rightarrow\infty.
$ 
Therefore,
$$\|\proj_{\LL}\proj_{\I^\perp}\|
=\sqrt{\|\proj_{\LL}\proj_{\I^\perp}\proj_{\I^\perp}\proj_{\LL}\|} =\sqrt{\|\proj_{\LL}\proj_{\I^\perp}\proj_{\LL}\|}\le\sqrt{1-q+q\epsilon}$$
holds with high probability.
\end{proof}

\begin{corollary}\label{cor:hdp-2}
    For any $\delta\in(0,1]$, there exists a $\theta>0$ depending on $\delta$ only, such that $\|\proj_{\LL}\proj_{\I(\CS)}\|\le\delta$ holds with high probability as long as Assumption~\ref{assump:assump} holds for $\theta_0\le\theta$ and $\rho\le\frac{\delta^2}{2-\delta^2}$.
\end{corollary}
\begin{proof}
Since $\I(\CS)\sim\operatorname{Bernoulli}(\rho)$ by Assumption~\ref{assump:assump}, we have $\I^\perp(\CS)\sim\operatorname{Bernoulli}(1-\rho)$. Let $\kappa_0$ be the constant associated with $\epsilon=\frac{\delta^2}{2}$ in Corollary~\ref{cor:hdp-1} and further let $\theta=\frac{\ln{2}}{\kappa_0 d^2}$. Since $r_0\le \theta_0\frac{ (1-\rho)n_1}{u_0\ln^2{n_d}}$ in Assumption~\ref{assump:assump} and $\theta_0\le\theta\le\frac{\ln{n_d}}{\kappa_0 d^2}$, we have $1-\rho\ge\frac{u_0r_0\ln^2n_d}{\theta_0n_1}\ge \kappa_0\frac{{d^2u_0r_0\ln{n_d}}}{n_1}$, as required by Corollary~\ref{cor:hdp-1}.
% As an aside, as long as $\theta_0\le ()^{-1} $, (Here, the coefficient $\ln{2}$ is only used to deal with the case when $n_d=2$.) which is the constant $\theta$ in the statement, where , we also have by that
% $$
%     r_0\le \theta_0\frac{ (1-\rho)n_1}{u_0\ln^2{n_d}}\rightarrow 1-\rho\ge\frac{u_0r_0\ln^2n_d}{\theta_0n_1}\ge \kappa_0\frac{{d^2u_0r_0\ln{n_d}}}{n_1},
% $$
% % $r_0\le \theta_0 {(1-\rho)n_1}/{(u_0\ln^2{n_d})}$
% as required by Corollary~\ref{cor:hdp-1}. % Notice that $\sqrt{1-q+q\epsilon}$ can be made arbitrarily small provided $q$ is sufficiently large (close to one). 
Therefore, if we let $q=1-\rho$ and $\epsilon=\frac{\delta^2}{2}$ in Corollary~\ref{cor:hdp-1}, then
$$\|\proj_{\LL}\proj_{\I(\CS)}\|\le\sqrt{1-q+q\epsilon}=\sqrt{\rho+\frac{\delta^2(1-\rho)}{2}}\le\sqrt{\frac{\delta^2}{2-\delta^2}\mleft(1-\frac{\delta^2}{2}\mright)+\frac{\delta^2}{2}}=\delta$$
holds with high probability.% since $\rho\le\frac{\delta^2}{2-\delta^2}$.
\end{proof}    
% We remark in closing this subsection that for the purpose of fulfilling the conditions in Lemma~\ref{lma:relax}, it suffices to set $\delta$ as, e.g., $0.4$ in Corollary~\ref{cor:hdp-2}.

\subsection{Dual certificate: Low-rank part via the golfing scheme}

In this subsection, we construct the low-rank part of the dual certificate via the golfing scheme~\cite{gross2011recovering,candes2011robust}; let us start with a couple of probability bounds before introducing the scheme. The first one is a direct consequence of~\cite[Theorem~2.1]{zhou2021sparse}: If $d$th order tensors $\CX\in\R^{n\times n\times \dots\times n}$ and $\CY\in\{0,1\}^{n\times n\times \dots\times n}$ with $\CX$ given and $\I(\CY)\sim\operatorname{Bernoulli}(q)$ %($\Prob\{y_\bi=1\in\I(\CS)\}=\rho$ and $\Prob\{y_\bi=1\in\I(\CS)\}=1-\rho$ for every $\bi\in\I^d$ independently)
% \begin{equation}\label{eq:distribution-Y}
%     y_{i_1 i_2\dots i_d}=
%     \begin{dcases}
%         1 & \text{with probability $q$}\\
%         0 & \text{with probability $1-q$}
%     \end{dcases}
%     \quad\text{for any $i_1,i_2,\dots,i_d\in[n]$ independently}
% \end{equation} 
for $q\ge \theta_1 \frac{\ln n}{n}$, then
    \begin{equation}\label{eq:zhou2021sparse}
        \Prob\mleft\{\bigl\|\Exp(\CX\odot \CY)-\CX\odot\CY\bigr\|_\sigma\le \kappa_1\|\CX\|_\infty\sqrt{qn}\ln^{d-2}{n}\mright\}\ge 1-n^{-\kappa_2},
    \end{equation}
where $\odot$ is the Hadamard product. %As an aside, we remark that clearly, $\I(\CY)\sim\operatorname{Bernoulli}(q)$.
\begin{corollary}\label{cor:concent-inf-Fro}
    % [{FIXME: This only works for rectangular cases.}]
    If $\CX\in\R^{n_1\times n_2\times\dots\times n_d}$ and $\I\sim\operatorname{Bernoulli}(q)$ with $q\ge \theta_1 \frac{\ln{n_d}}{n_d}$, then
    \begin{equation*}
        \Prob\mleft\{\bigl\|(\proj_{\I^d}-q^{-1}\proj_{\I})(\CX)\bigr\|_\sigma\le \kappa_1\|\CX\|_\infty\sqrt{q^{-1}n_d}\ln^{d-2}{n_d}\mright\}\ge 1-n_d^{-\kappa_2}.
    \end{equation*}
    % \begin{equation*}
    %     \Prob\mleft\{\mleft\|\mleft(\proj_{\I^d}-q^{-1}\proj_{\I}\mright)(\CZ)\mright\|\ge t\mright\}\le n^{-\Omega(1)}
    % \end{equation*}
\end{corollary}
Although~\eqref{eq:zhou2021sparse} requires that the tensor space to be $n\times n\times \dots\times n$, we can embed any tensor in $\R^{n_1\times n_2\times\dots\times n_d}$ into the space $\R^{n_d\times n_d\times\dots\times n_d}$ by appending zero entries. Obviously the embedding makes no changes to the spectral norm and the $\ell_\infty$-norm. By inspecting the tensor $\CX\odot\CY$ entrywisely, it is easy to see that 
$$
\bigl(\proj_{\I^d}-q^{-1}\proj_{\I(\CY)}\bigr)(\CX)=q^{-1}\bigl(\Exp(\CX\odot\CY)-\CX\odot\CY\bigr) \text{ if } \CY\in\{0,1\}^{n\times n\times \dots\times n} \text{ and } \I(\CY)\sim\operatorname{Bernoulli}(q).
$$
% as long as $\I\sim\operatorname{Bernoulli}(q)$, we have
% $$
% \bigl(\proj_{\I^d}-q^{-1}\proj_{\I}\bigr)(\CX)\stackrel{d}{=}q^{-1}\bigl(\Exp(\CX\odot\CY)-\CX\odot\CY\bigr)\text{ if $\CY$ follows the distribution in~\eqref{eq:distribution-Y}};
% % \mbox{ if } \I(\CY)\sim\operatorname{Bernoulli}(q).
% $$
% here and in what follows, the notation $\stackrel{d}{=}$ stands for equality in distribution. By combining the above pieces, 
Corollary~\ref{cor:concent-inf-Fro} then follows immediately from~\eqref{eq:zhou2021sparse}. 

We are ready to present the main probability bound in this subsection.

\begin{lemma}\label{lma:concent-linf}
    Under Assumption~\ref{assump:assump}, if $\CX\in\LL\setminus\{\CO\}$ and $\I\sim\operatorname{Bernoulli}(q)$, then
    \begin{equation*}
        \Prob\mleft\{\bigl\|\proj_{\LL}(\proj_{\I^d}-q^{-1}\proj_{\I})\proj_{\LL}(\CX)\bigr\|_\infty\ge t\mright\}\le 2 \exp\mleft(\frac{-3 t^2q}{u_0\|\CX\|_\infty\bigl(6\|\CX\|_\infty+2t\bigr)\sum_{k=1}^d\frac{r_k}{n_k}}\mright)\prod_{k=1}^d n_k.
    \end{equation*}
    In particular, by letting $t=\frac{\|\CX\|_\infty}{2}$ and 
    $q\ge \kappa_3\frac{d^2u_0r_0\ln{n_d}}{n_1}\ge \kappa_3 u_0\mleft(\sum_{k=1}^d\frac{r_k}{n_k}\mright)\sum_{k=1}^d\ln{n_k}$,
    % where $\kappa_3 >0$ can be arbitrary, we have
    $$
    \bigl\|\proj_{\LL}\bigl(\proj_{\I^d}-q^{-1}\proj_{\I}\bigr)\proj_{\LL}(\CX)\bigr\|_\infty\le \frac{\|\CX\|_\infty}{2}%,\quad\text{with probability at least $1-2\mleft(\prod_{k=1}^d n_k\mright)^{1-\frac{3\kappa_3}{28}}$},
    $$
    holds with high probability.% as long as $\kappa_3$ is large enough.
\end{lemma}
\begin{proof}
    In order to show the bound, we apply the scalar Bernstein inequality to $\proj_{\LL}(\proj_{\I^d}-q^{-1}\proj_{\I})\proj_{\LL}(\CX)$ entrywisely. Since $\CX\in\LL$ and $\CX=\sum_{\bi\in\I^d}x_{\bi}\bigotimes_{k=1}^d\be_{i_k}$, we have
    \begin{align*}
    \proj_{\LL}(\proj_{\I^d}-q^{-1}\proj_{\I})\proj_{\LL}(\CX)=\sum_{\bi\in\I^d}\proj_{\LL}(\proj_{\bi}-q^{-1}\proj_{\bi}\proj_{\I})(\CX)
    = \sum_{\bi\in\I^d}x_{\bi}\proj_{\LL}(\proj_{\bi}-q^{-1}\proj_{\bi}\proj_{\I})\mleft(\bigotimes_{k=1}^d\be_{i_k}\mright). %\\
    % &=\proj_\LL\mleft(\sum_{\bi\in\I^d}x_{\bi}\bigotimes_{k=1}^d\be_{i_k}\mright)-q^{-1}\proj_\LL\mleft(\sum_{\bi\in\I}x_{\bi}\bigotimes_{k=1}^d\be_{i_k}\mright)\\
    % &=\sum_{\bi\in\I^d}x_{\bi}\proj_\LL\mleft(\bigotimes_{k=1}^d\be_{i_k}\mright)-q^{-1}\sum_{\bi\in\I}x_{\bi}\proj_\LL\mleft(\bigotimes_{k=1}^d\be_{i_k}\mright)\\
    % &=\sum_{\bi\in\I^d} z_i(\proj_{\bi}-q^{-1}\proj
    \end{align*}
    As a result, for any $\bj\in\I^d$,
        \begin{align*}
        \bigl(\proj_{\LL}(\proj_{\I^d}-q^{-1}\proj_{\I})\proj_{\LL}(\CX)\bigr)_{\bj} 
=\sum_{\bi\in\I^d}x_{\bi}\mleft\langle \proj_{\LL}(\proj_{\bi}-q^{-1}\proj_{\bi}\proj_{\I})\mleft(\bigotimes_{k=1}^d\be_{i_k}\mright),\bigotimes_{k=1}^d\be_{j_k}\mright\rangle=:\sum_{\bi\in\I^d}w_{\bi}.
% &=\sum_{\bi\in\I^d}x_{\bi}\mleft\langle\proj_\LL\mleft(\bigotimes_{k=1}^d\be_{i_k}\mright),\bigotimes_{k=1}^d\be_{j_k}\mright\rangle-q^{-1}\sum_{\bi\in\I}x_{\bi}\mleft\langle\proj_\LL\mleft(\bigotimes_{k=1}^d\be_{i_k}\mright),\bigotimes_{k=1}^d\be_{j_k}\mright\rangle
%         \\
%         &\le q^{-1}\|\CX\|_\infty\sum_{\bi\in\I^d}\mleft|\mleft\langle\proj_\LL\mleft(\bigotimes_{k=1}^d\be_{i_k}\mright),\bigotimes_{k=1}^d\be_{j_k}\mright\rangle\mright|
%         \\&=\sum_{\bi\in\I^d}\underbrace{z_{i_1 i_2\dots i_d}\mleft(1-q^{-1}\mathbbm{1}_{(i_1,i_2,\dots,i_d)\in\I}\mright)\mleft\langle\proj_{\LL}\mleft(\bigotimes_{k=1}^d\be_{i_k}\mright),\proj_{\LL}\mleft(\bigotimes_{k=1}^d \be_{j_k}\mright)\mright\rangle}_{:=m_{i_1 i_2\dots i_d}}.
        \end{align*}
    We first notice that $\Exp w_{\bi}=0$ as $\Exp \proj_{\I}=q \proj_{\I^d}$ and 
    \begin{align*}
        |w_{\bi}|&\le |x_{\bi}|\max\bigl\{1,|1-q^{-1}|\bigr\}\mleft|\mleft\langle\proj_\LL\mleft(\bigotimes_{k=1}^d\be_{i_k}\mright),\bigotimes_{k=1}^d\be_{j_k}\mright\rangle\mright|
        \\
        &\le q^{-1}\|\CX\|_\infty \mleft|\mleft\langle\proj_\LL\mleft(\bigotimes_{k=1}^d\be_{i_k}\mright),\proj_\LL\mleft(\bigotimes_{k=1}^d\be_{j_k}\mright)\mright\rangle\mright|\\
        &\le q^{-1}\|\CX\|_\infty 
        \mleft\|\proj_{\LL}\mleft(\bigotimes_{k=1}^d\be_{i_k}\mright)\mright\|_2\mleft\|\proj_{\LL}\mleft(\bigotimes_{k=1}^d \be_{j_k}\mright)\mright\|_2
    \\
    &\le q^{-1}u_0\|\CX\|_\infty\sum_{k=1}^d\frac{r_k}{n_k},
    \end{align*}
where the last inequality follows from Lemma~\ref{lma:incoherence}. % and similarly to (\ref{eq:operator-case-by-case}) that
        On the other hand, as the variance is no more than the second moment, we have 
        \begin{align*}
        \sum_{\bi\in\I^d}\Exp w_{\bi}^2 &\le\sum_{\bi\in\I^d}\Exp\mleft(q^{-1}x_{\bi}\mleft\langle \proj_{\LL}\proj_{\bi}\proj_{\I}\mleft(\bigotimes_{k=1}^d\be_{i_k}\mright),\bigotimes_{k=1}^d\be_{j_k}\mright\rangle\mright)^2
        \\
        &=\sum_{\bi\in\I^d}\mleft(q^{-1}x_{\bi}\mleft\langle \proj_{\LL}\mleft(\bigotimes_{k=1}^d\be_{i_k}\mright),\bigotimes_{k=1}^d\be_{j_k}\mright\rangle\mright)^2\Prob\mleft\{\bi\in\I\mright\}        \\        
&\le q^{-1}\|\CX\|_\infty^2\sum_{\bi\in\I^d}\mleft\langle\bigotimes_{k=1}^d\be_{i_k},\proj_{\LL}\mleft(\bigotimes_{k=1}^d \be_{j_k}\mright)\mright\rangle^2\\
&= q^{-1}\|\CX\|_\infty^2\mleft\|\proj_{\LL}\mleft(\bigotimes_{k=1}^d \be_{j_k}\mright)\mright\|_2^2\\
&\le q^{-1}u_0\|\CX\|_\infty^2\sum_{k=1}^d\frac{r_k}{n_k}.
    \end{align*}

By applying the scalar Bernstein inequality~\cite[Theorem~2.8.4]{vershynin2018high}, we have
    \begin{equation*}
        \Prob\mleft\{\mleft|\bigl(\proj_{\LL}(\proj_{\I^d}-q^{-1}\proj_{\I})\proj_{\LL}(\CX)\bigr)_{\bj}\mright|\ge t\mright\}\le 2\exp\mleft(\frac{-3t^2q}{u_0\|\CX\|_\infty\bigl(6\|\CX\|_\infty+2t\bigr)\sum_{k=1}^d\frac{r_k}{n_k}}\mright).
    \end{equation*}
Because the above inequality applies to any $\bj\in\I^d$, we have
    \begin{equation*}
        \Prob\mleft\{\bigl\|\proj_{\LL}(\proj_{\I^d}-q^{-1}\proj_{\I})\proj_{\LL}(\CX)\bigr\|_\infty\ge t\mright\}\le 2\exp\mleft(\frac{-3t^2q}{u_0\|\CX\|_\infty\bigl(6\|\CX\|_\infty+2t\bigr)\sum_{k=1}^d\frac{r_k}{n_k}}\mright)\prod_{k=1}^d n_k
    \end{equation*}
 by the union bound.
\end{proof}

% \begin{lemma}
%     Let $\CZ\in\R^{n_1\times n_2\times\dots\times n_d}$ be arbitrary. Then we have that
%     \begin{equation*}
%         \Prob\mleft(\mleft\{\mleft\|\mleft(\proj_{\LL}\mleft(\proj_{\I^d}-(1-\rho)^{-1}\proj_{\I^\perp(\CS)}\mright)\proj_{\LL}\mright)(\CZ)\mright\|_\infty\ge t\mright\}\mright)\le 2\mleft(\prod_{k=1}^d n_k\mright) \exp\mleft(\frac{-(1-\rho)t^2/2}{\mleft(\|\CZ\|_\infty+2t/3\mright)\cdot\|\CZ\|_\inftyu_0\sum_{k=1}^d\frac{r_k}{n_k}}\mright).
%     \end{equation*}
% \end{lemma}

Let us now introduce the golfing scheme. We decompose $\I^\perp(\CS)=\bigcup_{j=1}^m\I^\perp(\CS_j)$ with $m\in\N$ to be specified later, where $\I(\CS_j)\sim\operatorname{Bernoulli}(\varphi)$ for $j\in[m]$ are identical and independent of each other. Since $\Prob\{\bi\notin\I^\perp(\CS_j)\}=\varphi$, we have
$\Prob\bigl\{\bi\notin\bigcup_{j=1}^m\I^\perp(\CS_j)\bigr\}=\varphi^m$ and so $\Prob\bigl\{\bi\in\bigcup_{j=1}^m\I^\perp(\CS_j)\bigr\}=1-\varphi^m$. As a result, we must have $1-\varphi^m=1-\rho$, i.e., $\varphi=\sqrt[m]{\rho}$. Given a tensor $\CZ\in\Z(\CL)$, the golfing scheme~\cite{gross2011recovering,candes2011robust} recursively defines
\begin{equation} \label{eq:golf}
    \CZ_0=\CO \text{ and } \CZ_{j}=\CZ_{j-1}-(1-\varphi)^{-1}\proj_{\I^\perp(\CS_{j})}\bigl(\proj_{\LL}(\CZ_{j-1})-\CZ\bigr)\text{ for }j\in[m].
\end{equation}
We expect $\CZ_m$ to be a candidate for the low-rank part of the dual certificate.
% for a sufficiently large $m$. 

Let us take a close look at the $\CZ_j$'s. As $\CZ\in\Z(\CL)\in\LL$, by treating $\proj_{\LL}(\CZ_{j})-\CZ\in\LL$ as a residual that is corrected iteratively in the process, we have
\begin{equation}\label{eq:property-Gk}
\CZ_k=-\sum_{j=1}^{k}(1-\varphi)^{-1}\proj_{\I^\perp(\CS_j)}\bigl(\proj_{\LL}(\CZ_{j-1})-\CZ\bigr)\text{ for } k\in[m].
\end{equation}
By observing that for any $j\in[m]$
    \begin{align*}
        \proj_{\LL}(\CZ_{j})-\CZ&=\proj_{\LL}(\CZ_{j-1})-\CZ-(1-\varphi)^{-1}\proj_{\LL}\proj_{\I^\perp(\CS_j)}\bigl(\proj_{\LL}(\CZ_{j-1})-\CZ\bigr)
        \\&=\bigl(\proj_{\LL}-(1-\varphi)^{-1}\proj_{\LL}\proj_{\I^\perp(\CS_j)}\proj_{\LL}\bigr)\bigl(\proj_{\LL}(\CZ_{j-1})-\CZ\bigr),
\end{align*}
we also have for any $k\in[m]$ that
\begin{equation}\label{eq:urelation}
\proj_{\LL}(\CZ_k)-\CZ = \bigl(\proj_{\LL}-(1-\varphi)^{-1}\proj_{\LL}\proj_{\I^\perp(\CS_k)}\proj_{\LL}\bigr)\cdots\bigl(\proj_{\LL}-(1-\varphi)^{-1}\proj_{\LL}\proj_{\I^\perp(\CS_1)}\proj_{\LL}\bigr)(-\CZ).
\end{equation}
We next provide some estimates of $\CZ_m$.
\begin{lemma}\label{lma:golfing1}
    Under Assumption~\ref{assump:assump}, if $m=\kappa_4\ln{n_d}$ and $\CZ\in\Z(\CL)$, then
    $$
        \bigl\|\proj_{\LL}(\CZ_m)-\CZ\bigr\|_2\le\frac{\lambda}{8},~\|\CZ_m\|_\infty\le 2(1-\varphi)^{-1}\|\CZ\|_\infty,\text{ and }\bigl\|\proj_{\LL^\perp}(\CZ_m)\bigr\|_\sigma\le 2 \kappa_1\|\CZ\|_\infty\sqrt{\frac{n_d}{1-\varphi}}\ln^{d-2}{n_d}
    $$
    hold with high probability, where $\CZ_m$ is defined by~\eqref{eq:golf} for this $\CZ$.
\end{lemma}
\begin{proof} 
Assumption~\ref{assump:assump} implies that $1-\rho\ge \frac{u_0r_0\ln^2{n_d}}{\theta_0n_1}$ and the well-known Bernoulli inequality implies that 
$$
1-\varphi=1-\rho^{\frac{1}{m}}=1-\bigl(1-(1-\rho)\bigr)^{\frac{1}{m}}\ge1-\mleft(1-\frac{1}{m}(1-\rho)\mright)=\frac{1-\rho}{m}.
$$
{\color{black} As a result}, for a sufficiently large $\kappa_3>0$,
$$
1-\varphi\ge \frac{u_0r_0\ln^2{n_d}}{\theta_0n_1m}=\frac{u_0r_0\ln{n_d}}{\theta_0\kappa_4n_1}\ge\max\mleft\{\theta_1\frac{\ln{n_d}}{n_d},\frac{\kappa_3d^2u_0r_0\ln{n_d}}{n_1}\mright\}
$$
as long as $\theta_0>0$ is sufficiently small. Therefore, $\I^\perp(\CS_j)$ satisfies the conditions in Corollary~\ref{cor:concent-inf-Fro} and Lemma~\ref{lma:concent-linf} for any $j\in[m]$. Besides, by Proposition~\ref{prop:hdp-1-1}, 
$\bigl\|\proj_{\LL}\bigl(\proj_{\I^d}-(1-\varphi)^{-1}\proj_{\I^\perp(\CS_j)}\bigr)\proj_{\LL}\bigr\|\le \frac{1}{2}$
holds with probability at least
    \begin{equation*}
        1-\exp\mleft(\frac{-3(1-\varphi)}{28u_0\sum_{k=1}^d\frac{r_k}{n_k}}\mright)\prod_{k=1}^d n_k \ge 1- \exp\mleft(\frac{-3\kappa_3}{28}\sum_{k=1}^d \ln{n_k}\mright)\prod_{k=1}^d n_k = 1-\mleft(\prod_{k=1}^d n_k\mright)^{1-\frac{3\kappa_3}{28}}
        \ge 1-\frac{1}{n_d}
    \end{equation*}
for a sufficiently large $\kappa_3$.

By the well-known bound between the spectral and Frobenius norms of a tensor (see e.g.,~\cite{li2018orthogonal})
$$
{\|\CT\|_\sigma}\ge\|\CT\|_2\mleft(\prod_{k=1}^{d-1}n_k\mright)^{-\frac{1}{2}} \text{ for any $\CT\in\R^{n_1\times n_2\times \dots \times n_d}$},
$$
we have $\|\CZ\|_2\le\sqrt{\prod_{k=1}^{d-1}n_k}$ since $\|\CZ\|_\sigma=1$. Therefore, by~\eqref{eq:urelation},
\begin{equation*}
    \bigl\|\proj_{\LL}(\CZ_m)-\CZ\bigr\|_2 \le\mleft(\prod_{j=1}^m\bigl\|\proj_{\LL}-(1-\varphi)^{-1}\proj_{\LL}\proj_{\I^\perp(\CS_j)}\proj_{\LL}\bigr\|\mright)\mleft\|\CZ\mright\|_2\le\frac{1}{2^m}\sqrt{\prod_{k=1}^{d-1}n_k}\le\frac{1}{8\sqrt{n_d}}=\frac{\lambda}{8}
\end{equation*}
holds with probability at least 
% $$
%     \mleft(1-\mleft(\prod_{k=1}^d n_k\mright)^{1-\frac{3\kappa_3}{28}}\mright)^m=\mleft(1-\mleft(\prod_{k=1}^d n_k\mright)^{1-\frac{3\kappa_3}{28}}\mright)^{\kappa_4\ln n_d}\ge 1-\kappa_4\ln n_d\mleft(\prod_{k=1}^d n_k\mright)^{1-\frac{3\kappa_3}{28}},
% $$
$(1-\frac{1}{n_d})^m$, where the last inequality requires that $2^{m-3}\ge\sqrt{\prod_{k=1}^d n_k}$, guaranteed by $m=\kappa_4\ln{n_d}$. This is a high probability since $(1-\frac{1}{n_d})^m=(1-\frac{1}{n_d})^{\kappa_4\ln n_d}$ tends to $1$ as $n_d$ tends to infinity.

Using a similar argument, we have by Lemma~\ref{lma:concent-linf} and~\eqref{eq:urelation} that
$\bigl\|\proj_{\LL}(\CZ_k)-\CZ\bigr\|_\infty\le\frac{\|\CZ\|_\infty}{2^k}$ holds with high probability for any $k\in[m]$. This, together with~\eqref{eq:property-Gk}, further implies that
    \begin{equation*}
        \|\CZ_m\|_\infty
        \le(1-\varphi)^{-1}\sum_{j=1}^{m}
        \bigl\|\proj_{\I^\perp(\CS_j)}\bigl(\proj_{\LL}(\CZ_{j-1})-\CZ\bigr)\bigr\|_\infty
        \le (1-\varphi)^{-1}\sum_{j=1}^{m}
        \frac{\|\CZ\|_\infty}{2^{j-1}}
\le 2(1-\varphi)^{-1}\|\CZ\|_\infty
    \end{equation*}
holds with high probability. Finally, by~\eqref{eq:property-Gk} again, we have
        \begin{align*}
        \bigl\|\proj_{\LL^\perp}(\CZ_m)\bigr\|_\sigma &=\mleft\|\proj_{\LL^\perp}\mleft(-\sum_{j=1}^{m}(1-\varphi)^{-1}\proj_{\I^\perp(\CS_j)}\bigl(\proj_{\LL}(\CZ_{j-1})-\CZ\bigr)\mright)\mright\|_\sigma\\
        &=\mleft\|\sum_{j=1}^{m}\proj_{\LL^\perp}\mleft(\proj_{\LL}-(1-\varphi)^{-1}\proj_{\I^\perp(\CS_j)}\mright)\bigl(\proj_{\LL}(\CZ_{j-1})-\CZ\bigr)\mright\|_\sigma\\
        &=\mleft\|\sum_{j=1}^{m}\proj_{\LL^\perp}\mleft(\proj_{\I^d}-(1-\varphi)^{-1}\proj_{\I^\perp(\CS_j)}\mright)\bigl(\proj_{\LL}(\CZ_{j-1})-\CZ\bigr)\mright\|_\sigma\\
        &\le\sum_{j=1}^{m}\bigl\|\bigl(\proj_{\I^d}-(1-\varphi)^{-1}\proj_{\I^\perp(\CS_j)}\bigr)\bigl(\proj_{\LL}(\CZ_{j-1})-\CZ\bigr)\bigr\|_\sigma
        \\
        &\le \kappa_1\sqrt{\frac{n_d}{1-\varphi}}\ln^{d-2}{n_d}\sum_{j=1}^{m}\bigl\|\proj_{\LL}(\CZ_{j-1})-\CZ\bigr\|_\infty \\
        &\le \kappa_1\sqrt{\frac{n_d}{1-\varphi}}\ln^{d-2}{n_d}\sum_{j=1}^{m}\frac{\|\CZ\|_\infty}{2^{j-1}} \\
        &\le 2\kappa_1\|\CZ\|_\infty\sqrt{\frac{n_d}{1-\varphi}}\ln^{d-2}{n_d}
        % \\&\lesssim\sum_{j=1}^{m}\mleft(\frac{\|\proj_{\LL}(\CZ_{j-1})-\CZ\|_\infty}{u_0\sum_{k=1}^d\frac{r_k}{n_k}}+\frac{\|\proj_{\LL}(\CZ_{j-1})-\CZ\|_2}{\sqrt{u_0\sum_{k=1}^d\frac{r_k}{n_k}}}\mright).
        \end{align*}
holds with high probability, where the last equality holds because $\proj_{\LL}(\CZ_{j-1})-\CZ\in\LL$ and the second inequality 
% is to due
\textcolor{black}{is due to}
Corollary~\ref{cor:concent-inf-Fro}. The proof is then completed by an overall union bound.
\end{proof}

We are now in a position to conclude this subsection.
\begin{proposition}\label{prop:Dl}
    Under Assumption~\ref{assump:assump}, there exists a $\CD_1\in\R^{n_1\times n_2\times \dots\times n_d}$ such that
    $$
    \min_{\CZ\in\Z(\CL)}\bigl\|\proj_{\LL}(\CD_1)-\CZ\bigr\|_2\le\frac{\lambda}{8},~ \bigl\|\proj_{\LL^\perp}(\CD_1)\bigr\|_\sigma\le \frac{1}{4},~\proj_{\I(\CS)}(\CD_1)=\CO, \text{ and }  \|\CD_1\|_\infty\le\frac{\lambda}{4}
      $$
    % \begin{enumerate*}
    %     \item $\proj_{\I(\CS)}(\CD_1)=\CO$;
    %     \item $\mleft\|\proj_{\LL}(\CD_1)-\CZ\mright\|_2\allowbreak\le\lambda/8$;
    %     \item $\mleft\|\CD_1\mright\|_\infty\le\lambda/4$;
    %     \item $\|\proj_{\LL^\perp}(\CD_1)\|_\sigma\le 1/4$.
    % \end{enumerate*}
    hold with high probability.
\end{proposition}
\begin{proof}
By Assumption~\ref{assump:assump}, there exists a $\CZ\in\Z(\CL)$ such that
\begin{equation*}
  \|\CZ\|_\infty\le \sqrt{\frac{u_0}{n_1 n_d \ln^{\max\{2d-5,0\}}{n_d}}}\cdot\sqrt{\theta_0 \frac{(1-\rho)n_1}{u_0\ln^2{n_d}}}=\sqrt{\frac{\theta_0(1-\rho)}{n_d\ln^{\max\{2d-3,2\}}{n_d}}}.
\end{equation*}
Let us consider $\CD_1=\CZ_m$ defined by~\eqref{eq:golf} for this $\CZ$ with $m=\kappa_4\ln n_d$. It is obvious from~\eqref{eq:property-Gk} that $\proj_{\I(\CS)}(\CD_1)=\CO$
% from the construction of the $\CZ_m$, or 
as $\I^\perp(\CS_j)\subseteq\I^\perp(\CS)$ for any $j\in[m]$.
Besides, $\min_{\CZ\in\Z(\CL)}\bigl\|\proj_{\LL}(\CD_1)-\CZ\bigr\|_2\le\frac{\lambda}{8}$ has been shown in Lemma~\ref{lma:golfing1}.

For the remaining two statements, by Lemma~\ref{lma:golfing1}, we have
\begin{equation*}
  \mleft\|\CD_1\mright\|_\infty\le \frac{2}{1-\varphi}\|\CZ\|_\infty\le \frac{2m}{1-\rho}\sqrt{\frac{\theta_0(1-\rho)}{n_d\ln^{\max\{2d-3,2\}}{n_d}}}\le\mleft(\frac{2\kappa_4\sqrt{\theta_0}}{\sqrt{(1-\rho)n_d \ln^{\max\{2d-5,0\}}{n_d}}}\mright)\lambda\le\frac{\lambda}{4}
\end{equation*}
for a sufficiently small $\theta_0$. Finally by Lemma~\ref{lma:golfing1} again, we have 
$$\bigl\|\proj_{\LL^\perp}(\CD_1)\bigr\|_\sigma
%\le 2 \kappa_1 \|\CZ\|_\infty\sqrt{\frac{n_d}{1-\varphi}}\ln^{d-2}{n_d}
\le 2 \kappa_1 \sqrt{\frac{\theta_0(1-\rho)}{n_d\ln^{\max\{2d-3,2\}}{n_d}}}\cdot\sqrt{\frac{n_dm}{1-\rho}}\ln^{d-2}{n_d} \le\frac{2 \kappa_1\sqrt{\kappa_4 \theta_0}}{\sqrt{\ln^{\max\{5-2d,0\}}{n_d}}}\le \frac{1}{4}
$$
for a sufficiently small $\theta_0$.
\end{proof}

\subsection{Dual certificate: Sparse part via the least squares method}

In this subsection, % following~\cite{candes2011robust}, 
we construct the sparse part of the dual certificate via the least squares method, aiming to achieve the following.
\begin{proposition}\label{prop:Ds}
    Under Assumption~\ref{assump:assump}, there exists a $\CD_2\in\R^{n_1\times n_2\times \dots\times n_d}$ such that
    $$
    \proj_{\LL}(\CD_2)=\CO,~ \bigl\|\proj_{\LL^\perp}(\CD_2)\bigr\|_\sigma<\frac{1}{4},~ \proj_{\I(\CS)}(\CD_2)=\lambda\CE,\text{ and }
      \bigl\|\proj_{\I^\perp(\CS)}(\CD_2)\bigr\|_\infty<\frac{\lambda}{4}
      $$
    hold with high probability.
\end{proposition}

To start with, let us consider the following convex optimization problem
\begin{equation*}
    \min\bigl\{\|\CD\|_2^2:\proj_{\LL}(\CD)=\CO,\,\proj_{\I(\CS)}(\CD)=\lambda\CE\bigr\}.
\end{equation*}
The purpose is to make sure that any feasible solution satisfies the two equality conditions in Proposition~\ref{prop:Ds} and that the norm minimization meets the other two conditions as well. Using the standard method of Lagrange multipliers, it is not difficult to show that the optimal solution of the problem is
\begin{equation*}\label{eq:least-squares}    \CD_2=\lambda\proj_{\LL^\perp}\sum_{k=0}^\infty(\proj_{\I(\CS)}\proj_{\LL}\proj_{\I(\CS)})^k(\CE).
\end{equation*}
The rest of this subsection is devoted to prove the two inequality conditions in Proposition~\ref{prop:Ds} for $\CD_2$, shown in Lemma~\ref{lma:least-square1} and Lemma~\ref{lma:least-square2}, respectively.

To begin, we present a technical lemma that approximates the tensor spectral norm by using the so-called $\epsilon$-net of the unit sphere, i.e., a set of unit vectors such that the Euclidean distance from any unit vector to the set is no more than $\epsilon$.
\begin{lemma}\label{thm:net}
If $\CT\in\R^{n_1\times n_2\times \dots\times n_d}$ and $\E_k$ is an $\epsilon$-net of $\SI^{n_k}$ for $k\in[d]$ and $\epsilon\in(0,\frac{1}{d})$, then
$$
\|\CT\|_\sigma\le\frac{1}{1-d\epsilon}\max_{\bv_k\in\E_k\,\forall\,k\in[d]}\mleft\langle\CT,\bv_1\otimes \bv_2\otimes\dots\otimes\bv_d\mright\rangle.
$$
In particular, by letting $\epsilon=\frac{1}{2d}$, there exist $\E_k\subseteq\SI^{n_k}$ with $|\E_k|\le(1+4d)^{n_k}$ for $k\in[d]$ such that 
\begin{equation}\label{eq:net}
\|\CT\|_\sigma\le2\max_{\bv_k\in\E_k\,\forall\,k\in[d]}\langle\CT,\bv_1\otimes \bv_2\otimes\dots\otimes\bv_d\rangle.
\end{equation}
\end{lemma}
\begin{proof}
Let $\|\CT\|_\sigma=\mleft\langle\CT,\bx_1\otimes \bx_2\otimes\dots\otimes\bx_d\mright\rangle$ where $\bx_k\in\SI^{n_k}$ for $k\in[d]$. By choosing $\bv_k\in\E_k$ with $\|\bx_k-\bv_k\|_2\le \epsilon$ for $k\in[d]$, we have
        \begin{align*}
        &\,\|\CT\|_\sigma-\mleft\langle\CT,\bv_1\otimes \bv_2\otimes\dots\otimes\bv_d\mright\rangle\\
        \le&\,\mleft|\mleft\langle\CT,\bigotimes_{k=1}^d \bx_k\mright\rangle-\mleft\langle\CT,\bigotimes_{k=1}^d \bv_k\mright\rangle\mright|
        \\=&\,\mleft|\sum_{k=1}^d\mleft(\mleft\langle\CT,\mleft(\bigotimes_{j=1}^{k-1}\bv_j\mright)\otimes\mleft(\bigotimes_{j=k}^d\bx_j\mright)\mright\rangle-\mleft\langle\CT,\mleft(\bigotimes_{j=1}^{k}\bv_j\mright)\otimes\mleft(\bigotimes_{j=k+1}^d\bx_j\mright)\mright\rangle\mright)\mright|
        \\
        \le&\,\sum_{k=1}^d\|\CT\|_\sigma\mleft(\prod_{j=1}^{k-1}\|\bv_j\|_2\mright)\|\bx_k-\bv_k\|_2\mleft(\prod_{j=k+1}^d\|\bx_j\|_2\mright)\\
        \le&\, \epsilon d\|\CT\|_\sigma,
    \end{align*}
    which further implies that
    \begin{equation*}%\label{eq:temp-concent-spectral-norm}
        \|\CT\|_\sigma\le\frac{1}{1-d\epsilon}\mleft\langle\CT,\bv_1\otimes \bv_2\otimes\dots\otimes\bv_d\mright\rangle.
    \end{equation*}
    The result then follows by taking maximum over all $\E_k$'s.

    In order to show~\eqref{eq:net}, the existence of relevant $\epsilon$-nets is required. In particular, it is stated in~\cite[Corollary~4.2.13]{vershynin2018high} that for any $\epsilon>0$, there exists an $\epsilon$-net of $\SI^n$ whose cardinality is no more than $(1+\frac{2}{\epsilon})^{n}$. 
\end{proof}
% In particular, for any $\epsilon>0$, there exists an $\epsilon$-net of $\SI^n$ with cardinality no more than $(1+\frac{2}{\epsilon})^{n}$; see, e.g.,~\cite[Corollary~4.2.13]{vershynin2018high}. If we apply $\epsilon=\frac{1}{2d}$ in Lemma~\ref{thm:net}, then there exists $\frac{1}{2d}$-nets $\E_k$ of $\SI^{n_k}$ with $|\E_k|\le(1+4d)^{n_k}$ for $k\in[d]$, such that 
% ; see e.g.,~\cite{vershynin2018high,hu2022complexity,he2023approx}
For relevant studies on the tensor norms by sphere covering, we refer interested readers to~\cite{hu2022complexity,he2023approx,guan2024lp}. 

We next bound the spectral norm of $\CE$ by applying the result of $\epsilon$-net.
\begin{lemma}\label{lma:espec}
    Under Assumption~\ref{assump:assump}, $\|\CE\|_\sigma\le\frac{\kappa_5}{\sqrt{-\ln\rho}}\sqrt{\sum_{k=1}^dn_k}$ holds with high probability.
\end{lemma}
\begin{proof}
By the condition of $\CS$ in Assumption~\ref{assump:assump} and $\CE=\sign(\CS)$, we have % $|e_{\bi}|=0$ or 1, and so 
$\bigl(\Exp |e_{\bi}|^x\bigr)^\frac{1}{x}=\rho^{\frac{1}{x}}$ for any $\bi\in\I^d$ and $x\ge 1$. The function $\rho^{\frac{1}{x}}/\sqrt{x}$ achieves the maximum over $[1,\infty)$ at $x=-2\ln\rho$ as long as $\rho\le \frac{1}{\sqrt{e}}$, which is certainly the case as $\rho$ is assumed to be sufficiently small. As a result, 
    $$
        \bigl(\Exp |e_{\bi}|^x\bigr)^\frac{1}{x}\le \sqrt{\frac{x}{-2e\ln{\rho}}} \text{ for any } x\ge 1.
    $$
    % $$
    %     \le g(\rho) \sqrt{p},\quad\text{where $g(\rho):=
    %     \begin{dcases}
    %     \rho & \text{if $\rho\ge\frac{1}{\sqrt{e}}$}\\
    %     \frac{1}{\sqrt{e\ln{\rho^{-2}}}} & \text{otherwise},
    %     \end{dcases}$};
    % $$
    This, together with~\cite[Proposition~2.5.2]{vershynin2018high}, further implies that
    $$
        \Exp \exp(t e_{\bi})\le\exp\mleft(\frac{\kappa_6}{-2e\ln{\rho}} t^2\mright)\text{ for any $\bi\in\I^d$ and $t\in\R$}.
    $$
    The result then follows directly by combining the above with~\cite[Lemma~1]{tomioka2014spectral} and~\cite[Theorem~1]{tomioka2014spectral}.
\end{proof}

We are now ready to make two key claims to conclude this subsection.
\begin{lemma}\label{lma:least-square1}
Under Assumption~\ref{assump:assump}, $\|\proj_{\LL^\perp}(\CD_2)\|_\sigma<\frac{1}{4}$ holds with high probability.
% where $\CD_2$ is defined in~\eqref{eq:least-squares}.
\end{lemma}
\begin{proof}
    We rewrite 
    \begin{equation*}
        \proj_{\LL^\perp}(\CD_2)=\lambda\proj_{\LL^\perp}(\CE)+\lambda\proj_{\LL^\perp}\sum_{k=1}^\infty(\proj_{\I(\CS)}\proj_{\LL}\proj_{\I(\CS)})^k(\CE)
    \end{equation*}
    and bound their spectral norms separately. By Lemma~\ref{thm:spec-subspace} and Lemma~\ref{lma:espec}, %(and recall $\LL^\perp=\TT^{[d]}(\CL)$), as long as $\rho$ is small enough, we have with high probability that
    $$\bigl\|\lambda\proj_{\LL^\perp}(\CE)\bigr\|_\sigma\le\lambda\|\CE\|_\sigma\le 
    \frac{1}{\sqrt{n_d}}\cdot \frac{\kappa_5}{\sqrt{-\ln\rho}}\sqrt{\sum_{k=1}^dn_k}\le \frac{\kappa_5\sqrt{d}}{\sqrt{-\ln\rho}}\le\frac{1}{8}
    $$
    holds with high probability for a sufficiently small $\rho$. Besides, by Corollary~\ref{cor:hdp-2}, 
    % and $\theta_0\le\theta$ with $\rho\le{\delta^2}/{(2-\delta^2)}$
    $\|\proj_{\LL}\proj_{\I(\CS)}\|\le\delta$ holds with high probability for any $\delta\in(0,1]$. Therefore, it suffices to show that 
    $$
    \Prob\mleft\{\mleft\|\lambda\proj_{\LL^\perp}\sum_{k=1}^\infty(\proj_{\I(\CS)}\proj_{\LL}\proj_{\I(\CS)})^k(\CE)\mright\|_\sigma<\frac{1}{8}\middle|\|\proj_{\LL}\proj_{\I(\CS)}\|\le\delta\mright\}
    $$
    is a high probability.

    In fact, by Lemma~\ref{thm:spec-subspace} and~\eqref{eq:net} in Lemma~\ref{thm:net},
    \begin{align*}
    \mleft\|\lambda\proj_{\LL^\perp} \sum_{k=1}^\infty(\proj_{\I(\CS)}\proj_{\LL}\proj_{\I(\CS)})^k(\CE)\mright\|_\sigma&\le \lambda\mleft\|\sum_{k=1}^\infty(\proj_{\I(\CS)}\proj_{\LL}\proj_{\I(\CS)})^k(\CE)\mright\|_\sigma\\
    &\le 2\lambda\max_{\bv_k\in\E_k\,\forall\, k\in[d]}\mleft\langle\CE,\sum_{k=1}^\infty(\proj_{\I(\CS)}\proj_{\LL}\proj_{\I(\CS)})^k\mleft(\bigotimes_{k=1}^d \bv_k\mright)\mright\rangle,
    \end{align*}
    where $\E_k$'s are the $\epsilon$-nets used in~\eqref{eq:net}. For any given $\I(\CS)$, the nonzero entries of $\CE$ are i.i.d.\ symmetric Bernoulli random variables (taking $\pm1$ with equal probability), according to Assumption~\ref{assump:assump}. Hence, by Hoeffding's inequality~\cite[Theorem~2.2.2]{vershynin2018high}, we have
    % \begin{equation*}
        \begin{align*}
        &\Prob\mleft\{\mleft\|\lambda\proj_{\LL^\perp}\sum_{k=1}^\infty(\proj_{\I(\CS)}\proj_{\LL}\proj_{\I(\CS)})^k(\CE)\mright\|_\sigma\ge t\middle|\I(\CS)\mright\}\\
        \le{}&\Prob\mleft\{\max_{\bv_k\in\E_k\,\forall\, k\in[d]}\mleft\langle\CE,\sum_{k=1}^\infty(\proj_{\I(\CS)}\proj_{\LL}\proj_{\I(\CS)})^k\mleft(\bigotimes_{k=1}^d \bv_k\mright)\mright\rangle\ge \frac{t}{2\lambda}\middle|\I(\CS)\mright\}\\
        \le{}&\sum_{\bv_k\in\E_k\,\forall\, k\in[d]}\Prob\mleft\{\mleft\langle\CE,\sum_{k=1}^\infty(\proj_{\I(\CS)}\proj_{\LL}\proj_{\I(\CS)})^k\mleft(\bigotimes_{k=1}^d \bv_k\mright)\mright\rangle\ge \frac{t}{2\lambda}\middle|\I(\CS)\mright\}\\
        \le{} &\sum_{\bv_k\in\E_k\,\forall\,k\in[d]}\exp\mleft(\frac{-t^2}{8\lambda^2\bigl\|\sum_{k=1}^\infty(\proj_{\I(\CS)}\proj_{\LL}\proj_{\I(\CS)})^k(\bigotimes_{k=1}^d \bv_k)\bigr\|_2^2}\mright) \\
        \le{} &\exp\mleft(\frac{-t^2}{8\lambda^2\bigl\|\sum_{k=1}^\infty(\proj_{\I(\CS)}\proj_{\LL}\proj_{\I(\CS)})^k\bigr\|^2}\mright)\prod_{k=1}^d\mleft(1+4d\mright)^{n_k},
    \end{align*}
    where the last inequality is due to $\bigl\|\bigotimes_{k=1}^d \bv_k\bigr\|_2=1$ and $|\E_k|\le(1+4d)^{n_k}$ for $k\in[d]$. 

    % \begin{align*}
    %         &\,\Prob\mleft\{\mleft\|\lambda\proj_{\LL^\perp}\sum_{k=1}^\infty(\proj_{\I(\CS)}\proj_{\LL}\proj_{\I(\CS)})^k(\CE)\mright\|_\sigma\ge t\mright\}
    %         \\
    %         =&\,\Prob\mleft\{\mleft\|\lambda\proj_{\LL^\perp}\sum_{k=1}^\infty(\proj_{\I(\CS)}\proj_{\LL}\proj_{\I(\CS)})^k(\CE)\mright\|_\sigma\ge t\middle|\|\proj_{\I(\CS)}\proj_{\LL}\proj_{\I(\CS)}\|\le\delta^2\mright\}\Prob\mleft\{\|\proj_{\I(\CS)}\proj_{\LL}\proj_{\I(\CS)}\|\le\delta^2\mright\}
    %         \\
    %         &+\Prob\mleft\{\mleft\|\lambda\proj_{\LL^\perp}\sum_{k=1}^\infty(\proj_{\I(\CS)}\proj_{\LL}\proj_{\I(\CS)})^k(\CE)\mright\|_\sigma\ge t\middle|\|\proj_{\I(\CS)}\proj_{\LL}\proj_{\I(\CS)}\|>\delta^2\mright\}\Prob\mleft\{\|\proj_{\I(\CS)}\proj_{\LL}\proj_{\I(\CS)}\|>\delta^2\mright\}
    %         \\            
    %         \le&\,\Prob\mleft\{\mleft\|\lambda\proj_{\LL^\perp}\sum_{k=1}^\infty(\proj_{\I(\CS)}\proj_{\LL}\proj_{\I(\CS)})^k(\CE)\mright\|_\sigma\ge t\middle|\|\proj_{\I(\CS)}\proj_{\LL}\proj_{\I(\CS)}\|\le\delta^2\mright\}
    %         + \Prob\mleft\{\|\proj_{\I(\CS)}\proj_{\LL}\proj_{\I(\CS)}\|>\delta^2\mright\}
    %         \\
    %         \le&\,\exp\mleft(-\frac{(1-\delta^2)^2 t^2}{8\lambda^2\delta^4}\mright)\prod_{k=1}^d\mleft(1+4d\mright)^{n_k}+\Prob\mleft\{\mleft\|\proj_{\LL}\mleft(\proj_{\I^d}-(1-\rho)^{-1}\proj_{\I^\perp(\CS)}\mright)\proj_{\LL}\mright\|>\frac{\delta^2}{2}\mright\},
    %     \end{align*}
    
    As $\|\proj_{\LL}\proj_{\I(\CS)}\|\le\delta$ implies that $\bigl\|\sum_{k=1}^\infty(\proj_{\I(\CS)}\proj_{\LL}\proj_{\I(\CS)})^k\bigr\|\le\sum_{k=1}^\infty\delta^{k}=\frac{\delta}{1-\delta}$, by letting $t=\frac{1}{8}$ in the above bound and recalling $\lambda=\frac{1}{\sqrt{n_d}}$, we have %by the law of total probability that
        \begin{align*}
     \Prob\mleft\{\mleft\|\lambda\proj_{\LL^\perp}\sum_{k=1}^\infty(\proj_{\I(\CS)}\proj_{\LL}\proj_{\I(\CS)})^k(\CE)\mright\|_\sigma\ge \frac{1}{8} \middle|\|\proj_{\LL}\proj_{\I(\CS)}\|\le\delta \mright\}
            &\le \exp\mleft(\frac{-(1-\delta)^2 n_d}{8^3\delta^2}\mright)\prod_{k=1}^d\mleft(1+4d\mright)^{n_k} \\
            &\le  \exp\mleft( n_d\mleft(-\frac{(1-\delta)^2}{8^3\delta^2}+d\ln(1+4d)\mright)\mright),
        \end{align*}
    which is a low probability by choosing a sufficiently small $\delta$ such that 
    $-\frac{(1-\delta)^2}{8^3\delta^2}+d\ln(1+4d)<0$.
\end{proof}

\begin{lemma}\label{lma:least-square2}
Under Assumption~\ref{assump:assump}, $\bigl\|\proj_{\I^\perp(\CS)}(\CD_2)\bigr\|_\infty<\frac{\lambda}{4}$ holds with high probability.
% , where $\CD_2$ is defined in~\eqref{eq:least-squares}.
\end{lemma}
\begin{proof}
% By noticing that $\CS\in\I(\CS)$, 
To begin with, we observe that
    \begin{align*}
        \bigl\|\proj_{\I^\perp(\CS)}(\CD_2)\bigr\|_\infty&=\mleft\|\lambda\proj_{\I^\perp(\CS)}(\proj_{\I^d}-\proj_{\LL})\sum_{k=0}^\infty(\proj_{\I(\CS)}\proj_{\LL}\proj_{\I(\CS)})^k(\CE)\mright\|_\infty\\
        &=\mleft\|-\lambda\proj_{\I^\perp(\CS)}\proj_{\LL}\proj_{\I(\CS)}\sum_{k=0}^\infty(\proj_{\I(\CS)}\proj_{\LL}\proj_{\I(\CS)})^k(\CE)\mright\|_\infty     \\
        &=
\lambda\max_{\bi\in\I^d}\,\mleft|\mleft\langle\CE,\mleft(\sum_{k=0}^\infty(\proj_{\I(\CS)}\proj_{\LL}\proj_{\I(\CS)})^k\mright)\proj_{\I(\CS)}\proj_{\LL}\proj_{\I^\perp(\CS)}\mleft(\bigotimes_{k=1}^d\be_{i_k}\mright)\mright\rangle\mright|\\
&=
\lambda\max_{\bi\in\I^\perp(\CS)}\,\mleft|\mleft\langle\CE,\mleft(\sum_{k=0}^\infty(\proj_{\I(\CS)}\proj_{\LL}\proj_{\I(\CS)})^k\mright)\proj_{\I(\CS)}\proj_{\LL}\mleft(\bigotimes_{k=1}^d\be_{i_k}\mright)\mright\rangle\mright|.
    \end{align*}

By Corollary~\ref{cor:hdp-2}, % and $\rho\le\frac{\delta^2}{2-\delta^2}$, 
$\|\proj_{\LL}\proj_{\I(\CS)}\|\le\delta$ holds with high probability for any $\delta\in(0,1]$, which further implies that
% $\|\proj_{\LL}\proj_{\I(\CS)}\|\le\delta$ which implies that 
$\bigl\|\sum_{k=0}^\infty(\proj_{\I(\CS)}\proj_{\LL}\proj_{\I(\CS)})^k\bigr\|\le\sum_{k=0}^\infty\delta^{k}=\frac{1}{1-\delta}$. As a result, we have for any $\bi\in\I^\perp(\CS)$ that
    \begin{equation*}
\mleft\|\mleft(\sum_{k=0}^\infty(\proj_{\I(\CS)}\proj_{\LL}\proj_{\I(\CS)})^k\mright)\proj_{\I(\CS)}\proj_{\LL}\mleft(\bigotimes_{k=1}^d\be_{i_k}\mright)\mright\|_2^2\le\frac{\|\proj_{\I(\CS)}\proj_{\LL}\|^2}{(1-\delta)^2}\mleft\|\proj_{\LL}\mleft(\bigotimes_{k=1}^d\be_{i_k}\mright)\mright\|_2^2\le\frac{\delta^2u_0\sum_{k=1}^d\frac{r_k}{n_k}}{(1-\delta)^2},
    \end{equation*}
where the last inequality is due to Lemma~\ref{lma:incoherence}.
    
For any given $\I(\CS)$, the nonzero entries of $\CE$ are i.i.d.\ symmetric Bernoulli random variables. By the two-sided version of Hoeffding’s inequality~\cite[Theorem~2.2.5]{vershynin2018high},
% Using the same derivation as that in the proof of Lemma~\ref{lma:least-square1}, we have
        \begin{align*}
            &\Prob\mleft\{\mleft\|\proj_{\I^\perp(\CS)}(\CD_2)\mright\|_\infty\ge \frac{\lambda}{4}\middle|\I(\CS)\mright\} \\
            ={}&\Prob\mleft\{\max_{\bi\in\I^\perp(\CS)}\mleft|\mleft\langle\CE,\mleft(\sum_{k=0}^\infty(\proj_{\I(\CS)}\proj_{\LL}\proj_{\I(\CS)})^k\mright)\proj_{\I(\CS)}\proj_{\LL}\mleft(\bigotimes_{k=1}^d\be_{i_k}\mright)\mright\rangle\mright| \ge \frac{1}{4}\middle|\I(\CS)\mright\}       \\
            \le{}&\sum_{\bi\in\I^\perp(\CS)}\Prob\mleft\{\mleft|\mleft\langle\CE,\mleft(\sum_{k=0}^\infty(\proj_{\I(\CS)}\proj_{\LL}\proj_{\I(\CS)})^k\mright)\proj_{\I(\CS)}\proj_{\LL}\mleft(\bigotimes_{k=1}^d\be_{i_k}\mright)\mright\rangle\mright| \ge \frac{1}{4}\middle|\I(\CS)\mright\}       \\
        \le {}&\,2\sum_{\bi\in\I^\perp(\CS)}\exp\mleft(-\mleft(32\,\mleft\|\mleft(\sum_{k=0}^\infty(\proj_{\I(\CS)}\proj_{\LL}\proj_{\I(\CS)})^k\mright)\proj_{\I(\CS)}\proj_{\LL}\mleft(\bigotimes_{k=1}^d\be_{i_k}\mright)\mright\|_2^2\mright)^{-1}\mright)\\
        \le {}&\,2\exp\mleft(-\frac{(1-\delta)^2}{32\delta^2u_0\sum_{k=1}^d\frac{r_k}{n_k}}\mright)\prod_{k=1}^dn_k \\
        \le {}&\,2\exp\mleft(-\frac{(1-\delta)^2\ln^2n_d}{32\delta^2\theta_0d(1-\rho)}+d\ln n_d\mright), %\\
    %    \le &\,\frac{2}{n_d}
    \end{align*}
    where the last inequality follows from Assumption~\ref{assump:assump}. This is clearly a low probability by choosing a sufficiently small $\delta$.
    % this is clearly a low probability, as desired. The proof is then completed by the law of total probability.
    % which can be sufficiently small by choosing a sufficiently small $\delta$.
    % where $\max_{k\in[d]}r_k\le \theta_0 \frac{(1-\rho)n_1}{u_0\ln^2{n_d}}$ in Assumption~\ref{assump:assump} has been applied in the last step. 
        % \begin{align*}
        %     &~~~\,\Prob\mleft\{\mleft\|\proj_{\I^\perp(\CS)}(\CD_2)\mright\|_\infty\ge \frac{\lambda}{4}\mright\} \\
        %     &=\sum_{\mathbb{A}\in\{\mathbb{F},\mathbb{F}^\complement\}}\Prob\mleft\{\mleft\|\proj_{\I^\perp(\CS)}(\CD_2)\mright\|_\infty\ge t\middle|\OO\in\mathbb{A}\mright\}\Prob\OO\in\mathbb{A}
        %     \\&\le\Prob\mleft\{\mleft\|\proj_{\I^\perp(\CS)}(\CD_2)\mright\|_\infty\ge t\middle|\OO\in\mathbb{F}\mright\}+\Prob\OO\in\mathbb{F}^\complement
        %     \\&\le 2\mleft(\prod_{k=1}^d n_k\mright)\exp\mleft(\frac{-(1-\delta^2)^2 t^2 \ln^2{n_d}}{2 \kappa_{11}\lambda^2\delta^2(1-\rho)}\mright)+\Prob\mleft\{\mleft\|\proj_{\LL}\mleft(\proj_{\I^d}-(1-\rho)^{-1}\proj_{\I^\perp(\CS)}\mright)\proj_{\LL}\mright\|\ge \frac{\delta^2}{2}\mright\}.
        % \end{align*}
\end{proof}

\subsection{Putting everything together}\label{sec:putting-everything-together}

We arrive at the following result by combining the results of previous two subsections with an overall union bound. 

% This, together with Corollary~\ref{cor:hdp-2} and Lemma~\ref{lma:relax}, directly shows Theorem~\ref{thm:main-thm}.

\begin{proposition}\label{prop:existence-dual-certificate}
Under Assumption~\ref{assump:assump}, there exists a $\CD\in\R^{n_1\times n_2\times \dots\times n_d}$ such that
$$
\min_{\CZ\in\Z(\CL)}\bigl\|\proj_{\LL}(\CD)-\CZ\bigr\|_2\le \frac{\lambda}{8},~\bigl\|\proj_{\LL^\perp}(\CD)\bigr\|_\sigma< \frac{1}{2},~\bigl\|\proj_{\I(\CS)}(\CD)-\lambda\CE\bigr\|_2\le\frac{\lambda}{8},\text{ and }\bigl\|\proj_{\I^\perp(\CS)}(\CD)\bigr\|_\infty<\frac{\lambda}{2}
$$
hold with high probability.
\end{proposition}
\begin{proof}
By Proposition~\ref{prop:Dl} and Proposition~\ref{prop:Ds},
\begin{align*}
       \min_{\CZ\in\Z(\CL)}\bigl\|\proj_{\LL}(\CD_1+\CD_2)-\CZ\bigr\|_2&=\min_{\CZ\in\Z(\CL)}\bigl\|\proj_{\LL}(\CD_1)-\CZ\bigr\|_2\le\frac{\lambda}{8},\\
        \bigl\|\proj_{\LL^\perp}(\CD_1+\CD_2)\bigr\|_\sigma&\le\bigl\|\proj_{\LL^\perp}(\CD_1)\bigr\|_\sigma+\bigl\|\proj_{\LL^\perp}(\CD_2)\bigr\|_\sigma < \frac{1}{2},\\
        \bigl\|\proj_{\I(\CS)}(\CD_1+\CD_2)-\lambda\CE\bigr\|_2&=\bigl\|\proj_{\I(\CS)}(\CD_2)-\lambda\CE\bigr\|_2=0\le\frac{\lambda}{8},\\
        \bigl\|\proj_{\I^\perp(\CS)}(\CD_1+\CD_2)\bigr\|_\infty&\le\|\CD_1\|_\infty+\bigl\|\proj_{\I^\perp(\CS)}(\CD_2)\bigr\|_\infty <\frac{\lambda}{2}
\end{align*}
% \begin{itemize}
%     \item $\displaystyle\min_{\CZ\in\Z(\CL)}\mleft\|\proj_{\LL}(\CD_1+\CD_2)-\CZ\mright\|_2=\min_{\CZ\in\Z(\CL)}\mleft\|\proj_{\LL}(\CD_1)-\CZ\mright\|_2\le\frac{\lambda}{8}$;
%     \item $\displaystyle\|\proj_{\LL^\perp}(\CD_1+\CD_2)\|_\sigma\le\allowbreak\|\proj_{\LL^\perp}(\CD_1)\|_\sigma+\|\proj_{\LL^\perp}(\CD_2)\|_\sigma < \frac{1}{2}$;
%     \item $\displaystyle\|\proj_{\I(\CS)}(\CD_1+\CD_2)-\lambda\CE\|_2=\|\proj_{\I(\CS)}(\CD_2)-\lambda\CE\|_2=0\le\frac{\lambda}{8}$;
%     \item $\displaystyle\|\proj_{\I^\perp(\CS)}(\CD_1+\CD_2)\|_\infty\le\|\CD_1\|_\infty+\|\proj_{\I^\perp(\CS)}(\CD_2)\|_\infty <\frac{\lambda}{2}$;
% \end{itemize}
hold with high probability. The proof is completed by letting $\CD=\CD_1+\CD_2$.
\end{proof}

{\color{black}
We are now ready to prove Theorem~\ref{thm:main-thm}.
\begin{myproof}{Theorem~\ref{thm:main-thm}}
    By letting $\delta$ to be, e.g., $\frac{1}{3}$ in Corollary~\ref{cor:hdp-2}, we have $\|\proj_{\LL}\proj_{\I(\CS)}\|<\frac{1}{2}$ with high probability. Besides, $\lambda=\frac{1}{\sqrt{n_d}}\le \frac{1}{\sqrt{2}}<1$ since $n_d\ge 2$ is assumed. Moreover, Proposition~\ref{prop:existence-dual-certificate} implies that there exists a $\CD\in\R^{n_1\times n_2\times \dots\times n_d}$ satisfying (\ref{eq:existence-dual-certificate}) with high probability.
    By combining the above facts with Lemma~\ref{lma:relax}, it follows that $(\CL,\CS)$ is the unique optimal solution of~\eqref{eq:PCP} with high probability, i.e.,~\eqref{eq:PCP} exactly recovers $\CL$ and $\CS$ with high probability.
\end{myproof}
}

%We remark that we may also perform the above analyses starting from the inclusion (\ref{eq:subdiff-incoherent}), but we also highlight that by doing so the resulting proof process will no longer possess consistency with its matrix counterpart, which is unfavorable and destroys the nature and beauty. As an exemplary consequence, this will lead to poor estimates for various constants, such as in Lemma~\ref{lma:uniqueness1}, we need to demand $\|\proj_{\LL^\perp}(\CD)\|_\sigma< 2/d(d-1)$ instead of $\|\proj_{\LL^\perp}(\CD)\|_\sigma< 1$, which is disadvantageous.

\section{Concluding remarks}\label{sec:conclusion}

In this paper, we systematically study the decomposability and the subdifferential of the tensor nuclear norm. We show that the tensor nuclear norm is decomposable over a pair of subspaces that have at least two disjoint modes, naturally generalizing the result for the matrix case. The same property holds for the tensor spectral norm as well. Based on the decomposability, we propose novel inclusions of the subdifferential of the tensor nuclear norm. They strictly enlarge the inclusion proposed in~\cite[Theorem~1]{yuan2017incoherent}, the only known subdifferential inclusion for tensors of an arbitrary order. While various bounds for the subdifferential of the tensor nuclear norm and several interesting examples suggest that there is no general way to explicitly characterize the subdifferential of the tensor nuclear norm in general as in the matrix case, we examine subgradients of the tensor nuclear norm in all relevant subspaces and estimate their inner and outer approximations. % lower and upper bounds of the subdifferential.
% that being said, we examine subgradients of the tensor nuclear norm in all subspaces and derive their lower bounds and upper bounds of the subdifferential.

We believe that these new results on the tensor nuclear norm have great potential in applications. For instance, the new inclusions of the subdifferential of the tensor nuclear norm can potentially be applied to analyze the statistical performance of a variety of nuclear-norm-based tensor learning problems. As a precursor, we study one immediate application, the tensor robust PCA. As shown in Theorem~\ref{thm:main-thm}, the exact recovery applies to tensors of an arbitrary order for the first time in the literature, generalizing the result in the matrix case when $d=2$. %, but also enjoys looser conditions for the exact recovery than those assumed in the literature when $d=3$. 
In light of our analysis, we propose a natural conjecture concerning the conditions for exact recovery of the tensor robust PCA.
\begin{conjecture}\label{conj:remove}
    The factor $\ln^{\max\{2d-5,0\}}{n_d}$ in Assumption~\ref{assump:assump} can be removed in Theorem~\ref{thm:main-thm}.
\end{conjecture}
In fact, the conjecture holds true if the factor $\ln^{d-2}{n_d}$ in Corollary~\ref{cor:concent-inf-Fro} can be removed. This has already been conjectured similarly in~\cite[Section~11]{zhou2021sparse}. If Conjecture~\ref{conj:remove} is true, then the conditions for exact recovery of the tensor robust PCA of every order, including the matrix case, would become identical. 

While our study indicates that there is no general way to explicitly characterize the
subdifferential of the tensor nuclear norm, it is definitely interesting to look into specific classes of tensors, in particular those from practical applications, such that a complete characterization is possible. Moreover, despite the remarkable statistical performance of the tensor robust PCA, the underlying optimization problem~\eqref{eq:PCP} is computationally intractable due to the NP-hardness to compute the tensor nuclear norm. Developing tractable tight approximations of the tensor robust PCA is important. For example, one may resort to the sum-of-squares relaxation; see, e.g.,~\cite{barak2016noisy,barak2022noisy} for its application to tensor completion and a related work~\cite{xia2019polynomial}.

Apart from practical applications, our developments may facilitate the study of a variety of other more in-depth properties and implications of the tensor nuclear norm.   
% Recalling that the study in Section~\ref{sec:subdiff} merely pertains to first-order variational properties of the tensor nuclear norm, 
For example, they may help to deduce the $C^2$-cone reducibility~\cite[Definition~3.135]{bonnans2013perturbation}, subdifferential metric subregularity~\cite[Section~3.H]{dontchev2009implicit}, and twice epi-differentiability~\cite[Definition~13.6(b)]{rockafellar2009variational} of the tensor nuclear norm;
% and parabolic regularity~\cite[Definition~13.65]{rockafellar2009variational}
see, e.g.,~\cite[Proposition~3.2]{cui2017quadratic},~\cite[Proposition~11]{zhou2017unified} and~\cite[Proposition~3.8]{cui2017quadratic}, and~\cite[Corollary~3.6]{he2024twice} and~\cite{mohammadi2025parabolic}, respectively for their matrix counterparts. 
Besides, the developments can also be useful to understand the neural collapse~\cite{papyan2020prevalence} in training tensor-based neural networks and the relations between various nuclear-norm-regularized tensor optimization problems and their Burer-Monteiro reparameterizations~\cite{burer2003nonlinear,burer2005local}; see, e.g.,~\cite[Appendix~C]{zhu2021geometric} and~\cite{li2019non,boumal2024,mcrae2024low,ouyang2025burer}, respectively for their matrix counterparts.

\section*{Acknowledgments}

% \textcolor{blue}{
This work was partially supported by the National Natural Science Foundation of China [Grants 72394364,
72394360, and 72171141].
% The first author would like to express his sincere gratitude to Professor Simai He for his generous subsidy support when the first author was studying in Shanghai. 
% Part of the work of the first author was done when he was a research assistant at School of Information Management and Engineering, Shanghai University of Finance and Economics.
The authors would like to thank the two anonymous reviewers for their insightful comments that have helped to improve this paper from its original version. In particular, the authors are indebted to the reviewer who has brought the references~\cite{sturmfels2016tensors,robeva2017singular,hashemi2018spectral,vannieuwenhoven2012new,tropp2015introduction,li2003norm} to their attention and suggested the last part of Lemma~\ref{prop:nuclear-U-equiv}.
The first author would like to express his sincere gratitude to Professor Simai He for his generous subsidy support when he was studying at School of Information Management and Engineering, Shanghai University of Finance and Economics, where part of the work in this paper was carried out. The first author would also like to thank his current advisor, Professor Anthony Man-Cho So, for the continued encouragement and support in pursuing this project.
% }

% This work was partially supported by the National Natural Science Foundation of China [Grants 72394364, 72394360, and 72171141]. 
% % The first author would like to express his sincere gratitude to Professor Simai He for his generous subsidy support when the first author was studying in Shanghai.
% % %, as well as Professor Anthony Man-Cho So, the current supervisor of the first author, for his kind permission to take time on this paper. 
% % Part of the work of the first author was done when he was a research assistant at School of Information Management and Engineering, Shanghai University of Finance and Economics.
% The first author would like to express his sincere gratitude to Professor Simai He for his generous subsidy support when the first author was studying at School of Information Management and Engineering, Shanghai University of Finance and Economics, where part of the work by the first author was conducted.

\appendix

\section{Computing tensor spectral norms}

\begin{lemma}\label{thm:E1}
$\bigl\|\CZ(t)\bigr\|_\sigma=1$ if and only if $-1\le t\le 1$, where
$$
\CZ(t)=\sum_{i=1}^3\be_i\otimes\be_i\otimes\be_i+t\,\be_1\otimes\be_2\otimes\be_3\in\R^{3\times3\times 3}.
$$
\end{lemma}
\begin{proof}
    By the definition of the spectral norm, $\bigl\|\CZ(t)\bigr\|_\sigma\ge\bigl\langle\CZ(t),\be_1\otimes\be_1\otimes\be_1\bigr\rangle=1$ for any $t\in\R$ and $\bigl\|\CZ(t)\bigr\|_\sigma\ge\bigl\langle\CZ(t),\sign(t)\,\be_1\otimes\be_2\otimes\be_3\bigr\rangle=|t|>1$ for any $|t|>1$. It suffices to show that $\bigl\|\CZ(t)\bigr\|_\sigma\le 1$ for any $-1\le t\le 1$. 
    
    To this end, consider the first-order optimality condition of the following problem
    \begin{equation}\label{eq:problem-A.1}
        \bigl\|\CZ(t)\bigr\|_\sigma=\max\mleft\{\sum_{i=1}^3 x_i y_i z_i+t x_1 y_2 z_3:\bx,\by,\bz\in\SI^3\mright\}.
    \end{equation}
    As an optimization problem of a smooth function over the oblique manifold, it follows from~\cite[Proposition~4.5]{boumal2023introduction},~\cite[Proposition~4.6]{boumal2023introduction},~\cite[Exercise~3.67]{boumal2023introduction}, and~\cite[(7.10)]{boumal2023introduction} that any local maximizer $(\bu,\bv,\bw)$ satisfies
    $$
        \begin{dcases}
            v_1 w_1 + t v_2 w_3 =\lambda u_1 \\
            v_2 w_2 = \lambda u_2 \\
            v_3 w_3 = \lambda u_3,
        \end{dcases}~
        \begin{dcases}
            u_1 w_1 = \lambda v_1 \\
            u_2 w_2 + t u_1 w_3 = \lambda v_2 \\
            u_3 w_3 = \lambda v_3,
        \end{dcases}\text{ and }
        \begin{dcases}
            u_1 v_1 = \lambda w_1 \\
            u_2 v_2 = \lambda w_2 \\
            u_3 v_3 + t u_1 v_2 = \lambda w_3,
        \end{dcases}
    $$
    where $\lambda=\sum_{i=1}^3 u_i v_i w_i+t u_1 v_2 w_3$, i.e., the objective function value of $(\bu,\bv,\bw)$ in~\eqref{eq:problem-A.1}.
    
    If $v_1w_1\ne 0$, then by $u_1 w_1 = \lambda v_1$ and $u_1 v_1 = \lambda w_1$, we have $\lambda^2=u_1^2\le1$, i.e., $|\lambda| \le 1$. For the same reason, either $u_2w_2\ne 0$ or $u_3v_3\ne 0$ implies that $|\lambda|\le1$. It remains to consider the case that $v_1w_1=u_2w_2=u_3v_3=0$, under which $|\lambda|=\bigl|\sum_{i=1}^3 u_i v_i w_i+t u_1 v_2 w_3\bigr|=|t u_1 v_2 w_3|\le |t|\le 1$ if $-1\le t\le 1$. Therefore, the objective value of any local maximizer in~\eqref{eq:problem-A.1} is no more than $1$ if $-1\le t\le 1$.    
    % As a result, it suffices to consider the case where the following three conditions simultaneously hold:
    % \begin{itemize}
    %     \item $v_1= 0$ or $w_1= 0$;
    %     \item $u_2= 0$ or $w_2= 0$;
    %     \item $u_3= 0$ or $v_3= 0$;
    % \end{itemize}
    % however, these conditions would further imply that
    % $$
    %     \mleft|\sum_{i=1}^3 x_i y_i z_i+t u_1 v_2 w_3\mright|=|t u_1 v_2 w_3|\le |t|.
    % $$
    % Hence, we conclude that
    % $$
    %     \bigl\|\CZ(t)\bigr\|_\sigma=\max\mleft\{\sum_{i=1}^3 x_i y_i z_i+t x_1 y_2 z_3:\bx,\by,\bz\in\SI^3\mright\}\le 1\text{ for any $-1\le t\le 1$}.
    % $$
    % The proof is complete.
\end{proof}

\begin{lemma}%[Spiral staircase]
\label{thm:E2}
 $\bigl\|\CZ+\CX(t)\bigr\|_\sigma=1$ if and only if $-1\le t\le 1$ and $\bigl\|\CZ+\CY(t)\bigr\|_\sigma=1$ if and only if $t=0$, where
$$
    \CZ=\sum_{i=1}^2\be_i\otimes\be_i\otimes\be_i,~ \CX(t)=t\,\be_1\otimes\be_2\otimes\be_3,\text{ and }\CY(t)=t\,\be_1\otimes\be_1\otimes\be_3, \text{ all in } \R^{2\times 2\times 3}.
$$
\end{lemma}
\begin{proof}
%The proof is given in the space of $\R^{3\times 3\times 3}$ by appending zero entries for any relevant tensor and keeping its original notation since its spectral norm is unchanged. 
By Lemma~\ref{thm:E1}, we have $\|\CU(t)\|_\sigma=1$ for any $-1\le t\le 1$, where 
$$
\CU(t)=\sum_{i=1}^3\be_i\otimes\be_i\otimes\be_i+t\,\be_1\otimes\be_2\otimes\be_3\in\R^{3\times 3\times 3}.
$$
Since $\CZ+\CX(t)$ is a subtensor of $\CU(t)$, it is easy to see (or by~\cite[Theorem 3.1]{li2016bounds}) that $$\bigl\|\CZ+\CX(t)\bigr\|_\sigma\le\bigl\|\CU(t)\bigr\|_\sigma=1 \text{ for any } {-1}\le t\le 1.$$ This, together with the fact that 
\begin{align*}
\bigl\|\CZ+\CX(t)\bigr\|_\sigma&\ge\bigl\langle\CZ+\CX(t),\be_1\otimes\be_1\otimes\be_1\bigr\rangle=1 \text{ for any }t\in\R \text{ and } \\ \bigl\|\CZ+\CX(t)\bigr\|_\sigma&\ge\bigl\langle\CZ+\CX(t),\sign(t)\,\be_1\otimes\be_2\otimes\be_3\bigr\rangle=|t|>1 \text{ for any }|t|>1
\end{align*}
immediately implies the first result. 

To show the second result, we observe that the vector $\bv(t)=(1,0,t)^{\T}$ is a subtensor of $\CZ+\CY(t)$. Therefore, $\bigl\|\CZ+\CY(t)\bigr\|_\sigma\ge\bigl\|\bv(t)\bigr\|_\sigma=\sqrt{1+t^2}>1$ if $t\ne0$. The conclusion then follows from the fact that $\bigl\|\CZ+\CY(0)\bigr\|_\sigma=\|\CZ\|_\sigma=1$.
\end{proof}

\begin{lemma}\label{thm:E3}
$\bigl\|\CX(t)\bigr\|_\sigma=\frac{2|t|}{\sqrt{3}}$ for any $t\in\R$ and $\bigl\|\CZ+\CX(t)\bigr\|_\sigma=1$ if and only if $-1\le t\le \frac{1}{2}$, where
$$
    \CZ=\be_1\otimes\be_1\otimes\be_1\text{ and }\CX(t)=t(\be_1\otimes\be_2\otimes\be_2 + \be_2\otimes\be_1\otimes\be_2 + \be_2\otimes\be_2\otimes\be_1),\text{ both in } \R^{2\times2\times 2}.
$$
\end{lemma}
\begin{proof}
     Since $\CX(t)$ is a symmetric tensor, it follows from Banach's theorem (see e.g.,~\cite{banach1938homogene} and~\cite[Corollary~4.2]{chen2012maximum}) that
     \begin{align*}
        \bigl\|\CX(t)\bigr\|_\sigma&=\max\bigl\{\bigl\langle\CX(t),\bx\otimes\bx\otimes\bx\bigr\rangle: \bx\in\SI^2\bigr\} \\
        &=\max\{3 t x_1 x_2^2:x_1^2+x_2^2=1\}\\
        &=\max\bigl\{3 t x (1-x^2):-1\le x\le 1\bigr\}\\
        &=\frac{2|t|}{\sqrt{3}}.
    \end{align*}     
     $\CZ+\CX(t)$ is also a symmetric tensor. By Banach's theorem again,
     \begin{align*}
        \bigl\|\CZ+\CX(t)\bigr\|_\sigma&=\max\bigl\{\bigl\langle\CZ+\CX(t),\bx\otimes\bx\otimes\bx\bigr\rangle: \bx\in\SI^2\bigr\} \\
       % &=\max\bigl\{|x_1^3+3 t x_1 x_2^2|:x_1^2+x_2^2=1\bigr\}\\
        &=\max\bigl\{x_1^3+3 t x_1 x_2^2:x_1^2+x_2^2=1\bigr\}\\
        &=\max\bigl\{x^3+3 t x (1-x^2):-1\le x\le 1\bigr\}.
    \end{align*}
    It can be calculated that
     $$
        \max\bigl\{x^3+3 t x (1-x^2):-1\le x\le 1\bigr\}=\begin{dcases}
            1 & -1\le t\le {\textstyle\frac{1}{2}}\\
            2{\textstyle\sqrt{\frac{t^3}{-1+3t}}} & \text{otherwise}.
        \end{dcases}
     $$
     This, together with the fact that 
     $2\sqrt{\frac{t^3}{-1+3t}}>1$ for any $t\in(-\infty, -1)\cup(\frac{1}{2},\infty)$, immediately implies that $\bigl\|\CZ+\CX(t)\bigr\|_\sigma=1$ if and only if $-1\le t\le \frac{1}{2}$.
\end{proof}

We remark that $\bigl\|\CX(\frac{1}{\sqrt{3}})\bigr\|_\sigma$ has been calculated in~\cite[Lemma~6.2]{friedland2018nuclear} and can be easily extended to $\bigl\|\CX(t)\bigr\|_\sigma$ for any $t$. $\bigl\|\CZ+\CX(-1)\bigr\|_\sigma$ has also been calculated in~\cite[Lemma~6.1]{friedland2018nuclear}. In fact, the tensor $\CZ+\CX(-1)$ is known as an orthogonal tensor in the literature; see~\cite[Theorem~3.5]{li2018orthogonal} for more details. The result of $\bigl\|\CZ+\CX(t)\bigr\|_\sigma$ in Lemma~\ref{thm:E3} provides a better understanding of the generalization of such a tensor.

% \begin{lemma}\label{lma:222-decomposability}
%     Suppose that $\CX\in\R^{2\times 2\times 2}$ with only $x_{1 2 2}$, $x_{2 1 2}$, and $x_{2 2 1}$ can possibly be nonzero. Then, it holds that $\|\CX\|_*=x_{1 2 2}+x_{2 1 2}+x_{2 2 1}$.
% \end{lemma}
% \begin{proof}
%     Let $\CY\in\R^{2\times 2\times 2}$ with nonzero entries $y_{1 2 2}=y_{2 1 2}=y_{2 2 1}=1$ and $y_{1 1 1}=-1$. The tensor is known to be an orthogonal tensor~\cite[Definition~3.1]{li2018orthogonal} and so $\|\CY\|_\sigma=1$~\cite[Proposition~3.4]{li2018orthogonal}. It follows from definition that
%     $$
%         \langle\CY,\be_1\otimes\be_2\otimes\be_2\rangle=\langle\CY,\be_2\otimes\be_1\otimes\be_2\rangle=\langle\CY,\be_2\otimes\be_2\otimes\be_1\rangle=1=\|\CY\|_\sigma;
%     $$
%     this, together with~\cite[Lemma~4.1]{friedland2018nuclear}, further implies that
%     $$
%         x_{1 2 2}\cdot\be_1\otimes\be_2\otimes\be_2+x_{2 1 2}\cdot\be_2\otimes\be_1\otimes\be_2+x_{2 2 1}\cdot\be_2\otimes\be_2\otimes\be_1
%     $$
%     is a nuclear decomposition~\cite[Section~4]{friedland2018nuclear} of $\CX$, and the desired result thus follows.
% \end{proof}

\section{Computing optimization problems}

\begin{lemma}\label{thm:optimization1} 
$\max\{
x_1y_2z_2+x_2: 
%\begin{gathered}
x_1y_2z_2+x_2\le 1+x_1y_1z_1,\,
\bx,\by,\bz\in\SI^2\cap\R_{+}^2
% \|\bx\|_2=\|\by\|_2=\|\bz\|_2=1\\
% \bx,\by,\bz \ge{\bf 0}
%\end{gathered}
% \mright
\}
=\frac{1+\sqrt{2}}{2}.$
\end{lemma}
\begin{proof}
Let $v$ be the optimal value of the problem. By letting $x_1=x_2=\frac{\sqrt{2}}{2}$, $y_1=z_1=\sqrt{1-\frac{\sqrt{2}}{2}}$, and $y_2=z_2=\sqrt{\frac{\sqrt{2}}{2}}$, we have $v\ge x_1y_2z_2+x_2=\frac{1+\sqrt{2}}{2}$. It remains to show that $v\le\frac{1+\sqrt{2}}{2}$.
%  Let us next consider the vectors
% $$
%     \bx^{\star}:=\begin{pmatrix}
%         {\sqrt{2}}/{2} \\ {\sqrt{2}}/{2}
%     \end{pmatrix}\in\SI^2\cap\R_{+}^2,~
%     \by^{\star}:=\begin{pmatrix}
%         \sqrt{1-{\sqrt{2}}/{2}} \\
%         \sqrt{{\sqrt{2}}/{2}}
%     \end{pmatrix}\in\SI^2\cap\R_{+}^2,
%     \text{ and}~
%     \bz^{\star}:=\begin{pmatrix}
%         \sqrt{1-{\sqrt{2}}/{2}} \\
%         \sqrt{{\sqrt{2}}/{2}}
%     \end{pmatrix}\in\SI^2\cap\R_{+}^2.
% $$
% It is easy to see that
% $$
%     x_1^{\star}y_2^{\star}z_2^{\star}+x_2^{\star}=\frac{1+\sqrt{2}}{2}=1+x_1^{\star}y_1^{\star}z_1^{\star},
% $$
% and hence the bound mentioned earlier is attained by $\bx^{\star},\by^{\star},\bz^{\star}$. The proof is complete.

If $x_1y_2z_2+x_2\le 1$, then $v\le 1\le\frac{1+\sqrt{2}}{2}$. It remains to consider the case that $x_1y_2z_2+x_2>1$. This also implies that $x_1+x_2\ge x_1y_2z_2+x_2>1$ by the obvious fact that $x_i,y_i,z_i\in[0,1]$ for $i\in[2]$. As a result, 
\begin{align*}
    x_1y_2z_2+x_2\le 1+x_1y_1z_1 
    &\implies (x_1 y_2 z_2+x_2-1)^2\le x_1^2 y_1^2 z_1^2=x_1^2 (1-y_2^2) (1-z_2^2)\\
    &\implies 2x_1 y_2 z_2(x_2-1)+x_1^2(y_2^2+z_2^2)\le x_1^2-(x_2-1)^2 \\
    &\implies 2x_1 y_2 z_2(x_2-1)+2x_1^2 y_2 z_2\le x_1^2-(x_2-1)^2 \\
    &\implies x_1 y_2 z_2+x_2\le \frac{x_1^2-(x_2-1)^2}{2(x_1+x_2-1)}+x_2=\frac{x_1 x_2}{x_1+x_2-1}.
\end{align*}
Therefore,
\begin{align*}
v&\le \max\mleft\{
\frac{x_1 x_2}{x_1+x_2-1}: 
%\begin{gathered}
x_1y_2z_2+x_2\le 1+x_1y_1z_1,\,x_1+x_2>1,\,
\bx,\by,\bz\in\SI^2\cap\R_{+}^2
% \|\bx\|_2=\|\by\|_2=\|\bz\|_2=1\\
% \bx,\by,\bz \ge{\bf 0}
%\end{gathered}
\mright\}\\
&\le \max \mleft\{
\frac{x_1 x_2}{x_1+x_2-1}:x_1+x_2>1,\,
\bx\in\SI^2\cap\R_{+}^2
\mright\}=\frac{1+\sqrt{2}}{2},
\end{align*}
under the condition that $x_1y_2z_2+x_2>1$.
\end{proof}

\begin{lemma}\label{thm:optimization2} 
$\max\{x_1y_1z_2+x_1y_2+x_2: x_1y_1z_2+x_1y_2+x_2\le 1+x_1y_1z_1,\, \bx,\by,\bz\in\SI^2\cap\R_{+}^2
\}
=\frac{3}{2}.
$
% We have
% \begin{align*}
% \max\mleft\{x_1y_1z_2+x_1y_2+x_2:
% \begin{gathered}
% x_1y_1z_2+x_1y_2+x_2\le1+x_1y_1z_1\\
% \bx,\by,\bz\in\SI^2\cap\R_{+}^2
% \end{gathered}
% \mright\}
% =\frac{3}{2}.
% \end{align*}
\end{lemma}
\begin{proof}
Let $v$ be the optimal value of the problem. By letting $x_1=\frac{\sqrt{3}}{2}$, $x_2=\frac{1}{2}$, $y_1=\frac{\sqrt{6}}{3}$, $y_2=\frac{\sqrt{3}}{3}$, and $z_1=z_2=\frac{\sqrt{2}}{2}$, we have $v\ge x_1y_1z_2+x_1y_2+x_2=\frac{3}{2}$. It remains to show that $v\le\frac{3}{2}$.

If $x_1y_1z_2+x_1y_2+x_2\le 1$, then $v\le 1\le\frac{3}{2}$. It remains to consider the case that $x_1y_1z_2+x_1y_2+x_2>1$. This, together with the constraint $x_1y_1z_2+x_1y_2+x_2\le 1+x_1y_1z_1$, implies that
\begin{align*}
    (x_1y_1z_2+x_1y_2+x_2-1)^2 \le x_1^2 y_1^2 z_1^2=1-x_2^2-x_1^2 y_2^2 -x_1^2 y_1^2 z_2^2\le 1-\frac{(x_1y_1z_2+x_1y_2+x_2)^2}{3}.
\end{align*}
By solving the above inequality with respect to $x_1y_1z_2+x_1y_2+x_2$, we get $x_1y_1z_2+x_1y_2+x_2\le\frac{3}{2}$, implying that $v\le\frac{3}{2}$.
\end{proof}

\section{Subdifferential of the nuclear norm for fourth-order tensors}

The following example generalizes Example~\ref{ex:yuan3} from $d=3$ to $d=4$.
\begin{example}\label{ex:yuan4}
Let $\CT=\be_1\otimes\be_1\otimes\be_1\otimes\be_1\in\R^{2\times 2\times 2\times 2}$. We have $\spn_k(\CT)=\spn(\be_1)$ for $k\in[4]$. Let $\CZ=\CT\in\Z(\CT)$ since $\langle\CZ,\CT\rangle=1=\|\CT\|_*$ and $\|\CZ\|_\sigma=1$. 

Let $\CX(t)=t\be_1\otimes(\be_1\otimes\be_2\otimes\be_2 + \be_2\otimes\be_1\otimes\be_2 + \be_2\otimes\be_2\otimes\be_1) + t\be_2\otimes(\be_1\otimes\be_1\otimes\be_2 + \be_1\otimes\be_2\otimes\be_1 + \be_2\otimes\be_1\otimes\be_1)\in\bigoplus_{|\I|\ge2,\,\I\subseteq[4]}\TT^\I(\CT)$. It can be verified that $\bigl\|\CZ+\CX(t)\bigr\|_\sigma=1$ if and only if $-\frac{1+\sqrt{2}}{3}\le t\le \frac{1}{3}$, following a similar proof to that of Lemma~\ref{thm:E3}. Therefore, $\CZ+\CX(t)\in\partial\|\CT\|_*$ for any $-\frac{1+\sqrt{2}}{3}\le t\le \frac{1}{3}$. However, it can also be verified that $\bigl\|\CX(t)\bigr\|_\sigma=\frac{3|t|}{2}$, again following a similar proof to that of Lemma~\ref{thm:E3}.
% Let 
% $$
% \CX(t)=t\mleft(\sum_{|\I|=2,\,\I\subseteq[4]}\bigotimes_{k=1}^d
% \begin{dcases}
%     \be_2& k\in\I \\
%     \be_1& k\notin\I
% \end{dcases}
% \mright)
% % (\be_1\otimes\be_2\otimes\be_2 + \be_2\otimes\be_1\otimes\be_2 + \be_2\otimes\be_2\otimes\be_1) 
% \in\bigoplus_{|\I|\ge2,\,\I\subseteq[4]}\TT^\I(\CT).
% $$
% It can be verified that $\bigl\|\CZ+\CX(t)\bigr\|_\sigma=1$ if and only if $-(1+\sqrt{2})/3\le t\le \frac{1}{3}$; this follows similarly to Lemma~\ref{thm:E3}. Hence, we have $\CZ+\CX(t)\in\partial\|\CT\|_*$ for all $-(1+\sqrt{2})/3\le t\le \frac{1}{3}$. However, it can also be verified that $\bigl\|\CX(t)\bigr\|_\sigma={3|t|}/{2}$ for all $t\in\R$; this also follows similarly to Lemma~\ref{thm:E3}.
% in particular, this implies that $\|\CX(-1)\|_\sigma=\frac{2}{\sqrt{3}}>1$. As an aside, the allowed stretches of $\CX(t)$ are different along two opposite directions: Positive $t$ can make at most $\|\CX(\frac{1}{2})\|_\sigma=\frac{1}{\sqrt{3}}$, but negative $t$ can lead to even $\|\CX(-1)\|_\sigma=\frac{2}{\sqrt{3}}$.
\end{example}
%$\underline{\tau}\mleft(\bigoplus_{|\I|\ge2,\,\I\subseteq[4]}\TT^\I\mright)\ge\frac{1}{3}$
%As an aside, it is also possible to 
The following result generalizes~\cite[Lemma~1]{yuan2016tensor} from $d=3$ to $d=4$.
\begin{lemma}\label{lemma:yuan-lemma1-d=4}
    If $\CT\in\R^{n_1\times n_2\times n_3\times n_4}$ is nonzero, then $\|\CZ+\CX\|_\sigma\le1$ for any $\CZ\in\Z(\CT)$ and $\CX\in\bigoplus_{|\I|\ge2,\,\I\subseteq[4]}\TT^\I(\CT)$ with $\|\CX\|_\sigma\le \frac{1}{3}$.    
    % $$
    %     \max\mleft\{\|\CZ+\CX\|_\sigma:\text{$\CZ\in\Z(\CT)$ and $\CX\in\bigoplus_{|\I|\ge2,\,\I\subseteq[4]}\TT^\I(\CT)$ with $\|\CX\|_\sigma\le \frac{1}{3}$}\mright\}= 1.
    % $$
\end{lemma}
\begin{proof}
Let us denote
\begin{align*}
\CX_i&=\proj_{\TT^{\{i,4\}}(\CT)}(\CX) \text{ for }i\in[3], \\
\CX_4&=\proj_{\TT^{\{1,3\}}(\CT)\oplus\TT^{\{1,3,4\}}(\CT)}(\CX), \\
\CX_5&=\proj_{\TT^{\{2,3\}}(\CT)\oplus\TT^{\{2,3,4\}}(\CT)}(\CX), \\
\CX_6&=\proj_{\U^{\{1,2\}}(\CT)}(\CX).
\end{align*}
It is not difficult to see that $\CX=\sum_{i=1}^6\CX_i$ % with each $\CX_i$ being in a different subspace. 
 and $\|\CX_i\|_\sigma\le\|\CX\|_\sigma\le \frac{1}{3}$ for $i\in[6]$ by Lemma~\ref{thm:spec-subspace}.

Let $\bv_k\in\SI^{n_k}$ for $k\in[4]$ such chat 
\begin{align*}
        \|\CZ+\CX\|_\sigma&=\langle\CZ+\CX,\bv_1\otimes\bv_2\otimes\bv_3\otimes\bv_4\rangle \\
       &=\langle \CZ,\bv_1\otimes\bv_2\otimes\bv_3\otimes\bv_4\rangle +\sum_{i=1}^6\langle \CX_i,\bv_1\otimes\bv_2\otimes\bv_3\otimes\bv_4\rangle\\
        &\le a_1 b_1 c_1 d_1+\frac{1}{3}(a_2 b_1 c_1 d_2+a_1 b_2 c_1 d_2+a_1 b_1 c_2 d_2+a_2 b_1 c_2+a_1 b_2 c_2+ a_2 b_2),
\end{align*}
where the inequality is due to that $\|\CZ\|_\sigma=1$ and $\CZ\in\TT(\CT)$ as $\CZ\in\Z(\CT)$, $\|\CX_i\|_\sigma\le \frac{1}{3}$ for $i\in[6]$, and 
$$
\begin{array}{llll}
    a_1=\bigl\|\proj_{\spn_1(\CT)}(\bv_1)\bigr\|_2, &b_1=\bigl\|\proj_{\spn_2(\CT)}(\bv_2)\bigr\|_2, &c_1=\bigl\|\proj_{\spn_3(\CT)}(\bv_3)\bigr\|_2, &d_1=\bigl\|\proj_{\spn_4(\CT)}(\bv_4)\bigr\|_2,\\
a_2=\bigl\|\proj_{\spn_1^\perp(\CT)}(\bv_1)\bigr\|_2, &b_2=\bigl\|\proj_{\spn_2^\perp(\CT)}(\bv_2)\bigr\|_2, &c_2=\bigl\|\proj_{\spn_3^\perp(\CT)}(\bv_3)\bigr\|_2, &d_2=\bigl\|\proj_{\spn_4^\perp(\CT)}(\bv_4)\bigr\|_2.
\end{array}
$$

Finally, by applying a similar proof to that of Lemma~\ref{thm:optimization2}, we obtain
$$
    \max_{\ba,\bb,\bc,\mathbf{d}\in\SI^2\cap\R_{+}^2} \mleft(a_1 b_1 c_1 d_1+\frac{1}{3}(a_2 b_1 c_1 d_2+a_1 b_2 c_1 d_2+a_1 b_1 c_2 d_2+a_2 b_1 c_2+a_1 b_2 c_2+ a_2 b_2)\mright)=1,
$$
implying that $\|\CZ+\CX\|_\sigma\le1$.
\end{proof}

\bibliographystyle{abbrv}
\bibliography{references.bib}{}

\end{document}